\documentclass[10pt,reqno]{amsart}

\usepackage[pass]{geometry}
\newlength\DX
\paperwidth=\dimexpr\paperwidth-\DX\relax
\hoffset=\dimexpr\hoffset-.5\DX\relax
\newlength\DY
\paperheight=\dimexpr\paperheight-\DY\relax
\voffset=\dimexpr\voffset-.5\DY-.5\footskip\relax

\usepackage{mathtools}

\usepackage[T1]{fontenc}

\usepackage{amsmath,amsfonts,amsbsy,amsgen,amscd,mathrsfs,amssymb,amsthm}
\usepackage{enumerate}
\usepackage{bm}

\usepackage{booktabs}

\usepackage{stmaryrd}
\SetSymbolFont{stmry}{bold}{U}{stmry}{m}{n} 

\usepackage[usenames,dvipsnames]{xcolor}
\usepackage[colorlinks=true,citecolor=blue,linkcolor=blue]{hyperref}

\usepackage[font=small,margin=10pt]{caption}
\usepackage{tikz}

\newtheorem{thm}{Theorem}[section]
\newtheorem{lem}[thm]{Lemma}

\newtheorem{prop}[thm]{Proposition}
\newtheorem{cor}[thm]{Corollary}

\theoremstyle{definition}

\newtheorem{defn}[thm]{Definition}

\theoremstyle{remark}

\newtheorem{example}[thm]{Example}

\newtheorem{rem}[thm]{Remark}
\newtheorem*{rem*}{Remark}

\numberwithin{equation}{section} 
\numberwithin{figure}{section}
\numberwithin{table}{section}

\makeatletter

\let\oldtocsection=\tocsection
\let\oldtocsubsection=\tocsubsection
\let\oldtocsubsubsection=\tocsubsubsection
\renewcommand{\tocsection}[2]{\hspace{-1.2em}\oldtocsection{#1}{#2}}
\renewcommand{\tocsubsection}[2]{\hspace{-.2em}\oldtocsubsection{#1}{#2}}
\renewcommand{\tocsubsubsection}[2]{\hspace{0.8em}\oldtocsubsubsection{#1}{#2}}

\DeclareRobustCommand{\gobblefive}[5]{}
\newcommand*{\SkipTocEntry}{\addtocontents{toc}{\gobblefive}}

\makeatother

\newcommand{\ntr}{\mathop{\mathrm{tr}}}
\newcommand{\tr}{\mathop{\mathrm{Tr}}}

\newcommand{\M}{\mathrm{M}}
\newcommand{\EE}{\mathbf{E}}
\newcommand{\Cov}{\mathrm{Cov}}
\newcommand{\PP}{\mathbf{P}}
\newcommand{\spc}{\mathrm{sp}}

\newcommand{\id}{\mathbf{1}}
\newcommand{\dH}{\mathrm{d_H}}
\newcommand{\RR}{\mathrm{R}}

\begin{document}

\title[Universality and matrix concentration]{Universality and sharp 
matrix concentration inequalities}

\author{Tatiana Brailovskaya}
\address{Fine Hall 218, Princeton University, Princeton, NJ 08544, USA}
\email{tatianab@princeton.edu}

\author{Ramon van Handel}
\address{Fine Hall 207, Princeton University, Princeton, NJ 08544, USA}
\email{rvan@math.princeton.edu}

\begin{abstract}
We show that, under mild assumptions, the spectrum of a sum of independent 
random matrices is close to that of the Gaussian random matrix whose 
entries have the same mean and covariance. This nonasymptotic universality 
principle yields sharp matrix concentration inequalities for general 
sums of independent random matrices when combined with the Gaussian theory 
of Bandeira, Boedihardjo, and Van Handel. A key feature of the resulting 
theory is that it is applicable to a broad class of random 
matrix models that may have highly nonhomogeneous and dependent entries, 
which can be far outside the mean-field situation considered in classical 
random matrix theory. We illustrate the theory in
applications to random graphs, 
matrix concentration inequalities for smallest singular values, 
sample covariance matrices, 
strong asymptotic freeness, 
and phase transitions in spiked models.
\end{abstract}



\maketitle

\thispagestyle{empty}
{\small
\setcounter{tocdepth}{2}
\tableofcontents
}

\section{Introduction}
\label{sec:intro}

\subsection{Matrix concentration inequalities}

Let $Z_1,\ldots,Z_n$ be 
independent $d\times d$ random matrices with zero mean, and let
\begin{equation}
\label{eq:intmodel}
	X := \sum_{i=1}^n Z_i.
\end{equation}
Random matrices of this form arise in numerous 
applications. As guiding examples, the reader may keep in mind the 
following very special cases:
\medskip
\begin{enumerate}[$\bullet$]
\itemsep\medskipamount
\item Any random matrix $X$ with centered jointly Gaussian entries may be 
represented in this form by setting $X=\sum_{i=1}^n g_iA_i$ for
suitable deterministic matrices $A_i$, where $g_1,\ldots,g_n$ are i.i.d.\ 
standard Gaussian variables.
\item Any random matrix $X$ with centered independent entries may be 
represented in this form as $X= \sum_{i,j=1}^d \eta_{ij} e_ie_j^*$, where 
$\eta_{ij}$ are independent centered random variables and $e_1,\ldots,e_d$ 
denotes the standard basis of $\mathbb{C}^d$.
\end{enumerate}
\medskip
Many other kinds of summands $Z_i$ arise naturally in a diverse range of 
pure and applied mathematical problems; cf.\ \cite{Tro15} and the
references therein, and the applications that are discussed in 
sections \ref{sec:introapp} and \ref{sec:app} below.

Already in the special cases highlighted above, it is clear that random 
matrices of the form \eqref{eq:intmodel} can possess a nearly arbitrary 
structure: the model allows for essentially any pattern of entry 
variances, dependencies, and distributions. Such general models are 
outside the reach of classical random matrix theory, which is primarily 
concerned with the asymptotic behavior of highly symmetric models such as 
matrices with i.i.d.\ entries or invariant ensembles \cite{AGZ10,Tao12}.

Rather surprisingly, one of the most fruitful ideas that has been 
developed in the present setting is that one can treat the model 
\eqref{eq:intmodel} essentially as though it is a sum of independent 
\emph{scalar} random variables. This approach results in a somewhat crude 
but extremely versatile family of nonasymptotic \emph{matrix concentration 
inequalities}. Two important examples of such inequalities are:\footnote{%
	Here and in the sequel, $\|M\|$ denotes the operator norm
	(i.e., the largest singular value) of a 
	matrix $M$, and $a\lesssim b$ denotes $a\le Cb$ for a universal 
	constant $C$.}
\medskip
\begin{enumerate}[$\bullet$]
\itemsep\medskipamount
\item 
For a self-adjoint random matrix $X$ with centered jointly 
Gaussian entries, the noncommutative Khintchine inequality of
Lust-Piquard and Pisier
\cite[\S 9.8]{Pis03} yields
\begin{equation}
\label{eq:nck}
	\EE\|X\| \lesssim
	\|\EE X^2\|^{\frac{1}{2}}\sqrt{\log d}.
\medskip
\end{equation}
\item 
For a self-adjoint random matrix $X$ of the form 
\eqref{eq:intmodel} with $\|Z_i\|\le R$ a.s., the 
matrix Bernstein inequality of Oliveira and Tropp \cite{Oli10,Tro15} 
yields
\begin{equation}
\label{eq:bernstein}
	\EE\|X\| \lesssim
	\|\EE X^2\|^{\frac{1}{2}} \sqrt{\log d} +  R\log d.
\end{equation}
\end{enumerate}
\medskip
To understand the significance of these inequalities,
note that
$(\EE\|X\|^2)^{\frac{1}{2}} \ge \|\EE X^2\|^{\frac{1}{2}}$ by Jensen's 
inequality. The above bounds can therefore capture the norm of very 
general random matrices up to a logarithmic dimensional factor. The 
dimensional factor proves to be suboptimal, however, even for the simplest 
random matrix models (such as those with i.i.d.\ entries).

The inefficiency of classical matrix concentration inequalities stems from 
the fact that by mimicking the proofs of scalar concentration 
inequalities, these bounds ignore noncommutativity of the summands $Z_i$ 
in \eqref{eq:intmodel}. 
In the setting of Gaussian random matrices, a significant step toward 
addressing this inefficiency was recently made by Bandeira, Boedihardjo, 
and the second author \cite{BBV21}, who developed a new class of 
\emph{sharp} matrix concentration inequalities that capture 
noncommutativity. For example, if $X$ is a self-adjoint random matrix with 
centered jointly Gaussian entries, \cite[Corollary 2.2]{BBV21} yields
\begin{equation}
\label{eq:bbv}
	\EE\|X\| \le
	\|X_{\rm free}\| + C\|\EE X^2\|^{\frac{1}{4}}
	\|\mathrm{Cov}(X)\|^{\frac{1}{4}} (\log d)^{\frac{3}{4}}
\end{equation}
for a universal constant $C$. Here $\mathrm{Cov}(X)$ denotes the 
$d^2\times d^2$ covariance matrix of the entries of $X$, while $X_{\rm 
free}$ is a certain noncommutative model of $X$ that arises from free 
probability theory.  As $\|X_{\rm free}\|\le 2\|\EE X^2\|^{\frac{1}{2}}$,
the inequality \eqref{eq:bbv} shows that the dimensional factor in 
\eqref{eq:nck} can be removed as soon as $\|\mathrm{Cov}(X)\|\ll (\log 
d)^{-3}\|\EE X^2\|$, which is a mild assumption in many applications. 
The theory of \cite{BBV21} yields much more, however: both 
the support and the empirical distribution of the spectrum of $X$ is close to 
that of $X_{\rm free}$, and similar results hold for polynomials of 
such matrices. Such results open the door to developing a nonasymptotic 
random matrix theory for nearly arbitrarily structured random matrices.

In view of these developments, it is of considerable interest to extend 
the Gaussian theory of \cite{BBV21} to the much more general setting 
\eqref{eq:intmodel} of sums of independent random matrices. For classical 
matrix concentration inequalities, this extension has been achieved in two 
distinct ways: one may either derive both \eqref{eq:nck} and 
\eqref{eq:bernstein} by a common method of proof \cite{Oli10,Tro15}, or 
deduce \eqref{eq:bernstein} from \eqref{eq:nck} by a symmetrization 
argument as in \cite{Rud99,Tro16}. Unfortunately, neither of these 
approaches appears to give rise to a satisfactory extension of the theory 
of \cite{BBV21}. The methods of \cite{BBV21} rely heavily on Gaussian 
analysis, and it is unclear how to adapt them to non-Gaussian situations. 
On the other hand, sharp inequalities are fundamentally inaccessible by 
symmetrization, as is explained in \cite[\S 8.2.2]{BBV21}.

\subsection{Universality}

In this paper, we take an entirely different viewpoint on such 
problems. To motivate the form of our main results, let us note
that if the last term in the inequality \eqref{eq:bernstein} is 
negligible, then \eqref{eq:bernstein} has exactly the same form as the 
Gaussian inequality \eqref{eq:nck}. The main theme of this paper is 
that this phenomenon has nothing to do with matrix concentration 
inequalities themselves, but is rather a consequence of a general 
\emph{universality principle}:
\begin{quote}
\emph{If $\max_{1\le i\le n}\|Z_i\| \ll\|\EE X^2\|^{\frac{1}{2}} (\log 
d)^{-\beta}$ \emph{(}for an appropriate $\beta>0$\emph{)}, then the 
spectrum of a self-adjoint
random matrix $X=\sum_{i=1}^n Z_i$ as in
\eqref{eq:intmodel} nearly coincides with that
of the Gaussian random matrix $G$ whose entries have the same mean and 
covariance as $X$.}
\end{quote}
This principle directly reduces the study of the spectrum of sums of 
independent random matrices to that of Gaussian matrices, regardless of 
what theory is applied to the Gaussian matrices. In particular, it 
simultaneously explains the phenomenon behind 
\eqref{eq:nck}---\eqref{eq:bernstein}, and enables us to fully extend the 
sharp matrix concentration theory of \cite{BBV21} to the model 
\eqref{eq:intmodel}.

The universality principle was stated above in an informal manner. A 
detailed formulation of our results will be given in section 
\ref{sec:main} below. In particular, we will obtain nonasymptotic 
inequalities that establish closeness both of the spectral distributions 
of $X$ and $G$, and of the spectra themselves in Hausdorff distance. 
(These results apply in a more general setting than \eqref{eq:intmodel}, 
where the random matrices may have an arbitrary mean.) We further 
formulate resulting sharp matrix concentration inequalities that 
arise from the theory of \cite{BBV21}.

Universality phenomena have been widely investigated in classical random 
matrix theory. As in many previous works on this topic, the starting point 
for our analysis is the cumulant expansion of Barbour \cite{Bar86} and 
Lytova and Pastur \cite{LP09}, which has been primarily applied to 
classical random matrix models with independent entries. A rather 
complicated extension of the cumulant expansion to dependent models 
appears in \cite{EKS19}, where it is used to study random matrices whose 
entries exhibit decay of correlations. A straightforward extension of 
Barbour's method to the dependent setting will be formulated in section 
\ref{sec:cumulant}; such an extension does not in itself require any new 
idea as compared to \cite{Bar86,LP09}.

The core contribution of this paper lies in the mechanism that gives rise 
to universality. To the best of our knowledge, prior universality results 
are essentially limited to the ``classical random matrix regime'' where 
the entry variances are of order $d^{-\frac{1}{2}}$ and the entries are 
independent or nearly independent (in the sense that they exhibit decay of 
correlations). In other words, these results rely on restrictive 
mean-field assumptions. In contrast, the independent sum models 
\eqref{eq:intmodel} of the present paper can lie far outside the 
mean-field regime: they can be highly nonhomogeneous, sparse, and exhibit 
strong dependence among the entries, and are not assumed to possess any 
special structure or symmetries. The properties of these 
models therefore cannot be explained by previous universality results that 
rely heavily on the special structure of the underlying models.

The central idea of this paper is that universality arises in these models 
in a different manner through an operator-theoretic mechanism: a key 
ingredient of our approach are high-order trace inequalities (section 
\ref{sec:tools}) that enable us to control the contributions of the terms 
in the cumulant expansion without imposing any correlation decay or 
mean-field assumptions. This operator-theoretic viewpoint on universality, 
together with a number of other new tools (such as nonstandard 
concentration inequalities for spectral statistics), provides access to 
many applications that are not captured by classical random matrix models.

\begin{rem}
Much of the literature on universality of classical random 
matrix models has focused on establishing universality at or near the 
scale of the fluctuations of the eigenvalues; see, e.g., \cite{EKS19} and 
the references therein. It should be emphasized that the results of this 
paper do not provide any information at the scale of the fluctuations, but 
rather only at scales at which the spectral statistics exhibit 
concentration. At the level of generality considered in this paper, the 
scale of the fluctuations is strongly model-dependent (see, e.g., 
\cite{Sod10,CDF09}), so that it is unclear how a meaningful result for 
generally structured models could be formulated. Even 
a plausible conjecture in this direction would be of considerable 
interest.
\end{rem}

\subsection{Applications}
\label{sec:introapp}

To illustrate the main results of this paper, we will develop several 
applications that we briefly describe here (see section \ref{sec:app} for 
detailed statements). Beyond their independent interest, we emphasize that 
the completely general universality phenomenon described by our main 
results is the common mechanism underlying all these rather diverse 
applications.

\subsubsection*{Random graphs and expanders}

The expansion properties of random regular graphs have been extensively 
studied for graphs of bounded degree $k$.  In particular, such graphs are 
nearly Ramanujan, i.e., they have the smallest possible (by \cite{Nil91}) 
second eigenvalue $\lambda_2= (1+o(1))2\sqrt{k-1}$ to 
leading order \cite{Fri08}. While the strong expansion properties of such 
graphs are expected to persist when the degree is allowed to diverge, this 
situation remains much more poorly understood; see, e.g., the survey 
\cite{Vu08}. The universality principles of this paper enable us to 
address this question both in classical and in new situations:
\medskip
\begin{enumerate}[$\bullet$]
\itemsep\medskipamount
\item The permutation model of random regular graphs with $n$ 
vertices of degree $k$ is nearly Ramanujan when $k\gg (\log n)^4$, 
addressing a well known question \cite[\S 1.4]{BKY17}. To date, the best 
known bound in this setting was $\lambda_2=O(\sqrt{k})$ 
\cite{FKS89,DJPP13,CGJ18}.
\item A classical result of Alon and Roichman \cite{AR94} states that if
$\Gamma$ is any finite group and $k\gg\log|\Gamma|$ generators are chosen 
uniformly at random, the resulting Cayley graph is an expander. This 
result cannot be improved for abelian groups. Here we show that under mild 
assumptions that hold, e.g., for all nonabelian finite simple groups, the 
Cayley graph defined by choosing $k\gg(\log|\Gamma|)^4$ random generators 
is nearly Ramanujan. This appears to be the first result of its kind.
\item A fundamental result of Bordenave and Collins \cite{BC19} states
that for any fixed base graph $H$, the new eigenvalues of 
its random $n$-lift are bounded as $n\to\infty$ by the spectral 
radius of the universal cover of $H$. Here we show that this conclusion 
remains valid for any sequence of base graphs $H_n$ whose maximal degrees 
grow at least polylogarithmically in the number of vertices of their 
random lifts. Moreover, in this setting we uncover a new phenomenon:
when the base graphs are simple, random $2$-lifts already achieve the 
optimal bound.
\end{enumerate}

\subsubsection*{Matrix concentration inequalities for smallest singular values}

By their nature, classical matrix concentration inequalities can only 
control the largest singular value of nonhomogeneous random matrices. In 
contrast, the universality principles of this paper apply not only to the 
largest singular value but also to the entire spectrum. By combining our 
results with the Gaussian theory of \cite{BBV21}, we are therefore able to 
obtain sharp matrix concentration inequalities for the smallest singular 
value of random matrices of the form \eqref{eq:intmodel} that may be 
viewed as a nonasymptotic, nonhomogeneous form of the classical Bai-Yin 
law \cite{BY93}. Let us emphasize that even if one is interested in 
suboptimal bounds on the smallest singular value, such information is 
fundamentally inacceassible by the methods used to prove classical matrix 
concentration inequalities for general models of the form 
\eqref{eq:intmodel}.

A direct application yields bounds for the smallest singular value of 
sparse nonhomogeneous bipartite Erd\H{o}s-R\'enyi graphs that are sharp to 
leading order. To date, the best known bounds \cite{DZ22} were suboptimal 
for nonhomogeneous graphs.

\subsubsection*{Sample covariance matrices}

Let $Y_1,\ldots,Y_n$ be independent, centered random vectors in 
$\mathbb{R}^d$. The $d\times d$ random matrix defined by $S := 
\sum_{i=1}^n Y_iY_i^*$ is called the (nonhomogeneous) sample 
covariance matrix. Equivalently, $S=YY^*$, where 
$Y:=\sum_{i=1}^n Y_i e_i^*$ is the $d\times n$ matrix whose columns are 
$Y_1,\ldots,Y_n$.

A central problem in this setting is to control the deviation of the 
sample covariance matrix from its mean $\|S-\EE S\|$. A curious feature of 
this problem is that we may express $S$ in terms of a model of the form 
\eqref{eq:intmodel} in two different ways: we may either consider $S$ 
itself as a model of the form \eqref{eq:intmodel}, or we may consider $Y$ 
as a model of the form \eqref{eq:intmodel}. These two representations give 
rise to distinct universality principles: roughly speaking, applying our 
universality principles to $S$ is efficient when $n$ is sufficiently large 
compared to $d$, while applying universality to $Y$ is efficient when $d$ 
is sufficiently large compared to $n$.

We will illustrate this phenomenon in the setting of nonhomogeneous 
Gaussian sample covariance matrices with arbitrary covariance matrices of 
$Y_1,\ldots,Y_n$, for which we obtain nonasymptotic bounds on $\|S-\EE 
S\|$ that are sharp for a wide range of parameters. No sharp bounds appear 
to be known in the literature at this level of generality. We also discuss 
non-Gaussian sample covariance matrices, which will be developed further 
in forthcoming work \cite{PvH23}.

\subsubsection*{Strong asymptotic freeness}

A celebrated result of Voiculescu \cite{Voi91} states that the traces of 
polynomials of independent $N\times N$ Wigner matrices converge as 
$N\to\infty$ to the traces of polynomials of certain limiting objects that 
arise in free probability theory. In an important breakthrough, Haagerup 
and Thorbj{\o}rnsen \cite{HT05} showed that this convergence holds not 
only for the trace but also for the norm:
\begin{equation}
\label{eq:introsaf}
	\lim_{N\to\infty}\|p(X_1^N,\ldots,X_m^N)\|=
	\|p(s_1,\ldots,s_m)\|\quad\mbox{a.s.}
\end{equation}
for every noncommutative polynomial $p$, where $X_1^N,\ldots,X_m^N$ are 
independent $N\times N$ complex Gaussian Wigner matrices and 
$s_1,\ldots,s_m$ is a free semicircular family. This property, called 
strong asymptotic freeness, is of fundamental importance both to random 
matrices and in the theory of operator algebras.

Whether \eqref{eq:introsaf} holds for more general models of random 
matrices $X_i^N$ is far from clear from the original rather delicate 
proofs. Previously, the state-of-the-art \cite{And13} was that 
\eqref{eq:introsaf} holds for matrices with i.i.d.\ centered entries with 
unit variance and bounded fourth moment. Very recently, however, the sharp 
matrix concentration theory of \cite{BBV21} made it possible to establish 
\eqref{eq:introsaf} for an extremely general class of Gaussian random 
matrices, showing that this phenomenon is much more ubiquitous than was 
previously understood. Our universality principles extend this conclusion 
even further to general non-Gaussian random matrices of the form 
\eqref{eq:intmodel} under mild assumptions that allow for 
significant sparsity and dependence.

\subsubsection*{Phase transitions in spiked models}

The behavior of low-rank perturbations of random matrices (so-called 
``spiked'' models) has attracted much attention in pure and applied random 
matrix theory since the work of Baik, Ben Arous and P\'ech\'e 
\cite{BBP05}. The characteristic feature of such models is that they 
exhibit a phase transition depending on the size of the perturbation: 
there is a threshold above which one or more isolated eigenvalues detach 
from the bulk of the spectrum. Most of the literature on this topic is 
concerned with Wigner matrices or with unitarily invariant models; see, 
e.g., the survey \cite{CD17}. The universality principles of this paper 
enable us to investigate such phenomena in much more general situations, 
including models that exhibit significant sparsity and dependence. Little 
appears to be known in this direction: previous work on a special type of 
low-rank perturbations of sparse random matrices appeared only very 
recently in \cite{LM22}.

As our primary aim here is to illustrate the main results of this paper, 
we will focus our attention on sparse and dependent models whose behavior 
can be reduced by universality to the classical spiked Wigner model. In 
this setting, we will show how our universality principles enable us to 
capture the number and locations of the outlier eigenvalues, as well as 
the overlaps of the associated eigenvectors with those of the low-rank 
perturbation. However, much more general situations become amenable to 
analysis in combination with the Gaussian theory of \cite{BBV21}, which 
makes it possible to investigate analogous phase transition phenomena in 
nonhomogeneous models. The computations involved in the nonhomogenous 
setting are unrelated to universality, and are treated in
detail in \cite{BCSV23}.

\subsection{Organization of this paper}

The remainder of this paper is organized as follows. In section 
\ref{sec:main}, we formulate the main results of this paper. Section 
\ref{sec:app} is devoted to a detailed formulation of the applications 
described above.

In section \ref{sec:cumulant}, we provide a brief self-contained treatment 
of the multivariate cumulant expansion. Section \ref{sec:tools} develops 
some key tools that are used in the proofs of our main results: high-order 
trace inequalities that provide the main mechanism for controlling the 
terms in the cumulant expansion, and certain nonstandard concentration of 
measure inequalities. The following three sections are devoted to the 
proofs of our main results. Section \ref{sec:univstat} proves the 
universality principles for spectral statistics, while section 
\ref{sec:univspec} proves the universality principle for the support of 
the spectrum. Section \ref{sec:trunc} is devoted to a truncation argument 
that extends our main results to models that satisfy minimal moment 
assumptions. Finally, section \ref{sec:applproofs} is devoted to the 
proofs of the various applications discussed in section \ref{sec:app}.

\subsection{Notation}

The following notation will be used throughout the paper. We write 
$[n]:=\{1,\ldots,n\}$ for $n\in\mathbb{N}$. The algebra of $d\times d$ 
matrices with values in a *-algebra $\mathcal{A}$ is denoted as 
$\M_d(\mathcal{A})$, and its subspace of self-adjoint matrices is denoted 
as $\M_d(\mathcal{A})_{\rm sa}$. For a matrix or operator $X$, we denote 
by $\|X\|$ its operator norm, by $\spc(X)$ its spectrum, and by
$|X|:=(X^*X)^{\frac{1}{2}}$. 
For self-adjoint $X,Y$, 
we denote by $X\le Y$ the positive semidefinite order.
The identity matrix or operator is denoted as 
$\id$. For $M\in\M_d(\mathbb{C})$, we denote by $\tr M:=\sum_{i=1}^d 
M_{ii}$ the unnormalized trace and by $\ntr M := \frac{1}{d}\tr M$ the 
normalized trace. We denote by $W^{k,1}(\mathbb{R})$ the Sobolev 
space of $f:\mathbb{R}\to\mathbb{C}$ so that
$\|f\|_{W^{k,1}(\mathbb{R})} :=
\sum_{i=0}^k \int_{-\infty}^\infty \big|\frac{d^i}{dx^i}f(x)\big|\,dx
<\infty$.
Finally, we use the convention that when a functional is followed by 
square brackets, it is applied before any other operations; for example,
$\mathbf{E}[X]^\alpha := (\mathbf{E}X)^\alpha$ and
$\ntr[M]^\alpha := (\ntr M)^\alpha$.

\section{Main results}
\label{sec:main}

\subsection{Random matrix models and matrix parameters}
\label{sec:model}

\subsubsection{The general model}

The basic random matrix model of this paper is defined as follows.
Fix $d\ge 2$ and $n\in\mathbb{N}$, let $Z_0\in\M_d(\mathbb{C})_{\rm sa}$ 
be any deterministic $d\times d$ self-adjoint matrix, and let 
$Z_1,\ldots,Z_n$ be any independent $d\times d$ self-adjoint random 
matrices with zero mean $\EE[Z_i]=0$ and complex-valued entries. We define
\begin{equation}
\label{eq:model}
	X:=Z_0+\sum_{i=1}^n Z_i.
\end{equation}
Note that this model is slightly more general than the model 
\eqref{eq:intmodel} discussed in the introduction, in 
that we allow for an arbitrary mean.

\begin{rem}
\label{rem:nonsa}
The assumption that $X$ is self-adjoint is made primarily for notational
convenience. Our main results extend directly to non-self-adjoint 
matrices as follows.
For any matrix $M\in\M_d(\mathbb{C})$, define its dilation
$\breve{M}\in\M_{2d}(\mathbb{C})_{\rm sa}$ as
$$
	\breve{M} := \begin{bmatrix} 0 & M \\ M^* & 0 \end{bmatrix}.
$$
If we denote by $M=U|M|$ the polar decomposition of $M$, it follows that
$$
	\breve{M} = V \begin{bmatrix} -|M| & 0 \\ 0 & |M| 
	\end{bmatrix}V^*\qquad\mbox{with}\qquad
	V:=\frac{1}{\sqrt{2}}\begin{bmatrix} U & U \\ -\id & \id
	\end{bmatrix},
$$
where $|M|:=(M^*M)^{\frac{1}{2}}$. As $V$ is unitary,
this shows that the eigenvalues of $\breve{M}$ coincide precisely 
(including multiplicities) with $\{\pm\sigma_i:i\in[d]\}$, where 
$\sigma_1,\ldots,\sigma_d$ are the singular values of $M$. 
Consequently, by applying our results to $\breve X$, we can 
immediately extend their conclusions on the eigenvalues of 
self-adjoint random matrices to the singular values of 
non-self-adjoint random matrices.

Singular values of rectangular random matrices are readily reduced to
those of square matrices by adding additional zero rows or columns.
On the other hand, we emphasize that the results of this paper do not
provide bounds on the complex eigenvalues of non-self-adjoint matrices.
For further comments on the non-self-adjoint case, see \cite[Remark 
2.6]{BBV21} and Corollary \ref{cor:userfriendly}.
\end{rem}

Associated with the random matrix $X$ are two models that capture its 
structure in an idealized manner. We introduce these models presently.

\subsubsection{The Gaussian model}

Throughout this paper, we denote by $G$ the Gaussian model that has the 
same mean and covariance structure as $X$. More precisely, denote by
$\Cov(X)$ the $d^2\times d^2$ covariance matrix
of the entries of $X$, that is,
\begin{equation}
\label{eq:covdefn}
	\Cov(X)_{ij,kl} := \EE[(X-\EE X)_{ij} \overline{(X-\EE X)_{kl}}].
\end{equation}
We define $G$ to be the $d\times d$ self-adjoint random matrix such that:
\begin{enumerate}[1.]
\itemsep\abovedisplayskip
\item
$\{\mathrm{Re}\,G_{ij},\mathrm{Im}\,G_{ij}:i,j\in[d]\}$ are jointly
Gaussian;
\item $\EE[G] =\EE[X]$ and $\Cov(G)=\Cov(X)$.
\end{enumerate}
Note that as $G$ is a self-adjoint matrix $G_{lk}=\overline{G_{kl}}$, the 
covariance matrix of the real-valued Gaussian vector 
$\{\mathrm{Re}\,G_{ij},\mathrm{Im}\,G_{ij}:i,j\in[d]\}$ is fully specified 
by $\Cov(G)$. Thus the above properties uniquely define the distribution 
of $G$.

\subsubsection{The noncommutative model}
\label{sec:noncmodel}

We now introduce a noncommutative model $X_{\rm free}$ that has the 
same mean and covariance structure as $X$. To this end, we must recall 
some basic notions from free probability theory; we refer to 
\cite{NS06} for precise definitions and a comprehensive treatment.

Fix a $C^*$-probability space $(\mathcal{A},\tau)$, that is, a 
unital $C^*$-algebra $\mathcal{A}$ endowed with a faithful trace $\tau$.
The following may be viewed as a noncommutative analogue of 
jointly Gaussian variables with mean $\mu$ and covariance $C$, cf.\
\cite[p.\ 128]{NS06}.

\begin{defn}
\label{defn:semi}
A family of self-adjoint elements $s_1,\ldots,s_m\in\mathcal{A}$ is said 
to be a \emph{semicircular family with mean $\mu$ and covariance $C$}
if 
$$
	\tau(s_k)=\mu_k,\qquad
	\tau((s_{k_1}-\mu_{k_1}\id)\cdots (s_{k_p}-\mu_{k_p}\id)) 
	=
	\sum_{\pi\in\mathrm{NC}_2([p])}
	\prod_{\{i,j\}\in\pi} C_{k_ik_j}
$$
for all $p\ge 1$ and $k,k_1,\ldots,k_p\in[m]$, where $\mathrm{NC}_2([p])$ 
denotes the collection of noncrossing pair partitions of $[p]$.
\end{defn}

A $d\times d$ matrix $Y\in\M_d(\mathcal{A})$ with 
$\mathcal{A}$-valued entries 
is naturally identified with an element of the $C^*$-algebra 
$\M_d(\mathbb{C})\otimes\mathcal{A}$, which we endow with the normalized 
trace ${\ntr}\otimes\tau$, cf.\ \cite[Chapter 9]{MS17}.
Define the entry covariance matrix $\Cov(Y)$ as
$$
	\Cov(Y)_{ij,kl}:=\tau((Y_{ij}-\tau(Y_{ij})\id)(Y_{kl}-\tau(Y_{kl})\id)^*).
$$
With these definitions in place, we can now define $X_{\rm 
free}\in\M_d(\mathcal{A})_{\rm sa}$ as follows:
\begin{enumerate}[1.]
\itemsep\abovedisplayskip
\item $\{\mathrm{Re}\,(X_{\rm 
free})_{ij},\mathrm{Im}\,(X_{\rm free})_{ij}:i,j\in[d]\}$ is a 
semicircular family;
\item $(\mathrm{id}\otimes\tau)(X_{\rm free}) =\EE[X]$ and 
$\Cov(X_{\rm free})=\Cov(X)$.
\end{enumerate}
Here we write $\mathrm{Re}\,a:= \frac{1}{2}(a+a^*)$ and
$\mathrm{Im}\,a := \frac{1}{2i}(a-a^*)$ for $a\in\mathcal{A}$.

\begin{rem}
\label{rem:coeffrep}
As jointly Gaussian variables can always be written as linear combinations 
of independent standard Gaussian variables, $G$ may be expressed as
$$
	G = Z_0 + \sum_{i=1}^N A_i g_i
$$
for some deterministic matrices $A_1,\ldots,A_N\in\M_d(\mathbb{C})_{\rm 
sa}$ and
i.i.d.\ (real) standard Gaussians $g_1,\ldots,g_N$ (note that this 
representation is not unique). Given any such a representation, it is 
readily verified that one may express $X_{\rm free}$ as
$$
	X_{\rm free} = Z_0\otimes\id + \sum_{i=1}^N A_i\otimes s_i,
$$
where $s_1,\ldots,s_N$ is a \emph{free} semicircular family, that is,
with zero mean and identity covariance matrix.
Thus the present definition of $X_{\rm free}$ agrees with the one 
in \cite{BBV21}.
\end{rem}

\subsubsection{Matrix parameters}
\label{sec:matparm}

Let $X$ be a self-adjoint random matrix as in \eqref{eq:model}. The 
following basic parameters will appear in our main results:
\begin{align}
	\sigma(X) & := \big\|\EE[(X-\EE X)^2]\big\|^{\frac{1}{2}},
	\phantom{\Big\|}
\label{eq:parfirst}
\\
	\sigma_*(X) &:= 
	\sup_{\|v\|=\|w\|=1} \EE\big[
	|\langle v,(X-\EE X)w\rangle|^2\big]^{\frac{1}{2}},
	\phantom{\Big\|}
\\
	v(X) &:= \|\Cov(X)\|^{\frac{1}{2}}, 
	\phantom{\Big\|}
\\
	R(X) &:= \Big\|\max_{1\le i\le n}\|Z_i\|\Big\|_\infty,
\end{align}
where $\|Y\|_\infty$ denotes the essential supremum of the random
variable $|Y|$.

The significance of these parameters may be summarized as follows. The 
parameter $\sigma(X)$ roughly captures the spread of the spectrum of 
$X-\EE X$, as was explained in the introduction. The parameter 
$\sigma_*(X)$ controls the fluctuations of the spectral statistics of
$X$ and $G$, see, e.g., section \ref{sec:conc} below. The parameter $v(X)$ 
quantifies the degree to which the spectral properties of $G$ are captured 
by those of $X_{\rm free}$: this is the main outcome of the theory of 
\cite{BBV21}. Finally, the universality principles of this paper will show 
that the parameter $R(X)$ quantifies the degree to which the spectral 
properties of $X$ are captured by those of $G$.

The parameter $R(X)$ is meaningful only when the random matrices $Z_i$ are 
uniformly bounded. We will also prove versions of our main results that 
apply to unbounded summands under minimal moment assumptions. The 
formulation of these results requires the following modified matrix 
parameters:
\begin{align}
	\sigma_q(X) & := 
	\big(\ntr \EE[(X-\EE X)^2]^{\frac{q}{2}}\big)^{\frac{1}{q}},
	\phantom{\Big\|}
\\
	R_q(X) &:= \big({\textstyle\sum_{i=1}^n \EE[\ntr |Z_i|^q]}
	\big)^{\frac{1}{q}},
	\phantom{\Big\|}
\\
	\bar R(X) &:= \EE\Big[\max_{1\le i\le n}\|Z_i\|^2
	\Big]^{\frac{1}{2}}
\label{eq:parlast}
\end{align}
for $q<\infty$, and $\sigma_\infty(X):=\sigma(X)$, $R_\infty(X):=R(X)$.

\smallskip
\begin{rem}
Let us emphasize the following basic facts.
\medskip
\begin{enumerate}[$\bullet$]
\itemsep\medskipamount
\item
All these parameters depend only on $X-\EE X$, i.e., they do 
not depend on $Z_0$.
\item $\sigma(X),\sigma_q(X),\sigma_*(X),v(X)$ only depend 
on the covariance of the entries of $X$, and therefore capture the 
universal behavior that is shared between $X$, $G$, and $X_{\rm free}$.
\item
In contrast, $R(X),R_q(X),\bar R(X)$ are specific to the non-Gaussian model.
While denote these as parameters of $X$ for 
notational simplicity,
we emphasize that these parameters depend on how $X$ is 
represented as a sum of $Z_i$ as in \eqref{eq:model}.
\item Recall the basic inequalities $\sigma_*(X)\le\sigma(X)$ and
$\sigma_*(X)\le v(X)$ \cite[\S 2.1]{BBV21}.
\end{enumerate}
\end{rem}

\begin{table}
\centering
\begin{small}
\begin{tabular}{@{}lcccc@{}}
\toprule
Model & $\sigma(X)$ & $v(X)$ & $R(X)$ & $\bar R(X)$ \\
\midrule
(a) Erd\H{o}s-R\'enyi graph $\mathrm{G}(d,q)$ 
(section \ref{sec:indent}) & $\sqrt{dq}$ &
$\sqrt{q}$ & $1$ & $O(1)$ \\
(b) Random $k$-regular graph (section \ref{sec:rregdet}) &
$\sqrt{k}$ & $\sqrt{\frac{k}{d}}$ & $1$ & 
$\phantom{\big(^{\frac{2}{p}}}1\phantom{\big)^{\frac{2}{p}}}$ \\
(c) Wigner matrix: $p$ moments (section \ref{sec:indent})
 & $\sqrt{d}$ & $1$ & $\infty$ &
$O\big(d^{\frac{2}{p}}\big)$ \\
(d) Band matrix: width $k$, $p$ moments 
(section \ref{sec:appsaf}) &
$\sqrt{k}$ & $1$ & $\infty$  & $O\big((kd)^{\frac{1}{p}}\big)$ \\
\bottomrule 
\end{tabular}
\end{small}
\caption{Order of magnitude of the matrix parameters for 
some classical models: (a) adjacency matrix of
Erd\H{o}s-R\'enyi $\mathrm{G}(d,q)$ graph; (b) adjacency matrix of random 
$k$-regular graph with $d$ vertices; (c)
Wigner matrix $X$ of dimension $d$ with
$\mathbf{E}[X_{ij}]=0$, $\mathbf{E}[X_{ij}^2]=1$, $\|X_{ij}\|_p=O(1)$;
(d) Band matrix $\bar X$ defined by $\bar X_{ij}=X_{ij} 1_{|i-j|\le 
\frac{k-1}{2}}$.
\label{tab:extable}}
\end{table}
\begin{rem}
In most applications of our theory
$v(X),\bar R(X)\ll\sigma(X)$ (up to a logarithmic factor in 
dimension): in this case, our results will show that the spectral edges of 
$X$ agree with those of $X_{\rm free}$ to leading order, and that the 
spectral distributions of $X$ and $X_{\rm free}$ agree at a mesoscopic 
scale.

To help the reader gain some insight into these parameters, we list their 
order of magnitude for some classical random matrix models in Table 
\ref{tab:extable}. In all these models, the condition 
$v(X),\bar R(X)\ll\sigma(X)$ is essentially optimal for 
universality. For random graphs, this requires that the (average) degree 
diverges; this is necessary, as the spectral distribution of both 
Erd\H{o}s-R\'enyi and random regular graphs of constant (average) degree 
does not match the semicircle distribution of the associated Gaussian 
model. For Wigner matrices whose entries have $p$ bounded moments, we need 
$p>4$; this is nearly optimal for the spectral edges, as it is well known 
that the spectrum has outliers when $p<4$. The analogous condition for 
random band matrices is also nearly optimal (cf.\ Corollary 
\ref{cor:sparsesaf}).

For these classical models, much more precise results have been achieved 
at the scale of the fluctuations of the eigenvalues using problem-specific 
methods. Our results do not provide any information at this scale, but are 
instead able to establish universality for the leading order behavior of 
the bulk and edges of the spectrum in far more general situations; see 
section \ref{sec:app} for a diverse range of applications.
\end{rem}

\subsection{Universality}
\label{sec:univmain}

We now provide precise formulations of the universality principle. We 
prove several results that capture different aspects of the spectrum.

\subsubsection{Universality of the spectrum}
\label{sec:specmain}

In this section, we formulate the universality principle for the spectrum 
$\spc(X)$ of $X$. Recall that the Hausdorff distance between two subsets 
$A,B\subseteq\mathbb{R}$ of the real line is defined as 
$$
	\dH(A,B) :=
	\inf\{\varepsilon>0:
	A\subseteq B+[-\varepsilon,\varepsilon]
	\mbox{ and }
	B\subseteq A+[-\varepsilon,\varepsilon]\}.
$$
Our main result is the following.

\begin{thm}[Spectrum universality]
\label{thm:specuniv}
For any $t\ge 0$, we have
$$
	\mathbf{P}\big[\dH(\spc(X),\spc(G)) >
	C\varepsilon(t)\big] \le de^{-t},
$$
where $C$ is a universal constant and
$$
	\varepsilon(t) =
	\sigma_*(X)\, t^{\frac{1}{2}} +
	R(X)^{\frac{1}{3}}\sigma(X)^{\frac{2}{3}} t^{\frac{2}{3}} +
	R(X)\,t.
$$
Moreover,
$$
	\EE\big[\dH(\spc(X),\spc(G))\big] \lesssim
	\sigma_*(X)\, (\log d)^{\frac{1}{2}} +
	R(X)^{\frac{1}{3}}\sigma(X)^{\frac{2}{3}} (\log d)^{\frac{2}{3}} +
	R(X)\log d.
$$
\end{thm}

Note that while we defined the distributions of $X$ and $G$ in section 
\ref{sec:model}, we did not specify their joint distribution. However,
the conclusion of Theorem \ref{thm:specuniv} is valid regardless of how
$X$ and $G$ are defined on the same probability space due to the strong 
concentration properties of random matrices.

Theorem \ref{thm:specuniv} readily yields a universality principle for the 
spectral edge. In the following, we denote by $\lambda_{\rm max}(X):=
\sup\spc(X)$ the upper edge of the spectrum. (Inequalities for the
lower edge follow readily as $\inf\spc(X)= -\lambda_{\rm max}(-X)$.)

\begin{cor}[Edge universality]
\label{cor:normuniv}
For any $t\ge 0$, we have
$$
	\mathbf{P}\big[
	|\lambda_{\rm max}(X)-\lambda_{\rm max}(G)| >
	C\varepsilon(t)\big] \le de^{-t},
$$
as well as
$$
	\mathbf{P}\big[
	|\lambda_{\rm max}(X)-\EE\lambda_{\rm max}(G)| >
	C\varepsilon(t)\big] \le de^{-t},
$$
where $C$ is a universal constant and
$\varepsilon(t)$ is as in Theorem \ref{thm:specuniv}.
Moreover,
$$
	|\EE\lambda_{\rm max}(X)-\EE\lambda_{\rm max}(G) | \lesssim
	\sigma_*(X)\,(\log d)^{\frac{1}{2}} +
	R(X)^{\frac{1}{3}} 
	\sigma(X)^{\frac{2}{3}}
	(\log d)^{\frac{2}{3}} +
	R(X)\log d.	
$$
The same bounds hold if $\lambda_{\rm max}(X),\lambda_{\rm max}(G)$
are replaced by $\|X\|,\|G\|$, respectively.
\end{cor}

The proofs of Theorem \ref{thm:specuniv} and Corollary \ref{cor:normuniv} 
are given in section \ref{sec:univspec}.

The above results are meaningful only when $R(X)<\infty$, which requires 
that the matrices $Z_i$ are uniformly bounded. However, this restriction 
is almost entirely removed by the following result that is proved in 
section \ref{sec:trunc}.

\begin{thm}[Spectrum universality: unbounded case]
\label{thm:specheavy}
We have
$$
	\PP\Big[\dH(\spc(X),\spc(G)) > C
	\varepsilon_R(t),
	~\max_{1\le i\le n}\|Z_i\|\le R\Big] 
	\le de^{-t}
$$
for all $t\ge 0$ and $R\ge \bar R(X)^{\frac{1}{2}}\sigma(X)^{\frac{1}{2}}+
2^{\frac{1}{2}}\bar R(X)$, where
$$
	\varepsilon_R(t) = \sigma_*(X)\,t^{\frac{1}{2}} +
	R^\frac{1}{3}\sigma(X)^{\frac{2}{3}}t^{\frac{2}{3}} +
	Rt
$$
and $C$ is a universal constant. Moreover,
$$
	\EE\big[\dH(\spc(X),\spc(G))\big] \lesssim
	\sigma_*(X)\,(\log d)^{\frac{1}{2}} +
	\bar R(X)^{\frac{1}{6}}\sigma(X)^{\frac{5}{6}}\log d
$$
whenever $\bar R(X)\, (\log d)^3 \lesssim \sigma(X)$.
\end{thm}

\subsubsection{Universality of spectral statistics}
\label{sec:specstatmain}

We now complement the above results by formulating universality principles 
for various spectral statistics.

We begin by establishing universality of even moments.

\begin{thm}[Moment universality]
\label{thm:momentuniv}
For any $p\in\mathbb{N}$ and $2p\le q\le\infty$, we have
$$
	\big|\EE[\ntr X^{2p}]^{\frac{1}{2p}} -
	\EE[\ntr G^{2p}]^{\frac{1}{2p}}\big|
	\lesssim
	R_q(X)^{\frac{1}{3}}\sigma_q(X)^{\frac{2}{3}} p^{\frac{2}{3}}
	+ R_q(X)\,p
$$
as well as
$$
	\big|\EE[\ntr X^{2p}]^{\frac{1}{2p}} -
	\EE[\ntr G^{2p}]^{\frac{1}{2p}}\big|
	\lesssim R_q(X)\,p^2.
$$
\end{thm}

The first inequality has a better dependence on $p$, while the second 
inequality yields a sharper estimate when $R_q(X)\,p^2\ll\sigma_q(X)$. 
Both inequalities are variations of the same proof, which is given in 
section \ref{sec:univstat}.

Moment bounds provide limited information on the spectrum of a random 
matrix. Complementary information can be extracted from the resolvent. 

\begin{thm}[Resolvent universality]
\label{thm:smuniv}
We have
$$
	\big\|\EE[(z\id-X)^{-1}] - \EE[(z\id-G)^{-1}]\big\|
	\lesssim
	\frac{R(X)\sigma(X)^2+R(X)^3\log d}{(\mathrm{Im}\,z)^4} 
$$
for every $z\in\mathbb{C}$ with $\mathrm{Im}\,z>0$. Consequently,
$$
	\big\|\EE[\varphi(X)] - \EE[\varphi(G)] \big\|
	\lesssim
	\big(R(X)\sigma(X)^2+R(X)^3\log d\big)
	\|\varphi\|_{W^{5,1}(\mathbb{R})}
$$
for every $\varphi\in W^{5,1}(\mathbb{R})$.
\end{thm}

We can also generalize Theorem \ref{thm:smuniv} to unbounded 
random matrices, at the expense of somewhat worse quantitative error bounds.

\begin{thm}[Resolvent universality: unbounded case]
\label{thm:smunivheavy}
We have
$$
	\big\|\EE[(z\id -X)^{-1}]-
	\EE[(z\id -G)^{-1}]\big\| 
	\lesssim
	\frac{\sigma_*(X) +
	\bar R(X)^{\frac{1}{10}}\sigma(X)^{\frac{9}{10}}}{(\mathrm{Im}\,z)^2}
$$
for every $z\in\mathbb{C}$ with $\mathrm{Im}\,z>0$ and
$$
	\big\|\EE[\varphi(X)] - \EE[\varphi(G)] \big\|
	\lesssim
	\big(\sigma_*(X) + \bar 
	R(X)^{\frac{1}{10}}\sigma(X)^{\frac{9}{10}}
	\big)
	\|\varphi\|_{W^{3,1}(\mathbb{R})}
$$
for every $\varphi\in W^{3,1}(\mathbb{R})$,
provided that $\bar R(X) (\log d)^{\frac{5}{3}}\lesssim \sigma(X)$.
\end{thm}

The proofs of Theorems \ref{thm:smuniv} and \ref{thm:smunivheavy} are
given in sections \ref{sec:univstat} and \ref{sec:trunc}.

While the above universality principles suffice for many applications, our 
proofs can be readily adapted to the study of other spectral statistics. 
In particular, a universality principle for moments of the resolvent 
$\EE[\ntr |z\id - X|^{-2p}]$ (Theorem \ref{thm:resuniv}) will play a 
central role in the proof of Theorem \ref{thm:specuniv}. Let us also note 
that the quantitative bounds of Theorems \ref{thm:smuniv} and 
\ref{thm:smunivheavy} can be considerably improved if one is interested 
only in the trace of the resolvent, cf.\ Remark \ref{rem:stieltjes}. 
However, the norm bounds given here are particularly useful as they 
contain information on the eigenvectors of the random matrices, as will
be explained in section \ref{sec:appspiked}.

\begin{rem}
For simplicity, we formulated the results of this 
section only for expected spectral statistics. However, corresponding 
tail bounds follow by combining these bounds with concentration
inequalities for the relevant spectral statistics; cf.\ 
Lemma \ref{lem:pfsafconc} for moments, and Proposition \ref{prop:specconc}
for general spectral statistics.
\end{rem}

\begin{rem}
Theorems \ref{thm:smuniv} and \ref{thm:smunivheavy} provide access to the 
mesoscopic distribution of the random matrices (at a scale determined by 
$R(X)$ or $\bar R(X)$). Their proofs are much simpler than 
our universality results for the spectral edges, and could also be 
approached by more 
traditional methods \cite{Cha06,Cha14}; they are included here for 
completeness. While these results apply in principle to arbitrary spectral 
statistics $\varphi$, they do not provide useful bounds on $p$th 
moments for large $p$ (and thus for the spectral edges) as 
this would give rise to analogues of Theorems 
\ref{thm:momentuniv} and \ref{thm:resuniv} with constants growing 
exponentially rather than polynomially in $p$.
\end{rem}

\subsection{Matrix concentration inequalities}
\label{sec:mainmtx}

Universality principles show that a non-Gaussian random matrix $X$ behaves 
as a Gaussian random matrix $G$, but do not explain in themselves what the 
spectra of these matrices look like. To apply these results to specific 
models, our universality principles must be combined with suitable bounds 
for Gaussian random matrices. We will presently show that both classical 
matrix concentration inequalities, and new sharp matrix concentration 
inequalities, arise directly from our main results.

\subsubsection{Classical matrix concentration inequalities}

We begin by briefly illustrating how two classical matrix concentration 
inequalities can be recovered from our main results. While direct proofs 
of these inequalities \cite{Tro15,MJCFT14} are considerably simpler (and 
yield better numerical constants), this provides a new explanation for the 
form of these inqualities and serves as the simplest illustration of our 
results.

\begin{example}[Matrix Bernstein]
As $\sigma_*(X)\le\sigma(X)$ and as
$R(X)^{\frac{1}{3}}\sigma(X)^{\frac{2}{3}}t^{\frac{2}{3}} \lesssim
\sigma(X)\,t^{\frac{1}{2}}+R(X)\,t$ by Young's inequality,
Corollary \ref{cor:normuniv} implies
$$
	\EE\|X\| \lesssim
	\EE\|G\| + \sigma(X)\sqrt{\log d} + R(X)\log d.
$$
We may therefore view the matrix Bernstein inequality \eqref{eq:bernstein}
as a direct consequence of the Gaussian bound \eqref{eq:nck} and 
Corollary \ref{cor:normuniv}. The tail bound of \cite[Theorem 
6.1.1]{Tro15} can also easily be recovered up to universal constants
from Corollary \ref{cor:normuniv}.
\end{example}

\begin{example}[Matrix Rosenthal]
\label{ex:rosenthal}
Suppose that $\EE[X]=0$. 
Then the noncommutative Khintchine inequality \cite[\S 9.8]{Pis03} 
states that for every $p\in\mathbb{N}$, we have
$$
	\EE[\ntr G^{2p}]^{\frac{1}{2p}}\lesssim
	\sigma_{2p}(X)\sqrt{p}
$$
(the norm bound \eqref{eq:nck} follows directly from this estimate by 
choosing $p\sim\log d$).
Combining the noncommutative Khintchine inequality with Theorem 
\ref{thm:momentuniv} yields
$$
	\EE[\ntr X^{2p}]^{\frac{1}{2p}} \lesssim
	\sigma_{2p}(X)\sqrt{p}
	+R_{2p}(X)\,p,
$$
where we used that 
$R_{2p}(X)^{\frac{1}{3}}\sigma_{2p}(X)^{\frac{2}{3}}p^{\frac{2}{3}} 
\lesssim 
\sigma_{2p}(X)\sqrt{p}+R_{2p}(X)\,p$ by Young's inequality. This matrix 
Rosenthal inequality \cite[Corollary 7.4]{MJCFT14} may therefore be 
viewed as another consequence of the universality principle.
\end{example}

\subsubsection{Sharp matrix concentration inequalities}
\label{sec:sharpmmain}

A primary motivation behind our universality principles is that they may 
be combined with the Gaussian theory of \cite{BBV21} to obtain a powerful 
new family of sharp matrix concentration inequalities for sums of 
independent random matrices. These inequalities reduce the study of a very 
large family of nonhomogeneous random matrices to explicit computations. 
Let us state a prototypical inequality of this kind for sake of 
illustration.

\begin{thm}[Sharp matrix concentration]
\label{thm:smconc}
For any $t\ge 0$, we have
$$
	\mathbf{P}\big[\spc(X) \subseteq \spc(X_{\rm free}) +
	C\big\{v(X)^{\frac{1}{2}}\sigma(X)^{\frac{1}{2}}
	(\log d)^{\frac{3}{4}}+\varepsilon(t)\big\}[-1,1]\big] 
	\ge 1-2de^{-t},
$$
where $C$ is a universal constant and
$\varepsilon(t)$ is as in Theorem \ref{thm:specuniv}. In particular,
$$
	\mathbf{P}\big[\lambda_{\rm max}(X)\ge 
	\lambda_{\rm max}(X_{\rm free}) +
	Cv(X)^{\frac{1}{2}}\sigma(X)^{\frac{1}{2}}
        (\log d)^{\frac{3}{4}}+C\varepsilon(t)\big] \le 2de^{-t}
$$
and
\begin{multline*}
	\EE\lambda_{\rm max}(X) \le \lambda_{\rm max}(X_{\rm free}) +
 \\
	C\big\{v(X)^{\frac{1}{2}}\sigma(X)^{\frac{1}{2}}
        (\log d)^{\frac{3}{4}} +
	R(X)^{\frac{1}{3}}\sigma(X)^{\frac{2}{3}}(\log d)^{\frac{2}{3}} +
	R(X)\log d\big\}.
\end{multline*}
The same bounds hold if $\lambda_{\rm max}(X),\lambda_{\rm max}(X_{\rm 
free})$ are replaced by $\|X\|,\|X_{\rm free}\|$.
\end{thm}

\begin{proof}
This follows immediately by the union bound from Theorem 
\ref{thm:specuniv} and an application of 
\cite[Theorem 2.1]{BBV21} to the Gaussian matrix $G$.
\end{proof}

Theorem \ref{thm:smconc} shows that when $v(X)$ and $R(X)$ are 
sufficiently small, the 
spectrum of $X$ is controlled by that of its noncommutative model
$X_{\rm free}$. The latter admits explicit computations using tools of 
free probability. For example, it was shown by Lehner 
\cite[Corollary 1.5]{Leh99} (cf. \cite[\S 4.1]{BBV21}) that
$$
	\lambda_{\rm max}(X_{\rm free}) =
	\inf_{B>0}\lambda_{\rm max}\big(
	B^{-1}+\EE X + \EE[(X-\EE X)B(X-\EE X)]\big),
$$
where the infimum is over positive definite $B\in\M_d(\mathbb{C})_{\rm 
sa}$ (the infimum may be further restricted to $B$ for which the matrix in 
$\lambda_{\rm max}(\cdots)$ is a multiple of the identity). On the other 
hand, one may also easily deduce ``user-friendly'' bounds in the spirit 
of \cite{Tro15} whose statements make no reference to free probability.

\begin{cor}[``User-friendly'' bound]
\label{cor:userfriendly}
Let $Y=\sum_{i=1}^n Z_i$, where $Z_1,\ldots,Z_n$ are independent 
(possibly non-self-adjoint) $d\times d$ random matrices with $\EE[Z_i]=0$. 
Then
\begin{multline*}
	\mathbf{P}\big[\|Y\|\ge
	\|\EE Y^*Y\|^{\frac{1}{2}} +
        \|\EE YY^*\|^{\frac{1}{2}} + \\
        C\big\{v(Y)^{\frac{1}{2}}\sigma(Y)^{\frac{1}{2}}
        (\log d)^{\frac{3}{4}}+
        \sigma_*(Y)\, t^{\frac{1}{2}} +
        R(Y)^{\frac{1}{3}}\sigma(Y)^{\frac{2}{3}} t^{\frac{2}{3}} +
        R(Y)\,t\big\}
	\big] \le 4de^{-t}
\end{multline*}
for a universal constant $C$ and all $t\ge 0$, and
\begin{multline*}
	\EE\|Y\| \le
	\|\EE Y^*Y\|^{\frac{1}{2}} +
        \|\EE YY^*\|^{\frac{1}{2}} +
\\
	C\big\{v(Y)^{\frac{1}{2}}\sigma(Y)^{\frac{1}{2}}
        (\log d)^{\frac{3}{4}} +
	R(Y)^{\frac{1}{3}}\sigma(Y)^{\frac{2}{3}}(\log d)^{\frac{2}{3}} +
	R(Y)\log d\big\}.
\end{multline*}
Here we define $\sigma(Y) := \max(\|\EE Y^*Y\|^{\frac{1}{2}},
\|\EE YY^*\|^{\frac{1}{2}})$ in the non-self-adjoint case, while
$\sigma_*(Y),v(Y),R(Y)$ are defined as in section \ref{sec:matparm}.
\end{cor}

\begin{proof}
Combine Theorem \ref{thm:smconc}, Remark 
\ref{rem:nonsa}, and \cite[Lemmas 2.5 and 4.10]{BBV21}.
\end{proof}

There are many other possible combinations of the results of \cite{BBV21} 
and in section \ref{sec:univmain} above. For example, unbounded variants 
of Theorem \ref{thm:smconc} and Corollary \ref{cor:userfriendly} that 
involve the parameter $\bar R(X)$ instead of $R(X)$ are readily deduced 
from Theorem \ref{thm:specheavy}, and we obtain two-sided bounds for the 
moments and other spectral statistics of $X$ in terms of those of $X_{\rm 
free}$ by combining the results of section \ref{sec:specstatmain} with 
those of \cite[\S 2.2]{BBV21}. In the interest of space we do not spell 
out further combinations of this kind here; the appropriate results are 
easily applied directly in any given application.

\section{Applications}
\label{sec:app}

In this section, we provide precise formulations of the applications that 
were introduced in section \ref{sec:introapp} above, and discuss how they 
arise from our main results. Some technical proofs are postponed until 
section \ref{sec:applproofs}.

\subsection{Independent entries}
\label{sec:indent}

Before we turn to the main applications of this paper, we begin in this 
short section by bounding the matrix parameters that appear in our main 
results in the special case of random matrices with independent entries. 
While much stronger results and more complicated models will be considered 
in the sequel, our aim here is to help the reader gain some 
insight into the magnitudes of the parameters that appear in our bounds in 
the simplest setting.

Let $X$ be a $d\times d$ self-adjoint random matrix so that
$(X_{ij})_{i\ge j}$ are independent real random variables. Then we may 
write $X$ in the form \eqref{eq:model} as
\begin{equation}
\label{eq:indentex}
	X = \mathbf{E}[X] + \sum_{i\ge j}  \xi_{ij} E_{ij}
\end{equation}
where $\xi_{ij}:=X_{ij}-\mathbf{E} X_{ij}$,
$E_{ij}:=e_ie_j^*+e_je_i^*$ for $i>j$, and $E_{ii}:=e_ie_i^*$.

\begin{lem}
\label{lem:indepexp}
For the model \eqref{eq:indentex}, we have
$$
	\sigma(X)^2 = \max_i \sum_j \mathrm{Var}(X_{ij}),\qquad
	\sigma_*(X)^2\le v(X)^2 \le 2 \max_{i,j} \mathrm{Var}(X_{ij}),
$$
and for any $p\ge 1$
$$
	R(X) \le \max_{ij}\|\xi_{ij}\|_\infty,\qquad
	\bar R(X) \le
	\bigg(
	\sum_{i,j} \EE\big[\xi_{ij}^{2p}\big]
	\bigg)^{1/2p}.
$$
\end{lem}

\begin{proof}
As $\sigma(X),\sigma_*(X),v(X)$ are unchanged if we replace $X$ by its
Gaussian model, the first equation display follows from
\cite[Lemma 3.1]{BBV21}.  The bound on $R(X)$ follows immediately from the 
definition using $\|E_{ij}\|\le 1$. To bound $\bar R(X)$, we note that
$$
	\bar R(X)^{2p} \le
	\mathbf{E}\bigg[\max_{i,j}\|\xi_{ij}E_{ij}\|^{2p}\bigg]
	\le
	\sum_{i,j}
	\mathbf{E}\big[\|\xi_{ij}\|^{2p}\big],
$$
where we used Jensen in the first and $\|E_{ij}\|\le 1$
in the second inequality.
\end{proof}

The parameter $R(X)$ is finite when the entries of $X$ are uniformly 
bounded, as is often the case for discrete random matrices; otherwise, 
$\bar R(X)$ must be used.  Let us illustrate these parameters in two 
classical examples.

\begin{example}[Erd\H{o}s-R\'enyi]
\label{ex:indER}
Let $X_{ij}\sim\mathrm{Bern}(q)$ for all $i,j$, i.e., $X$ is the 
adjacency matrix of an Erd\H{o}s-R\'enyi graph with $d$ vertices and edge 
probability $q$. Then
$$
	\sigma(X)=\sqrt{dq(1-q)},\qquad\sigma_*(X)\le 
	v(X)\lesssim \sqrt{q(1-q)},\qquad R(X)\lesssim 1.
$$
The Gaussian model $G-\EE G$ associated to 
$X-\EE X$ is a Wigner matrix with entry variance $q(1-q)$, for 
which it is well known that $\mathbf{E}\|G-\EE 
G\|=(1+o(1))2\sqrt{dq(1-q)}$
(see, e.g., \cite{AGZ10}). Thus Corollary \ref{cor:normuniv} yields
$$
	\mathbf{E}\|X-\EE X\|=(1+o(1))2\sqrt{dq(1-q)}
	\quad\text{when}\quad dq(1-q) \gg (\log d)^4.
$$
The latter condition is optimal up to the power of the logarithm 
\cite{BBK19}. Note that it is readily verified in this example that $\bar 
R(X)\gtrsim 1$
as along as $\min(q,1-q)\gtrsim \frac{1}{d^2}$, so that the application of
Theorem \ref{thm:specheavy} does not yield any improvement here.
\end{example}

\begin{example}[Wigner matrix with $p$ moments]
\label{ex:indW}
Suppose that $(X_{ij})_{i\ge j}$ are i.i.d.\ with
zero mean, unit variance, and $\|X_{ij}\|_p\le C$ for some $p\ge 2$.
Then
$$
	\sigma(X)=\sqrt{d},\qquad \sigma_*(X)\le v(X)\lesssim 1,\qquad
	\bar R(X) \lesssim Cd^{\frac{2}{p}}.
$$
As the Gaussian model $G$ is a Wigner matrix with entry variance $1$, 
it is classical that $\mathbf{E}\|G\|=(1+o(1))2\sqrt{d}$ as above.
Thus Theorem \ref{thm:specheavy} yields
$$
	\mathbf{E}\|X\|=(1+o(1))2\sqrt{d}
	\quad\text{when}\quad d\gg C^2d^{\frac{4}{p}}(\log d)^{12}.
$$
In particular, the latter condition holds as soon as $p>4$. This is again 
optimal up to the power of the logarithm, cf.\ Corollary 
\ref{cor:sparsesaf} below.
\end{example}

We emphasize that the results of this paper do not provide any new 
information on the classical Examples \ref{ex:indER} and \ref{ex:indW}, 
whose spectral properties 
have been studied in stunning detail in the random matrix theory 
literature down to microscopic scales inaccessible by our bounds. However, 
for nonhomogeneous random matrices with independent entries, the 
theory of this paper provides new results that were not accessible by 
prior methods; see, e.g., Corollary \ref{cor:sparsesaf} below.

\begin{rem}
Combining Lemma \ref{lem:indepexp} with Corollary 
\ref{cor:userfriendly} yields ``user-friendly'' 
inequalities for random matrices with independent entries. Similar 
inequalities were obtained in \cite[\S 4.3]{LVY18} by different 
methods that are specific to the independent entry case.
The latter yield slightly better quantitative results in some examples 
(such as for the Erd\H{o}s-R\'enyi model, cf.\ \cite[Example 4.1]{LVY18}),
but cannot capture the sharp leading-order term 
$\|X_{\rm free}\|$ as in Theorem \ref{thm:smconc}.
For this reason, the bounds of this paper are often sharper than
those of \cite{LVY18} even for independent entry models.
\end{rem}

\subsection{Random graphs and expanders}
\label{sec:appgraph}

While random regular graphs, random Cayley graphs, and random $n$-lifts 
appear at first sight to be rather different models, our analysis of all 
these models will ultimately be based on a basic observation regarding 
random matrices defined by group representations. We first introduce some 
general facts, and then consider each of the above models in turn.

\subsubsection{Random matrices defined by group representations}

Let $\Gamma$ be a finite group, let $\rho:\Gamma\to U(d)$ be a nontrivial
unitary representation of dimension $d$, and let $g_1,\ldots,g_k$ be 
i.i.d.\ random variables drawn uniformly from $\Gamma$. Then 
\begin{equation}
\label{eq:mtxgpgen}
	X = \sum_{i=1}^k (\rho(g_i)+\rho(g_i)^*)
\end{equation}
defines a random matrix of the form \eqref{eq:model}. As 
$\rho(g_i)$ are unitary, it is easy to see that $\sigma(X)\asymp\sqrt{k}$ 
and $R(X)\asymp 1$. Therefore, our universality principles show that the 
spectrum of such matrices behaves as the associated Gaussian model when 
$k\gg (\log d)^\beta$ for a suitable $\beta>0$. The random regular 
graph, Cayley graph, and $n$-lift models will all arise as variations on 
this theme.

To understand the behavior of such matrices, it then remains to understand 
the behavior of the Gaussian model associated to $X$. The following 
standard group-theoretic facts will suffice for this purpose.

\begin{lem}
\label{lem:groupthy}
Suppose that $\rho$ is irreducible. Then there exists $s\in\{-1,0,1\}$ so 
that
$$
	\EE[\rho(g_i)]=0,\qquad 
	\Cov(\rho(g_i))=\frac{1}{d}\id,\qquad
	\EE[\rho(g_i)^2] = \frac{s}{d}\id.
$$
\end{lem}

\begin{proof}
The first two statements are standard facts about nontrivial irreducible 
representations \cite[Proposition 4.3.1 and Corollary 4.3.9]{Kow14}. To 
prove the last statement, we first observe that $\tr 
\EE[\rho(g_i)^2]=:s\in\{-1,0,1\}$ by a theorem of Frobenius and Schur 
\cite[Theorem 6.2.3]{Kow14}. On the other hand, for any $h\in\Gamma$, the 
random variables $g_i,h^{-1}g_i,g_i h^{-1}$ are uniformly distributed on 
$\Gamma$. Therefore,
$$
	\rho(h)\,\EE[\rho(g_i)^2]\, \rho(h^{-1}) =
	\EE[\rho(g_i h^{-1} g_i)]\, \rho(h^{-1}) =
	\EE[\rho(g_i)^2]
$$ 
for every $h\in\Gamma$, and the conclusion follows by \cite[Proposition 
4.3.4]{Kow14}.
\end{proof}

\begin{rem}
Let us emphasize that as $\rho(g_i)$ is not self-adjoint, the entry 
covariance matrix $\Cov(\rho(g_i))$ as defined in \eqref{eq:covdefn} does 
not fully determine the covariance of the real and imaginary parts of 
the entries of $\rho(g_i)$. In fact, by the Frobenius-Schur theorem used in 
the proof, it is the value $s$ that determines whether the representation 
is real $(s=1)$, complex $(s=0)$, or quaternionic $(s=-1)$.
\end{rem}

\subsubsection{Random regular graphs}
\label{sec:rregdet}

In this subsection, let $\Pi_1,\ldots,\Pi_k$ be i.i.d.\ uniformly 
distributed random $d\times d$ permutation matrices. Then
$$
	X = \sum_{i=1}^k (\Pi_i+\Pi_i^*)
$$
is the adjacency matrix of a (not necessarily simple) $2k$-regular graph 
with $d$ vertices. This is the \emph{permutation model} of random 
regular graphs. Before we proceed, let us recall a basic fact
about the adjacency matrix of any regular graph.

\begin{lem}[Alon-Boppana]
\label{lem:alonbop}
Let $A$ be the adjacency matrix of an $m$-regular graph with $d$ vertices. 
Then the largest and second largest eigenvalue of $A$ satisfy
$$
	\lambda_1(A) = m,\qquad\quad
	\lambda_2(A) \ge \bigg(1-C\,\frac{\log m}{\log d}
	\bigg)2\sqrt{m-1}
$$
for a universal constant $C$. Moreover, $1\in\mathbb{R}^d$ (the vector all 
of whose entries are one) is an eigenvector of $A$ with eigenvalue $m$.
\end{lem}

\begin{proof}
The statement about the largest eigenvalue and eigenvector follow 
immediately from the Perron-Frobenius theorem
and $A1=m1$. The bound on the second 
eigenvalue is a classical result of Alon-Boppana (in this 
form, see \cite{Nil91}).
\end{proof}

In the seminal paper \cite{Fri08}, Friedman shows that when $k$ is fixed 
and $d\to\infty$, the second largest eigenvalue of the adjacency matrix 
$X$ of the permutation model satisfies $\lambda_2(X)\le 
(1+o(1))2\sqrt{2k-1}$ with high probability. Thus, by Lemma 
\ref{lem:alonbop}, such graphs have the largest possible spectral gap to 
leading order, that is, they are ``nearly Ramanujan''. It is an old 
question whether this conclusion persists when both $k,d\to\infty$; see, 
e.g., \cite[\S 1.4]{BKY17}, and \cite{Vu08} for further questions of this 
kind. While some quantitative information can be extracted from proofs of 
Friedman's theorem (for example, a special case of \cite[Theorem 
1.4]{BC23} shows that Friedman's result remains valid when $k\ll 
\frac{\log d}{(\log\log d)^2}$), all known proofs appear to break down for 
larger $k$. In the latter regime, it is known \cite{FKS89,DJPP13,CGJ18} 
that $\lambda_2(X)=O(\sqrt{k})$, but these results cannot recover the 
optimal constant.

We presently settle this question for $k\gg (\log d)^4$. This leaves only 
a narrow range of parameters $\frac{\log d}{(\log\log d)^2}\lesssim k 
\lesssim (\log d)^4$ open.\footnote{%
This remaining range of parameters was recently settled in the
subsequent work \cite{CGV24}.}

\begin{thm}
\label{thm:rreg}
Denote by $X^\perp$ the restriction of $X$ to $1^\perp$.
Then for every $a>0$, there is a constant $C>0$ depending only on $a$ so that
$$
	\|X^\perp\| \le
	\bigg(
	1 + C \frac{(\log d)^{\frac{3}{4}}}{d^{\frac{1}{4}}} +
	C \frac{(\log d)^{\frac{2}{3}}}{k^{\frac{1}{6}}} + 
	C \frac{\log d}{k^{\frac{1}{2}}}
	\bigg)
	2\sqrt{2k}
$$
with probability at least $1-d^{-a}$. In particular,
$\lambda_2(X)\le (1+o(1))2\sqrt{2k-1}$ with probability $1-o(1)$
whenever $d,k\to\infty$ with $(\log d)^4=o(k)$.
\end{thm}

\begin{proof}
The random matrix $X$ is a special case of the model of the
previous section, where we choose $\rho$ to be the permutation 
representation of the symmetric group $\mathrm{S}_d$. Moreover, 
$\rho^\perp=\rho|_{1^\perp}$ is an irreducible representation of 
$\mathrm{S}_d$ of dimension $d-1$. We can therefore compute using Lemma 
\ref{lem:groupthy}
$$
	\EE[X^\perp]=0,\qquad
	\|\EE[(X^\perp)^2]\| = 
	2k\bigg(1+\frac{s}{d-1}\bigg)
$$
with $|s|\le 1$, as well as (using that $v(A+B)\le v(A)+v(B)$)
$$
	\sigma_*(X^\perp)\le 
	v(X^\perp) \le
	2\,v\bigg(\sum_{i=1}^k\rho^\perp(g_i)\bigg)
	=2\sqrt{\frac{k}{d-1}},\qquad
	R(X^\perp) \le 2.
$$
The bound on $\|X^\perp\|$ follows directly by applying the tail bound 
of Corollary \ref{cor:userfriendly} with $t=(a+3)\log d$.
As $1$ is an eigenvector of $X$ with eigenvalue $\lambda_1(X)$, we clearly
have $\lambda_2(X)\le\|X^\perp\|$ and the proof is complete.
\end{proof}

\begin{rem}
Theorem \ref{thm:rreg} yields an upper bound on $\lambda_2(X)$ for the 
permutation model, which agrees with the Alon-Boppana lower bound that 
holds for \emph{any} regular graph. Note, however, that the Alon-Boppana 
bound (Lemma \ref{lem:alonbop}) is only meaningful when $\log k\ll\log d$. 
A variant of Theorem \ref{thm:rreg} readily shows that for the permutation 
model, the lower bound $\lambda_2(X)\ge (1+o(1))2\sqrt{2k-1}$ remains 
valid for \emph{any} $k\gg (\log d)^4$, even when the Alon-Boppana bound 
fails (e.g., combine Corollary \ref{cor:normuniv} with \cite[Corollary 
2.11]{BBV21}). However, when the Alon-Boppana bound fails it need not be 
the case that random regular graphs are optimal expanders.
\end{rem}

\subsubsection{Random Cayley graphs}

In this subsection we let $\Gamma$ be a finite group, and let 
$g_1,\ldots,g_k$ be i.i.d.\ variables drawn uniformly from $\Gamma$. We 
consider the Cayley graph defined by the generating set 
$\{g_1,\ldots,g_k,g_1^{-1},\ldots,g_k^{-1}\}$, and denote its adjacency 
matrix by $X$. This model is a special case of \eqref{eq:mtxgpgen} where 
$\rho:\Gamma\to U(\ell_2(\Gamma))$ is the right-regular representation of 
$\Gamma$, that is, $(\rho(g)f)_h = f_{hg}$.

A classical result of Alon and Roichman \cite{AR94} states that if we 
choose $k\gg\log|\Gamma|$ generators, then the random Cayley graph is an 
expander, that is, $\lambda_2(X)=o(k)$ (see also \cite{Nao12} and the 
references therein for alternative proofs and extensions).\footnote{%
	The asymptotic notation used here implicity assumes a family
	of groups $\Gamma_n$ and degrees $k_n$ with $n\to\infty$; we
	dropped the indexing for simplicity. Note, however, that 
	our results are nonasymptotic; asymptotic statements
	are given only to clarify their qualitative features.
}
The remarkable 
feature of this result is that it holds for \emph{any} finite group 
$\Gamma$. Here we exhibit a new phenomenon: for many groups, choosing 
$k\gg(\log|\Gamma|)^4$ generators suffices to ensure that the random 
Cayley graph is nearly Ramanujan, that is, that $\lambda_2(X)\le 
(1+o(1))2\sqrt{2k-1}$. This cannot happen for an arbitrary group, as the 
weaker estimates of \cite{AR94} are essentially optimal for abelian 
groups. Rather, we show that this is the case for general nonabelian 
groups as soon as the dimensions of the nontrivial irreducible 
representations are not too small.

\begin{thm}
\label{thm:cayley}
Denote by $d_0\le d_1\le \ldots\le d_m$ the dimensions of the 
isomorphism classes of irreducible representations of $\Gamma$, where 
$d_0=1$ corresponds to the trivial representation. Define the parameter
$$
	\alpha(\Gamma) := 
	\max_{1\le i\le m} \sqrt{\frac{\log(i+1)}{d_i}}.
$$
Then for every $a>0$, 
there is a constant $C>0$ depending only on $a$ so that
$$
	\|X^\perp\| \le 
	\bigg(1
	+C\alpha(\Gamma)
	+C\frac{(\log d_1)^{\frac{3}{4}}}{d_1^{\frac{1}{4}}}
	+C \frac{ (\log |\Gamma|)^{\frac{2}{3}}}{k^{\frac{1}{6}}}
	+C \frac{\log|\Gamma|}{k^{\frac{1}{2}}}
	\bigg)2\sqrt{2k}
$$
with probability at least $1-d_1^{-a}$. In particular,
$\lambda_2(X)\le (1+o(1))2\sqrt{2k-1}$ with probability $1-o(1)$ whenever
$\alpha(\Gamma)=o(1)$ and $(\log |\Gamma|)^4 = o(k)$.
\end{thm}

\begin{proof}
The Peter-Weyl theorem \cite[Theorem 5.4.1]{Kow14} states that the 
right-regular representation decomposes as a direct sum
$\rho=\bigoplus_{i=0}^m(\rho_i\otimes\id_{d_i})$ of non-isomorphic 
irreducible representations $\rho_i$ of $\Gamma$, each of which appears 
with multiplicity $d_i$ (here $\id_d$ is the identity matrix of dimension 
$d$). In other words, there is a choice of basis in which $X$ is 
block-diagonal with blocks $X_i\otimes\id_{d_i}$, where each $X_i$ is of 
the form \eqref{eq:mtxgpgen} for a distinct irreducible representation of 
$\Gamma$. As the trivial representation accounts for the action on the 
eigenvector $1$, we obtain $X^\perp = \bigoplus_{1\le i\le m} X_i\otimes 
\id_{d_i}$.

As in the proof of Theorem \ref{thm:rreg}, we can compute
using Lemma \ref{lem:groupthy}
$$
	\EE[X_i] = 0,\quad
	\sigma(X_i)^2 \le 2k\bigg(1+\frac{1}{d_i}\bigg),\quad
	\sigma_*(X_i)^2 \le v(X_i)^2 \le \frac{4k}{d_i},\quad
	R(X_i)\le 2
$$
for $1\le i\le m$. Thus Corollary \ref{cor:userfriendly} yields
$$
	\mathbf{P}\bigg[
	\|X_i\| \ge 
	\bigg(
	1
	+ C\frac{(\log d_i)^{\frac{3}{4}}}{d_i^{\frac{1}{4}}}
	+ C\frac{t^{\frac{1}{2}}}{d_i^{\frac{1}{2}}}
	+ C\frac{t^{\frac{2}{3}}}{k^{\frac{1}{6}}}
	+ C\frac{t}{k^{\frac{1}{2}}}
	\bigg)
	2\sqrt{2k}\bigg]
	\le 4d_i e^{-t}
$$
for all $1\le i\le m$ and $t\ge 0$. Choosing $t=2\log(i+1) + (a+3)\log d_i$ 
yields
$$
	\mathbf{P}\bigg[
	\|X_i\| \ge 
	\bigg(
	1
	+ C\alpha(\Gamma)
	+ C\frac{(\log d_1)^{\frac{3}{4}}}{d_1^{\frac{1}{4}}}
	+ C\frac{(\log|\Gamma|)^{\frac{2}{3}}}{k^{\frac{1}{6}}}
	+ C\frac{\log|\Gamma|}{k^{\frac{1}{2}}}
	\bigg)
	2\sqrt{2k}\bigg]
	\le 
	\frac{d_1^{-a}}{(i+1)^2}
$$
for all $1\le i\le m$, where $C$ depends only on $a$. (Here we used that 
$t\le (a+5)\log|\Gamma|$ as $d_i\le|\Gamma|$ and $m+1\le|\Gamma|$.)
Using $\|X^\perp\|=\max_{1\le i\le m}\|X_i\|$, applying a union bound,
and noting that $\sum_{i=1}^\infty (i+1)^{-2}<1$ concludes the proof.
\end{proof}

The parameter $\alpha(\Gamma)$ in Theorem \ref{thm:cayley} controls the 
growth rate of the dimensions of the irreducible representations of 
$\Gamma$. The condition $\alpha(\Gamma)=o(1)$ holds as soon as the 
irreducible representations are sufficiently high-dimensional. This 
condition is satisfied in many examples. For example, the following result 
is a slightly stronger form of \cite[Lemma 9]{BKMZ22}; we include the 
proof for completeness.

\begin{lem}
For any sequence $\Gamma_n$ of nonabelian finite simple groups such that 
$|\Gamma_n|\to\infty$, we have $\alpha(\Gamma_n)\to 0$.
\end{lem}

\begin{proof}
The nonabelian finite simple groups are classified \cite[\S 5.1]{KL90} 
into 6 families of classical simple groups, 10 families of exceptional 
simple groups, the alternating groups, and a finite number of sporadic 
groups. As $|\Gamma_n|\to\infty$, the sporadic groups are irrelevant. In 
the following, we fix a non-sporadic finite simple group with 
$d_1\le\ldots\le d_m$ ($m+1\le|\Gamma|$)
defined as in Theorem \ref{thm:cayley}.

If $\Gamma$ is a classical or exceptional simple group, \cite[Tables 
5.1.A--B and 5.3.A]{KL90} yield $\log |\Gamma| \asymp r^2\log q$ and 
$d_1\gtrsim q^{cr}$ for some $r\in\mathbb{N}$ and prime power $q$, where 
$c$ is a universal constant. Thus $\alpha(\Gamma)^2 \le 
\frac{\log|\Gamma|}{d_1} \lesssim \frac{(\log q^{cr})^2}{q^{cr}} = o(1)$ 
as $|\Gamma|\to\infty$.

Now let $\Gamma=\mathrm{Alt}(r)$ be an alternating group with $r\ge 9$; 
then $\log |\Gamma| \asymp r\log r$, and it follows from \cite[\S 
5.1]{FH91} and \cite[Result 2]{Ras77} that $d_1=r-1$ and $d_2\ge cr^2$ for 
a universal constant $c$. Thus $\alpha(\Gamma)^2 \le \max(\frac{\log 
2}{d_1},\frac{\log|\Gamma|}{d_2})=o(1)$ as $|\Gamma|\to\infty$. 
\end{proof}

Theorem \ref{thm:cayley} therefore implies that for nonabelian finite 
simple groups, the Cayley graph defined by choosing $k\gg 
(\log|\Gamma|)^4$ random generators is nearly Ramanujan with high 
probability. When we are in addition in the domain of validity of the 
Alon-Boppana theorem, that is, when $\log k \ll \log|\Gamma|$ (cf.\ Lemma 
\ref{lem:alonbop}), it follows that these random Cayley graphs are optimal 
expanders.

The formulation of Theorem \ref{thm:cayley} was inspired by \cite{BKMZ22}, 
where a variant of the Gaussian model $G$ associated to $X$ was introduced 
on an ad-hoc basis to illustrate certain subtleties in the formulation of 
matrix concentration inequalities \cite[\S 8.1]{BBV21}. The key point 
here, however, is that the universality principles of this paper make 
random matrices of this kind appear in a fundamental manner in the study 
of the expansion properties of random Cayley graphs.

\begin{rem} To achieve a sharp bound, Theorem \ref{thm:cayley} requires at 
least that the smallest dimension $d_1$ of a nontrivial irreducible 
representation diverges; that is, Theorem \ref{thm:cayley} is concerned 
with quasirandom groups in the sense of \cite{Gow08}. This condition also 
plays a key role in deep results on expansion in finite simple groups of 
Lie type that were pioneered by Bourgain and Gamburd \cite{BG08,BGGT15}. 
In contrast to Theorem \ref{thm:cayley}, the latter results require only a 
bounded number of random generators but do not achieve the near-Ramanujan 
property of the associated Cayley graphs. Whether both properties can be 
achieved simultaneously is a long-standing question, see \cite{RS19} for 
further discussion and numerical evidence.

It is clear that $d_1\to\infty$ is 
also necessary to obtain an optimal expander with high probability. For 
example, if $\Gamma=\mathrm{S}_d$ is the symmetric group, $d_1=1$ 
corresponds to the sign representation, and the $1\times 1$ block 
$X_1$ in the proof of Theorem \ref{thm:cayley} equals twice the sum of $k$ 
i.i.d.\ symmetric Bernoulli variables. Thus $\lambda_2(X)\ge\|X_1\|$ already 
exceeds, say, $3\sqrt{2k-1}$ with constant probability. A similar argument 
applies to any sequence of groups for which $d_1\not\to\infty$. However, 
in this situation it is still possible to obtain $\lambda_2(X)\le 
(1+o(1))2\sqrt{2k-1}$ with constant probability as long as the 
number of low-dimensional irreducible representations is bounded.
Conditions for this to hold follow along the same lines as in
Theorem \ref{thm:cayley}.
\end{rem}

\subsubsection{Random lifts}

For an $m$-regular graph, the Alon-Boppana lower bound on the second 
eigenvalue arises from the fact that $2\sqrt{m-1}$ is the spectral radius 
of the infinite $m$-regular tree.  This suggests that an analogue of Lemma 
\ref{lem:alonbop} for a \emph{non-regular} graph $H$ should lower bound 
its second eigenvalue by the spectral radius $\varrho(\hat H)$ of its 
universal covering tree $\hat H$. This is captured, at least 
qualitatively, by \cite[Theorem 6.6]{HWL06}. A non-regular graph $H$ may 
thus be viewed as an optimal expander if its second eigenvalue is bounded 
by $(1+o(1))\varrho(\hat H)$ \cite[\S 6]{HWL06}.

Amit and Linial \cite{AL02} and Friedman \cite{Fri03} proposed a model of 
random graphs that is designed to achieve such optimal expansion 
properties. Given any base graph $H=([d],E_H)$ with $d$ vertices, its 
\emph{random $n$-lift} $H^{(n)}=([d]\times[n],E_{H^{(n)}})$ is obtained by 
duplicating each vertex and edge of the base graph $n$ times, and randomly 
scrambling the duplicate edges among the duplicate vertices. That is, for 
each $e\in E_H$ with $e=(i,j)$, $i\le j$, we construct $e_k\in 
E_{H^{(n)}}$, $k=1,\ldots,n$ with $e_k=((i,k),(j,\sigma_e(k)))$, where 
$\sigma_e$ is a random permutation that is chosen independently for each 
$e\in E_H$. The adjacency matrix $X^{(n)}$ of $H^{(n)}$ is 
$$
	X^{(n)} = 
	\sum_{e\in E_H} 
	(A_e \otimes \Pi_e^{(n)} + A_e^*\otimes\Pi_e^{(n)*}),
$$
where $\Pi_e^{(n)}$ are i.i.d.\ uniformly distributed $n\times n$ random 
permutation matrices and $A_e$ are the $d\times d$ matrices 
$A_e=e_ie_j^*$ for $e=(i,j)$, $i\le j$. 

It is important to note that for every $n$, the restriction of $X^{(n)}$ 
to $\mathbb{C}^d\otimes \mathbb{C}1$ coincides with the adjacency matrix 
$X^{(1)}$ of $H$. Thus every eigenvalue of $H$ is also an eigenvalue of 
$H^{(n)}$. The \emph{new} eigenvalues that are introduced by the random 
lift are the eigenvalues of $X^{(n)\perp}$, the restriction of $X^{(n)}$ 
to $\mathbb{C}^d\otimes 1^\perp$. The long-standing conjecture that 
$\|X^{(n)\perp}\|\le (1+o(1))\varrho(\hat H)$ as $n\to\infty$ for fixed 
$H$ was proved by Bordenave and Collins in \cite{BC19}. This shows that 
random $n$-lifts are optimal expanders provided the base graph is an 
optimal expander.

As in the case of random regular graphs, however, it is far from clear 
whether this phenomenon persists if one considers random $n$-lifts of an 
unbounded sequence of base graphs $H_n$. The best bound to date 
\cite[Theorem 1.4]{BC23} is restricted to $n$-lifts of graphs $H$ with 
$|E_H|\ll \frac{\log n}{(\log\log n)^2}$ edges. The following result 
addresses the complementary regime where the maximal degree of $H$ grows 
at least polylogarithmically in the number of vertices $nd$ of its 
$n$-lift $H^{(n)}$.

\begin{thm}
\label{thm:lift}
Let $H=([d],E_H)$ be an (undirected, not necessarily simple) graph
without self-loops.
Denote by $\mathrm{D}(H)$ the maximal degree of a vertex of $H$ and by 
$\mathrm{M}(H)$ the 
maximal multiplicity of an edge of $H$. Then for every $a>0$, there is a 
constant $C>0$ depending only on $a$ so that the new eigenvalues of 
$H^{(n)}$ satisfy
$$
	\|X^{(n)\perp}\| \le
	\bigg(
	1
	+C\frac{\mathrm{M}(H)^{\frac{1}{4}}}{n^{\frac{1}{4}}}
	\frac{(\log nd)^{\frac{3}{4}}}{\mathrm{D}(H)^{\frac{1}{4}}}
	+C\frac{(\log nd)^{\frac{2}{3}}}{\mathrm{D}(H)^{\frac{1}{6}}} 
	+C\frac{\log nd}{\mathrm{D}(H)^{\frac{1}{2}}}
	\bigg)
	\varrho(\hat H)
$$
with probability at least $1-(nd)^{-a}$. In particular,
$\|X^{(n)\perp}\| \le (1+o(1))\varrho(\hat H)$ with probability $1-o(1)$
whenever $(\log nd)^4=o(\mathrm{D}(H))$ and $\mathrm{M}(H)=O(n)$.
\end{thm}

A surprising aspect of Theorem \ref{thm:lift} is that when $H$ is a simple 
graph and $\mathrm{D}(H)\gg (\log d)^4$, the conclusion holds already for 
$n=2$, that is, for random $2$-lifts. This is stark contrast to the 
bounded degree case, where one must in general let $n\to\infty$ to achieve 
an $(1+o(1))\varrho(\hat H)$ upper bound. On the other hand, when 
$n\to\infty$, a quantitative Alon-Boppana type theorem of \cite[Theorem 
1.7]{BC23} shows that there is a broad range of parameters where Theorem 
\ref{thm:lift} yields the smallest possible new eigenvalues among all (not 
necessarily random) $n$-lifts.

\begin{rem}
The assumption that $H$ has no self-loops was made for simplicity. The 
proof of Theorem \ref{thm:lift} will allow for self-loops, but in this 
case we must let $n\to\infty$ to achieve a $(1+o(1))\varrho(\hat H)$ 
bound. We already discussed a special case of this setting: when
$H$ is the graph with $1$ vertex and $k$ self-loops, $H^{(n)}$ 
coincides with the permutation model of random regular graphs of section 
\ref{sec:rregdet}.
\end{rem}

\begin{rem}
When $H$ is a simple $m$-regular graph (for which
$\varrho(\hat H)=2\sqrt{m-1}$), a result along the lines 
of Theorem \ref{thm:lift} can be obtained in a much simpler manner by 
comparing the norm of $X^\perp$ to that of a Wigner matrix
\cite{BvH16,BD21}. The primary interest of Theorem \ref{thm:lift} is
that it yields the correct upper bound for general $H$.
\end{rem}

The proof of Theorem \ref{thm:lift} is given in section \ref{sec:pflift}. 
Let however briefly outline the argument. It is a basic fact (see, e.g.,
\cite{BC19}) that $\varrho(\hat H)$ may be computed as
\begin{equation}
\label{eq:sprcover}
	\varrho(\hat H) = 
	\Bigg\|
	\sum_{e\in E_H} (A_e \otimes \lambda(g_e) + 
		A_e^* \otimes \lambda(g_e)^*)
	\Bigg\|,
\end{equation}
where $(g_e)_{e\in E_H}$ are the free generators of the free group 
$\mathrm{F}_{|E_H|}$ and $\lambda$ denotes the left-regular 
representation. On the other hand, Theorem \ref{thm:smconc} enables us to 
bound the norm of $X^{(n)\perp}$ by that of $X^{(n)\perp}_{\rm free}$.
This almost yields the desired conclusion, except that in 
$X^{(n)\perp}_{\rm free}$ the free generators $\lambda(g_i)$ are replaced 
by certain deformed circular variables. While in general
$\|X^{(n)\perp}_{\rm free}\|>\varrho(\hat H)$,
we will show these quantities coincide 
to leading order when $\mathrm{D}(H)\to\infty$, concluding the proof.

\subsection{Matrix concentration inequalities for smallest singular 
values}
\label{sec:appsmsing}

The theory behind classical matrix concentration inequalities \cite{Tro15} 
is inherently limited to the extreme eigenvalues of random matrices. In 
contrast, our results control the entire spectrum. This makes it possible, 
for example, to obtain matrix concentration inequalities for the smallest 
singular value of non-self-adjoint random matrices. Such results are 
fundamentally outside the scope of classical matrix concentration 
inequalities for general models of the form \eqref{eq:model} (unless one 
imposes special structure, see, e.g., \cite[\S 5.2.1]{Tro15} for an 
example).

We presently state a general result of this kind. In the following, we 
define the smallest singular value of $Y$ as
$\mathrm{s_{min}}(Y):=\inf\spc(|Y|)$.

\begin{thm}
\label{thm:ssing}
Let $Y=Z_0+\sum_{i=1}^n Z_i$, where $Z_0$ is a nonrandom $d\times m$ 
matrix and $Z_1,\ldots,Z_n$ are independent centered $d\times m$ random 
matrices, with $d\ge m$. Then
\begin{multline*}
	\mathbf{P}\big[
	\mathrm{s_{min}}(Y) \le \mathrm{s_{min}}(Y_{\rm free}) -
	C\big\{
	v(Y)^{\frac{1}{2}}\sigma(Y)^{\frac{1}{2}}(\log d)^{\frac{3}{4}}
\\
	- \sigma_*(Y)t - R(Y)^{\frac{1}{3}}\sigma(Y)^{\frac{2}{3}}
	t^{\frac{2}{3}} - R(Y)t\big\}
	\big]
	\le de^{-t}
\end{multline*}
for all $t\ge 0$, where $C$ is universal constant.
Here $\sigma(Y),\sigma_*(Y),v(Y),R(Y)$ are defined as in Corollary 
\ref{cor:userfriendly}, and $Y_{\rm free}$ is 
the $d\times m$ matrix so that the real and imaginary parts of its 
entries is a semicircular family with the same mean and covariance as the 
real and imaginary parts of the entries of $Y$.
\end{thm}

The proof of Theorem \ref{thm:ssing} will be given in section 
\ref{sec:pfssing} below. A variant of Theorem \ref{thm:ssing} for 
unbounded random matrices can also be deduced along the same lines, by 
using Theorem \ref{thm:specheavy} instead of Theorem \ref{thm:specuniv} in 
the proof.

In the case that $\EE[Y]=0$, a simple ``user-friendly'' bound
\begin{equation}
\label{eq:userfsmsg}
	\mathrm{s_{min}}(Y_{\rm free}) \ge
	\mathrm{s_{min}}(\EE Y^*Y)^{\frac{1}{2}} -
	\|\EE YY^*\|^{\frac{1}{2}}
\end{equation}
was obtained in \cite[Lemma 3.15]{BBV21}. This bound gives rise to very 
simple explicit estimates, but may be far from sharp for nonhomogeneous 
random matrices. In this case, the quantity $\mathrm{s_{min}}(Y_{\rm 
free})$ can also be computed exactly using an explicit variational formula 
(in the spirit of \cite{Leh99}) that is obtained in \cite{EvH23}.

\begin{example}[Bipartite random graphs]
Consider a bipartite random graph with vertex set $[d]\sqcup[m]$ ($d\ge 
m$) in which each edge $(i,j)$ with $i\in[d],j\in[m]$ is included 
independently with probability $p_{ij}$. This is a nonhomogeneous and 
bipartite analogue of the classical Erd\H{o}s-R\'enyi model. 
The adjacency matrix $A$ of this graph is the $d\times m$ matrix with 
independent entries $A_{ij}\sim\mathrm{Bern}(p_{ij})$.

A basic question of interest in this setting (cf.\ \cite{DZ22} and the 
references therein) is to bound the largest and smallest singular values 
of $A-\EE A$. 
\end{example}

\begin{cor}
\label{cor:bipartite}
Denote by $\rho := \min_j \sum_i p_{ij}(1-p_{ij})$, 
$\gamma := \max_i \sum_j p_{ij}(1-p_{ij})$, and $k :=
\max\big\{\max_i\sum_j p_{ij}(1-p_{ij}), \max_j\sum_i 
p_{ij}(1-p_{ij})\big\}$. Then for every $a>0$, there is a 
constant $C>0$ that depends only on $a$ so that
\begin{align*}
	\|A-\EE A\| &\le
	\sqrt{\rho}+\sqrt{\gamma} +
	C k^{\frac{1}{3}}(\log d)^{\frac{2}{3}},
\\
	\mathrm{s_{min}}(A-\EE A) &\ge
	\sqrt{\rho}-\sqrt{\gamma} -
	C k^{\frac{1}{3}}(\log d)^{\frac{2}{3}}
\end{align*}
with probability at least $1-d^{-a}$,
provided that $k\ge \log d$.
\end{cor}

\begin{proof}
We can express $Y:=A-\EE A$ as $Y = \sum_{ij} 
Z_{ij}$ where $Z_{ij}=Y_{ij}e_ie_j^*$ are independent centered random 
matrices. Then we readily compute $\mathrm{s_{min}}(\EE Y^*Y)=\rho$,
$\|\EE YY^*\|=\gamma$, $\sigma^2(X)\le k$, 
$\sigma_*(Y)^2\le v(Y)^2 \le \max_{ij}p_{ij}(1-p_{ij})$, and $R(Y)\le 1$.
The conclusion now follows directly from Corollary \ref{cor:userfriendly},
Theorem \ref{thm:ssing}, and \eqref{eq:userfsmsg}.
\end{proof}

The simplest example of this result is the homogeneous case where 
$p_{ij}=p<1$. In this case, the above bounds reduce to 
$$
	1-\sqrt{\frac{m}{d}} -
	\frac{C(\log d)^{\frac{2}{3}}}{(dp)^{\frac{1}{6}}}
	\le
	\frac{\mathrm{s_{min}}(A-\EE A)}{\sqrt{dp(1-p)}} \le 
	\frac{\|A-\EE A\|}{\sqrt{dp(1-p)}} \le
	1+\sqrt{\frac{m}{d}} + 
	\frac{C(\log d)^{\frac{2}{3}}}{(dp)^{\frac{1}{6}}}.
$$
This shows that the classical Bai-Yin law \cite{BY93}, which applies
to dense graphs with constant $0<p,\frac{m}{d}<1$ as $d\to\infty$, remains 
valid for sparse graphs with average degree $dp\gg (\log d)^4$. 
In this homogeneous setting, the results of \cite{DZ22} establish the same 
conclusion in the slightly larger range $dp\gg \log d$. However, for 
nonhomogeneous graphs, the best known bounds due to \cite{DZ22} are 
already weaker than those of Corollary \ref{cor:bipartite} to leading 
order, cf.\ \cite[Remark 2.6]{DZ22}.

On the other hand, we have formulated the simple bounds of Corollary 
\ref{cor:bipartite} for sake of illustration only: the same proof yields 
bounds in which $\sqrt{\rho}+\sqrt{\gamma}$ and 
$\sqrt{\rho}-\sqrt{\gamma}$ are replaced by the optimal leading-order 
terms $\|Y_{\rm free}\|$ and $\mathrm{s_{min}}(Y_{\rm free})$, 
respectively (where $Y=A-\EE A$), which can be computed in terms of 
explicit variational principles \cite{Leh99,EvH23}. In other words, in 
contrast to previous results, we obtain sharp Bai-Yin laws for 
sparse nonohomogeneous random matrices.

\begin{rem}
More generally, Theorem \ref{thm:ssing} may be viewed as a nonasymptotic, 
nonhomogeneous Bai-Yin law that is sharp to 
leading order. It should be emphasized, however, that it can only locate 
the smallest singular value of $Y$ near that of its noncommutative model 
$Y_{\rm free}$. In particular, Theorem \ref{thm:ssing} sheds no light on 
the invertibility of $Y$ when $\mathrm{s_{min}}(Y_{\rm free})=0$, as is 
the case, e.g., for square matrices with i.i.d.\ 
entries. The latter question is of a fundamentally different nature, which 
is presently understood for nonhomogeneous models only under restrictive 
assumptions \cite{RZ16,Coo18} (see, however, \cite{Tik23} for significant 
recent progress in this direction).
\end{rem}

\subsection{Sample covariance matrices}
\label{sec:appscov}

Let $Y_1,\ldots,Y_n$ be independent, centered random
vectors in $\mathbb{R}^d$. The $d\times d$ random matrix
\begin{equation}
\label{eq:smtx}
        S = \sum_{i=1}^n Y_iY_i^*
\end{equation}
is called the (nonhomogeneous) \emph{sample covariance matrix} associated
to the data $Y_1,\ldots,Y_n$. Equivalently, we may express 
$S=YY^*$, where
\begin{equation}
\label{eq:ymtx}
	Y = \sum_{i=1}^n Y_i e_i^*
\end{equation}
is the $d\times n$ random matrix with independent columns
$Y_1,\ldots,Y_n$.
In the classical setting where $Y_1,\ldots,Y_n$ are identically 
distributed, $\frac{1}{n}\EE S$ is the covariance matrix of 
$Y_i$, and $\frac{1}{n}S$ may be viewed as a statistical estimator of this 
covariance matrix. A central problem is then to bound the deviation 
$\frac{1}{n}\|S-\EE S\|$ of the estimated covariance matrix from the 
actual covariance matrix. Here we allow for a more general 
nonhomogeneous situation where the data $Y_1,\ldots,Y_n$ need not be 
identically distributed, which is of independent interest (see, e.g., 
\cite{CHZ20}).

From the viewpoint of this paper, sample covariance matrices may 
approached in two different ways: we may either view $S$ itself as a model 
of the form \eqref{eq:model}, or we may view $Y$ as a model of the form 
\eqref{eq:model}. These two interpretations give rise to distinct 
universality principles. As we will see below, neither approach subsumes 
the other: they control the behavior of $S$ in complementary regimes.

For simplicity, we focus in this section on ``user-friendly'' explicit 
bounds on the expected deviation $\EE\|S-\EE S\|$; sharp bounds in terms 
of $S_{\rm free}$ and $Y_{\rm free}$, as well as high-probability bounds, 
may be obtained analogously.

\subsubsection{Gaussian sample covariance matrices}
\label{sec:gaussscov}

In this section we consider Gaussian sample covariance matrices, that is, 
\eqref{eq:smtx} where $Y_1,\ldots,Y_n$ are independent Gaussian random 
vectors $Y_i\sim N(0,\Sigma_i)$. In this case, $Y$ is a Gaussian random 
matrix, to which the Gaussian theory of \cite{BBV21} can be applied.

\begin{thm}[Gaussian bound]
\label{thm:sgbd}
Let $Y_i\sim N(0,\Sigma_i)$. Then we have
\begin{multline*}
	\EE\|S-\EE S\| \le
	2\bigg\|
	\sum_{i=1}^n \tr[\Sigma_i]\Sigma_i
	\bigg\|^{\frac{1}{2}} +
	\max_{i\le n}\tr \Sigma_i
	\\ +
	C
	\bigg(\bigg\|\sum_{i=1}^n\Sigma_i\Bigg\|+
	\max_{i\le n}\tr\Sigma_i\bigg)^{\frac{3}{4}}
	\max_{i\le n} \|\Sigma_i\|^{\frac{1}{4}} 
	\log^{\frac{3}{2}}(d+n).
\end{multline*}
\end{thm}

\begin{proof}
By \cite[Lemma 3.8]{BBV21}, we have $\sigma(Y)^2 = \|\sum_i 
\Sigma_i\|\vee \max_i\tr\Sigma_i$ and $v(Y)^2=\max_i\|\Sigma_i\|$. As 
$v(Y)\le\sigma(Y)$, \cite[Theorem 3.11 and Proposition 3.12]{BBV21} 
yield
$$
	\EE\|S-\EE S\| \le
	2\|\EE[Y\,\EE[Y^*Y]\,Y^*\|^{\frac{1}{2}} + \|\EE Y^*Y\| +
	C\sigma(Y)^{\frac{3}{2}}v(Y)^{\frac{1}{2}}\log^{\frac{3}{2}}(d+n)
$$
for a universal constant $C$. The leading terms are readily computed.
\end{proof}

On the other hand, even when $Y$ is Gaussian, we may view $S$ as a 
non-Gaussian random matrix of the form \eqref{eq:model} with 
$Z_i=Y_iY_i^*$, to which the universality principles of this paper may be 
applied. For example, applying Theorem \ref{thm:momentuniv} yields the 
following bound, whose proof is given in section \ref{sec:pfsguniv}.

\begin{thm}[$S$-universality bound]
\label{thm:sguniv}
Let $Y_i\sim N(0,\Sigma_i)$. Then for any $\varepsilon\in (0,1]$
$$
	\EE\|S-\EE S\| \le
	(1+\varepsilon)\,
	2\bigg\|
	\sum_{i=1}^n \tr[\Sigma_i]\Sigma_i
	\bigg\|^{\frac{1}{2}} 
	+
	\frac{C}{\varepsilon^3}\bigg(
	\bigg\|\sum_{i=1}^n \Sigma_i^2\bigg\|^{\frac{1}{2}}
	+
	\max_{i\le n}\tr\Sigma_i\bigg)
	\log^3(d+n).
$$
\end{thm}

The fundamental distinction between these bounds is that Theorem
\ref{thm:sguniv} models $S$ by the noncommutative model $S_{\rm free}$, 
while Theorem \ref{thm:sgbd} models $S=YY^*$ by the noncommutative 
model $Y_{\rm free}Y_{\rm free}^*$. Somewhat surprisingly, these distinct 
interpretations have complementary (partially overlapping) domains of 
validity, which is already illustrated by the simplest possible example.

\begin{example}
\label{ex:baiyinscov}
Suppose that $Y_1,\ldots,Y_n$ are i.i.d.\ standard Gaussian vectors in 
$\mathbb{R}^d$, that is, $\Sigma_i=\id$ for all $i$. In this setting,
the classical Bai-Yin law \cite{BY93} implies that
$\EE\|S-\EE S\| = (1+o(1))(2\sqrt{nd}+d)$ when $n,d\to\infty$ with
$\frac{n}{d}$ fixed.

Let us now verify what Theorems \ref{thm:sgbd} 
and \ref{thm:sguniv} yield for this model.
\begin{enumerate}[$\bullet$]
\itemsep\medskipamount
\item First, note that the Gaussian bound of Theorem \ref{thm:sgbd} 
yields
$$
	\EE\|S-\EE S\| \le 2\sqrt{nd} + d +
	C(n+d)^{\frac{3}{4}}\log^{\frac{3}{2}}(d+n).
$$
Here the leading terms agree with the Bai-Yin law, but the error term
is of smaller order if and only if
$n\to\infty$ and $d\gg n^{\frac{1}{2}}(\log n)^3$. This includes the 
$n\propto d$ setting of the classical Bai-Yin law,
but excludes cases where $n$ is much larger than $d$.
\item
On the other hand, the universality bound of Theorem \ref{thm:sguniv} 
yields
$$
	\EE\|S-\EE S\| \le (1+\varepsilon)\,2\sqrt{nd} +
	C\varepsilon^{-3}(\sqrt{n}+d)\log^3(d+n).
$$
In this bound, the leading term agrees with the Bai-Yin law only when
$n\gg d$, and the error term is of smaller order if and only if 
$(\log n)^6\ll d\ll \frac{n}{(\log n)^6}$. This regime excludes the 
setting of the classical Bai-Yin law, but covers precisely the situation 
that the Gaussian bound fails to capture.
\end{enumerate}
Combining the above bounds yields $\EE\|S-\EE S\|\le 
(1+o(1))(2\sqrt{nd}+d)$ whenever $n\to\infty$ and $d\gg(\log n)^6$. This 
very general conclusion hides the fact that two complementary approaches 
were used to capture the large $d$ and large $n$ regimes.
\end{example}

The homogeneous setting of Example \ref{ex:baiyinscov} is special in that 
$\Sigma_i=\id$ implies $\|S-\EE S\| = \|YY^*-n\id\| = 
\max\{\|Y\|^2-n,n-\mathrm{s_{min}}(Y^*)^2\}$, so that this case can also 
be approached using the methods of section \ref{sec:appsmsing}. Such a 
reduction fails, however, for nonhomogeneous sample covariance matrices. 
In the general setting, Theorems \ref{thm:sgbd} and \ref{thm:sguniv} 
control the behavior of Gaussian sample covariance matrices in 
complementary regimes that together span a wide range of parameters.

\begin{rem}
Let us note for completeness that Gaussian sample covariance matrices 
always satisfy $\EE\|S-\EE S\|\gtrsim \|\sum_i 
\tr[\Sigma_i]\Sigma_i\|^{\frac{1}{2}}+\max_i\tr\Sigma_i$, cf.\ section 
\ref{sec:pfsimplelower}. Thus the leading terms in Theorems
\ref{thm:sgbd} and \ref{thm:sguniv} are also lower bounds on $\EE\|S-\EE 
S\|$ up to a universal constant. On the other hand, the proofs of these 
results can even capture the sharp leading term predicted by free 
probability.
\end{rem}

\subsubsection{Non-Gaussian models}

We now consider more general models where the data $Y_1,\ldots,Y_n$ may be
non-Gaussian. This makes little difference in the setting of Theorem 
\ref{thm:sguniv}: here we already interpreted $S$ itself as a non-Gaussian 
matrix, and applied the universality principle to compare it with its 
Gaussian model (the assumption that $Y$ is Gaussian was not used in a 
fundamental way in the proof).

On the other hand, in order to extend Theorem \ref{thm:sgbd} to the 
non-Gaussian setting, we must compare $S=YY^*$ with $HH^*$, where $H$ is 
the Gaussian model associated to $Y$. Such a comparison can be deduced 
from our universality principles by means of a linearization argument as 
in \cite[\S 3.3]{BBV21}. Note that the setting of the following result is 
far more general than that of the model \eqref{eq:smtx}.

\begin{thm}[$Y$-universality]
\label{thm:scovlin}
Let $Y=Z_0+\sum_{i=1}^n Z_i$ be a $d\times m$ random matrix defined as in 
Theorem \ref{thm:ssing}, and let $H$ be the $d\times m$ random matrix
so that the real and imaginary parts of its entries are jointly Gaussian
with the same mean and covariance as the real and imaginary parts 
of the entries of $Y$. Then
$$
        \big|\EE\|YY^*-\EE YY^*\| -
        \EE\|HH^*-\EE HH^*\| \big| \lesssim
        \delta\,\EE\|H\| + \delta^2
$$
with
$$
        \delta =
        \sigma_*(Y)\log^{\frac{1}{2}}(d+m) +
        R(Y)^{\frac{1}{3}}\sigma(Y)^{\frac{2}{3}}
	\log^{\frac{2}{3}}(d+m)
        +R(Y)\log(d+m).
$$
\end{thm}

The proof of Theorem \ref{thm:scovlin} is in section 
\ref{sec:pfscovlin}. We state the result for bounded random matrices for 
simplicity; similar results for unbounded matrices are obtained by using 
Theorem \ref{thm:specheavy} rather than Theorem \ref{thm:specuniv} in the 
proof.

As $Y$ is already expressed as a sum of independent random matrices in 
\eqref{eq:ymtx}, it is tempting to attempt to apply Theorem 
\ref{thm:scovlin} with $Z_i=Y_ie_i^*$. Unfortunately, as is illustrated in 
the following example, such a straightforward application of the 
universality principle fails to yield meaningful results. The reason is 
simple: \eqref{eq:ymtx} captures only the independence of the data 
$Y_1,\ldots,Y_n$, but independent data alone does not suffice to ensure 
universality of sample covariance matrices.

\begin{example}
\label{ex:baiyinsecond}
Let us revisit the Bai-Yin setting of Example \ref{ex:baiyinscov} in the 
non-Gaussian case, that is, we now assume only that $Y_1,\ldots,Y_n$ are 
i.i.d.\ centered random vectors with unit covariance matrix $\Sigma_i=\id$ 
for all $i$.

In this setting, the Gaussian model $H$ associated to $Y$ in 
Theorem \ref{thm:scovlin} is precisely the $d\times n$ random matrix with 
i.i.d.\ standard Gaussian entries. In particular, the classical Bai-Yin 
law \cite{BY93} implies that $\EE\|HH^*-\EE HH^*\| = 
(1+o(1))(2\sqrt{\gamma}+1)d$ when $n,d\to\infty$ with $\frac{n}{d}=\gamma$ 
fixed. On the other hand, if we write $Y=\sum_{i=1}^n Z_i$ with
$Z_i=Y_ie_i^*$, we clearly have $R(Y)^2\ge \max_i \EE\|Z_i\|^2 = d$, so 
that
$$
	\delta\, \EE\|H\|+\delta^2
	\gtrsim
	\frac{(\log d)^2}{2\sqrt{\gamma}+1}\,
	\EE\|HH^*-\EE HH^*\|.
$$
Thus Theorem \ref{thm:scovlin} cannot yield universality of the 
Bai-Yin law in this manner, as its error term is always of larger order
than the Gaussian quantity of interest.

The problem that arises here is not an inefficiency of our universality
principles, however, but is a genuine phenomenon: at the 
present level of generality, universality of the Bai-Yin law is simply 
false. For example, let $Y_i=\sqrt{d}\,\varepsilon_i
e_{I_i}$, where $\varepsilon_1,\ldots,\varepsilon_n$ are i.i.d.\
random signs and $I_1,\ldots,I_n$ are i.i.d.\ 
uniformly distributed variables on $[d]$. Then $Y_1,\ldots,Y_n$ are 
i.i.d.\ with zero mean and unit covariance. However, as $YY^*$ is diagonal 
with multinomially distributed diagonal entries, we have \cite{RS98}
$$
        \EE\|YY^*-\EE YY^*\| = (1+o(1))
        \frac{d\log d}{\log(\gamma^{-1} \log d)}
	\gg
	\EE\|HH^*-\EE HH^*\|
$$
when $n,d\to\infty$ with $\frac{n}{d}=\gamma$ fixed. Thus universality
of the Bai-Yin law fails.
\end{example}

Example \ref{ex:baiyinsecond} illustrates that even in the classical 
setting of Bai-Yin law, some additional assumption on the distribution of 
the vectors $Y_i$ is needed to achieve universality. The additional 
assumption that would enable us to apply Theorem~\ref{thm:scovlin} is that 
each $Y_i$ is itself a sum of independent random vectors of sufficiently 
small norm. This situation arises naturally in random matrix theory: we 
presently provide one example of such a model, where the above results 
yield a considerable improvement on the best known nonasymptotic bounds.

\begin{example}[Product of random and deterministic matrices]
\label{ex:embed}
Let $A$ be a $N\times n$ random matrix with independent real entries that
have zero mean and unit variance, and let $B$ be a $d\times N$ 
nonrandom matrix. We are interested in the sample covariance
matrix $S=YY^*$ where $Y=BA$. The difficulty of analyzing such models
is that even though $A$ has independent entries, the matrix $Y$ generally 
has highly dependent entries which renders many standard tools of 
nonasymptotic random matrix theory inapplicable.
Here we obtain the following.
\end{example}

\begin{thm}
\label{thm:versh}
Let $S=YY^*$ with $Y=BA$, where $A$ is a $N\times n$ random matrix with
$\EE[A_{ij}]=0$, $\mathrm{Var}(A_{ij})=1$, and 
$\|A_{ij}\|_\infty\le\alpha$, and let $B$ be a $d\times N$ 
nonrandom matrix. Assume that $\alpha\le\sqrt{n}$ and $\|B\|_{\rm 
HS}\ge\alpha\|B\|$. Then we have
\begin{multline*}
	\EE\|S-\EE S\| \le \\
	\bigg(
	1+
	C
	\bigg\{
	\bigg(\frac{\alpha}{\sqrt{n}}\bigg)^{\frac{1}{15}}
	+
	\bigg(\frac{\alpha\|B\|}{\|B\|_{\rm HS}}
	\bigg)^{\frac{1}{4}}
	\bigg\}
	\log^3(d+n)
	\bigg)
	\big(
	2\|B\|_{\rm HS}\|B\|\sqrt{n}
	+ \|B\|_{\rm HS}^2
	\big),
\end{multline*}
where $C$ is a universal constant, $\EE S =n BB^*$, and
$\|M\|_{\rm HS}^2 := \tr |M|^2$.
\end{thm}

The proof is given in section \ref{sec:pfversh}. While we 
have formulated a single bound, it should be emphasized that the proof is 
once again a combination of two distinct universality principles. (We have 
made no effort to optimize the exponents in the lower-order terms, which 
are not expected to be optimal.)

It is instructive to compare Theorem \ref{thm:versh} with previous 
nonasymptotic results in this setting, which establish bounds analogous to 
Theorem \ref{thm:versh} up to a multiplicative factor that depends on the 
moments of $A_{ij}$: see \cite{Zhi22} and the references therein 
for subgaussian or subexponential entries, and \cite{Ver11} for a slightly 
weaker result for entries with bounded fourth moment. The advantage of 
Theorem \ref{thm:versh} is twofold. First, Theorem 
\ref{thm:versh} reproduces the correct leading-order behavior in the 
Bai-Yin law (i.e., the case $N=d,B=\id$), while previous results lose at 
least a multiplicative factor. Second, Theorem \ref{thm:versh} is 
applicable to sparse random matrices, while previous bounds are 
fundamentally inefficient in the sparse setting.

To illustrate this point, suppose that $A_{ij}$ are symmetric Bernoulli 
variables with $\mathbf{P}[A_{ij}=0]=1-p$ and 
$\mathbf{P}[A_{ij}=p^{-\frac{1}{2}}]=\frac{p}{2}$. Then Theorem 
\ref{thm:versh} yields 
$$
	\EE\|S-\EE S\| \le (1+o(1))\big(2\|B\|_{\rm 
	HS}\|B\|\sqrt{n} + \|B\|_{\rm HS}^2\big)
	\quad\mbox{for}\quad
	p\gg \frac{\log^\beta(d+n)}{n\wedge r}
$$
for a suitable $\beta$, where 
$r=\|B\|_{\rm HS}^2\|B\|^{-2}$ is the effective rank of $B$.
On the other hand, as $\EE |A_{ij}|^4 = \frac{1}{p}$ diverges
as soon as $p\to 0$, the results of \cite{Ver11,Zhi22} fail to achieve
even the correct order of magnitude of the norm in the sparse setting.

\begin{rem}
We have formulated Theorem \ref{thm:versh} for the case
that the entries $A_{ij}$ are uniformly bounded, while prior results
\cite{Ver11,Zhi22} also consider unbounded entries. However, our
restriction to bounded entries was made for simplicity of exposition only,
and is not a fundamental restriction of the proof of Theorem 
\ref{thm:versh}. A related inequality for the unbounded case may be
found in Remark \ref{rem:fourmomsc}.
\end{rem}

While Example \ref{ex:embed} provides a natural model where each $Y_i$ is 
itself a sum of independent random vectors, such an assumption can be 
restrictive for more general models of sample covariance matrices. On the 
other hand, in the special (homogeneous) setting of the Bai-Yin law, it 
was shown in \cite{CT18} that a much weaker assumption suffices to achieve 
universal behavior: in this case one need only assume that each $Y_i$ 
satsifies certain concentration of measure properties, which rules out the 
counterexample of Example \ref{ex:baiyinsecond}. In forthcoming work 
\cite{PvH23}, the universality principles of this paper are further 
refined to capture such concentration assumptions for general 
(nonhomogeneous) sample covariance matrices.

More generally, the above considerations highlight the broader question 
whether the universality principles of this paper extend to random 
matrices that admit more general dependence structures than can be 
captured by the model \eqref{eq:model}; such principles could enable the 
analysis of natural models that are outside the scope of this 
paper. Progress in this direction may be found in \cite{PvH23,SvW23}.

\subsection{Strong asymptotic freeness}
\label{sec:appsaf}

The celebrated asymptotic freeness theorem of Voiculescu \cite{Voi91} 
states that if $X_1^N,\ldots,X_m^N$ are independent $N\times N$ Wigner 
matrices and $s_1,\ldots,s_m$ is a free semicircular family (i.e., a 
semicircular family as in Definition \ref{defn:semi} with zero mean and 
unit covariance), then 
$$
	\lim_{N\to\infty}
	\ntr p(X_1^N,\ldots,X_m^N) = \tau(p(s_1,\ldots,s_m))
	\quad\mbox{a.s.}
$$
for every noncommutative polynomial $p$. This makes it possible to 
compute the limiting spectral distributions of polynomials Wigner 
matrices using tools of free probabilty; see, e.g., \cite[Chapter 5]{AGZ10}.
That the convergence holds also in norm
$$
	\lim_{N\to\infty}
	\|p(X_1^N,\ldots,X_m^N)\| = \|p(s_1,\ldots,s_m)\|
	\quad\mbox{a.s.}
$$
is a deep result of Haagerup and Thorbj{\o}rnsen \cite{HT05}, who proved 
it for GUE matrices. The latter strong asymptotic 
freeness property is of fundamental importance both to 
random matrices and in the theory of operator algebras.

The methods of \cite{HT05} are rather delicate, and their extension even 
to random matrices with i.i.d.\ entries with bounded fourth moment 
requires considerable effort \cite{And13}. It was therefore long unclear 
whether the strong asymptotic freeness pheonomenon could be expected to 
hold in the absence of strong symmetry assumptions. That this is indeed 
the case is a notable application of the sharp matrix concentration theory 
of \cite{BBV21}, which made it possible to establish strong asymptotic 
freeness of an extremely general class of Gaussian random matrix models.

Here we extend the latter results to an even more general family of 
non-Gaussian random matrices. To this end, we must show that our 
universality principles for random matrices of the form \eqref{eq:model} 
imply universality for polynomials of such matrices. In the following 
(asymptotic) result, whose proof is given in section \ref{sec:pfsaf}, this 
is accomplished by a direct application of known linearization arguments 
\cite{HT05,DlS10} that we use as a black box. However, while we do not 
develop this direction systematically in this paper, our methods can also 
be used to obtain nonasymptotic bounds for polynomials of random matrices: 
for example, Theorem \ref{thm:scovlin} may be viewed as a result of this 
kind for a certain quadratic polynomial.

\begin{thm}[Strong asymptotic freeness]
\label{thm:saf}
Let $s_1,\ldots,s_m$ be a free semicircular family.
For each $N\ge 1$, let
Let $H_1^N,\ldots,H_m^N$ be independent self-adjoint random matrices of 
dimension $d_N\ge N$ defined by
$$
	H_k^N = Z_{k0}^N + \sum_{i=1}^{M_N} Z_{ki}^N,
$$
where $Z_{k0}^N$ is a deterministic self-adjoint matrix and
$Z_{k1}^N,\ldots,Z_{kM_N}^N$ are independent self-adjoint random matrices 
with zero mean. Suppose that
$$
	\lim_{N\to\infty}\|\EE[H_k^N]\|=
	\lim_{N\to\infty}\|\EE[(H_k^N)^2]-\id\|=
	\lim_{N\to\infty}\bar R(H_k^N)=0
$$
and that
$$
	\lim_{N\to\infty}
	(\log d_N)^{\frac{3}{2}} 
	v(H_k^N) = 0,\qquad
	\lim_{N\to\infty} (\log d_N)^{2}\max_{1\le i\le M_N}\|Z_{ki}^N\|
	=0\quad\mbox{a.s.}
$$
for every $1\le k\le m$. Then
\begin{align*}
	\lim_{N\to\infty}\ntr p(H_1^N,\ldots,H_m^N) =
	\tau(p(s_1,\ldots,s_m))\quad\mbox{a.s.},\\
	\lim_{N\to\infty}\|p(H_1^N,\ldots,H_m^N)\| =
	\|p(s_1,\ldots,s_m)\|\quad\mbox{a.s.}
\end{align*}
for every noncommutative polynomial $p$.
Moreover, the analogous result holds if a.s.\ 
convergence is replaced by convergence in probability.
\end{thm}

Theorem \ref{thm:saf} applies to a large family of random matrices with 
non-Gaussian, nonhomogeneous, and dependent entries. In order to 
illustrate some characteristic features of this result, let us develop one 
example in more detail.

\begin{example}[Sparse Wigner matrices]
\label{ex:sparsew}
We consider matrices with a deterministic sparsity pattern, where all 
nonzero entries of the matrix are i.i.d. We emphasize that
such random matrices may be highly nonhomogeneous.
\end{example}

\begin{defn}
Let $(\eta_{ij})_{1\le i\le j<\infty}$ be i.i.d.\ real-valued random 
variables with zero mean and unit variance, and let
$\mathrm{G}=([d],E)$ be a $k$-regular graph with $d$ vertices.
Then the \emph{$(\mathrm{G},\eta)$-sparse Wigner matrix} is the $d\times 
d$ self-adjoint random matrix $X$ with entries $X_{ij}=k^{-\frac{1}{2}}
\eta_{ij}1_{\{i,j\}\in E}$ for $1\le i\le j\le d$.
\end{defn}

The proof of of the following result is given in section 
\ref{sec:pfcorsparse}.

\begin{cor}
\label{cor:sparsesaf}
Let $(\eta_{rij})_{1\le r\le m,1\le i\le j<\infty}$ be i.i.d.\ centered
random variables with unit variance and $\EE[|\eta_{kij}|^p]<\infty$
for some $p>2$, and 
let $\mathrm{G}_N$ be a $k_N$-regular graph with $d_N\ge N$ 
vertices. Let $H_r^N$ be the 
$(\mathrm{G}_N,\eta_r)$-sparse Wigner matrix.
\begin{enumerate}[a.]
\itemsep\medskipamount
\item If $k_N\gg d_N^{\frac{2}{p-2}}(\log d_N)^{\frac{4p}{p-2}}$, then 
\begin{align*}
	\lim_{N\to\infty}\ntr p(H_1^N,\ldots,H_m^N) =
	\tau(p(s_1,\ldots,s_m))~~\mbox{in probability},\\
	\lim_{N\to\infty}\|p(H_1^N,\ldots,H_m^N)\| =
	\|p(s_1,\ldots,s_m)\|~~\mbox{in probability}
\end{align*}
for every noncommutative polynomial $p$.
If the graphs $\mathrm{G}_N=([d_N],E_N)$ are increasing
\emph{(}i.e., $E_N\subseteq E_{N+1}$ for all $N$\emph{)}, the 
convergence also holds a.s.
\item If $k_N\ll d_N^{\frac{2}{p-2}}(\log d_N)^{-\frac{2p}{p-2}}$, 
then the conclusion of part a.\ must fail for some entry distribution
satisfying the assumptions.
\end{enumerate}
\end{cor}

A fundamental phenomenon that is captured by this result is that strong 
asymptotic freeness requires a tradeoff between sparsity and integrability 
of the entries. For dense Wigner matrices $k_N=d_N$, we obtain strong 
asymptotic freeness as soon as the entries have $4+\varepsilon$ moments 
for some $\varepsilon>0$, as was previously shown in \cite{And13}. On the 
other hand, as we bound more moments, increasingly sparse random matrices 
can still achieve strong asymptotic freeness. This tradeoff is captured 
nearly optimally by Corollary \ref{cor:sparsesaf}, up to logarithmic 
factors.

In the opposite extreme, when the entries $\eta_{rij}$ are uniformly 
bounded, it follows directly from Theorem \ref{thm:saf} that strong 
asymptotic freeness holds (in the a.s.\ sense) as soon as $k_N\gg(\log 
d_N)^4$. By taking $\eta_{rij}$ to be symmetric Bernoulli variables, this 
shows that one can construct $d\times d$ random matrices with independent 
entries that achieve strong asymptotic freeness using only $O(d\log^5 d)$ 
bits of randomness.

\begin{rem}
The sparse Wigner model of Example \ref{ex:sparsew} is only one special 
case of the very general setting captured by Theorem \ref{thm:saf}, which 
also includes several of the examples that were discussed in the previous 
sections: e.g., random matrices defined by group representations 
as in section \ref{sec:appgraph}, or centered adjacency matrices of sparse 
Erd\H{o}s-R\'enyi graphs. Such examples further extend the scope of the 
strong asymptotic freeness phenomenon beyond what was previously known.
\end{rem}

\subsection{Phase transitions in spiked models}
\label{sec:appspiked}

A widely studied phenomenon in random matrix theory, which dates back to 
the work of Baik, Ben Arous and P\'ech\'e \cite{BBP05}, is that low-rank 
perturbations of random matrices (known as ``spiked'' models) give rise to 
phase transitions: small perturbations do not affect the limiting 
eigenvalue statistics, while large perturbations give rise to the 
appearance of outlier eigenvalues. Several closely related forms of this 
phenomenon have been investigated by many authors; see, e.g., the survey 
\cite{CD17}. For sake of illustration, we focus here on the following 
prototypical phenomenon of this kind.\footnote{%
	The assumption that $A_d$ is positive semidefinite is made here
	exclusively to simplify the notation; any negative 
	eigenvalues of $A_d$ exhibit a completely analogous transition at 
	$\theta_i=-1$.
}

\begin{thm}[BBP transition for spiked GOE \cite{BN11}]
\label{thm:bbp}
Let $G_d$ be a $d\times d$ self-adjoint 
random matrix whose entries $(G_{dij})_{i\ge j}$ are independent real 
Gaussian variables with mean $0$ and variance $\frac{1+1_{i=j}}{d}$. Let 
$A_d$ be a nonrandom $d\times d$ positive semidefinite matrix of rank $r$ 
whose eigenvalues $\theta_1\ge\cdots\ge \theta_r>0$ are independent of $d$.
\begin{enumerate}[a.]
\itemsep\medskipamount
\item We have for $1\le i\le r$
$$
	\lambda_i(A_d+G_d) \xrightarrow[\text{a.s.}]{d\to\infty}
	\begin{cases} 
	\theta_i + \frac{1}{\theta_i} & \mbox{for }\theta_i>1,\\
	2 &\mbox{for }\theta_i\le 1,
	\end{cases},\qquad
	\lambda_{r+1}(A_d+G_d) \xrightarrow[\text{a.s.}]{d\to\infty} 2,
$$
where $\lambda_1(M)\ge\cdots\ge\lambda_d(M)$ are the eigenvalues
of $M$.
\item For every $1\le i, j\le r$ such that $\theta_i>1$ and
$\theta_j\ne\theta_i$, we have
$$
	\|P_i(A_d)v_i(A_d+G_d)\|^2 \xrightarrow[\text{a.s.}]{d\to\infty}
	1-\frac{1}{\theta_i^2},
	\qquad
	\|P_j(A_d)v_i(A_d+G_d)\|^2 \xrightarrow[\text{a.s.}]{d\to\infty}0,
$$
where $P_i(M)$ is the projection on the eigenspace of $M$ 
associated to the eigenvalue $\lambda_i(M)$, and $v_i(M)$ is any 
unit norm eigenvector of $M$ with eigenvalue $\lambda_i(M)$.
\end{enumerate}
\end{thm}

We aim to understand whether the phenomena described in Theorem 
\ref{thm:bbp} are universal: do the conclusions remain valid if the GOE 
matrix $G_n$ is replaced by another random matrix $H_d$ whose entries have 
the same mean and covariance? Previous results have extended Theorem 
\ref{thm:bbp} to the setting where $H_d$ has i.i.d.\ entries above the 
diagonal under distributional assumptions that require at least some 
bounded moments of higher order, cf.\ \cite{CD17} and the references 
therein. (Analogues of Theorem \ref{thm:bbp} are also known to hold
for non-Gaussian homogeneous models where $H_d$ is invariant under a
symmetry group; such models are rather different in spirit from the
kind of universality phenomena considered here.)

When applied to this setting, our universality 
principles can capture many new situations, including sparse and 
dependent models.

\begin{thm}
\label{thm:spiked}
Let $G_d,A_d$ be as in Theorem \ref{thm:bbp}, and let $H_d$ be any $d\times 
d$ self-adjoint real random matrix of the form \eqref{eq:model} whose 
entries have the same mean and covariance as those of $G_d$. Suppose that 
$(\log d)^2R(H_d)\to 0$ as $d\to\infty$. Then all the conclusions of Theorem 
\ref{thm:bbp} remain valid if $G_d$ is replaced by $H_d$.
\end{thm}

The proof of this result is given in section \ref{sec:pfspiked}. Let us 
however briefly outline the main ingredients of the proof. On the one hand, 
Theorem \ref{thm:specuniv} shows that the eigenvalues of $A_d+H_d$ 
concentrate at the locations predicted by Theorem \ref{thm:bbp}. On the 
other hand, for $\theta_i>1$, let $\varphi_i$ be a mollification of the 
indicator function of a small interval around 
$\theta_i+\frac{1}{\theta_i}$. Then $\varphi_i(A_d+H_d)$ coincides with 
high probability with the projection onto the linear span of the 
eigenvectors of $A_d+H_d$ whose eigenvalues concentrate at 
$\theta_i+\frac{1}{\theta_i}$. We can therefore apply the second part of 
Theorem \ref{thm:smuniv} to establish universality of these 
eigenprojections.

To illustrate Theorem \ref{thm:spiked}, we briefly discuss
one simple example.

\begin{example}[Planted clique in the permutation model]
Let $X_d$ be the adjacency matrix of a random $2k_d$-regular graph with 
$d$ vertices in the permutation model defined in section 
\ref{sec:rregdet}, where $(\log d)^4\ll k_d\ll d^2$. Choose a subset 
$E_d\subset[d]$ of vertices
so that $|E_d|=(1+o(1)) \theta \sqrt{2k_d}$. Then 
$1_{E_d}1_{E_d}^*+X_d$ is the adjacency matrix of the random 
graph in which we planted a clique with vertices $E_d$.

By Lemma \ref{lem:groupthy}, the random matrix 
$(2k_d)^{-\frac{1}{2}}X_d^\perp$, where $X_d^\perp$ is the restriction of 
$X_d$ to $1^\perp$, has the same mean and covariance as a GOE matrix of 
dimension $d-1$. Furthermore, the assumptions on $k_d$ and $E_d$ imply that 
there exist unit vectors $v_d\in 1^\perp$ so that 
$\|(2k_d)^{-\frac{1}{2}}1_{E_d}1_{E_d}^*-\theta v_dv_d^*\|\to 0$. Thus 
applying Theorem \ref{thm:spiked} with 
$H_{d-1}=(2k_d)^{-\frac{1}{2}}X_d^\perp$ and $A_{d-1}=\theta v_dv_d^*$ 
shows that the adjacency matrix of the planted model has an outlier 
eigenvalue (beside its Perron-Frobenius eigenvalue) if and only if 
$\theta>1$. In other words, the detectability of a planted clique by an 
outlier in the spectrum exhibits a phase transition at $|E_d| = 
\sqrt{2k_d}$.

Let us emphasize that the random matrices that arise in this example are 
both dependent and may be highly sparse. (For the classical study of 
spectral detection of planted cliques in dense Erd\H{o}s-R\'enyi graphs, 
see \cite{AKS98}.)
\end{example}

\begin{rem}
Even if we consider $H_d$ with i.i.d.\ entries, sparse 
matrices are not captured by previous extensions of Theorem 
\ref{thm:bbp} as their entries have unbounded moments of order $p>2$ (cf.\ 
Example \ref{ex:embed}). Some results for sparse 
matrices were obtained very recently in \cite{LM22}, but rely on a special 
choice of $A_d$.
\end{rem}

In this section we have used the classical Gaussian result of Theorem 
\ref{thm:bbp} as input for the universality theory of this paper. However, 
much more general results can be obtained in the Gaussian setting by 
applying the sharp matrix concentration theory of \cite{BBV21}. This 
approach has two key advantages: it is nonasymptotic, and it yields 
analogous phenomena in nonhomogeneous situations. The latter are of 
particular interest in many applications, but are much less well 
understood than the homogeneous setting \cite{BM21,LL22,LM22,Au23}. 
A detailed study of phase transitions in nonhomogeneous models using
the methods of \cite{BBV21} and of this paper appears in
\cite{BCSV23}.

\section{The cumulant method}
\label{sec:cumulant}

The aim of this section is to introduce the basic device that we 
will use to prove universality throughout this paper. The general setting 
that will be considered in this section is the following. Let 
$Y_1,\ldots,Y_n$ be independent random vectors in $\mathbb{R}^N$, and 
let $U_1,\ldots,U_n$ be independent Gaussian random vectors such 
that 
$Y_i$ and $U_i$ have the same mean and covariance. Given a function 
$f:\mathbb{R}^{Nn}\to\mathbb{C}$, we aim to bound the deviation from the 
Gaussian model 
$$
	\Delta := \EE[f(Y_1,\ldots,Y_n)]-\EE[f(U_1,\ldots,U_n)].
$$
There are various classical approaches to such problems. For example, the 
Lindeberg method replaces $Y_i$ by $U_i$ one term at a time, and then uses 
Taylor expansion to third order to control the error of each term; similar 
bounds arise from Stein's method \cite[\S 5]{Cha14}. Unfortunately, in the 
setting of this paper such methods appear to give rise to very poor 
bounds. For example, in the context of Theorem \ref{thm:momentuniv}, 
classical methods yield bounds where the parameter 
$\sigma_q(X)^2\le\sigma(X)^2:=\|\sum_{i=1}^n \EE Z_i^2\|$ is replaced by 
at least $\sum_{i=1}^n \EE\|Z_i\|^2$, which is typically much larger.

The reason for the inefficiency of classical approaches to universality is 
that they require the independent variables to be bounded term by term. In 
the present setting, bounding the contribution of each summand $Z_i$ in 
\eqref{eq:model} separately ignores the noncommutativity of the summands. 
To surmount this problem, we will work instead with an exact formula for 
the deviation $\Delta$ in terms of a series expansion in the cumulants of 
the underlying variables. For our purposes, the advantage of this exact 
formula is that it will enable us to keep the summands $Z_i$ together, and 
estimate the resulting terms efficiently using trace inequalities without 
destroying their noncommutativity. The price we pay for this is that we 
must expand the deviation $\Delta$ to high order in order to obtain 
efficient estimates.

In the univariate case $N=1$, the cumulant expansion dates back to the 
work of Barbour \cite{Bar86}, and has been routinely applied to the study 
of random matrices with independent entries since the work of Lytova and 
Pastur \cite{LP09}. In the remainder of this section, we recall the 
relevant arguments of \cite{Bar86,LP09} and spell out their immediate
extension to the multivariate case $N>1$.

\subsection{Cumulants}
\label{sec:cumdef}

Let $W_1,\ldots,W_m$ be bounded real-valued random variables. Then 
their log-moment generating function is analytic with power 
series expansion
$$
	\log \EE[e^{\sum_{i=1}^m t_iW_i}]
	 =
	\sum_{k=1}^\infty 
	\sum_{j_1,\ldots,j_k=1}^m
	\frac{1}{k!}\,
	\kappa(W_{j_1},\ldots,W_{j_k})
	\,t_{j_1}\cdots t_{j_k}.
$$
The coefficient $\kappa(W_1,\ldots,W_k)$ is called the \emph{joint 
cumulant} of the random variables $W_1,\ldots,W_k$. Joint cumulants are 
multilinear in their arguments and invariant under permutation of their
arguments. Moreover, for jointly Gaussian random variables, all joint 
cumulants of order $k\ge 3$ vanish.

For any subset $J\subseteq[m]:=\{1,\ldots,m\}$, denote by 
$W_J:=(W_j)_{j\in J}$ the associated subset of random variables. 
Moreover, denote by $\mathrm{P}([m])$ the collection of all partitions 
of $[m]$. The following fundamental result \cite[Proposition 3.2.1]{PT11} 
expresses the relation between joint cumulants and moments.

\begin{lem}[Leonov-Shiryaev]
\label{lem:shir}
We can write
$$
	\EE[W_1\cdots W_m] =
	\sum_{\pi\in\mathrm{P}([m])}
	\prod_{J\in\pi}\kappa(W_J).
$$
Conversely, we have
$$
	\kappa(W_1,\ldots,W_m) =
	\sum_{\pi\in\mathrm{P}([m])}
	(-1)^{|\pi|-1}(|\pi|-1)!
	\prod_{J\in\pi}
	\EE\Bigg[\prod_{j\in J}W_j\Bigg].
$$
\end{lem}

The significance of cumulants for our purposes is the following identity. 
The univariate ($m=1$) case was proved in \cite[Lemma 1]{Bar86} and 
\cite[Proposition 3.1]{LP09}; the multivariate case follows precisely in 
the same manner.

\begin{lem}
\label{lem:ibp}
For any polynomial $f:\mathbb{R}^m\to\mathbb{C}$ and $i\in[m]$, we have
\begin{multline*}
	\EE[W_if(W_1,\ldots,W_m)] = \\
	\sum_{k=0}^\infty
	\sum_{j_1,\ldots,j_k=1}^m
	\frac{1}{k!}\,
	\kappa(W_i,W_{j_1},\ldots,W_{j_k})\,
	\EE\bigg[
	\frac{\partial^k f}{\partial x_{j_1}\cdots\partial x_{j_k}}
	(W_1,\ldots,W_k)
	\bigg].
\end{multline*}
\end{lem}

\begin{proof}
Let $\varphi(x_1,\ldots,x_m):= e^{\sum_{j=1}^m t_j x_j}$. Then
\begin{align*}
	&\EE[W_i \varphi(W_1,\ldots,W_m)] =
	\EE[\varphi(W_1,\ldots,W_m)]
	\frac{\partial}{\partial t_i}
	\log\EE[e^{\sum_{j=1}^m t_j W_j}] \\
	&=
	\sum_{k=0}^\infty 
	\sum_{j_1,\ldots,j_k=1}^m
	\frac{1}{k!}\,
	\kappa(W_i,W_{j_1},\ldots,W_{j_k})
	\,t_{j_1}\cdots t_{j_k}\,
	\EE[\varphi(W_1,\ldots,W_m)] \\
	&=
        \sum_{k=0}^\infty
        \sum_{j_1,\ldots,j_k=1}^m
        \frac{1}{k!}\,
        \kappa(W_i,W_{j_1},\ldots,W_{j_k})\,
        \EE\bigg[
        \frac{\partial^k \varphi}{\partial x_{j_1}\cdots\partial x_{j_k}}
        (W_1,\ldots,W_k)
        \bigg].
\end{align*}
As any monomial is given by $W_{i_1}\cdots W_{i_l} =
\frac{\partial^l}{\partial t_{i_1}\cdots\partial 
t_{i_l}}\varphi(W_1,\ldots,W_m)\big|_{t_1,\ldots,t_m=0}$, the conclusion 
follows readily by differentiating the above identity.
\end{proof}

Note that the first two cumulants are given by $\kappa(W)=\EE[W]$ and 
$\kappa(W_1,W_2)=\mathrm{Cov}(W_1,W_2)$. Thus if $W_1,\ldots,W_m$ are 
centered and jointly Gaussian (so that the cumulants of order $k\ge 3$ 
vanish), the identities of Lemmas \ref{lem:shir} and \ref{lem:ibp} reduce 
to
\begin{equation}
\label{eq:wick}
	\EE[W_1\cdots W_m] = \sum_{\pi\in\mathrm{P}_2([m])}
	\prod_{\{i,j\}\in\pi} \mathrm{Cov}(W_i,W_j)
\end{equation}
(where $\mathrm{P}_2([m])$ is the collection of pair partitions of $[m]$)
and
\begin{equation}
\label{eq:gibp}
	\EE[W_i f(W_1,\ldots,W_m)] = 
	\sum_{j=1}^m \mathrm{Cov}(W_i,W_j)\, \EE\bigg[
	\frac{\partial f}{\partial x_j}(W_1,\ldots,W_m)\bigg].
\end{equation}
These are none other than the well-known Wick formula and integration by 
parts formula for centered Gaussian measures.

\subsection{Cumulant expansion}
\label{sec:cumex}

We can now express the basic principle that will be used to prove 
universality. This principle is a direct extension of the method of 
\cite{Bar86,LP09} to the multivariate case; see, e.g., \cite[Corollary 
3.1]{LP09}.

\begin{thm}
\label{thm:cumpoly}
Let $Y_1,\ldots,Y_n$ be independent centered and bounded random vectors in 
$\mathbb{R}^N$, and let $U_1,\ldots,U_n$ be independent centered Gaussian 
random vectors in $\mathbb{R}^N$ such that $Y_i$ and $U_i$ have the same 
covariance. Assume that $Y=(Y_1,\ldots,Y_n)$ and $U=(U_1,\ldots,U_n)$ are 
independent of each other, and define
$$
	Y(t) := \sqrt{t}\,Y + \sqrt{1-t}\,U.
$$
Then we have
$$
	\frac{d}{dt}\EE[f(Y(t))] =
	\frac{1}{2}\sum_{k=3}^\infty
	\sum_{i=1}^n
	\sum_{j_1,\ldots,j_k=1}^N
	\frac{t^{\frac{k}{2}-1}}{(k-1)!}
	\,\kappa(Y_{ij_1},\ldots,Y_{ij_k})\,
	\EE\bigg[
	\frac{\partial^kf}{\partial y_{ij_1}\cdots y_{ij_k}}
	(Y(t))
	\bigg]
$$
for any polynomial $f:\mathbb{R}^{Nn}\to\mathbb{C}$ and $t\in[0,1]$.
\end{thm}

\begin{proof}
We readily compute
$$
	\frac{d}{dt}\EE[f(Y(t))] =
	\frac{1}{2}
	\sum_{i=1}^n \sum_{j=1}^N
	\bigg\{
	\frac{1}{\sqrt{t}}\,
	\EE\bigg[Y_{ij}
	\frac{\partial f}{\partial y_{ij}}(Y(t))\bigg]
	-
	\frac{1}{\sqrt{1-t}}\,
	\EE\bigg[U_{ij}
	\frac{\partial f}{\partial y_{ij}}(Y(t))
	\bigg]\bigg\}.
$$
The conclusion follows by applying Lemma \ref{lem:ibp} conditionally 
on $\{U,(Y_k)_{k\ne i}\}$ to compute the first term in the sum,
and applying \eqref{eq:gibp} conditionally on $\{Y,(U_k)_{k\ne i}\}$
to compute the second term in the sum.
\end{proof}

The model $Y(t)$ should be viewed as an interpolation between the original 
model $Y$ and the associated Gaussian model $U$. In particular,
Theorem \ref{thm:cumpoly} yields a bound on the Gaussian 
deviation by the fundamental theorem of calculus
$$
	\EE[f(Y)] - \EE[f(U)] =
	\int_0^1 \frac{d}{dt}\EE[f(Y(t))]\, dt.
$$
We will however often find it necessary to perform a change of variables 
before applying the fundamental theorem of calculus.

When the function $f$ is not a polynomial, it must be approximated by a 
polynomial before we can apply Theorem \ref{thm:cumpoly}. The following 
result is a straightforward combination of Theorem \ref{thm:cumpoly} with 
Taylor expansion to order $p-1$.

\begin{thm}
\label{thm:cumsm}
Let $Y_1,\ldots,Y_n$ be independent centered and bounded random vectors in 
$\mathbb{R}^N$, and let $U_1,\ldots,U_n$ be independent centered Gaussian 
random vectors in $\mathbb{R}^N$ such that $Y_i$ and $U_i$ have the same 
covariance. Assume that $Y=(Y_1,\ldots,Y_n)$ and $U=(U_1,\ldots,U_n)$ are 
independent of each other, and define
$$
	Y(t) := \sqrt{t}\,Y + \sqrt{1-t}\,U.
$$
Let $p\ge 3$ and $f:\mathbb{R}^{Nn}\to\mathbb{C}$ be a smooth function.
Then we have
\begin{multline*}
	\frac{d}{dt}\EE[f(Y(t))] =\\
	\frac{1}{2}\sum_{k=3}^{p-1}
	\sum_{i=1}^n
	\sum_{j_1,\ldots,j_k=1}^N
	\frac{t^{\frac{k}{2}-1}}{(k-1)!}
	\,\kappa(Y_{ij_1},\ldots,Y_{ij_k})\,
	\EE\bigg[
	\frac{\partial^kf}{\partial y_{ij_1}\cdots y_{ij_k}}
	(Y(t))
	\bigg] + \mathcal{R}
\end{multline*}
for any $t\in[0,1]$, where the reminder term satisfies
\begin{align*}
	& |\mathcal{R}| \lesssim
	\sup_{s,t\in[0,1]}
	\Bigg\{
	\bigg|
	\sum_{i=1}^n
	\sum_{j_1,\ldots,j_p=1}^N
	\EE\bigg[
	Y_{ij_1}\cdots Y_{ij_p}
        \frac{\partial^pf}{\partial y_{ij_1}\cdots y_{ij_p}}
        (Y(t,i,s))
	\bigg]\bigg| +
	\\
	&\; \max_{2\le k\le p-1} 
	\bigg|
	\sum_{i=1}^n
	\sum_{j_1,\ldots,j_p=1}^N
	\frac{\kappa(Y_{ij_1},\ldots,Y_{ij_k})}{(k-1)!}
	\,
	\EE\bigg[
	Y_{ij_{k+1}}\cdots Y_{ij_p}
        \frac{\partial^pf}{\partial y_{ij_1}\cdots y_{ij_p}}
        (Y(t,i,s))
	\bigg]\bigg|\Bigg\}
\end{align*}
with $Y_j(t,i,s) := s^{1_{i=j}}\sqrt{t}\,Y_j+\sqrt{1-t}\,U_j$.
\end{thm}

\begin{proof}
Let $g:\mathbb{R}^{Nn}\to\mathbb{C}$ be a smooth function, and
let $g_i$ be the Taylor expansion of 
$t\mapsto g(y_1,\ldots,y_{i-1},ty_i,y_{i+1},\ldots,y_n)$ to order
$p-1$ around $0$ (evaluated at $t=1$):
$$
	g_i(y) :=
	\sum_{l=0}^{p-1}
	\sum_{j_1,\ldots,j_l=1}^N
	\frac{1}{l!}\,
	y_{ij_1}\cdots y_{ij_l}\,
	\frac{\partial^lg}{\partial y_{ij_1}\cdots\partial y_{ij_l}}
	(y_1,\ldots,y_{i-1},0,y_{i+1},\ldots,y_n).
$$
Then
\begin{multline*}
	\frac{\partial^kg}{\partial y_{ij_1}\cdots\partial y_{ij_k}}(y) = 
	\frac{\partial^kg_i}{\partial y_{ij_1}\cdots\partial y_{ij_k}}(y)
	+ 
	\int_0^1 \frac{(1-s)^{p-k-1}}{(p-k-1)!}\cdot\mbox{}
	\\
	\sum_{j_{k+1},\ldots,j_p=1}^N
	y_{ij_{k+1}}\cdots y_{ij_p}
	\frac{\partial^pg}{\partial y_{ij_1}\cdots\partial y_{ij_p}}
	(y_1,\ldots,y_{i-1},sy_i,y_{i+1},\ldots,y_n)\,ds
\end{multline*}
for all $0\le k\le p-1$ and $j_1,\ldots,j_k$.
Choosing $g(Y) := f(\sqrt{t}\,Y+\sqrt{1-t}\,U)$ yields
\begin{align*}
	\frac{1}{2\sqrt{t}}
	\sum_{i=1}^n \sum_{j=1}^N
        \EE\bigg[Y_{ij}
        \frac{\partial f}{\partial y_{ij}}(Y(t))\bigg]
	&=
	\frac{1}{2t}
	\sum_{i=1}^n \sum_{j=1}^N
        \EE\bigg[Y_{ij}
        \frac{\partial g}{\partial y_{ij}}(Y)\bigg]
	\\ &=
	\frac{1}{2t}
	\sum_{i=1}^n\sum_{j=1}^N
        \EE\bigg[Y_{ij}
        \frac{\partial g_i}{\partial y_{ij}}(Y)\bigg]
	+ \mathcal{R}_1,
\end{align*}
where
$$
	\mathcal{R}_1=
	\frac{t^{\frac{p}{2}-1}}{2}
	\int_0^1 \frac{(1-s)^{p-2}}{(p-2)!}
	\sum_{i=1}^n
	\sum_{j_1,\ldots,j_p=1}^N
        \EE\bigg[
	Y_{ij_1} \cdots Y_{ij_p}
	\frac{\partial^pf}{\partial u_{ij_1}\cdots\partial y_{ij_p}}
	(Y(t,i,s))
	\bigg]\,ds.
$$
As $y_i\mapsto g_i(y)$ is a polynomial of degree $p-1$ and
$\kappa(Y_{ij})=\EE[Y_{ij}]=0$, we
can now apply Lemma \ref{lem:ibp} conditionally
on $\{U,(Y_k)_{k\ne i}\}$ to compute
\begin{align*}
	&\frac{1}{2t}
	\sum_{i=1}^n\sum_{j=1}^N
        \EE\bigg[Y_{ij}
        \frac{\partial g_i}{\partial y_{ij}}(Y)\bigg] \\
	&=
	\frac{1}{2t}
	\sum_{k=2}^{p-1}
	\sum_{i=1}^n
	\sum_{j_1,\ldots,j_k=1}^N
	\frac{1}{(k-1)!}\, \kappa(Y_{ij_1},\ldots,Y_{ij_k})\,
        \EE\bigg[
        \frac{\partial g_i}{\partial y_{ij_1}
	\cdots \partial y_{ij_k}}(Y)\bigg]
	\\
	&=
	\frac{1}{2}
	\sum_{k=2}^{p-1}
	\sum_{i=1}^n
	\sum_{j_1,\ldots,j_k=1}^N
	\frac{t^{\frac{k}{2}-1}}{(k-1)!}\, \kappa(Y_{ij_1},\ldots,Y_{ij_k})\,
        \EE\bigg[
        \frac{\partial f}{\partial y_{ij_1}
	\cdots \partial y_{ij_k}}(Y(t))\bigg]
	- \mathcal{R}_2,
\end{align*}
where
\begin{multline*}
	\mathcal{R}_2 =
	\frac{t^{\frac{p}{2}-1}}{2}
	\sum_{k=2}^{p-1}
	\int_0^1 \frac{(1-s)^{p-k-1}}{(p-k-1)!}
	\cdot\mbox{}\\
	\sum_{i=1}^n
	\sum_{j_1,\ldots,j_p=1}^N
	\frac{\kappa(Y_{ij_1},\ldots,Y_{ij_k})}{(k-1)!}\, 
        \EE\bigg[
	Y_{ij_{k+1}}\cdots Y_{ij_p}
	\frac{\partial^pf}{\partial y_{ij_1}\cdots\partial y_{ij_p}}
	(Y(t,i,s))\bigg]\,ds.
\end{multline*}
Thus the identity in the statement follows precisely
as in the proof of Theorem \ref{thm:cumpoly} with $\mathcal{R}=
\mathcal{R}_1-\mathcal{R}_2$.
The estimate on $|\mathcal{R}|$ now follows readily by noting
that
$$
	\sum_{k=1}^{p-1}
	\int_0^1 \frac{(1-s)^{p-k-1}}{(p-k-1)!}\,ds =
	\sum_{k=1}^{p-1}
	\frac{1}{(p-k)!}
	\le e-1,
$$
concluding the proof.
\end{proof}

\section{Basic tools}
\label{sec:tools}

The aim of this section is to develop two important tools that will be 
needed in the proofs of our main results. In section \ref{sec:trace}, we 
prove a trace inequality that will enable us to control the derivatives 
that arise in the cumulant expansion of various spectral statistics. In 
section \ref{sec:conc}, we develop concentration of measure inequalities 
for the resolvent and for more general spectral statistics.

\subsection{A trace inequality}
\label{sec:trace}

Let $L_p(S_p^d)$ be the Banach space of $d\times d$ random matrices $M$
(that is, $\M_d(\mathbb{C})$-valued random variables on an underlying 
probability space that we consider fixed throughout the paper)
such that $\|M\|_p<\infty$. Here
$$
	\|M\|_p := \begin{cases}
	\big(\EE[\ntr |M|^p]\big)^{\frac{1}{p}} & \mbox{if }1\le p<\infty,\\
	\| \|M\| \|_\infty & \mbox{if }p=\infty,
	\end{cases}
$$
where we recall that $\ntr$ denotes the normalized trace. In particular, 
when $M\in\M_d(\mathbb{C})$ is a deterministic matrix, $\|M\|_p$ is the 
(normalized) Schatten-$p$ norm. In this notation, we can write
$$
	\sigma_q(X) := 
	\Bigg\|\Bigg(\sum_{i=1}^n \EE 
	Z_i^2\Bigg)^{\frac{1}{2}}\Bigg\|_q,
	\qquad
	R_q(X):= \Bigg(\sum_{i=1}^n \|Z_i\|_q^q\Bigg)^{\frac{1}{q}}
$$
for $q<\infty$ (cf.\ section \ref{sec:matparm}).

The following trace inequality will play a key role throughout this paper.

\begin{prop}
\label{prop:trace}
Fix $k\ge 2$. 
Let $(Z_{ij})_{i\in[n],j\in[k]}$ be a collection of (possibly dependent) 
$d\times d$ self-adjoint random matrices such that $Z_{ij}$ has the same 
distribution as $Z_i$ for each $i,j$.
Let $1\le p_1,\ldots,p_k,q\le\infty$
satisfy $\sum_{j=1}^{k}\frac{1}{p_j}=1-\frac{k}{q}$. Then
$$
	\Bigg|
	\sum_{i=1}^n
	\EE[\ntr Z_{i1} Y_1 Z_{i2} Y_2 \cdots Z_{ik} Y_k]
	\Bigg|
	\le
	R_q(X)^{\frac{(k-2)q}{q-2}}
	\sigma_q(X)^{\frac{2(q-k)}{q-2}}
	\prod_{j=1}^k \|Y_j\|_{p_j}
$$
for any (possibly dependent) $d\times d$ random matrices $Y_1,\ldots,Y_k$ 
that are independent of the random matrices $(Z_{ij})_{i\in[n],j\in[k]}$.
\end{prop}

In preparation for the proof of this result, we recall some fundamental 
tools that will be needed below. We first state a variant of
the Riesz-Thorin interpolation theorem for Schatten classes. (The 
application of complex interpolation in this context was inspired by 
\cite{Tro18}, and was previously used in \cite[Lemma 4.5]{BBV21}.)

\begin{lem}
\label{lem:calderon}
Let $F:(L_\infty(S_\infty^d))^k\to\mathbb{C}$ be a multilinear functional.
Then the map
$$
	\bigg(\frac{1}{p_1},\ldots,\frac{1}{p_k}\bigg)
	\mapsto
	\log \sup_{M_1,\ldots,M_k}
	\frac{|F(M_1,\ldots,M_k)|}{\|M_1\|_{p_1}\cdots\|M_k\|_{p_k}}
$$
is convex on $[0,1]^k$.
\end{lem}

\begin{proof}
This follows immediately from the classical complex interpolation theorem 
for multilinear maps \cite[\S 10.1]{Cal64} and the fact that
the spaces $L_p(S_p^d)$ form a complex 
interpolation scale $L_r(S_r^d)=(L_p(S_p^d),L_q(S_q^d))_{\theta}$ with 
$\frac{1}{r}=\frac{1-\theta}{p}+\frac{\theta}{q}$ \cite[\S 2]{Pis03b}. 
\end{proof}

Next, we recall a H\"older inequality for Schatten classes. We include a 
proof in order to illustrate Lemma \ref{lem:calderon}; the same method 
will be used again below.

\begin{lem}
\label{lem:holder}
Let $1\le 
p_1,\ldots,p_k\le\infty$ satisfy $\sum_{i=1}^k\frac{1}{p_i}=1$. Then
$$
	|\EE[\ntr Y_1\cdots Y_k]| \le
	\|Y_1\|_{p_1}\cdots \|Y_k\|_{p_k}
$$
for any $d\times d$ random matrices $Y_1,\ldots,Y_k$.
\end{lem}

\begin{proof}
It suffices to prove the inequality for any
$Y_1,\ldots,Y_k\in L_\infty(S_\infty^d)$, that is,
we must show that the multilinear functional  
$F(Y_1,\ldots,Y_k) :=
\EE[\ntr Y_1\cdots Y_k]$ satisfies
$$
	\sup_{Y_1,\ldots,Y_k\in L_\infty(S_\infty^d)}
	\frac{|F(Y_1,\ldots,Y_k)|}{\|Y_1\|_{p_1}\cdots\|Y_k\|_{p_k}}
	\le 1\qquad\mbox{for all}\quad
	\bigg(\frac{1}{p_1},\ldots,\frac{1}{p_k}\bigg)\in \Delta,
$$
where $\Delta := \big\{x\in[0,1]:\sum_{i=1}^k x_i=1\big\}$. By Lemma 
\ref{lem:calderon}, it suffices to prove the claim only for the extreme 
points of $\Delta$, that is, when $p_i=1$ and $p_j=\infty$, $j\ne i$ for 
some $i$. But the latter case is elementary, as $|XY|^2 = Y^*X^*XY \le \|X\|^2 
|Y|^2$ and thus $|F(Y_1,\ldots,Y_k)|\le \|Y_{i+1}\cdots Y_k Y_1\cdots 
Y_{i-1}\|_\infty\|Y_i\|_1\le
\|Y_i\|_1\prod_{j\ne i}\|Y_j\|_\infty$. 
\end{proof}

Finally, we recall without proof the Lieb-Thirring inequality
\cite[Theorem 7.4]{Car10}.

\begin{lem}
\label{lem:lieb}
Let $Y,Z$ be $d\times d$ positive semidefinite random matrices. Then
$$
	\EE[\ntr {(ZYZ)^r}] \le \EE[\ntr Z^r Y^r Z^r]
$$
for every $1\le r<\infty$.
\end{lem}

We can now proceed to the proof of Proposition \ref{prop:trace}.

\begin{proof}[Proof of Proposition \ref{prop:trace}]
Throughout the proof we will assume without loss of generality that 
$R_q(X)<\infty$, as the conclusion is trivial otherwise.
Thus Lemma \ref{lem:holder} implies that 
the multilinear functional $\tilde F:(\M_d(\mathbb{C}))^k\to\mathbb{C}$ 
defined by
$$
	\tilde F(M_1,\ldots,M_k) :=
	\sum_{i=1}^n
	\EE[\ntr Z_{i1} M_1 Z_{i2} M_2 \cdots Z_{ik} M_k]
$$
satisfies $|\tilde F(M_1,\ldots,M_k)| \le \prod_{j=1}^k \|M_j\|_\infty$
for all $M_1,\ldots,M_k\in \M_d(\mathbb{C})$. Therefore
$$
	F(Y_1,\ldots,Y_k) := \mathbf{E}[\tilde F(Y_1,\ldots,Y_k)]
$$
defines a multilinear functional 
$F:(L_\infty(S_\infty^d))^k\to\mathbb{C}$. Throughout the proof,
it is implicit in the notation that $(Y_1,\ldots,Y_k)$ are taken to be 
independent of $(Z_{ij})_{i,j}$.

\medskip
\textbf{Step 1.} Our aim is to show that
$$
	\sup_{Y_1,\ldots,Y_k\in L_\infty(S_\infty^d)}
	\frac{|F(Y_1,\ldots,Y_k)|}{\|Y_1\|_{p_1}\cdots\|Y_k\|_{p_k}}
	\le 
	R_q(X)^{\frac{(k-2)q}{q-2}}
	\sigma_q(X)^{\frac{2(q-k)}{q-2}}
$$
for all $(\frac{1}{p_1},\ldots,\frac{1}{p_k})\in \Delta := \big\{
x\in[0,1]^k:\sum_{i=1}^k x_i = 1-\frac{k}{q}\big\}$. By Lemma 
\ref{lem:calderon}, it suffices to prove the claim only for 
$(\frac{1}{p_1},\ldots,\frac{1}{p_k})$
that are extreme points of the simplex $\Delta$, that is, when
$p_i = \frac{q}{q-k}$ and $p_j=\infty$, $j\ne i$ holds for some $i$.

By cyclic permutation of the trace, it suffices to consider the case
$p_1,\ldots,p_{k-1}=\infty$ and $p_k=\frac{q}{q-k}$. To further simplify 
the statement to be proved, let $I$ be a random variable that 
is uniformly distributed on $[n]$ and is independent of 
$(Y_j,Z_{ij})_{i,j}$, and define the random matrices $\mathbf{Z}_j := 
Z_{Ij}$. Then it suffices to show that
$$
	n |\EE[\ntr \mathbf{Z}_{1} Y_1 \mathbf{Z}_{2} Y_2 \cdots 
	\mathbf{Z}_{k} Y_k]|
	\le
	R_q(X)^{\frac{(k-2)q}{q-2}}
        \sigma_q(X)^{\frac{2(q-k)}{q-2}}
$$
whenever $\|Y_1\|_\infty=\cdots=\|Y_{k-1}\|_\infty=1$ and
$\|Y_k\|_{\frac{q}{q-k}}=1$. In the remainder of the proof,
we fix $Y_1,\ldots,Y_k$ satisfying the latter 
assumptions.

\medskip
\textbf{Step 2.}
The assumptions on $k,p_1,\ldots,p_k,q$ imply that
$q\ge k\ge 2$. In the case that $q=k$, we can estimate using
Lemma \ref{lem:holder}
$$
	n|\EE[\ntr \mathbf{Z}_{1} Y_1 \mathbf{Z}_{2} Y_2 \cdots
        \mathbf{Z}_{k} Y_k]|
	\le
	n\|\mathbf{Z}_{1}\|_k\|Y_1\|_\infty\cdots
	\|\mathbf{Z}_{k}\|_k\|Y_k\|_\infty
	=
	R_k(X)^k,
$$
completing the proof. We therefore assume in the rest of the proof 
that $q>k$.

\medskip
\textbf{Step 3.}
Suppose $k$ is even. 
Denote by $\mathbf{Z}_j=\mathbf{U}_j|\mathbf{Z}_j|$
and $Y_k=V_k|Y_k|$ the polar decompositions of $\mathbf{Z}_j$ and
$Y_k$, respectively. Then we can estimate for $r\ge 1$
\begin{align*}
	&
	|\EE[\ntr \mathbf{Z}_{1} Y_1 \mathbf{Z}_{2} Y_2 \cdots
        \mathbf{Z}_{k} Y_k]| 
	=
	|\EE[\ntr 
	|Y_k|^{\frac{1}{2}}\mathbf{Z}_{1} Y_1 \mathbf{Z}_{2} Y_2 \cdots
        \mathbf{Z}_{k} V_k |Y_k|^{\frac{1}{2}}]| \\
	&\le
	\EE[\ntr 
	\mathbf{Z}_{1} Y_1 \cdots
	\mathbf{Z}_{\frac{k}{2}} Y_{\frac{k}{2}}
	Y_{\frac{k}{2}}^* \mathbf{Z}_{\frac{k}{2}}
	\cdots Y_1^* \mathbf{Z}_{1} |Y_k|]^{\frac{1}{2}}\cdot
	\\ &\qquad\quad
	\EE[\ntr 
	\mathbf{Z}_{k} Y_{k-1}^*
	\mathbf{Z}_{k-1}
	\cdots 
	Y_{\frac{k}{2}+1}^* 
	\mathbf{Z}_{\frac{k}{2}+1}
	\mathbf{Z}_{\frac{k}{2}+1}
	Y_{\frac{k}{2}+1}
	\cdots
	\mathbf{Z}_{k-1}
	Y_{k-1}
        \mathbf{Z}_{k} V_k |Y_k| V_k^*]^{\frac{1}{2}}
\\
	&=
	\EE[\ntr 
	|\mathbf{Z}_{1}|^{1-\frac{1}{r}}
	\mathbf{U}_{1}^*
	Y_1
	\mathbf{Z}_{2} 
	\cdots
	Y_{\frac{k}{2}-1}
	\mathbf{Z}_{\frac{k}{2}} Y_{\frac{k}{2}}
	\cdot \\
	&\qquad\qquad\qquad\qquad
	Y_{\frac{k}{2}}^* \mathbf{Z}_{\frac{k}{2}}
	Y_{\frac{k}{2}-1}^*
	\cdots \mathbf{Z}_{2}
	Y_1^* \mathbf{U}_1
	|\mathbf{Z}_{1}|^{1-\frac{1}{r}}
	|\mathbf{Z}_{1}|^{\frac{1}{r}}
	|Y_k|  
	|\mathbf{Z}_{1}|^{\frac{1}{r}}
	]^{\frac{1}{2}}\cdot
	\\ &\phantom{\mbox{}=\mbox{}}
	\EE[\ntr 
        |\mathbf{Z}_{k}|^{1-\frac{1}{r}}
	\mathbf{U}_{k}^*
	Y_{k-1}^* \mathbf{Z}_{k-1}
	\cdots 
	Y_{\frac{k}{2}+1}^* 
	\mathbf{Z}_{\frac{k}{2}+1}\cdot\\
	&\qquad\qquad\qquad\qquad
	\mathbf{Z}_{\frac{k}{2}+1}
	Y_{\frac{k}{2}+1}
	\cdots
	\mathbf{Z}_{k-1}
	Y_{k-1} 
	\mathbf{U}_{k} 
        |\mathbf{Z}_{k}|^{1-\frac{1}{r}}
        |\mathbf{Z}_{k}|^{\frac{1}{r}}
	V_k |Y_k| V_k^*
        |\mathbf{Z}_{k}|^{\frac{1}{r}}]^{\frac{1}{2}}
\end{align*}
by Cauchy-Schwarz. Now let
$$
	r = \frac{q-2}{q-k} \in [1,\infty).
$$
Then we have $2\frac{1-\frac{1}{r}}{q} + (k-2)\frac{1}{q} + \frac{1}{r} 
= 1$. We can therefore estimate
\begin{align*}
	&|\EE[\ntr \mathbf{Z}_{1} Y_1 \mathbf{Z}_{2} Y_2 \cdots
        \mathbf{Z}_{k} Y_k]| \\
	&\le
	\|\mathbf{Z}_1\|_q^{1-\frac{1}{r}}
	\|\mathbf{Z}_2\|_q
	\cdots
	\|\mathbf{Z}_{\frac{k}{2}}\|_q\,
	\| |\mathbf{Z}_1|^{\frac{1}{r}}|Y_k|
		|\mathbf{Z}_1|^{\frac{1}{r}}\|_r^{\frac{1}{2}}
	\cdot \\
	&\qquad\quad
	\|\mathbf{Z}_k\|_q^{1-\frac{1}{r}}
	\|\mathbf{Z}_{\frac{k}{2}+1}\|_q
	\cdots
	\|\mathbf{Z}_{k-1}\|_q\,
	\| |\mathbf{Z}_k|^{\frac{1}{r}}V_k|Y_k|V_k^*
		|\mathbf{Z}_k|^{\frac{1}{r}}\|_r^{\frac{1}{2}}
	\\
	&=
	n^{-\frac{k-2}{q-2}}
	R_q(X)^{\frac{(k-2)q}{q-2}}\,
	\| |\mathbf{Z}_1|^{\frac{1}{r}}|Y_k|
		|\mathbf{Z}_1|^{\frac{1}{r}}\|_r^{\frac{1}{2}}\,
	\| |\mathbf{Z}_k|^{\frac{1}{r}}V_k|Y_k|V_k^*
		|\mathbf{Z}_k|^{\frac{1}{r}}\|_r^{\frac{1}{2}}
\end{align*}
by Lemma \ref{lem:holder}, where we used that
$\|\mathbf{Z}_j\|_q = n^{-\frac{1}{q}}R_q(X)$ and that
$k-\frac{2}{r}=\frac{(k-2)q}{q-2}$.
On the other hand, using Lemma \ref{lem:lieb} we obtain
$$
	\| |\mathbf{Z}_1|^{\frac{1}{r}}|Y_k|
                |\mathbf{Z}_1|^{\frac{1}{r}}\|_r^r
	\le
	\EE[\ntr |Y_k|^r \mathbf{Z}_1^2 ] =
	\EE[\ntr |Y_k|^r \EE[\mathbf{Z}_1^2]] \le
	\|\EE[\mathbf{Z}_1^2]\|_{\frac{q}{2}},
$$
where we used that $Y_k$ and $\mathbf{Z}_1$ are independent and
$\||Y_k|^r\|_{\frac{q}{q-2}}=
\|Y_k\|_{\frac{q}{q-k}}^r=1$.
The analogous term involving $\mathbf{Z}_k$ is estimated identically.
We therefore obtain
$$
	|\EE[\ntr \mathbf{Z}_{1} Y_1 \mathbf{Z}_{2} Y_2 \cdots
        \mathbf{Z}_{k} Y_k]|
	\le
	n^{-1}
	R_q(X)^{\frac{(k-2)q}{q-2}}\,
	\sigma_q(X)^{\frac{2(q-k)}{q-2}},
$$
where we used that
$\|\EE[\mathbf{Z}_j^2]\|_{\frac{q}{2}}=n^{-1}\sigma_q(X)^2$ for all $j$.
This concludes the proof of the inequality for the case that $k$ is even.

\medskip
\textbf{Step 4.}
Finally, suppose $k$ is odd. Then we apply Cauchy-Schwarz as follows:
\begin{align*}
	&
	|\EE[\ntr \mathbf{Z}_{1} Y_1 \mathbf{Z}_{2} Y_2 \cdots
        \mathbf{Z}_{k} Y_k]| 
	\\ &=
	|\EE[\ntr 
	|Y_k|^{\frac{1}{2}}\mathbf{Z}_{1} Y_1
	\cdots
	\mathbf{Z}_{\frac{k-1}{2}}	
	Y_{\frac{k-1}{2}}
	\mathbf{U}_{\frac{k+1}{2}}
	|\mathbf{Z}_{\frac{k+1}{2}}|^{\frac{1}{2}}
\cdot\\
&\qquad\qquad\qquad\qquad
	|\mathbf{Z}_{\frac{k+1}{2}}|^{\frac{1}{2}}
	Y_{\frac{k+1}{2}}
	\mathbf{Z}_{\frac{k+3}{2}}
	\cdots
	Y_{k-1}
        \mathbf{Z}_{k} V_k |Y_k|^{\frac{1}{2}}]| \\
	&\le
	\EE[\ntr 
	\mathbf{Z}_{1} Y_1
	\cdots
	\mathbf{Z}_{\frac{k-1}{2}}	
	Y_{\frac{k-1}{2}}
	\mathbf{U}_{\frac{k+1}{2}}
	|\mathbf{Z}_{\frac{k+1}{2}}|
	\mathbf{U}_{\frac{k+1}{2}}^*
	Y_{\frac{k-1}{2}}^*
	\mathbf{Z}_{\frac{k-1}{2}}
	\cdots
	Y_1^*
	\mathbf{Z}_{1}
	|Y_k|
	]^{\frac{1}{2}}\cdot \\
	&\qquad\quad
	\EE[\ntr 
	\mathbf{Z}_{k}
	Y_{k-1}^*
	\cdots
	\mathbf{Z}_{\frac{k+3}{2}}
	Y_{\frac{k+1}{2}}^*
	|\mathbf{Z}_{\frac{k+1}{2}}|
	Y_{\frac{k+1}{2}}
	\mathbf{Z}_{\frac{k+3}{2}}
	\cdots
	Y_{k-1}
        \mathbf{Z}_{k} V_k |Y_k|V_k^*
	]^{\frac{1}{2}}.
\end{align*}
The rest of the proof proceeds exactly as in the case that $k$ is 
even.
\end{proof}

\subsection{Concentration of measure}
\label{sec:conc}

In the proof of our main results, it will be necessary to control the norms 
of the resolvents $\|(z\id-X)^{-1}\|$ and $\|(z\id-G)^{-1}\|$ 
simultaneously over many points $z\in\mathbb{C}$. To this end, we will 
exploit the fact that these quantities are strongly concentrated around 
their means.

For the Gaussian model $G$, such concentration inequalities follow from a 
routine application of Gaussian concentration, as we recall in section 
\ref{sec:concgauss}. However, the non-Gaussian model $X$ does not appear 
to be amenable to off-the-shelf concentration inequalities: while convex 
Lipschitz functions of sums of independent random matrices (such as the 
norm $\|X\|$) can be treated using concentration inequalities due to 
Talagrand, such methods do not apply to the non-convex function 
$(Z_1,\ldots,Z_n)\mapsto \|(z\id -X)^{-1}\|$. In section 
\ref{sec:concnongauss}, we develop a specialized concentration inequality 
that will play a key role in the proofs of our main results.

Finally, in section \ref{sec:concspstat}, we obtain concentration 
inequalities for the spectral statistics $\langle v,\varphi(X)w\rangle$ 
both in the Gaussian and non-Gaussian situations, which may be used in 
conjunction with Theorem \ref{thm:smuniv} to obtain high probability 
universality bounds for spectral statistics. The proofs of these 
concentration inequalities rely on concentration of the resolvent as 
derived in the previous sections. An analogous concentration inequality 
for moments, which may be used in conjunction with Theorem 
\ref{thm:momentuniv}, is much simpler and follows from a routine 
application of Talagrand's concentration inequality; such an inequality is 
given in Lemma \ref{lem:pfsafconc}.

\subsubsection{Resolvent norm: the Gaussian case}
\label{sec:concgauss}

The Gaussian random matrix $G$ is amenable to a routine application of 
Gaussian concentration \cite[Theorem 5.6]{BLM13} as in \cite[Lemma 
6.5]{BBV21}. For completeness, we spell out the argument.

\begin{lem}
\label{lem:gconc}
Fix $z\in\mathbb{C}$ with $\mathrm{Im}\,z>0$. Then we have for any
$x\ge 0$
$$
	\mathbf{P}\bigg[
	\big|\|(z\id -G)^{-1}\| 
	-\EE\|(z\id -G)^{-1}\|\big| \ge
	\frac{\sigma_*(G)}{(\mathrm{Im}\,z)^2}\,x\bigg]
	\le 2e^{-x^2/2}.
$$
\end{lem}

\begin{proof}
Without loss of generality, we may express
$$
	G = A_0 + \sum_{i=1}^N g_i A_i 
$$
for some deterministic $A_0,\ldots,A_N\in\M_d(\mathbb{C})_{\rm sa}$ and
i.i.d.\ standard Gaussian variables $g_1,\ldots,g_N$ (cf.\ 
Remark \ref{rem:coeffrep}). Now consider the function
$f:\mathbb{R}^N\to\mathbb{R}$ defined by
$$
	f(x) :=
	\Bigg\|\bigg(z\id - A_0 - \sum_{i=1}^Nx_i A_i\bigg)^{-1}\Bigg\|.
$$
As $A^{-1}-B^{-1} = A^{-1}(B-A)B^{-1}$ for invertible matrices $A,B$, and
as $\|(z\id-Y)^{-1}\|\le (\mathrm{Im}\,z)^{-1}$ for any self-adjoint 
matrix $Y$, we obtain
\begin{align*}
	|f(x)-f(y)|  &\le
	\Bigg\|
	\bigg(z\id - A_0 - \sum_{i=1}^Nx_i A_i\bigg)^{-1}-
	\bigg(z\id - A_0 - \sum_{i=1}^Ny_i A_i\bigg)^{-1}
	\Bigg\|
	\\
	&\le
	\frac{1}{(\mathrm{Im}\,z)^2}\,\Bigg\|
	\sum_{i=1}^N(x_i-y_i)A_i\Bigg\|
	\le \frac{\sigma_*(G)}{(\mathrm{Im}\,z)^2}\,\|x-y\|,
\end{align*}
where we used that
\begin{align*}
	\Bigg\|
        \sum_{i=1}^N(x_i-y_i)A_i\Bigg\| &=
	\sup_{\|v\|=\|w\|=1}
	\Bigg|
        \sum_{i=1}^N(x_i-y_i)\langle v,A_iw\rangle
	\Bigg| \\
	&\le
	\sup_{\|v\|=\|w\|=1}
	\Bigg(
	\sum_{i=1}^N |\langle v,A_iw\rangle|^2
	\Bigg)^{\frac{1}{2}}\|x-y\| \\ &=	
	\sup_{\|v\|=\|w\|=1} 
	\EE\big[|\langle v,(G-\EE G)w\rangle|^2\big]^{\frac{1}{2}}\|x-y\|
	\phantom{\bigg|}
	\\ &= \sigma_*(G)\,\|x-y\|.
	\phantom{\bigg|}
\end{align*}
Thus $\|(z\id -G)^{-1}\|=f(g_1,\ldots,g_N)$ is a
$\frac{\sigma_*(G)}{(\mathrm{Im}\,z)^2}$-Lipschitz function of a standard 
Gaussian vector. The conclusion is therefore immediate from the Gaussian 
concentration inequality \cite[Theorem 5.6]{BLM13}, which states than an 
$L$-Lipschitz function of a standard Gaussian vector is $L^2$-subgaussian.
\end{proof}

\subsubsection{Resolvent norm: the non-Gaussian case}
\label{sec:concnongauss}

We now aim to prove an analogue of Lemma \ref{lem:gconc} for the 
non-Gaussian model $X$. To this end, we exploit the resolvent identity to 
prove a specialized concentration inequality using the entropy method 
\cite[Chapter 6]{BLM13}. The result takes a more complicated form than 
Lemma \ref{lem:gconc}, but will nonetheless suffice for the purposes of 
this paper.

\begin{prop}
\label{prop:ngconc}
Fix $z\in\mathbb{C}$ with $\mathrm{Im}\,z>0$. Then we have
\begin{multline*}
	\mathbf{P}\bigg[\big|\|(z\id -X)^{-1}\| -
	\EE\|(z\id -X)^{-1}\|\big|
	\ge
	\frac{\sigma_*(X)}{(\mathrm{Im}\,z)^2}\sqrt{x} + 
	\bigg\{
	\frac{R(X)}{(\mathrm{Im}\,z)^2}+
	\frac{R(X)^2}{(\mathrm{Im}\,z)^3}
	\bigg\}\,x 
\\ 
	+\bigg\{
	\frac{R(X)^{\frac{1}{2}}(\EE\|X-\EE X\|)^{\frac{1}{2}}}
	{(\mathrm{Im}\,z)^2} +
	\frac{R(X)(\EE\|X-\EE X\|^2)^{\frac{1}{2}}}{(\mathrm{Im}\,z)^3}
	\bigg\}
	\sqrt{x}
	\bigg]
	\le 2e^{-Cx}
\end{multline*}
for any $x\ge 0$, where $C$ is a universal constant.
\end{prop}

In preparation for the proof, we begin by estimating a type of discrete 
gradient of the function $(Z_1,\ldots,Z_n)\mapsto \|(z\id -X)^{-1}\|$.

\begin{lem}
\label{lem:W}
Let $(Z_1',\ldots,Z_n')$ be an independent copy of $(Z_1,\ldots,Z_n)$.
Let $X$ be as in \eqref{eq:model} and let
$X^{\sim i}:=Z_0+\sum_{j\ne i}Z_j + Z_i'$. Then
$$
	\|(z\id -X)^{-1}\|- \|(z\id -X^{\sim i})^{-1}\| \le
	\frac{2R(X)}{(\mathrm{Im}\,z)^2}
$$
for all $i$, and
$$
	\sum_{i=1}^n
	(\|(z\id -X)^{-1}\|- \|(z\id -X^{\sim i})^{-1}\|)_+^2
	\le W
$$
with
$$
	W:=
	\frac{2}{(\mathrm{Im}\,z)^4}
	\sup_{\|v\|=\|w\|=1}
	\sum_{i=1}^n
	|\langle v,(Z_i-Z_i')w\rangle|^2 +
	\frac{8}{(\mathrm{Im}\,z)^6}
	R(X)^2\Bigg\|\sum_{i=1}^n(Z_i-Z_i')^2\Bigg\|.
$$
\end{lem}

\begin{proof}
The first part of the statement follows as
$$
	\|(z\id -X)^{-1}\|- \|(z\id -X^{\sim i})^{-1}\| 
	\le \|(z\id- X)^{-1}(Z_i-Z_i')(z\id-X^{\sim i})^{-1}\|
	\le \frac{2R(X)}{(\mathrm{Im}\,z)^2}
$$
using the reverse triangle inequality and 
$A^{-1}-B^{-1}=A^{-1}(B-A)B^{-1}$ in the first inequality, and
$\|(z\id - A)^{-1}\|\le (\mathrm{Im}\,z)^{-1}$ in the second inequality.

To prove the second part of the statement, we must estimate more 
carefully.
Let $v_*,w_*$ be (random) vectors in the unit sphere such that
$$
	\|(z\id - X)^{-1}\| = 
	\sup_{\|v\|=\|w\|=1} |\langle v,(z\id - X)^{-1}w\rangle|
	=
	|\langle v_*,(z\id - X)^{-1}w_*\rangle|.
$$
Then
\begin{align*}
	&\|(z\id - X)^{-1}\| -
	\|(z\id - X^{\sim i})^{-1}\|
\\
	&\le
	|\langle v_*,(z\id - X)^{-1}w_*\rangle| -
	|\langle v_*,(z\id - X^{\sim i})^{-1}w_*\rangle|
\\
	&\le
	|\langle v_*,(z\id - X)^{-1}
	(Z_i-Z_i')(z\id - X^{\sim i})^{-1}
	w_*\rangle|
\\
	&\le
	|\langle v_*,(z\id - X)^{-1}
	(Z_i-Z_i')(z\id - X)^{-1}
	w_*\rangle|
\\
&\qquad+
	|\langle v_*,(z\id - X)^{-1}
	(Z_i-Z_i')(z\id - X^{\sim i})^{-1}
	(Z_i-Z_i')(z\id - X)^{-1}
	w_*\rangle|,
\end{align*}
where we used twice the identity $A^{-1}-B^{-1} = A^{-1}(B-A)B^{-1}$.
But as we have
$\|(z\id-X)^{-1}w_*\|\le (\mathrm{Im}\,z)^{-1}$ and
$\|(\bar{z}\id-X)^{-1}v_*\|\le (\mathrm{Im}\,z)^{-1}$,
we can estimate
\begin{align*}
	&\sum_{i=1}^n
	|\langle v_*,(z\id - X)^{-1}
        (Z_i-Z_i')(z\id - X)^{-1}
        w_*\rangle|^2 
\\ &\qquad\qquad\le
	\frac{1}{(\mathrm{Im}\,z)^4}
	\sup_{\|v\|=\|w\|=1}
	\sum_{i=1}^n
	|\langle v,
        (Z_i-Z_i')
        w\rangle|^2.
\end{align*}
On the other hand, we have
\begin{align*}
	&
	\sum_{i=1}^n
	|\langle v_*,(z\id - X)^{-1}
	(Z_i-Z_i')(z\id - X^{\sim i})^{-1}
	(Z_i-Z_i')(z\id - X)^{-1}
	w_*\rangle|^2 \\
	&\le
	\sum_{i=1}^n
	\|(Z_i-Z_i')(\bar z\id - X)^{-1}v_*\|^2
	\|(z\id - X^{\sim i})^{-1}
	(Z_i-Z_i')(z\id - X)^{-1}
	w_*\|^2 \\
	&\le
	\frac{4R(X)^2}{(\mathrm{Im}\,z)^4}
	\sum_{i=1}^n
	\langle (z\id - X)^{-1}
        w_*,(Z_i-Z_i')^2(z\id - X)^{-1}
	w_*\rangle \\
	&\le
	\frac{4R(X)^2}{(\mathrm{Im}\,z)^6}
	\Bigg\|
	\sum_{i=1}^n
        (Z_i-Z_i')^2\Bigg\|.
\end{align*}
The conclusion follows readily using $(a+b)^2\le 2a^2+2b^2$.
\end{proof}

Next, we bound the expectation of the random variable $W$.

\begin{lem}
\label{lem:EW}
Let $W$ be defined as in Lemma \ref{lem:W}. Then
$$
	\EE[W] \lesssim
	\frac{\sigma_*(X)^2}{(\mathrm{Im}\,z)^4}
	+
	\frac{R(X)\EE\|X-\EE X\|}{(\mathrm{Im}\,z)^4}
	+
	\frac{R(X)^2\EE\|X-\EE X\|^2}{(\mathrm{Im}\,z)^6}.
$$
\end{lem}

\begin{proof}
First note that as $(a-b)^2 \le 2a^2+2b^2$ and as $(Z_1,\ldots,Z_n)$ and
$(Z_1',\ldots,Z_n')$ have the same distribution, we can estimate
$$
	\EE[W] \le
	\frac{8}{(\mathrm{Im}\,z)^4}
	\EE\Bigg[
	\sup_{\|v\|=\|w\|=1}
	\sum_{i=1}^n
	|\langle v,Z_iw\rangle|^2\Bigg] +
	\frac{8}{(\mathrm{Im}\,z)^6}
	R(X)^2 \EE\Bigg\|\sum_{i=1}^n (Z_i-Z_i')^2\Bigg\|.
$$
To estimate the first term, we apply \cite[Theorem 11.8]{BLM13} to obtain
\begin{align*}
	&\EE\Bigg[
	\sup_{\|v\|=\|w\|=1}
	\sum_{i=1}^n
	(\mathrm{Re}\,\langle v,Z_iw\rangle)^2
	\Bigg]
	\\ &\le
	8R(X)\,
	\EE\Bigg[\sup_{\|v\|=\|w\|=1}
	\sum_{i=1}^n 
	\mathrm{Re}\,\langle
	v,Z_iw\rangle\Bigg] +
	\sup_{\|v\|=\|w\|=1} \sum_{i=1}^n
	\EE[(\mathrm{Re}\,\langle v,Z_iw\rangle)^2],
	\\ &\le
	8R(X)\,\EE\|X-\EE X\| + \sigma_*(X)^2,\phantom{\bigg]}
\end{align*}
and analogously when the real part is replaced by the imaginary part. Thus
$$
	\EE\Bigg[
        \sup_{\|v\|=\|w\|=1}
        \sum_{i=1}^n
        |\langle v,Z_iw\rangle|^2\Bigg] 
	\le
	2\sigma_*(X)^2+
	16R(X)\EE\|X-\EE X\|.
$$
To estimate the second term, note that
\begin{align*}
	\EE\Bigg\|\sum_{i=1}^n (Z_i-Z_i')^2\Bigg\| &=
	\EE\Bigg\|
	\mathbf{E}_\varepsilon\Bigg[
	\Bigg(\sum_{i=1}^n \varepsilon_i(Z_i-Z_i')\Bigg)^2\Bigg] 
	\Bigg\|
	\le
	\EE\Bigg\|
	\sum_{i=1}^n \varepsilon_i(Z_i-Z_i')
	\Bigg\|^2
	\\ &=
	\EE\Bigg\|
	\sum_{i=1}^n (Z_i-Z_i')
	\Bigg\|^2
	\le 4\EE\|X-\EE X\|^2,
\end{align*}
where $\varepsilon_1,\ldots,\varepsilon_n$ are i.i.d.\ random signs 
independent of $Z,Z'$ and $\EE_\varepsilon$ denotes the expectation with
respect to the variables $\varepsilon_i$ only. The first equality is 
trivial, the first inequality is by Jensen, the second equality holds by 
the exchangeability of $(Z_i,Z_i')$, and the second inequality follows by 
the triangle inequality and $(a+b)^2\le 2a^2+2b^2$. Combining the above 
estimates completes the proof.
\end{proof}

Finally, we show that $W$ has a self-bounding property.

\begin{lem}
\label{lem:Wsb}
Let $W$ be defined as in Lemma \ref{lem:W}, and define
$$
	W^{\sim i}:=
	\frac{2}{(\mathrm{Im}\,z)^4}
	\sup_{\|v\|=\|w\|=1}
	\sum_{j\ne i}
	|\langle v,(Z_j-Z_j')w\rangle|^2 +
	\frac{8}{(\mathrm{Im}\,z)^6}
	R(X)^2\Bigg\|\sum_{j\ne i}(Z_j-Z_j')^2\Bigg\|.
$$
Then $W^{\sim i}\le W$ and
$$
	\sum_{i=1}^n (W-W^{\sim i})^2 \le
	\bigg\{
	\frac{16R(X)^2}{(\mathrm{Im}\,z)^4}+
	\frac{64R(X)^4}{(\mathrm{Im}\,z)^6}
	\bigg\}W.
$$
\end{lem}

\begin{proof}
That $W^{\sim i}\le W$ is obvious. To prove the self-bounding inequality,
let $u_*,v_*,w_*$ be (random) vectors in the unit sphere such that
$$
	\sup_{\|v\|=\|w\|=1}
	\sum_{i=1}^n
	|\langle v,(Z_i-Z_i')w\rangle|^2
	=
	\sum_{i=1}^n
	|\langle v_*,(Z_i-Z_i')w_*\rangle|^2
$$
and
$$
	\Bigg\|\sum_{i=1}^n(Z_i-Z_i')^2\Bigg\| =
	\sup_{\|u\|=1}
	\sum_{i=1}^n \|(Z_i-Z_i')u\|^2
	=
	\sum_{i=1}^n \|(Z_i-Z_i')u_*\|^2.
$$
Then
$$
	W-W^{\sim i} \le
	\frac{2}{(\mathrm{Im}\,z)^4}
	|\langle v_*,(Z_i-Z_i')w_*\rangle|^2 +
	\frac{8}{(\mathrm{Im}\,z)^6}
        R(X)^2
	\|(Z_i-Z_i')u_*\|^2.
$$
Therefore,
\begin{align*}
	& \sum_{i=1}^n (W-W^{\sim i})^2 \\
	& \le
	\frac{8}{(\mathrm{Im}\,z)^8}
	\sum_{i=1}^n
	|\langle v_*,(Z_i-Z_i')w_*\rangle|^4 +
	\frac{128}{(\mathrm{Im}\,z)^{12}}
        R(X)^4
	\sum_{i=1}^n
	\|(Z_i-Z_i')u_*\|^4 \\
	&\le
	\frac{32}{(\mathrm{Im}\,z)^8}
	R(X)^2
	\sum_{i=1}^n
	|\langle v_*,(Z_i-Z_i')w_*\rangle|^2 +
	\frac{512}{(\mathrm{Im}\,z)^{12}}
        R(X)^6
	\sum_{i=1}^n
	\|(Z_i-Z_i')u_*\|^2 \\
	&\le
	\frac{16R(X)^2}{(\mathrm{Im}\,z)^4}
	\cdot
	\frac{2}{(\mathrm{Im}\,z)^4}
	\sup_{\|v\|=\|w\|=1}
	\sum_{i=1}^n
	|\langle v,(Z_i-Z_i')w\rangle|^2 
	\\ &\qquad+
	\frac{64R(X)^4}{(\mathrm{Im}\,z)^6}
	\cdot
	\frac{8}{(\mathrm{Im}\,z)^6}
        R(X)^2
	\Bigg\|\sum_{i=1}^n
	(Z_i-Z_i')^2\Bigg\|.
\end{align*}
The conclusion follows from the definition of $W$.
\end{proof}

We can now complete the proof of Proposition \ref{prop:ngconc}.

\begin{proof}[Proof of Proposition \ref{prop:ngconc}]
We begin by noting that the self-bounding property established in Lemma 
\ref{lem:Wsb} implies, by \cite[Theorem 6.19]{BLM13}, that
\begin{equation}
\label{eq:selfb}
	\log \EE[e^{W/a}]
	\le
	\frac{2}{a}\,\EE[W], \qquad
	a = 
        \frac{16R(X)^2}{(\mathrm{Im}\,z)^4}+
        \frac{64R(X)^4}{(\mathrm{Im}\,z)^6}.
\end{equation}
On the other hand, the estimate of Lemma \ref{lem:W} implies,
by the exponential Poincar\'e inequality
\cite[Theorem 6.16]{BLM13}, that for $0\le\lambda<a^{-\frac{1}{2}}$
$$
	\log \EE[e^{\lambda \{\|(z\id - X)^{-1}\|-\EE \|(z\id - 
	X)^{-1}\|\}}]
	\le \frac{\lambda^2a}{1-\lambda^2a}
	\log \EE[e^{W/a}].
$$
Combining these estimates with a Chernoff bound
\cite[p.\ 29]{BLM13} yields
$$
	\mathbf{P}\big[
	\|(z\id - X)^{-1}\|\ge \EE \|(z\id - X)^{-1}\| +
	\sqrt{8\EE[W]x} + \sqrt{a}\,x\big]
	\le e^{-x}
$$
for all $x\ge 0$. This yields a tail bound for deviation above the mean.

We must now prove a tail bound for deviation below the mean. 
This requires a variant of the second inequality of \cite[Theorem 
6.16]{BLM13}, whose proof we spell out for completeness. The last 
inequality of \cite[Theorem 6.15]{BLM13} and Lemma \ref{lem:W} imply
$$
	\mathrm{Ent}\big[e^{-\lambda \|(z\id - X)^{-1}\|}\big]
	\le
	\lambda^2\vartheta(\lambda b)\,
	\EE\big[
	We^{-\lambda \|(z\id - X)^{-1}\|}\big],\qquad
	b=\frac{2R(X)}{(\mathrm{Im}\,z)^2}
$$
for $\lambda\ge 0$, where $\mathrm{Ent}(Z):=\EE[Z\log 
Z]-\EE[Z]\log\EE[Z]$ and we used that $\vartheta(x):=\frac{e^x-1}{x}$ is a 
positive increasing function. In particular, as $b^2\le a$
and $\vartheta(1)\le 2$, 
$$
	\mathrm{Ent}\big[e^{-\lambda \|(z\id - X)^{-1}\|}\big]
	\le
	2\lambda^2\,
	\EE\big[
	We^{-\lambda \|(z\id - X)^{-1}\|}\big]
$$
for $0\le \lambda\le a^{-\frac{1}{2}}$. Applying the duality 
formula of entropy as in \cite[p.\ 187]{BLM13} yields
$$
	\mathrm{Ent}\big[e^{-\lambda \|(z\id - X)^{-1}\|}\big]
	\le
	\frac{2\lambda^2a}{
	1-2\lambda^2a}
	\log\EE[e^{W/a}]
	\, \EE\big[e^{-\lambda \|(z\id - X)^{-1}\|}\big]
$$
for $0\le \lambda\le (2a)^{-\frac{1}{2}}$. Therefore
$$
	\frac{d}{d\lambda} 
	\bigg(
	\frac{1}{\lambda} \log \EE\big[e^{-\lambda \|(z\id - X)^{-1}\|}\big]
	\bigg)
	=
	\frac{\mathrm{Ent}\big[e^{-\lambda \|(z\id - X)^{-1}\|}\big]}
	{\lambda^2\, \EE\big[e^{-\lambda \|(z\id - X)^{-1}\|}\big]}
	\le
	\frac{2a}{
        1-2\lambda^2a}
        \log\EE[e^{W/a}]
$$
for $0\le \lambda\le (2a)^{-\frac{1}{2}}$. Integrating both sides yields
\begin{align*}
	\log \EE\big[e^{-\lambda \{\|(z\id - X)^{-1}\|
	-\EE\|(z\id - X)^{-1}\|\}}\big]
	&\le
	\lambda\sqrt{2a}\, \mathrm{arctanh}(\lambda\sqrt{2a}) 
        \log\EE[e^{W/a}]
	\\
	&
	\le \frac{2\lambda^2a}{1-\lambda\sqrt{2a}}	
        \log\EE[e^{W/a}]
	\le \frac{4\EE[W]\lambda^2}{1-\lambda\sqrt{2a}},
\end{align*}
where we used $\mathrm{arctanh}(x)\le \frac{x}{1-x}$ in the second 
inequality and \eqref{eq:selfb} in the last inequality.
We can now apply a Chernoff bound
\cite[p.\ 29]{BLM13} to obtain
$$
        \mathbf{P}\big[
        \|(z\id - X)^{-1}\|\le \EE \|(z\id - X)^{-1}\| -
        4\sqrt{\EE[W]x} - \sqrt{2a}\,x\big]
        \le e^{-x}
$$
for all $x\ge 0$. This yields a tail bound for deviation below the mean.

To conclude the proof, it remains to combine the upper and lower tail 
bounds by the union bound, and to use
Lemma \ref{lem:EW} to estimate $\EE[W]$.
\end{proof}

\subsubsection{Spectral statistics}
\label{sec:concspstat}

Theorem \ref{thm:smuniv} establishes universality of the expectations of 
spectral statistics of the form $\langle v,\varphi(X)w\rangle$. A 
corresponding tail bound would follow if we can prove a concentration 
inequality for such spectral statistics. This problem turns out to be 
subtle even in the Gaussian case: even when $\varphi$ is Lipschitz, the 
Lipschitz property of $(g_1,\ldots,g_N)\mapsto\langle v,\varphi(G)w\rangle$ 
is not obvious. Deep results on the latter problem \cite{CMPS14} could be 
applied in the Gaussian setting, but do not appear to be sufficiently 
powerful to handle the non-Gaussian case.

Here we take a different approach. Using functional calculus \cite[\S 
2.2]{Dav95}, $\varphi(X)$ can be expressed as an integral of the 
resolvent of $X$. (We will use an essentially equivalent formulation that 
appears in the proof of \cite[Theorem 6.2]{HT05}.) With this representation 
in hand, we can readily repeat the proof of Proposition \ref{prop:ngconc} 
to obtain a concentration inequality. The main result of this section
is the following.

\begin{prop}
\label{prop:specconc}
For $\varphi\in W^{4,1}(\mathbb{R})$ and $v,w\in\mathbb{R}^d$ with
$\|v\|=\|w\|=1$, we have
$$
	\mathbf{P}\big[|\langle v,\varphi(G)w\rangle -
	\EE[\langle v,\varphi(G)w\rangle]|
	\ge \|\varphi\|_{W^{3,1}}\sigma_*(X)\sqrt{x}
	\big]\le 4e^{-Cx}
$$
and
\begin{multline*}
	\mathbf{P}\big[|\langle v,\varphi(X)w\rangle -
	\EE[\langle v,\varphi(X)w\rangle]|
	\ge \|\varphi\|_{W^{4,1}}\big\{
	(R(X)+R(X)^2)x \\
	+(\sigma_*(X) + 
	R(X)^{\frac{1}{2}}(\EE\|X-\EE X\|)^{\frac{1}{2}} +
	R(X)(\EE\|X-\EE X\|^2)^{\frac{1}{2}})\sqrt{x}
	\big\}
	\big]\le 4e^{-Cx}
\end{multline*}
for all $x\ge 0$, where $C$ is a universal constant.
\end{prop}

The basis for the proof is the following identity.

\begin{lem}
\label{lem:hellfer}
For any $M\in\M_d(\mathbb{C})_{\rm sa}$, $p\in\mathbb{N}$, and $\varphi\in 
C^\infty_c(\mathbb{R})$, we have
\begin{multline*}
	\varphi(M) = 
	-\frac{1}{\pi}
	\lim_{\varepsilon\downarrow 0}
	\mathrm{Im}\,
	\int_{-\infty}^\infty \bigg(1+\frac{d}{dx}\bigg)^p\varphi(x)
	\times \mbox{} \\
	\frac{(1+i)^p}{(p-1)!}
	\int_0^\infty 
	((x+t+i(\varepsilon+t))\id-M)^{-1}
	\,t^{p-1}e^{-(1+i)t}
	\,dt\,dx.
\end{multline*}
\end{lem}

\begin{proof}
The identity follows by following verbatim the proof of 
\cite[Theorem 6.2]{HT05}, noting that we can express $\langle 
v,\varphi(M)v\rangle = \int \varphi(x)\,\mu_v(dx)$ for some measure $\mu_v$ 
(the spectral distribution of $M$ with respect to the state
$\tau_v(M):=\langle v,Mv\rangle$).
\end{proof}

We can now use the above representation to establish discrete (and 
continuous) gradient bounds of $\langle v,\varphi(X)w\rangle$ along the 
lines of Lemma \ref{lem:W}.

\begin{lem}
\label{lem:V}
Let $X=Z_0+\sum_{j=1}^n Z_j$ as in \eqref{eq:model}, define $X^{\sim i}$ 
as in Lemma \ref{lem:W}, let $G=A_0+\sum_{i=1}^N g_iA_i$ as in the proof 
of Lemma \ref{lem:gconc}, and let $\|v\|=\|w\|=1$.
\medskip
\begin{enumerate}[a.]
\itemsep\medskipamount
\item
The function $(g_1,\ldots,g_N)\mapsto \langle v,\varphi(G)w\rangle$ is
$C\|\varphi\|_{W^{3,1}}\sigma_*(X)$-Lipschitz.
\item $|\langle v,\varphi(X)w\rangle-
\langle v,\varphi(X^{\sim i})w\rangle| \lesssim\|\varphi\|_{W^{3,1}}R(X)$.
\item $\sum_{i=1}^n |\langle v,\varphi(X)w\rangle-
\langle v,\varphi(X^{\sim i})w\rangle|^2 \lesssim
\|\varphi\|_{W^{4,1}}^2V$.
\end{enumerate}
\medskip
Here $C$ is a universal constant, and $V$ is defined as $W$ in Lemma 
\ref{lem:W} with $\mathrm{Im}\,z=1$.
\end{lem}

\begin{proof}
By a routine approximation argument, we may assume that $\varphi\in 
C^\infty_c(\mathbb{R})$. Now
note that $\langle v,(z\id-G)^{-1}w\rangle$ is 
$\frac{\sigma_*(X)}{(\mathrm{Im}\,z)^2}$-Lipschitz as in the proof of Lemma 
\ref{lem:gconc}. Thus
Lemma \ref{lem:hellfer} with $p=3$ shows that $\langle 
v,\varphi(G)w\rangle$ is Lipschitz with constant
$$
	\frac{1}{\pi}\cdot
	3\|\varphi\|_{W^{3,1}}\cdot
	\lim_{\varepsilon\downarrow 0}
	\frac{(\sqrt{2})^3}{2!}
	\int_0^\infty \frac{\sigma_*(X)}{(\varepsilon+t)^2}\,t^2e^{-t}\,dt
	\lesssim 
	\|\varphi\|_{W^{3,1}}\sigma_*(X),
$$
which establishes part a. Part b.\ follows in precisely the same manner 
using that $|\langle v,(z\id-X)^{-1}w\rangle-
\langle v,(z\id-X^{\sim i})^{-1}w\rangle|\le 
\frac{2R(X)}{(\mathrm{Im}\,z)^2}$ as in the proof of Lemma \ref{lem:W}.

For part c., we begin by applying Lemma \ref{lem:hellfer} with $p=4$ to 
estimate
\begin{multline*}
	\Bigg(
	\sum_{i=1}^n
	|\langle v,\varphi(X)w\rangle -
	\langle v,\varphi(X^{\sim i})w\rangle|^2\Bigg)^{\frac{1}{2}}
	\le 
	\frac{4}{3!\pi}
	\int_{-\infty}^\infty 
	\bigg|\bigg(1+\frac{d}{dx}\bigg)^4\varphi(x)\bigg|
	\times \mbox{} \\
	\int_0^\infty 
	\Bigg(
	\sum_{i=1}^n	
	|\langle v,((x+t+it)\id-X)^{-1}w\rangle -
	\langle v,((x+t+it)\id-X^{\sim i})^{-1}w\rangle|^2
	\Bigg)^{\frac{1}{2}}
	\,t^3 e^{-t}
	\,dt\,dx,
\end{multline*}
where we used the triangle inequality to bring the Euclidean norm with 
respect to the index $i$ inside the integral. But it follows as in the 
proof of Lemma \ref{lem:W} that the quantity inside the brackets on the 
second line is bounded by the random variable $W$ of Lemma \ref{lem:W} with 
$\mathrm{Im}\,z=t$. The conclusion follows readily.
\end{proof}

We can now complete the proof of Proposition \ref{prop:specconc}.

\begin{proof}[Proof of Proposition \ref{prop:specconc}]
By rescaling $\varphi$, we may assume without loss of generality that 
$\|\varphi\|_{W^{3,1}}=1$ (Gaussian case) or 
$\|\varphi\|_{W^{4,1}}=1$ (non-Gaussian case).
By a union bound, it suffices to consider
separately the real and imaginary parts of $\langle 
v,\varphi(G)w\rangle$ and $\langle v,\varphi(X)w\rangle$, respectively. The 
conclusion now follows using Lemma \ref{lem:V} by 
repeating the proofs of Lemma \ref{lem:gconc} and Proposition 
\ref{prop:ngconc} verbatim with $\mathrm{Im}\,z=1$.
\end{proof}

\section{Universality of spectral statistics}
\label{sec:univstat}

The aim of this section is to prove our main universality principles for 
spectral statistics. The basic idea behind the proofs is that we will 
interpolate between the non-Gaussian and Gaussian models, and estimate the 
rate of change along the interpolation by means of the cumulant expansion 
and trace inequalities. This program will be implemented for the moments, 
resolvent moments, and resolvent in sections \ref{sec:mompf}, 
\ref{sec:respf}, and \ref{sec:smpf}, respectively. Before we do so, 
however, we first introduce some basic constructions that are common to 
all the proofs.

\subsection{Preliminaries}
\label{sec:univstprelim}

We always fix a random matrix $X$ as in \eqref{eq:model}, and let $G$ be 
its Gaussian model. Throughout this section, we will further assume that 
$X$ and $G$ are independent of each other; this will entail no loss of 
generality, as the universality results that are proved in this 
section---Theorems \ref{thm:momentuniv}, \ref{thm:smuniv}, and 
\ref{thm:resuniv}---are independent of the joint distribution of $X$ 
and $G$. Define
$$
	X(t) := \EE X +\sqrt{t}\,(X-\EE X)+\sqrt{1-t}\,(G-\EE G),\qquad
	t\in[0,1].
$$
The random matrix $X(t)$ interpolates between $X(1)=X$ and $X(0)=G$, where
the interpolation is chosen so that $\EE X(t)$ and $\Cov(X(t))$ are 
independent of $t$.
The basic principle behind all the proofs of this section is that we aim 
to compute $\frac{d}{dt}\EE[f(X(t))]$ for the relevant spectral statistic 
$f:\M_d(\mathbb{C})_{\rm sa}\to\mathbb{C}$ using the cumulant 
expansion. To this end, we will choose the random vector $Y_i$ in the 
statements of Theorems \ref{thm:cumpoly} and \ref{thm:cumsm} to 
be the $2d^2$-dimensional vector of the real and imaginary parts of the 
entries of the random matrix $Z_i$. In the following, we will apply the
notation for partitions in section \ref{sec:cumdef} without further 
comment.

For our purposes, it will be convenient to reformulate the resulting 
expansions by combining them with the cumulant formula of Lemma 
\ref{lem:shir}. Before we can do so, we must introduce a simple 
construction that will facilitate working with the second identity of 
Lemma \ref{lem:shir}.
For any $k\in\mathbb{N}$ and partition $\pi\in\mathrm{P}([k])$, define
random matrices $Z_{i1|\pi},\ldots,Z_{ik|\pi}$ $(i\in[n])$ with the
following properties:
\begin{enumerate}[1.]
\item $(Z_{ij|\pi})_{i\in[n]}$ has the same distribution as
$(Z_i)_{i\in[n]}$.
\item $(Z_{ij|\pi})_{i\in[n]}=(Z_{ij'|\pi})_{i\in[n]}$ for indices $j,j'$
that belong to
the same element of $\pi$.
\item $(Z_{ij|\pi})_{i\in[n]}$ are independent for indices $j$ that belong
to distinct elements of $\pi$.
\item $(Z_{ij|\pi})_{i\in[n],j\in[k]}$ is independent
of $X$ and $G$.
\end{enumerate}
This construction will be fixed in the sequel. (We do not 
specify the joint distribution of these matrices for different $k,\pi$ 
as these will not arise in the analysis.) We can now state a version
of Theorem \ref{thm:cumpoly} in the present setting.

\begin{cor}
\label{cor:mtxcumpoly}
For any $f:\M_d(\mathbb{C})_{\rm sa}\to\mathbb{C}$
that is polynomial in the matrix entries,
\begin{multline*}
	\frac{d}{dt}\EE[f(X(t))] = \mbox{}\\
	\frac{1}{2}\sum_{k=3}^\infty
	\frac{t^{\frac{k}{2}-1}}{(k-1)!}
	\sum_{\pi\in\mathrm{P}([k])}
	(-1)^{|\pi|-1}(|\pi|-1)!\,
	\EE\Bigg[
	\sum_{i=1}^n
	\partial_{Z_{i1|\pi}}\cdots\partial_{Z_{ik|\pi}}f(X(t))
	\Bigg]
\end{multline*}
where $\partial_Bf$ denotes the directional derivative of
$f$ in the direction $B\in\M_d(\mathbb{C})_{\rm sa}$.
\end{cor}

\begin{proof}
We can clearly write $G-\EE G=\sum_{i=1}^n G_i$, where $G_i$ is the
Gaussian model associated to $Z_i$ and $G_1,\ldots,G_n$ are independent.

Let $\iota:\M_d(\mathbb{C})\to\mathbb{R}^{2d^2}$ be 
defined by 
$\iota(M):=(\mathrm{Re}\,M_{uv},\mathrm{Im}\,M_{uv})_{u,v\in[d]}$, 
let $Y_i=\iota(Z_i)$ and $U_i=\iota(G_i)$, and define $Y(t)$ as in Theorem 
\ref{thm:cumpoly}. Then we have
$$
	\iota(X(t)) = \iota(\EE X) + \sum_{i=1}^n Y_i(t).
$$
In particular, we can equivalently view $f(X(t))=f\big(\EE X + 
\sum_{i=1}^n \iota^{-1}(Y_i(t))\big)$ as a function of $Y(t)$. Applying 
Theorem \ref{thm:cumpoly} to the latter yields
\begin{multline*}
	\frac{d}{dt}\EE[f(X(t))] =
	\frac{1}{2}\sum_{k=3}^\infty
	\frac{t^{\frac{k}{2}-1}}{(k-1)!}\times\mbox{}\\
	\sum_{i=1}^n
	\sum_{\substack{(u_j,v_j,\alpha_j)\in\mathcal{I}\\
	j=1,\ldots,k}}
	\kappa((Z_i)_{u_1v_1}^{\alpha_1},\ldots,(Z_i)_{u_kv_k}^{\alpha_k})\,
	\EE\bigg[
	\frac{\partial^kf}{\partial  M_{u_1v_1}^{\alpha_1}\cdots
	\partial M_{u_kv_k}^{\alpha_k}}(X(t))
	\bigg],
\end{multline*}
where $\mathcal{I}:=[d]\times [d]\times \{\mathrm{R},\mathrm{I}\}$
and we denote $M_{uv}^\mathrm{R}:=\mathrm{Re}\,M_{uv}$ and
$M_{uv}^{\mathrm{I}}:=\mathrm{Im}\,M_{uv}$.
The conclusion follows by applying the second identity 
of Lemma \ref{lem:shir} to the cumulant, and using the independence 
structure of $Z_{ij|\pi}$ to merge the product of expectations in the 
resulting identity into a single expectation.
\end{proof}

The following is the analogous version of Theorem \ref{thm:cumsm}.

\begin{cor}
\label{cor:mtxcumsm}
For any $p\ge 3$ and smooth function $f:\M_d(\mathbb{C})\to\mathbb{C}$ we 
have
\begin{multline*}
	\frac{d}{dt}\EE[f(X(t))] = \mbox{}\\
	\frac{1}{2}\sum_{k=3}^{p-1}
	\frac{t^{\frac{k}{2}-1}}{(k-1)!}
	\sum_{\pi\in\mathrm{P}([k])}
	(-1)^{|\pi|-1}(|\pi|-1)!\,
	\EE\Bigg[
	\sum_{i=1}^n
	\partial_{Z_{i1|\pi}}\cdots\partial_{Z_{ik|\pi}}f(X(t))
	\Bigg]+ \mathcal{R},
\end{multline*}
where the remainder term satisfies
\begin{multline*}
	 |\mathcal{R}| \lesssim
	\sup_{s,t\in[0,1]}
	\Bigg\{
	\Bigg|
	\sum_{i=1}^n
	\EE[\partial_{Z_i}^pf(X(t,i,s))]\Bigg| +
	\\
	 \max_{2\le k\le p-1} 
	\Bigg|
	\sum_{\pi\in\mathrm{P}([k])}
	\frac{(-1)^{|\pi|-1}(|\pi|-1)!}{(k-1)!}\,
	\sum_{i=1}^n
	\EE[
	\partial_{Z_i}^{p-k}
	\partial_{Z_{i1|\pi}}\cdots\partial_{Z_{ik|\pi}}f(X(t,i,s))
	]\Bigg|\Bigg\}
\end{multline*}
with $X(t,i,s) := X(t) - (1-s)\sqrt{t}Z_i$.
\end{cor}

\begin{proof}
The conclusion follows from Theorem \ref{thm:cumsm} in exactly the same 
manner as we derived Corollary \ref{cor:mtxcumpoly} from Theorem 
\ref{thm:cumpoly}.
\end{proof}

\subsection{Moments}
\label{sec:mompf}

The aim of this section is to show that the moments
$\EE[\ntr X^{2p}]$ are close to their Gaussian analogues
$\EE[\ntr G^{2p}]$. To this end, we will first compute
$\frac{d}{dt}\EE[\ntr X(t)^{2p}]$ by means of the cumulant expansion, and 
then estimate the individual terms to obtain a differential inequality.

We begin by computing the derivatives of the moment function
$M\mapsto \ntr[M^{2p}]$.

\begin{lem}
\label{lem:momder}
Let $p\in\mathbb{N}$ and $B_1,\ldots,B_k\in\M_d(\mathbb{C})_{\rm sa}$. 
Then
\begin{multline*}
	\partial_{B_1}\cdots\partial_{B_k}\ntr[M^{2p}]
	=\mbox{}\\
	\sum_{\sigma\in\mathrm{Sym}(k)}
	\sum_{\substack{r_1,\ldots,r_{k+1}\ge 0\\
	r_1+\cdots+r_{k+1}=2p-k}}
	\ntr[M^{r_1}B_{\sigma(1)}
	M^{r_2}B_{\sigma(2)}\cdots
	M^{r_k}B_{\sigma(k)}M^{r_{k+1}}].
\end{multline*}
\end{lem}

\begin{proof}
This follows by applying the product rule $k$ times.
\end{proof}

We will also need the following estimate.

\begin{lem}
\label{lem:feller}
For any $k\in\mathbb{N}$, we have
$$
        \sum_{\pi\in\mathrm{P}([k])}
        (|\pi|-1)! 
	\le 2^k(k-1)!.
$$
\end{lem}

\begin{proof}
We first crudely estimate
$$
        \sum_{\pi\in\mathrm{P}([k])}
        (|\pi|-1)! 
	\le
        \sum_{\pi\in\mathrm{P}([k])}
        (|\pi|-1)! 
	\prod_{J\in\pi} |J|!.
$$
Now note that any partition of $[k]$ into $m$ parts can be generated by 
first choosing $r_1,\ldots,r_m\ge 1$ such that $r_1+\cdots+r_m=k$, and 
then choosing disjoint sets $J_1,\ldots,J_m$ with $|J_i|=r_i$. Moreover, 
each distinct partition is generated precisely $m!$ times in this manner, 
as relabeling the sets $J_i$ does not change the partition. Therefore
\begin{align*}
        \sum_{\pi\in\mathrm{P}([k])}
        (|\pi|-1)! 
	\prod_{J\in\pi} |J|! &=
	\sum_{m=1}^k (m-1)!\, \frac{1}{m!}
	\sum_{\substack{r_1,\ldots,r_m\ge 1
	\\ r_1+\cdots+r_m=k}}
	{k\choose r_1,\ldots,r_m}
	\prod_{j=1}^m r_j!
	\\ &=
	(k-1)!
	\sum_{m=1}^k 
	{k\choose m} = (2^k-1)(k-1)!,
\end{align*}
where the second equality follows as the number of $m$-tuples of
positive integers that sum to $k$ is ${k-1\choose m-1}$, and the 
last equality holds by the binomial theorem.
\end{proof}

We are now ready to apply the cumulant expansion.

\begin{prop}
\label{prop:momdiffeq}
For any $p\in\mathbb{N}$ with $p\ge 2$,
$2p\le q\le\infty$, and $t\in[0,1]$, we have
\begin{multline*}
	\bigg|\frac{d}{dt}\EE[\ntr X(t)^{2p}]\bigg| \le
	64p^3 \max\{R_q(X)\sigma_q(X)^2,R_q(X)^3\}\times\\
	\max\big\{
        \EE[\ntr X(t)^{2p}]^{1-\frac{3}{2p}},
	(8pR_q(X))^{2p-3}
	\big\}.
\end{multline*}
\end{prop}

\begin{proof}
Combining Corollary \ref{cor:mtxcumpoly} and Lemma \ref{lem:momder}
yields
\begin{multline*}
	\frac{d}{dt}\EE[\ntr X(t)^{2p}] = 
	\frac{1}{2}\sum_{k=3}^{2p}
	kt^{\frac{k}{2}-1}
	\sum_{\pi\in\mathrm{P}([k])}
	(-1)^{|\pi|-1}(|\pi|-1)! \times\mbox{}\\ 
	\sum_{\substack{r_1,\ldots,r_{k+1}\ge 0\\
	r_1+\cdots+r_{k+1}=2p-k}}
	\sum_{i=1}^n
	\EE[
	\ntr X(t)^{r_1}Z_{i1|\pi}
	X(t)^{r_2}Z_{i2|\pi}\cdots
	X(t)^{r_k}Z_{ik|\pi}X(t)^{r_{k+1}}].
\end{multline*}
Here we used that as 
$(Z_{i\sigma(j)|\pi})_{i\in[n],j\in[k]}$ and
$(Z_{ij|\sigma^{-1}(\pi)})_{i\in[n],j\in[k]}$ have the same distribution
for any permutation $\sigma$, we can eliminate the sum over $\sigma$ in 
Lemma \ref{lem:momder} by symmetry. 
Now let $r=\frac{(2p-k)q}{q-k}$, so that $2p-k\le r\le 2p$.
Let $p_j=\frac{r}{r_{j+1}}$ for $j<k$ and $p_k=\frac{r}{r_{k+1}+r_1}$.
Then we can apply Proposition \ref{prop:trace} to estimate
\begin{align*}
	&\Bigg|
	\sum_{i=1}^n
	\EE[
	\ntr X(t)^{r_1}Z_{i1|\pi}
	X(t)^{r_2}Z_{i2|\pi}\cdots
	X(t)^{r_k}Z_{ik|\pi}X(t)^{r_{k+1}}]
	\Bigg|
	\\ 
	&\qquad\le
	R_q(X)^{\frac{(k-2)q}{q-2}}\sigma_q(X)^{\frac{2(q-k)}{q-2}}
	\EE[\ntr X(t)^r]^{\frac{2p-k}{r}}
	\\ &\qquad\le
	R_q(X)^{\frac{(k-2)q}{q-2}}\sigma_q(X)^{\frac{2(q-k)}{q-2}}
	\EE[\ntr X(t)^{2p}]^{1-\frac{k}{2p}}
\end{align*}
for any $r_1,\ldots,r_{k+1}\ge 0$ with $r_1+\cdots+r_{k+1}=2p-k$.
It follows that
$$
	\bigg|\frac{d}{dt}\EE[\ntr X(t)^{2p}]\bigg| \le 
	\frac{1}{2}\sum_{k=3}^{2p}
	(4p)^k
	R_q(X)^{\frac{(k-2)q}{q-2}}\sigma_q(X)^{\frac{2(q-k)}{q-2}}
        \EE[\ntr X(t)^{2p}]^{1-\frac{k}{2p}}
$$
using Lemma \ref{lem:feller}, $t\le 1$, and that 
the number of $(k+1)$-tuples of nonnegative integers that sum to
$2p-k$ is ${2p\choose k}\le \frac{(2p)^k}{k!}$. 
To simplify the expression, we estimate
\begin{align*}
	\bigg|\frac{d}{dt}\EE[\ntr X(t)^{2p}]\bigg| &\le 
	\frac{1}{2}\sum_{k=3}^{2p} 2^{-k}
	(8p)^k
	R_q(X)^{\frac{(k-2)q}{q-2}}\sigma_q(X)^{\frac{2(q-k)}{q-2}}
        \EE[\ntr X(t)^{2p}]^{1-\frac{k}{2p}}
	\\ &\le
	\frac{1}{8}
	\max_{3\le k\le 2p}
	(8p)^k
	R_q(X)^{\frac{(k-2)q}{q-2}}\sigma_q(X)^{\frac{2(q-k)}{q-2}}
        \EE[\ntr X(t)^{2p}]^{1-\frac{k}{2p}}
	\\ &\le
	64p^3
	\max\big\{
	R_q(X)^{\frac{q}{q-2}}\sigma_q(X)^{\frac{2(q-3)}{q-2}}
        \EE[\ntr X(t)^{2p}]^{1-\frac{3}{2p}},
	\\
	&\phantom{\mbox{}\le 64p^3 \max\big\{\mbox{}}
	(8p)^{2p-3}
	R_q(X)^{\frac{(2p-2)q}{q-2}}\sigma_q(X)^{\frac{2(q-2p)}{q-2}}
	\big\}.
\end{align*}
Here we used that the term inside the maximum on the second line is convex 
as a function of $k$, so that the maximum 
is attained at one of the endpoints $k\in\{3,2p\}$.
The proof is readily concluded using that
$R^{\frac{q}{q-2}-1}\sigma^{\frac{2(q-3)}{q-2}}
\le (\max\{R,\sigma\})^2$ and
$R^{\frac{(2p-2)q}{q-2}-(2p-3)-1}\sigma^{\frac{2(q-2p)}{q-2}}
\le (\max\{R,\sigma\})^2$ (as $q\ge 2p\ge 4$ implies that 
in both cases the exponents on the left-hand side are positive and sum to 
$2$).
\end{proof}

It remains to solve the differential inequality in the statement of 
Proposition \ref{prop:momdiffeq}. To this end we will use the following 
simple lemma.

\begin{lem}
\label{lem:diffineq}
Let $f:[0,1]\to\mathbb{R}_+$, $C,K\ge 0$, and $\alpha\in[0,1]$. Suppose 
that
$$
	\bigg|\frac{d}{dt}f(t)\bigg| \le
	C\max\{f(t)^{1-\alpha},K^{1-\alpha}\}
$$
for all $t\in[0,1]$. Then
$$
	|f(1)^\alpha-f(0)^\alpha| \le C\alpha +
	K^\alpha.
$$
\end{lem}

\begin{proof}
It follows readily by the chain rule that
$$
	\bigg|\frac{d}{dt}(f(t)+K)^\alpha\bigg| =
	\alpha (f(t)+K)^{\alpha-1}
	\bigg|\frac{d}{dt}f(t)\bigg|
	\le
	C\alpha,
$$
so that
$$
	|(f(1)+K)^\alpha-(f(0)+K)^\alpha| =
	\bigg|
	\int_0^1  \frac{d}{dt}(f(t)+K)^\alpha\, dt\bigg|
	\le
	C\alpha.
$$
The conclusion follows as
$x^\alpha-y^\alpha \le
(x+K)^\alpha - (y+K)^\alpha + K^\alpha$
for any $x,y\ge 0$.
\end{proof}

We can now conclude the proof of Theorem \ref{thm:momentuniv}.

\begin{proof}[Proof of Theorem \ref{thm:momentuniv}: first inequality]
If $p=1$, then $\EE[\ntr X(t)^{2p}]=\EE[\ntr X^2]$ is independent of $t$ 
by construction, and the conclusion is trivial. If $p\ge 2$, we can 
apply Lemma \ref{lem:diffineq} with $\alpha=\frac{3}{2p}\in[0,1]$
and Proposition \ref{prop:momdiffeq} to obtain
$$
	|\EE[\ntr X^{2p}]^{\frac{3}{2p}}-
	\EE[\ntr G^{2p}]^{\frac{3}{2p}}|
	\le
	96 p^2 \max\{R_q(X)\sigma_q(X)^2,R_q(X)^3\}
	+ (8pR_q(X))^{3}.
$$
The conclusion follows as
$|{}x^{\frac{1}{3}}-y^{\frac{1}{3}}|\le |x-y|^{\frac{1}{3}}$ for 
$x,y\ge 0$.
\end{proof}

The second inequality of Theorem \ref{thm:momentuniv} follows by a slight
variation of the proof.

\begin{proof}[Proof of Theorem \ref{thm:momentuniv}: second inequality]
Note that $\sigma_{2p}(X)=\sigma_{2p}(X(t))$ for all $t$, as
the definition of $\sigma_{2p}(X)$ depends only on $\mathrm{Cov}(X)$.
We can therefore estimate
$$
	\sigma_{2p}(X)^{2p} =
	\ntr\, (\EE[X(t)^2] - \EE[X(t)]^2)^p
	\le
	\ntr \EE[X(t)^2]^p
	\le \EE[\ntr X(t)^{2p}]
$$
for every $t\in[0,1]$. Here we used that
$\ntr A^p \le \ntr B^p$ for $B\ge A\ge 0$, and
that $A\mapsto \ntr A^p$ is convex for $A\ge 0$
\cite[Theorem 2.10]{Car10}.
Furthermore, note that 
$$
	R_{2p}(X) 
	\le R_2(X)^{\frac{q-2p}{p(q-2)}} 
	R_q(X)^{\frac{(p-1)q}{p(q-2)}}
	\le \sigma_{2p}(X)^{\frac{q-2p}{p(q-2)}} 
	R_q(X)^{\frac{(p-1)q}{p(q-2)}}
$$
by the Riesz convexity theorem and as 
$R_2(X)=\sigma_2(X)\le\sigma_{2p}(X)$.
Consequently, we can bound the differential inequality in the
proof of Proposition \ref{prop:momdiffeq} as
\begin{align*}
	\bigg|\frac{d}{dt}\EE[\ntr X(t)^{2p}]\bigg| 
	&\le 
	\frac{1}{2}\sum_{k=3}^{2p}
	(4p)^k
	R_{2p}(X)^{\frac{(k-2)p}{p-1}}\sigma_{2p}(X)^{\frac{2p-k}{p-1}}
        \EE[\ntr X(t)^{2p}]^{1-\frac{k}{2p}}
\\
	&\le 
	\frac{1}{2}\sum_{k=3}^{2p}
	(4p)^k
        R_q(X)^{\frac{(k-2)q}{q-2}}
        \EE[\ntr X(t)^{2p}]^{1-\frac{k}{2p}+\frac{q-k}{p(q-2)}}
\\
	&\le 
	\frac{1}{8}
	\max_{3\le k\le 2p}
	(8p)^k
        R_q(X)^{\frac{(k-2)q}{q-2}}
        \EE[\ntr X(t)^{2p}]^{1-\frac{k}{2p}+\frac{q-k}{p(q-2)}}.
\end{align*}
By convexity, we may bound all terms in the maximum by their value at 
either $k=2+\frac{q-2}{q}\le 3$ or $k=2+2p\frac{q-2}{q}\ge 2p$, so that
$$
	\bigg|\frac{d}{dt}\EE[\ntr X(t)^{2p}]\bigg| \le
	\frac{(8p)^{2+\frac{q-2}{q}}
        R_q(X)}{8}
	\max\big\{
        \EE[\ntr X(t)^{2p}]^{1-\frac{1}{2p}},
	\big((8p)^{\frac{q-2}{q}}
        R_q(X)\big)^{2p-1}\big\}.
$$
The conclusion follows from Lemma \ref{lem:diffineq}.
\end{proof}

\begin{rem}
That there is considerable room in the proof of Theorem 
\ref{thm:momentuniv} is evident from the crude inequality in the first 
equation display of the proof of Lemma \ref{lem:feller}. This additional 
room can be used to capture models whose summands $Z_i$ are not 
uniformly bounded, but have subexponential tails. For the purposes of 
this paper, such an extension is not needed as the truncation method that 
will be developed in section \ref{sec:trunc} below yields far more general 
results. However, this extra room can be of significant utility in 
extending the approach of this paper to random matrices that are not 
captured by the independent sum model \eqref{eq:model}, cf.\ \cite{PvH23}. 
\end{rem}

\subsection{Resolvent moments}
\label{sec:respf}

The aim of this section is to prove the following universality principle 
for the moments of the resolvent. This result will form the basis for the 
proof of Theorem \ref{thm:specuniv} (which is given in section 
\ref{sec:univspec} below).

\begin{thm}[Resolvent moments universality]
\label{thm:resuniv}
We have
$$
	\big|\EE[\ntr |z\id-X|^{-2p}]^{\frac{1}{2p}} -
	\EE[\ntr |z\id-G|^{-2p}]^{\frac{1}{2p}}\big|
	\lesssim
	\frac{R(X)\sigma(X)^2p^2+R(X)^3p^3}{
	(\mathrm{Im}\,z)^4}
$$
for any $p\in\mathbb{N}$ and
$z\in\mathbb{C}$ with $\mathrm{Im}\,z>0$.
\end{thm}

The proof of Theorem \ref{thm:resuniv} is very similar in spirit to that 
of Theorem \ref{thm:momentuniv}. However, as the resolvent is not a 
polynomial (and does not have a globally convergent power series), we must 
truncate the cumulant expansion as in Corollary \ref{cor:mtxcumsm}.

We begin by computing the derivatives of $M\mapsto\ntr |z\id - M|^{-2p}$.

\begin{lem}
\label{lem:resder}
Let $z\in\mathbb{C}$ with $\mathrm{Im}\,z>0$,
$p\in\mathbb{N}$, and $M,B_1,\ldots,B_k\in\M_d(\mathbb{C})_{\rm sa}$.
Denote the resolvent of $M$ as $\RR_M(z) := (z\id - M)^{-1}$. Then
\begin{align*}
	&\partial_{B_1}\cdots\partial_{B_k}\ntr |z\id - M|^{-2p} 
	= \mbox{}
\\
	&\quad
	\sum_{\sigma\in\mathrm{Sym}(k)}
	\sum_{\substack{l,m\ge 0\\l+m=k}}
	\sum_{\substack{r_1,\ldots,r_{l+1}\ge 1\\
	r_1+\cdots+r_{l+1}=p+l}}
	\sum_{\substack{s_1,\ldots,s_{m+1}\ge 1\\
	s_1+\cdots+s_{m+1}=p+m}}
	\ntr[
	\RR_M(z)^{r_1} B_{\sigma(1)}
\\
	&\quad\qquad
	\cdots
	\RR_M(z)^{r_l} B_{\sigma(l)}
	\RR_M(z)^{r_{l+1}}
	\RR_M(\bar z)^{s_1} B_{\sigma(l+1)} 
	\cdots
	\RR_M(\bar z)^{s_m} B_{\sigma(k)} 
	\RR_M(\bar z)^{s_{m+1}}].
\end{align*}
In particular,
$$
	\big|\partial_{B_1}\cdots\partial_{B_k}\ntr |z\id - M|^{-2p}\big|
	\le
	\frac{(2p-1+k)!}{(2p-1)!}
	\frac{\|B_1\|_k\cdots \|B_k\|_k}{(\mathrm{Im}\,z)^{2p+k}}.
$$
\end{lem}

\begin{proof}
The identity follows by applying the product rule $k$ times to
$\ntr |z\id - M|^{-2p} = \ntr[\RR_M(z)^p \RR_M(\bar z)^p]$ and using that
$\partial_B \RR_M(z) = \RR_M(z)B\RR_M(z)$. To prove the inequality, note 
that 
each summand is bounded by 
$(\mathrm{Im}\,z)^{-2p-k}\|B_1\|_k\cdots\|B_k\|_k$ by
H\"older's inequality and $\|\RR_M(z)\|\le |\mathrm{Im}\,z|^{-1}$, 
while the sums have $\frac{(2p-1+k)!}{(2p-1)!}$ terms
(the latter is most easily seen by applying the first identity with 
$d=1$.)
\end{proof}

We can now apply the cumulant expansion.

\begin{prop}
\label{prop:resdiffeq}
For any $z\in\mathbb{C}$ with $\mathrm{Im}\,z>0$, $p\in\mathbb{N}$, and 
$t\in[0,1]$, we have
\begin{multline*}
	\bigg|\frac{d}{dt}\EE[\ntr |z\id - X(t)|^{-2p}]\bigg| 
\\	\lesssim
	\frac{p^3 R(X)\sigma(X)^2}{(\mathrm{Im}\,z)^4}
	\max\Bigg\{
	\EE[\ntr |z\id - X(t)|^{-2p}]^{1-\frac{1}{2p}},
	\frac{(32pR(X))^{6p-3}}{(\mathrm{Im}\,z)^{8p-4}}
	\Bigg\}.
\end{multline*}
\end{prop}

\begin{proof}
Combining Corollary \ref{cor:mtxcumsm} with Lemma \ref{lem:resder} yields
\begin{align*}
	&\frac{d}{dt}\EE[\ntr |z\id -X(t)|^{-2p}] = 
	\frac{1}{2}\sum_{k=3}^{6p-1}
	kt^{\frac{k}{2}-1}
	\sum_{\pi\in\mathrm{P}([k])}
	(-1)^{|\pi|-1}(|\pi|-1)!
	\times\mbox{}\\
&\enskip
        \sum_{\substack{l,m\ge 0\\l+m=k}}
        \sum_{\substack{r_1,\ldots,r_{l+1}\ge 1\\
        r_1+\cdots+r_{l+1}=p+l}}
        \sum_{\substack{s_1,\ldots,s_{m+1}\ge 1\\
        s_1+\cdots+s_{m+1}=p+m}}
	\sum_{i=1}^n
	\EE[
        \ntr
        \RR_{X(t)}(z)^{r_1} Z_{i1|\pi}
        \cdots
        \RR_{X(t)}(z)^{r_l} Z_{il|\pi}\cdot\mbox{}
\\
&\qquad\qquad\quad
        \RR_{X(t)}(z)^{r_{l+1}}
        \RR_{X(t)}(\bar z)^{s_1} Z_{i(l+1)|\pi}
        \cdots
        \RR_{X(t)}(\bar z)^{s_m} Z_{ik|\pi}
        \RR_{X(t)}(\bar z)^{s_{m+1}}
	]+ \mathcal{R}
\end{align*}
with
$$
	 |\mathcal{R}| \lesssim
	\frac{(8p-1)!}{(2p-1)!}
	\frac{2^{6p}}{(\mathrm{Im}\,z)^{8p}}
	\sum_{i=1}^n
	\EE[\ntr Z_i^{6p}].
$$
Here we eliminated the sum over permutations $\sigma$ in the identity as 
in the proof of Proposition \ref{prop:momdiffeq}, and we used Lemma 
\ref{lem:feller} and H\"older's inequality in the estimate of the 
remainder. To proceed, we apply Proposition \ref{prop:trace} with
$p_j=\frac{2p+k}{r_{j+1}}$ for $1\le j<l$,
$p_l=\frac{2p+k}{r_{l+1}+s_1}$,
$p_j=\frac{2p+k}{s_{j-l+1}}$ for $l<j<k$,
$p_k =\frac{2p+k}{s_{m+1}+r_1}$, and $q=\infty$ to estimate
\begin{align*}
	&\Bigg|\sum_{i=1}^n
	\EE[
        \ntr
        \RR_{X(t)}(z)^{r_1} Z_{i1|\pi}
        \cdots
        \RR_{X(t)}(z)^{r_l} Z_{il|\pi}
        \RR_{X(t)}(z)^{r_{l+1}}
        \RR_{X(t)}(\bar z)^{s_1} Z_{i(l+1)|\pi}
\\
&\qquad\quad
        \cdots
        \RR_{X(t)}(\bar z)^{s_m} Z_{ik|\pi}
        \RR_{X(t)}(\bar z)^{s_{m+1}}
	]\Bigg|
	\le
	R(X)^{k-2}\sigma(X)^2 \EE[\ntr |z\id - X(t)|^{-2p-k}].
\end{align*}
We can therefore estimate
\begin{multline*}
	\bigg|\frac{d}{dt}\EE[\ntr |z\id -X(t)|^{-2p}]\bigg| 
	\le
	C\frac{(8p-1)!}{(2p-1)!}
	\frac{2^{6p}}{(\mathrm{Im}\,z)^{8p}}
	\sum_{i=1}^n
	\EE[\ntr Z_i^{6p}]
	\\
	+
	\frac{1}{2}\sum_{k=3}^{6p-1}
	\frac{(2p-1+k)!}{(2p-1)!}
	2^k 
	R(X)^{k-2}\sigma(X)^2 \EE[\ntr |z\id - X(t)|^{-2p-k}]
\end{multline*}
for a universal constant $C$,
where we used Lemma \ref{lem:feller} and that the sums over
$l,m,r_j,s_j$ contain a total of ${2p-1+k\choose 2p-1}$ terms
(cf.\ the proof of Lemma \ref{lem:resder}). Thus
\begin{multline*}
	\bigg|\frac{d}{dt}\EE[\ntr |z\id -X(t)|^{-2p}]\bigg| 
	\le
	C
	\frac{(16p)^{6p}}{(\mathrm{Im}\,z)^{8p}}
	R(X)^{6p-2}\sigma(X)^2
	\\
	+
	\frac{1}{2}
	\sum_{k=3}^{6p-1}
	(16p)^k
	R(X)^{k-2}\sigma(X)^2
	\EE[\ntr |z\id - X(t)|^{-2p-k}],
\end{multline*}
where we used $\frac{(2p-1+k)!}{(2p-1)!} \le (8p)^k$
for $k\le 6p$ and $\sum_{i=1}^n\EE[\ntr Z_i^{6p}]
\le R(X)^{6p-2}\sigma(X)^2$.

We now proceed as in the proof of Proposition \ref{prop:momdiffeq}.
As the terms inside the sum are convex as a function of $k$, we can 
estimate
\begin{align*}
	&\frac{1}{2}
	\sum_{k=3}^{6p-1}
	(16p)^k
	R(X)^{k-2}\sigma(X)^2
	\EE[\ntr |z\id - X(t)|^{-2p-k}]
\\
&	\le
	\frac{1}{8}
	\max\big\{
	(32p)^3
	R(X)\sigma(X)^2
	\EE[\ntr |z\id - X(t)|^{-2p-3}],
\\
&
	\phantom{\mbox{}\le\frac{1}{8}\max\big\{\mbox{}}
	(32p)^{6p}
	R(X)^{6p-2}
	\sigma(X)^2
	\EE[\ntr |z\id - X(t)|^{-8p}]
	\big\}
\\
&	\le
	\frac{(32p)^3
	R(X)\sigma(X)^2}{8(\mathrm{Im}\,z)^4}
	\max\Bigg\{
	\EE[\ntr |z\id - X(t)|^{-2p}]^{1-\frac{1}{2p}},
	\frac{(32pR(X))^{6p-3}}{(\mathrm{Im}\,z)^{8p-4}}
	\Bigg\},
\end{align*}
where we used that
$\EE[\ntr |z\id - X(t)|^{-2p+1}]\le
\EE[\ntr |z\id - X(t)|^{-2p}]^{1-\frac{1}{2p}}$ by Jensen's inequality,
and that $\|(z\id - X(t))^{-1}\|\le (\mathrm{Im}\,z)^{-1}$. The conclusion 
follows readily.
\end{proof}

The proof of Theorem \ref{thm:resuniv} is now immediate.

\begin{proof}[Proof of Theorem \ref{thm:resuniv}]
Combine Proposition \ref{prop:resdiffeq} and Lemma 
\ref{lem:diffineq}.
\end{proof}

\subsection{Resolvent matrix}
\label{sec:smpf}

The aim of this section is to prove the resolvent universality principle 
of Theorem \ref{thm:smuniv}. In contrast to the universality principles 
for the moments and resolvent moments, the present result is much more 
classical in nature as its proof does not require the cumulant expansion; 
it could therefore also be approached by means of more traditional 
universality methods as in \cite{Cha06,Cha14}. In particular, we will 
apply Corollary \ref{cor:mtxcumsm} with $p=3$: in this special case, the 
proof of Corollary \ref{cor:mtxcumsm} uses only Taylor expansion and no 
cumulants appear.

Nonetheless, the present situation is somewhat different in nature than 
the previous universality results in that we bound the difference between 
the expected resolvents of $X$ and $G$ in norm (as opposed to their 
traces, see Remark \ref{rem:stieltjes} below). This introduces some 
additional subtleties that must be addressed in the proof.

We begin by applying Corollary \ref{cor:mtxcumsm} in the present setting.

\begin{lem}
\label{lem:smpf}
We have
$$
	\big\|\EE[(z\id-X)^{-1}] - \EE[(z\id-G)^{-1}]\big\|
	\lesssim
	\bigg(
	\frac{R(X)}{(\mathrm{Im}\,z)^4}	
	+
	\frac{R(X)^3}{(\mathrm{Im}\,z)^6}\bigg)
	\EE\Bigg\|\sum_{i=1}^n Z_i^2\Bigg\|.
$$
\end{lem}

\begin{proof}
We readily compute
\begin{multline*}
	\partial_{B_1}\partial_{B_2}\partial_{B_3}
	(z\id - M)^{-1} = \mbox{}\\
	\sum_{\sigma\in\mathrm{Sym}(3)}
	(z\id - M)^{-1}
	B_{\sigma(1)}
	(z\id - M)^{-1}
	B_{\sigma(2)}
	(z\id - M)^{-1}
	B_{\sigma(3)}
	(z\id - M)^{-1}.
\end{multline*}
Applying Corollary \ref{cor:mtxcumsm} to
$f(M) = \langle v,(z\id - M)^{-1}w\rangle$
with $p=3$ yields
\begin{multline*}
	\big\|\EE[(z\id-X)^{-1}]-\EE[(z\id -G)^{-1}]\big\|
	=
	\sup_{\|v\|=\|w\|=1}
	\bigg|\int_0^1
	\frac{d}{dt}\EE[\langle v,(z\id-X(t))^{-1}w\rangle]
	\,dt\bigg|
	\\
	\lesssim
	F(Z,Z,Z)+
	F(Z,Z',Z')+F(Z',Z,Z') + F(Z',Z',Z),
\end{multline*}
where $Z'=(Z_i')_{1\le i\le n}$ is an independent copy of $Z=(Z_i)_{1\le 
i\le n}$,
$$
	F(Z^{(1)},Z^{(2)},Z^{(3)}) =
	\sup_{\|v\|=\|w\|=1}
	\sup_{s,t\in[0,1]}
	\Bigg|
	\sum_{i=1}^n
	\EE[\langle v,
	\RR_{tis}Z_i^{(1)} \RR_{tis} Z_i^{(2)} \RR_{tis} Z_i^{(3)} \RR_{tis}
	w\rangle]
	\Bigg|,
$$ 
and $\RR_{tis}=(z\id - X(t,i,s))^{-1}$ with $X(t,i,s) = 
X(t)-(1-s)\sqrt{t}Z_i$. (Note that the term with $|\pi|=2$
in the bound on $|\mathcal{R}|$ in Corollary 
\ref{cor:mtxcumsm} vanishes as $\EE[Z_i]=0$.)

As $\|\RR_{tis}\|\le 
(\mathrm{Im}\,z)^{-1}$ and $\|Z_i\|\le R(X)$, we have
$$
	F(Z^{(1)},Z^{(2)},Z^{(3)}) \le
	\frac{R(X)}{(\mathrm{Im}\,z)^2}	
	\sup_{\|v\|=\|w\|=1}
	\sup_{s,t\in[0,1]}
	\sum_{i=1}^n
	\EE[\|Z_i^{(1)} \RR_{tis}^*v\|
	\|Z_i^{(3)} \RR_{tis}w\|].
$$
Now note that as 
$A^{-1}-B^{-1}=A^{-1}(B-A)B^{-1}$, we have
\begin{align*}
	\RR_{tis} &=
	\RR_t - (1-s)\sqrt{t}\,\RR_{tis} Z_i  \RR_t, \\
	\RR_{tis}^* &=
	\RR_t^* - (1-s)\sqrt{t}\,\RR_{tis} Z_i  \RR_t^*
\end{align*}
with $\RR_t = (z\id - X(t))^{-1}$. Thus
\begin{align*}
	\|Z_i^{(3)} \RR_{tis}w\| &\le
	\|Z_i^{(3)} \RR_tw\| +
	\frac{R(X)}{\mathrm{Im}\,z}
	\|Z_i  \RR_tw\|, \\
	\|Z_i^{(1)} \RR_{tis}^*v\| &\le
	\|Z_i^{(1)} \RR_t^*v\| +
	\frac{R(X)}{\mathrm{Im}\,z}
	\|Z_i  \RR_t^*v\|
\end{align*}
for $s,t\in[0,1]$. 
Using $(a+b)(c+d) \le a^2+b^2+c^2+d^2$ yields
\begin{multline*}
	\EE[\|Z_i^{(1)} \RR_{tis}^*v\|
        \|Z_i^{(3)} \RR_{tis}w\|]
	\le
	\EE[\langle \RR_tw,(Z_i^{(3)})^2\RR_tw\rangle] +
	\EE[\langle \RR_t^*v,(Z_i^{(1)})^2\RR_t^*v\rangle] +
	\\
	\frac{R(X)^2}{(\mathrm{Im}\,z)^2}
	\big(
	\EE[\langle \RR_tw,Z_i^2\RR_tw\rangle] +
	\EE[\langle \RR_t^*v,Z_i^2\RR_t^*v\rangle] 
	\big).	
\end{multline*}
We can therefore estimate
$$
	F(Z^{(1)},Z^{(2)},Z^{(3)}) \lesssim
	\bigg(
	\frac{R(X)}{(\mathrm{Im}\,z)^4}	
	+
	\frac{R(X)^3}{(\mathrm{Im}\,z)^6}\bigg)
	\EE\Bigg\|\sum_{i=1}^n Z_i^2\Bigg\|
$$
whenever $Z^{(k)}$ have the same distribution as $Z$, concluding the 
proof.
\end{proof}

We also need the following simple lemma.

\begin{lem}
\label{lem:posnck}
We have
$$
	\EE\Bigg\|\sum_{i=1}^n Z_i^2\Bigg\|
	\lesssim \sigma(X)^2 + R(X)^2\log d.
$$
\end{lem}

\begin{proof}
We can estimate
\begin{align*}
	\EE\Bigg\|\sum_{i=1}^n Z_i^2\Bigg\| & \le
	\sigma(X)^2 +
	\EE\Bigg\|\sum_{i=1}^n (Z_i^2-\EE Z_i^2)\Bigg\|
	\\
	&\lesssim
	\sigma(X)^2 +
	\Bigg\|\sum_{i=1}^n
	\EE[(Z_i^2-\EE Z_i^2)^2]
	\Bigg\|^{\frac{1}{2}}\sqrt{\log d} +
	R(X)^2\log d
	\\
	&\lesssim
	\sigma(X)^2 +
	R(X)\sigma(X)\sqrt{\log d} + R(X)^2\log d,
\end{align*}
where the second line follows from the matrix Bernstein inequality
\eqref{eq:bernstein}. The conclusion follows as
$R(X)\sigma(X)\sqrt{\log d} \le \sigma(X)^2+R(X)^2\log d$.
\end{proof}

We can now complete the proof of Theorem \ref{thm:smuniv}.

\begin{proof}[Proof of Theorem \ref{thm:smuniv}]
Lemmas \ref{lem:smpf} and \ref{lem:posnck} yield
$$
	\big\|\EE[(z\id-X)^{-1}] - \EE[(z\id-G)^{-1}]\big\|
	\lesssim
	\frac{\sigma(X)^2+R(X)^2\log d}{(\mathrm{Im}\,z)^3}
	\bigg(\frac{R(X)}{\mathrm{Im}\,z}	
	+
	\frac{R(X)^3}{(\mathrm{Im}\,z)^3}\bigg).
$$
This yields the first inequality of Theorem \ref{thm:smuniv} when 
$\mathrm{Im}\,z\ge R(X)$. On the other hand, when $\mathrm{Im}\,z<R(X)$,
we can crudely estimate
\begin{align*}
	\big\|\EE[(z\id-X)^{-1}] - \EE[(z\id-G)^{-1}]\big\|
	&\le
	\big\|\EE[(z\id-X)^{-1}]\big\| + \big\| \EE[(z\id-G)^{-1}]\big\|
	\\ &\le
	\frac{2}{\mathrm{Im}\,z}
	\le
	\frac{2R(X)^3}{(\mathrm{Im}\,z)^4},
\end{align*}
concluding the proof of the first inequality.

To prove the second inequality, it suffices to consider real-valued
$\varphi:\mathbb{R}\to\mathbb{R}$ (as otherwise we may apply the 
real-valued inequality separately to the real and imaginary parts of 
$\varphi$). Then $\varphi(X)$ and $\varphi(G)$ are self-adjoint, so 
we may express
\begin{align*}
	\big\|\EE[\varphi(X)] - \EE[\varphi(G)] \big\| &=
	\sup_{\|v\|=1}
	\big|\langle v,\EE[\varphi(X)]v\rangle -
	\langle v,\EE[\varphi(G)]v\rangle\big|
	\\
	&= \sup_{\|v\|=1}
	\bigg| \int \varphi \,d\mu_v - \int \varphi\,d\nu_v\bigg|
\end{align*}
where we defined the probability measures $\mu_v,\nu_v$ for $\|v\|=1$ so 
that $\int \varphi \,d\mu_v=\langle v,\EE[\varphi(X)]v\rangle$ and
$\int \varphi\,d\nu_v =\langle v,\EE[\varphi(G)]v\rangle$ for all
bounded continuous $\varphi$. As the first inequality of
Theorem \ref{thm:smuniv} implies that
$$
	\sup_{\|v\|=1}
	\bigg| \int \frac{1}{z-x} \,\mu_v(dx) - \int 
	\frac{1}{z-x}\,\nu_v(dx)\bigg| \lesssim
	\frac{R(X)\sigma(X)^2+R(X)^3\log d}{(\mathrm{Im}\,z)^4},
$$
the second inequality of Theorem \ref{thm:smuniv} follows from \cite[Lemma 
5.11]{BBV21}.
\end{proof}

\begin{rem}[An improved inequality for Stieltjes transforms]
\label{rem:stieltjes}
The main complication in the proof of Theorem \ref{thm:smuniv} arises from 
the fact that we aim to achieve a norm estimate. If we are only interested 
in establishing universality of the trace of the resolvent (that is, of 
the Stieltjes transform of the empirical spectral distribution), the 
proof simplifies greatly and yields a better bound. Indeed, the first 
equation display of the proof of Lemma \ref{lem:smpf} and H\"older's 
inequality readily yield
$$
	|\EE[
	\partial_{Y_1}\partial_{Y_2}\partial_{Y_3}
        \ntr\mbox{} (z\id - X(t,i,s))^{-1}]|
	\le
	\frac{6\|Y_1\|_3\|Y_2\|_3\|Y_3\|_3}{(\mathrm{Im}\,z)^4}
$$
for any random matrices $Y_1,Y_2,Y_3$. Combining this inequality with
the $p=3$ case of Corollary \ref{cor:mtxcumsm} immediately yields the 
estimate
$$
        |\EE[\ntr (z\id-X)^{-1}] - \EE[\ntr (z\id-G)^{-1}]|
        \lesssim
        \frac{1}{(\mathrm{Im}\,z)^4} 
        \sum_{i=1}^n \EE[\ntr |Z_i|^3].
$$
In particular, in this case the logarithmic dimension dependence is 
eliminated.
\end{rem}

\section{Universality of the spectrum}
\label{sec:univspec}

The aim of this section is to prove Theorem \ref{thm:specuniv} and 
Corollary \ref{cor:normuniv}. The main idea behind the proof is that 
universality of the spectrum can be deduced from the bound on the moments 
of the resolvent in Theorem \ref{thm:resuniv} using a technique that was 
developed for Gaussian random matrices in \cite[\S 6.2]{BBV21}. The 
difficulty in the present setting is that the resolvent of the 
non-Gaussian random matrix $X$ exhibits more complicated concentration 
properties than in the Gaussian case.

We first introduce some basic estimates in section 
\ref{sec:uspprelim}. In sections \ref{sec:uspupper} and 
\ref{sec:usplower}, we bound the probability that 
$\spc(X)\subseteq\spc(G)+[-\varepsilon,\varepsilon]$ and 
$\spc(G)\subseteq\spc(X)+[-\varepsilon,\varepsilon]$, respectively. 
Combining these bounds yields the Hausdorff distance bound of Theorem 
\ref{thm:specuniv}. Finally, Corollary \ref{cor:normuniv} will be proved 
in section \ref{sec:uspnorm}.

\subsection{Preliminaries}
\label{sec:uspprelim}

The basic principle behind the proof is the following deterministic lemma, 
which is a trivial modification of \cite[Lemma 6.4]{BBV21}.

\begin{lem}
\label{lem:rescomparison}
Let $C,K_1,K_2,K_3\ge 0$, and let $A,B\in\M_d(\mathbb{C})_{\rm sa}$ 
satisfy
$$
	\|(z\id - A)^{-1}\| \le C\|(z\id - B)^{-1}\| 
	+\frac{K_1}{(\mathrm{Im}\,z)^2}
	+\frac{K_2}{(\mathrm{Im}\,z)^3}
	+\frac{K_3}{(\mathrm{Im}\,z)^4}
$$
for all $z=\lambda+i\varepsilon$ with $\lambda\in\spc(A)$ and
$\varepsilon=6K_1\vee (6K_2)^{\frac{1}{2}}\vee (6K_3)^{\frac{1}{3}}$.
Then
$$
	\spc(A)\subseteq\spc(B)+2C\varepsilon [-1,1].
$$
\end{lem}

\begin{proof}
Fix $\lambda\in\spc(A)$ and $z=\lambda+i\varepsilon$, where $\varepsilon$ 
is as defined in the statement.
As $\|(z\id-A)^{-1}\| = (\mathrm{dist}(z,\spc(A)))^{-1}$, the assumption 
implies that
$$
	\frac{1}{\varepsilon} \le
	\frac{C}{\sqrt{\varepsilon^2+\mathrm{dist}(\lambda,\spc(B))^2}}
	+ \frac{K_1}{\varepsilon^2}
	+ \frac{K_2}{\varepsilon^3}
	+ \frac{K_3}{\varepsilon^4}.
$$
If $\mathrm{dist}(\lambda,\spc(B))>2C\varepsilon$, we would have
$\frac{1}{2} < \frac{K_1}{\varepsilon}
+ \frac{K_2}{\varepsilon^2}
+ \frac{K_3}{\varepsilon^3}
\le \frac{1}{2}$ by the definition of $\varepsilon$, which is impossible.
Thus $\mathrm{dist}(\lambda,\spc(B))\le
2C\varepsilon$ for all $\lambda\in\spc(A)$.
\end{proof}

The main idea behind the proof of Theorem \ref{thm:specuniv} is that we 
will engineer the assumption of Lemma \ref{lem:rescomparison} by using 
Theorem \ref{thm:resuniv} and concentration of measure. Before we turn to 
the details of the argument, let us prove a crude \emph{a priori} 
bound on the spectrum that will be needed below.

\begin{lem}
\label{lem:net}
We have
$$
	\mathbf{P}\big[\spc(X)\subseteq\spc(\EE X) +
	C\{\sigma_*(X)\sqrt{d+t}+
	R(X)(d+t)\}[-1,1]\big] \ge 1-e^{-t}
$$
and
$$
	\mathbf{P}\big[\spc(G)\subseteq\spc(\EE G) +
	C\sigma_*(X)\sqrt{d+t}\,[-1,1]\big] \ge 1-e^{-t}
$$
for all $t\ge 0$, where $C$ is a universal constant.
\end{lem}

\begin{proof}
Let $N\subset\mathbb{S}^{d-1}$ be a $\frac{1}{4}$-net of the 
unit sphere $\mathbb{S}^{d-1}:=\{x\in\mathbb{C}^d:\|x\|=1\}$,
that is, $\mathrm{dist}(x,N)\le\frac{1}{4}$ for all 
$x\in\mathbb{S}^{d-1}$. A routine estimate \cite[p.\ 110]{Tao12}
yields
\begin{align*}
	\mathbf{P}[\|X-\EE X\| \ge x] &\le
	\mathbf{P}\bigg[\max_{v,w\in N}|\langle v,(X-\EE X)w\rangle| \ge
	\frac{x}{4}\bigg] \\ &\le
	|N|^2 \sup_{v,w\in\mathbb{S}^{d-1}}
	\mathbf{P}\bigg[|\langle v,(X-\EE X)w\rangle| \ge
	\frac{x}{4}\bigg].
\end{align*}
By viewing $\mathbb{C}^d$ as a $2d$-dimensional real vector space, we may 
use a standard volume argument \cite[Lemma 2.3.4]{Tao12} to choose
the net $N$ so that $|N|\le C^d$ for a universal constant 
$C$. On the other hand, by Bernstein's inequality \cite[Theorem 
2.10]{BLM13}
\begin{align*}
	&\mathbf{P}\big[|\langle v,(X-\EE X)w\rangle| \ge
	2\sigma_*(X)\sqrt{x} + \sqrt{2}R(X)x
	\big]
	\\ &\qquad\le
	\mathbf{P}\Bigg[
	\Bigg|\sum_{i=1}^n\mathrm{Re}\,\langle v,Z_iw\rangle\Bigg|
	\ge 
	\sigma_*(X)\sqrt{2x} + R(X)x
	\Bigg]
	\\ &\qquad\qquad +
	\mathbf{P}\Bigg[
	\Bigg|\sum_{i=1}^n\mathrm{Im}\,\langle v,Z_iw\rangle\Bigg|
	\ge
	\sigma_*(X)\sqrt{2x} + R(X)x
	\Bigg]
	\le 4e^{-x}
\end{align*}
for all $x\ge 0$ and $v,w\in\mathbb{S}^{d-1}$.
Combining the above estimates yields
$$
	\mathbf{P}[\|X-\EE X\| \ge
	8\sigma_*(X)\sqrt{cd+t} + 4\sqrt{2}R(X)(cd+t)]
	\le
	C^{2d} e^{-cd-t} \le e^{-t}
$$
for all $t\ge 0$, provided the universal constant $c$ is chosen 
sufficiently large. The first inequality in the statement now follows
by noting that
$$
	\spc(X)\subseteq \spc(\EE X)+\|X-\EE X\|[-1,1]
$$
by Weyl's inequality $\max_i |\lambda_i(A)-\lambda_i(B)|\le\|A-B\|$ for 
self-adjoint matrices $A,B$ \cite[Corollary III.2.6]{Bha97} (here 
$\lambda_i(A)$ is the $i$th largest eigenvalue of $A$).

The inequality for the Gaussian matrix $G$ follows in the identical 
fashion, except that we replace Bernstein's inequality 
by the Gaussian bound
$$
	\mathbf{P}\big[|\langle v,(G-\EE G)w\rangle| \ge
        2\sigma_*(X)\sqrt{x} 
        \big] \le 4e^{-x}
$$
(this follows from the Gaussian tail bound \cite[p.\ 22]{BLM13} as
the real and imaginary parts of
$\langle v,(G-\EE G)w\rangle$ are Gaussian variables
with variance bounded by $\sigma_*(X)^2$).
\end{proof}

\subsection{Proof of Theorem \ref{thm:specuniv}: upper bound}
\label{sec:uspupper}

The aim of the present section is to prove that 
$\spc(X)\subseteq\spc(G)+[-\varepsilon,\varepsilon]$ with high probability 
for a suitable choice of $\varepsilon$. This will be accomplished by 
showing that the corresponding resolvent norm inequality of Lemma 
\ref{lem:rescomparison} holds with with high probability. 

We begin by showing that this is the case for a single choice of $z$. Note 
that the joint distribution of $X$ and $G$ is irrelevant to the following 
proofs.

\begin{lem}
\label{lem:uppersinglez}
Let $z\in\mathbb{C}$ with $\mathrm{Im}\,z>0$.  Then
\begin{align*}
	\mathbf{P}\bigg[
	\|(z\id -X)^{-1}\| &\ge
	C\bigg\{\|(z\id -G)^{-1}\|+
	\frac{\sigma_*(X)}{(\mathrm{Im}\,z)^2}\sqrt{x}
\\ &\qquad\qquad\mbox{}
	+ 
	\frac{R(X)\sigma(X)^2x^2+R(X)^3x^3}{(\mathrm{Im}\,z)^4}
	\bigg\}\bigg] \le 3e^{-x}
\end{align*}
for all $x\ge\log d$, where $C$ is a universal constant.
\end{lem}

\begin{proof}
We begin by noting that Markov's inequality implies
$$
	\mathbf{P}\Big[
	\|(z\id -X)^{-1}\| \ge
	e\, \EE[\|(z\id -X)^{-1}\|^{2p}]^{\frac{1}{2p}}
	\Big] \le e^{-2p}.
$$
By Theorem \ref{thm:resuniv}, the expectation inside the probability
satisfies
$$
	\EE[\|(z\id - X)^{-1}\|^{2p}]^{\frac{1}{2p}} \le
	d^{\frac{1}{2p}}
	\EE[\|(z\id - G)^{-1}\|^{2p}]^{\frac{1}{2p}}
	+ Cd^{\frac{1}{2p}}
	\frac{R(X)\sigma(X)^2p^2+R(X)^3p^3}{(\mathrm{Im}\,z)^4}
$$
for $p\in\mathbb{N}$, where $C$ is a universal constant. Here we used that
$\frac{1}{d}\|A\| \le \ntr A \le \|A\|$ for any positive 
semidefinite matrix
$A\in\M_d(\mathbb{C})_{\rm sa}$.

To proceed, note first that Lemma \ref{lem:gconc} implies
$$
	\EE[\|(z\id - G)^{-1}\|^{2p}]^{\frac{1}{2p}}
	\le
	\EE\|(z\id - G)^{-1}\| +
	C\sqrt{p}\frac{\sigma_*(X)}{(\mathrm{Im}\,z)^2}
$$
for $p\in\mathbb{N}$, where $C$ is a universal constant (this follows
as the $L^p$-norm of a $\sigma^2$-subgaussian random variable is
at most of order $\sigma\sqrt{p}$, cf.\ \cite[Theorem 2.1]{BLM13}).
Another application of Lemma \ref{lem:gconc} therefore yields
$$
        \mathbf{P}\bigg[
	\EE[\|(z\id - G)^{-1}\|^{2p}]^{\frac{1}{2p}}
	\ge
        \|(z\id -G)^{-1}\|+
	C_1\sqrt{p}\frac{\sigma_*(X)}{(\mathrm{Im}\,z)^2}
	\bigg]
        \le 2e^{-C_2p}
$$
for $p\in\mathbb{N}$ and universal constants $C_1,C_2$.

Combining the above bounds yields
\begin{align*}
	\mathbf{P}\bigg[
	\|(z\id -X)^{-1}\| &\ge
	e d^{\frac{1}{2p}}
        \|(z\id -G)^{-1}\|+
	C_1
	e d^{\frac{1}{2p}}
	\sqrt{p}\frac{\sigma_*(X)}{(\mathrm{Im}\,z)^2}
\\ &\qquad\mbox{}
	+ Ce d^{\frac{1}{2p}}
	\frac{R(X)\sigma(X)^2p^2+R(X)^3p^3}{(\mathrm{Im}\,z)^4}
	\bigg] \le e^{-2p}+2e^{-C_2p}
\end{align*}
for $p\in\mathbb{N}$. The conclusion follows readily
using $d^{\frac{1}{2p}}\le e^{\frac{1}{2}}$ for $p\ge\log d$.
\end{proof}

We are now ready to prove one direction of Theorem \ref{thm:specuniv}.

\begin{prop}
\label{prop:specunivupper}
For any $t\ge 0$, we have
$$
	\mathbf{P}\big[\spc(X)\subseteq\spc(G) + C\varepsilon(t)[-1,1]\big]
	\ge 1-de^{-t},
$$
where $C$ is a universal constant and $\varepsilon(t)$ is as defined
in Theorem \ref{thm:specuniv}.
\end{prop}

\begin{proof}
Define the set
$$
	\Omega_x := \spc(\EE X) + 
	C'\{\sigma_*(X)\sqrt{d+x}+R(X)(d+x)\}[-1,1],
$$
where $C'$ is the universal constant of Lemma \ref{lem:net}.
Then $\Omega_x$ is a union of $d$ intervals of length
$2C'\{\sigma_*(X)\sqrt{d+x}+R(X)(d+x)\}$. We can therefore find
$\mathcal{N}_x\subset\Omega_x$ with $|\mathcal{N}_x|\le 
\frac{4C'd(d+x)}{x}$
such that each $\lambda\in\Omega_x$ satisfies
$\mathrm{dist}(\lambda,\mathcal{N}_x)\le \sigma_*\sqrt{x}+R(X)x$.
In particular, for every $\lambda\in\Omega_x$, there exists
$\lambda'\in\mathcal{N}_x$ so that
$$
	\big|\|((\lambda+i\varepsilon)\id-X)^{-1}\|-
	\|((\lambda'+i\varepsilon)\id-X)^{-1}\|\big|
	\le
	\frac{\sigma_*\sqrt{x}+R(X)x}{\varepsilon^2}
$$
as well as the analogous bound where $X$ is replaced by $G$
(here we used the identity $A^{-1}-B^{-1}=A^{-1}(B-A)B^{-1}$).
We can therefore estimate
\begingroup
\allowdisplaybreaks
\begin{align*}
	&
	\mathbf{P}\bigg[
	\|(z\id -X)^{-1}\| \ge
	C\bigg\{\|(z\id -G)^{-1}\|+
	\frac{3\sigma_*(X)\sqrt{x} + 2R(X)x}{\varepsilon^2}
	\\ &\qquad\qquad\mbox{}
	+ 
	\frac{R(X)\sigma(X)^2x^2+R(X)^3x^3}{\varepsilon^4}
	\bigg\}\mbox{ for some }z\in\spc(X)+i\varepsilon\bigg]
\\
	&\le
	\mathbf{P}\bigg[
	\|(z\id -X)^{-1}\| \ge
	C\bigg\{\|(z\id -G)^{-1}\|+
	\frac{3\sigma_*(X)\sqrt{x} + 2R(X)x}{\varepsilon^2}
	\\ &\qquad\qquad\mbox{}
	+ 
	\frac{R(X)\sigma(X)^2x^2+R(X)^3x^3}{\varepsilon^4}
	\bigg\}\mbox{ for some }z\in\Omega_x+i\varepsilon\bigg]
	+ e^{-x}
\\
	&\le
	\mathbf{P}\bigg[
	\|(z\id -X)^{-1}\| \ge
	C\bigg\{\|(z\id -G)^{-1}\|+
	\frac{\sigma_*(X)}{\varepsilon^2}\sqrt{x}
	\\ &\qquad\qquad\mbox{}
	+ 
	\frac{R(X)\sigma(X)^2x^2+R(X)^3x^3}{\varepsilon^4}
	\bigg\}\mbox{ for some }z\in\mathcal{N}_x+i\varepsilon\bigg]
	+ e^{-x}
\\
	&\le (3|\mathcal{N}_x|+1)e^{-x} 
	\le \bigg(1+\frac{12C'd(d+x)}{x}\bigg) e^{-x}
\end{align*}
\endgroup
for $x\ge\log d$, where we used Lemma \ref{lem:net} in the first 
inequality and a union bound and Lemma \ref{lem:uppersinglez} in the 
third inequality (here $C>1$ is the constant of
Lemma \ref{lem:uppersinglez}).

Now let $x=Lt$ for a universal constant $L$. Recalling the standing 
assumption $d\ge 2$, it is readily seen that we may 
choose $L>1$ sufficiently large so that
$$
	\bigg(1+\frac{12C'd(d+x)}{x}\bigg) e^{-x}
	\le de^{-t}
$$
for all $t\ge\log d$. Then we have shown that
\begin{align*}
	&\mathbf{P}\bigg[
	\|(z\id -X)^{-1}\| \le
	3L^3C\bigg\{\|(z\id -G)^{-1}\|+
	\frac{\sigma_*(X)\sqrt{t} + R(X)t}{\varepsilon^2}
	\\ &\qquad\qquad\mbox{}
	+ 
	\frac{R(X)\sigma(X)^2t^2+R(X)^3t^3}{\varepsilon^4}
	\bigg\}\mbox{ for all }z\in\spc(X)+i\varepsilon\bigg]
	\ge 1-de^{-t}
\end{align*}
for all $t\ge\log d$. On the other hand, the same bound holds trivially 
for $t<\log d$ as then $1-de^{-t}<0$. The proof is concluded by applying 
Lemma \ref{lem:rescomparison}.
\end{proof}

\subsection{Proof of Theorem \ref{thm:specuniv}: lower bound}
\label{sec:usplower}

We now turn to the complementary inequality 
$\spc(G)\subseteq\spc(X)+[-\varepsilon,\varepsilon]$ with high 
probability. The proof is similar in spirit to that of the upper bound,
but we must now work with the more complicated concentration inequality
of Proposition \ref{prop:ngconc}. As before, we begin by establishing a 
resolvent norm inequality for a single choice of $z$.

\begin{lem}
\label{lem:lowersinglez}
Let $z\in\mathbb{C}$ with $\mathrm{Im}\,z>0$.  Then
\begin{multline*}
	\mathbf{P}\bigg[
	\|(z\id -G)^{-1}\| \ge
        C\bigg\{
	\|(z\id -X)^{-1}\| +
        \frac{R(X)\sigma(X)x+R(X)^2x^{\frac{3}{2}}}{(\mathrm{Im}\,z)^3}
	+
\mbox{}
\\
	\frac{R(X)\sigma(X)^2x^2+R(X)^3x^3}{(\mathrm{Im}\,z)^4}
	+
        \frac{
	\sigma_*(X)x^{\frac{1}{2}}+
	R(X)^{\frac{1}{2}}\sigma(X)^{\frac{1}{2}} x^{\frac{3}{4}} +
	R(X)x}{(\mathrm{Im}\,z)^2}
        \bigg\} 
	\bigg]
	\le 3e^{-x}
\end{multline*}
for all $x\ge\log d$, where $C$ is a universal constant.
\end{lem}

\begin{proof}
As in the proof of Lemma \ref{lem:uppersinglez}, we have
$$
	\mathbf{P}\Big[
	\|(z\id -G)^{-1}\| \ge
	e\, \EE[\|(z\id -G)^{-1}\|^{2p}]^{\frac{1}{2p}}
	\Big] \le e^{-2p}
$$
and
$$
	\EE[\|(z\id - G)^{-1}\|^{2p}]^{\frac{1}{2p}} \le
	d^{\frac{1}{2p}}
	\EE[\|(z\id - X)^{-1}\|^{2p}]^{\frac{1}{2p}}
	+ Cd^{\frac{1}{2p}}
	\frac{R(X)\sigma(X)^2p^2+R(X)^3p^3}{(\mathrm{Im}\,z)^4}
$$
for $p\in\mathbb{N}$ by Markov's inequality and Theorem \ref{thm:resuniv}.

To proceed, we use that Proposition \ref{prop:ngconc} implies
\begin{multline*}
	\EE[\|(z\id - X)^{-1}\|^{2p}]^{\frac{1}{2p}}
	\le
	\EE\|(z\id - X)^{-1}\| +
	C\bigg\{
	\frac{R(X)}{(\mathrm{Im}\,z)^2}+
	\frac{R(X)^2}{(\mathrm{Im}\,z)^3}
	\bigg\}\,p 
	\\
	+C\bigg\{
	\frac{\sigma_*(X)+R(X)^{\frac{1}{2}}(\EE\|X-\EE X\|)^{\frac{1}{2}}}
	{(\mathrm{Im}\,z)^2} +
	\frac{R(X)(\EE\|X-\EE X\|^2)^{\frac{1}{2}}}{(\mathrm{Im}\,z)^3}
	\bigg\}
	\sqrt{p}
\end{multline*}
for $p\in\mathbb{N}$ by \cite[Theorem 2.3]{BLM13}.
Another application of Proposition \ref{prop:ngconc} yields
\begin{multline*}
	\mathbf{P}\bigg[
	\EE[\|(z\id - X)^{-1}\|^{2p}]^{\frac{1}{2p}}
	\ge
	\|(z\id -X)^{-1}\| +
        C\bigg\{
        \frac{R(X)}{(\mathrm{Im}\,z)^2}+
        \frac{R(X)^2}{(\mathrm{Im}\,z)^3}
        \bigg\}\,p + \mbox{}
\\
\qquad
        C\bigg\{
        \frac{\sigma_*(X)+R(X)^{\frac{1}{2}}(\EE\|X-\EE X\|)^{\frac{1}{2}}}
        {(\mathrm{Im}\,z)^2} +
        \frac{R(X)(\EE\|X-\EE X\|^2)^{\frac{1}{2}}}{(\mathrm{Im}\,z)^3}
        \bigg\}
        \sqrt{p}
	\bigg]
	\le 2e^{-p}
\end{multline*}
for $p\in\mathbb{N}$, provided the universal constant $C$ is chosen 
sufficiently large. Now recall that the matrix Bernstein inequality 
\cite[eq.\ (6.1.4)]{Tro15} implies
$$
	(\EE\|X-\EE X\|^2)^{\frac{1}{2}} \lesssim
	\sigma(X)\sqrt{\log d} + R(X)\log d
	\le
	\sigma(X)\sqrt{p}+R(X)p
$$
for $p\ge\log d$. We can therefore further estimate
\begin{multline*}
	\mathbf{P}\bigg[
	\EE[\|(z\id - X)^{-1}\|^{2p}]^{\frac{1}{2p}}
	\ge
	\|(z\id -X)^{-1}\| +
        C\bigg\{
        \frac{R(X)\sigma(X)p+R(X)^2p^{\frac{3}{2}}}{(\mathrm{Im}\,z)^3}
	+
\mbox{}
\\
        \frac{
	\sigma_*(X)\sqrt{p}+
	R(X)^{\frac{1}{2}}\sigma(X)^{\frac{1}{2}} p^{\frac{3}{4}} +
	R(X)p}{(\mathrm{Im}\,z)^2}
        \bigg\} 
	\bigg]
	\le 2e^{-p}
\end{multline*}
for $p\ge \log d$, provided $C$ is chosen sufficiently large.
The proof is now readily concluded by
combining the above bounds
and using $d^{\frac{1}{2p}}\le e^{\frac{1}{2}}$ for $p\ge\log d$.
\end{proof}

\begin{rem}
We have emphasized in the introduction that the matrix Bernstein 
inequality may be viewed as a consequence of the universality principles 
of this paper. On the other hand, we have used the matrix Bernstein 
inequality in the proof of Lemma \ref{lem:lowersinglez} to estimate the
matrix norms that appear in Proposition \ref{prop:ngconc}. There is no 
circular reasoning here: the present section is only concerned 
with lower bounds on the spectrum of $X$, while the matrix Bernstein 
inequality already 
follows from the upper bound of Proposition \ref{prop:specunivupper} (or 
from Theorem \ref{thm:momentuniv} by choosing $p\asymp\log d$ and
$q=\infty$) and 
the noncommutative Khintchine inequality.

The same remark applies to the application of the matrix Bernstein 
inequality in the proof of Theorem \ref{thm:smuniv} (cf.\ Lemma 
\ref{lem:posnck} in section \ref{sec:smpf}).
\end{rem}

We are now ready to prove the converse direction of Theorem 
\ref{thm:specuniv}.

\begin{prop}
\label{prop:specunivlower}
For any $t\ge 0$, we have
$$
	\mathbf{P}\big[\spc(G)\subseteq\spc(X) + 
	C\varepsilon(t)[-1,1]\big]
	\ge 1-de^{-t},
$$
where $C$ is a universal constant and $\varepsilon(t)$ is as defined
in Theorem \ref{thm:specuniv}.
\end{prop}

\begin{proof}
By following exactly the same steps as in the proof of Proposition 
\ref{prop:specunivupper}, we can deduce using Lemmas \ref{lem:net} and 
Lemma \ref{lem:lowersinglez} the inequality
\begin{multline*}
	\mathbf{P}\bigg[
	\|(z\id -G)^{-1}\| \le
        C\bigg\{
	\|(z\id -X)^{-1}\| +
        \frac{R(X)\sigma(X)t+R(X)^2t^{\frac{3}{2}}}{\varepsilon^3}
	+
\mbox{}
\\
	\frac{R(X)\sigma(X)^2t^2+R(X)^3t^3}{\varepsilon^4}
	+
        \frac{
	\sigma_*(X)t^{\frac{1}{2}}+
	R(X)^{\frac{1}{2}}\sigma(X)^{\frac{1}{2}} t^{\frac{3}{4}} +
	R(X)t}{\varepsilon^2}
        \bigg\} 
\\
	\mbox{for all }z\in\spc(G)+i\varepsilon
	\bigg]
	\ge 1-de^{-t}
\end{multline*}
for all $t,\varepsilon\ge 0$, where $C$ is a universal constant.
Thus Lemma \ref{lem:rescomparison} implies
$$
	\mathbf{P}\big[\spc(G)\subseteq\spc(X) + 
	C\varepsilon'(t)[-1,1]\big]
	\ge 1-de^{-t}
$$
for all $t\ge 0$ and a universal constant $C$, where
$$
	\varepsilon'(t) =
	\sigma_*(X)t^{\frac{1}{2}}
	+R(X)^{\frac{1}{3}}\sigma(X)^{\frac{2}{3}}t^{\frac{2}{3}}
        +R(X)^{\frac{1}{2}}\sigma(X)^{\frac{1}{2}}t^{\frac{3}{4}} 
	+R(X)t.
$$
It remains to note that
$$
	R(X)^{\frac{1}{2}}\sigma(X)^{\frac{1}{2}}t^{\frac{3}{4}} \le
	\frac{3}{4}
	R(X)^{\frac{1}{3}}\sigma(X)^{\frac{2}{3}}t^{\frac{2}{3}} +
	\frac{1}{4}
	R(X)t
$$ by Young's inequality, concluding the proof.
\end{proof}

We now conclude the proof of Theorem \ref{thm:specuniv}.

\begin{proof}[Proof of Theorem \ref{thm:specuniv}]
Combining Propositions \ref{prop:specunivupper} and
\ref{prop:specunivlower} yields
$$
	\mathbf{P}[\dH(\spc(X),\spc(G)) >
	C\varepsilon(s)] \le 2de^{-s}
$$
for all $s\ge 0$ by the union bound. Choosing $s=2t$, we obtain
$$
	\mathbf{P}[\dH(\spc(X),\spc(G)) >
	2C\varepsilon(t)] \le 2de^{-2t} \le de^{-t}
$$
for $t\ge\log d$, as the latter implies $2e^{-t} \le \frac{2}{d}\le 1$
by the standing assumption $d\ge 2$. But for $t<\log d$ the
inequality is trivial as then $de^{-t}>1$. The tail bound follows.

To deduce the expectation bound, we note that
\begin{align*}
	\EE[\dH(\spc(X),\spc(G))] 
	&\le
	C\varepsilon(\log d)+\int_{C\varepsilon(\log d)}^\infty
	\mathbf{P}[\dH(\spc(X),\spc(G)) >
        x]\,dx
\\
	&=
	C\varepsilon(\log d)+C\int_{\log d}^\infty
	\mathbf{P}[\dH(\spc(X),\spc(G)) >
        C\varepsilon(t)]\,\frac{d\varepsilon(t)}{dt}\,dt
\\
	&\le
	C\varepsilon(\log d)+2dC\int_{\log d}^\infty
	e^{-t}\,\frac{d\varepsilon(t)}{dt}\,dt
	\lesssim \varepsilon(\log d)
\end{align*}
using $\int_a^{\infty} e^{-t} t^{\beta}\,dt
\le C_\beta e^{-a} a^{\beta}$ for $a>\frac{1}{4}$, $\beta\in\mathbb{R}$,
where $C_\beta$ depends only on $\beta$.
\end{proof}

\subsection{Proof of Corollary \ref{cor:normuniv}}
\label{sec:uspnorm}

Now that Theorem \ref{thm:specuniv} has been established, the proof of
Corollary \ref{cor:normuniv} follows by routine manipulations.	

\begin{proof}[Proof of Corollary \ref{cor:normuniv}]
We first note that
$$
	\spc(A)\subseteq \spc(B)+[-\varepsilon,\varepsilon]
$$
certainly implies
$$
	\lambda_{\rm max}(A)\le\lambda_{\rm max}(B)+\varepsilon
$$
for any $A,B\in\M_d(\mathbb{C})_{\rm sa}$ and $\varepsilon>0$. Thus
$$	|\lambda_{\rm max}(A)-\lambda_{\rm max}(B)|\le 
	\dH(\spc(A),\spc(B)),
$$
and the first and last bound of Corollary \ref{cor:normuniv} follow
immediately from Theorem \ref{thm:specuniv}.

To prove the middle bound, we note that a routine application of 
Gaussian concentration (see, e.g., \cite[Corollary 4.14]{BBV21}) yields
$$
	\mathbf{P}\big[ |\lambda_{\rm max}(G)-
	\EE\lambda_{\rm max}(G)| \ge \sigma_*(X)\sqrt{2t}\big]
	\le 2e^{-t}
$$
for all $t\ge 0$. Combined with the first bound of Corollary 
\ref{cor:normuniv}, we obtain
\begin{multline*}
	\mathbf{P}\big[
	|\lambda_{\rm max}(X)-\EE\lambda_{\rm max}(G)| >
	\sigma_*(X)\sqrt{2t}+
	C\varepsilon(t)\big]
\\	\le
	2e^{-t} +
	\mathbf{P}\big[
	|\lambda_{\rm max}(X)-\lambda_{\rm max}(G)| >
	C\varepsilon(t)\big]
	\le (d+2)e^{-t}
\end{multline*}
for all $t\ge 0$. The second inequality of Corollary \ref{cor:normuniv} 
follows for a suitable choice of the universal constant (as in the last 
step of the proof of Theorem \ref{thm:specuniv}).

The analogous bounds for $\|X\|,\|G\|$ are proved in an identical manner.
\end{proof}

\section{Truncation}
\label{sec:trunc}

The aim of this section is to prove Theorems \ref{thm:specheavy} and 
\ref{thm:smunivheavy}. The basic idea behind these results is the 
following truncation argument. Let $X$ be as in \eqref{eq:model}, and let 
$G$ be the associated Gaussian model. Define the truncated model 
$$
	\tilde X := Z_0 + \sum_{i=1}^n 1_{\|Z_i\|\le R} Z_i.
$$
Then $X=\tilde X$ on the event $\{\max_i\|Z_i\|\le R\}$, while
$R(\tilde X)\le R$. We can therefore obtain universality 
principles for unbounded $X$ by conditioning on the above event, and 
applying the results of the previous sections to $\tilde X$.

The problem with this approach is that it does not yield a comparison 
between the spectra of $X$ and $G$, but rather between the spectra of $X$ 
and $\tilde G$, where $\tilde G$ is the Gaussian model associated to 
$\tilde X$. The main difficulty in the implementation of the truncation 
argument is therefore to compare the spectra of the Gaussian models $G$ 
and $\tilde G$. To this end, we will first prove general comparison 
principles for the spectra of Gaussian random matrices in section 
\ref{sec:gaussc}. In section \ref{sec:truncerr}, we will upper bound the 
relevant parameters in the specific case of $G$ and $\tilde G$. Finally, 
we combine these estimates in section \ref{sec:truncpf} to complete the 
proof of Theorems \ref{thm:specheavy} and \ref{thm:smunivheavy}.

\subsection{Gaussian comparison principles}
\label{sec:gaussc}

The aim of this section is to prove general comparison principles for the 
spectra of Gaussian random matrices. We begin by stating a comparison 
principle for the resolvent moments.

\begin{lem}
\label{lem:gaussperturb}
Let $H,\tilde H$ be self-adjoint Gaussian random matrices. Then we have
$$
	\big|\EE[\ntr |z\id - H|^{-2p}]^{\frac{1}{2p}} -
	\EE[\ntr |z\id - \tilde H|^{-2p}]^{\frac{1}{2p}}\big|
	\le
	\frac{\|\EE H-\EE\tilde H\|}{(\mathrm{Im}\,z)^2} +
	2p\,\frac{\Delta(H,\tilde H)}{(\mathrm{Im}\,z)^3}
$$
for any $p\in\mathbb{N}$, where
$$
	\Delta(H,\tilde H) :=
	\sup_{\|M\|\le 1}\big\|
	\EE[(H-\EE H)M(H-\EE H)]-
	\EE[(\tilde H-\EE\tilde H)M(\tilde H-\EE\tilde H)]\big\|.
$$

\end{lem}

\begin{proof}
First, note that as $A^{-1}-B^{-1}=A^{-1}(B-A)B^{-1}$, we have
$$
	\|(z\id -H)^{-1}-(z\id - (H-\EE H+\EE\tilde H))^{-1}\| \le
	\frac{\|\EE H-\EE\tilde H\|}{(\mathrm{Im}\,z)^2}.
$$
Thus we can assume in the sequel that $\EE H=\EE\tilde H$.

Assume without loss of generality that $H,\tilde H$ are independent
and $\EE H=\EE\tilde H$, and let $Y,\tilde Y$ be independent copies of 
$H-\EE H$ and $\tilde H-\EE\tilde H$, respectively.
Define 
$$
	H(t) := \EE H + \sqrt{t}\,(H-\EE H)+\sqrt{1-t}\,(\tilde H-\EE\tilde H).
$$
By the Gaussian interpolation lemma \cite[Lemma 4.11]{BBV21}
$$
	\frac{d}{dt}\EE[\ntr |z\id - H(t)|^{-2p}] =
	\frac{
	\EE[\partial_{Y}^2 \ntr |z\id - H(t)|^{-2p}]
	-
	\EE[\partial_{\tilde Y}^2 \ntr |z\id - H(t)|^{-2p}]
	}{2}.
$$
By the product rule, we have
\begin{multline*}
	\partial_B^2 \ntr|z\id-H|^{-2p} =
	2p \sum_{k=0}^p \mathop{\mathrm{Re}} \ntr[
	B (z\id-H)^{-k-1} B (z\id-H)^{-p-1+k} (\bar z\id-H)^{-p}]
	\\ +2p \sum_{k=0}^{p-1} \mathop{\mathrm{Re}} \ntr[
	B (z\id-H)^{-p-1} 
	(\bar z\id-H)^{-k-1}
	B
	(\bar z\id-H)^{-p+k}].
\end{multline*}
We can therefore bound
$$
	\bigg|
	\frac{d}{dt}\EE[\ntr |z\id - H(t)|^{-2p}] \bigg|
	\le
	p(2p+1)\, \Delta(H,\tilde H)\,
	\EE[\ntr |z\id-H(t)|^{-2p-2}]
$$
by applying Lemma \ref{lem:calderon} to 
$F(M,M') := \EE[\ntr YMYM']-\EE[\ntr \tilde YM\tilde YM']$ as in the
proof of Lemma \ref{lem:holder}, and using 
$\sup_{\|M\|\le 1,\|M'\|_1\le 1}|F(M,M')| = \Delta(H,\tilde H)$.

It remains to note that $\EE[\ntr |z\id-H(t)|^{-2p-2}]\le
(\mathrm{Im}\,z)^{-3}\EE[\ntr |z\id-H(t)|^{-2p}]^{1-\frac{1}{2p}}$ and 
the chain rule readily yield the estimate
$$
	\bigg|
	\frac{d}{dt}\EE[\ntr |z\id - H(t)|^{-2p}]^{\frac{1}{2p}} \bigg|
	\le
	(p+\tfrac{1}{2})\,
	\frac{\Delta(H,\tilde H)}{(\mathrm{Im}\,z)^3}.
$$
The conclusion now follows by integrating over $t$ (using 
$p+\frac{1}{2}\le 2p$).
\end{proof}

A bound on the Hausdorff distance now follows along familiar lines.

\begin{prop}
\label{prop:gausscomp}
Let $H,\tilde H$ be self-adjoint Gaussian random matrices. Then
\begin{multline*}
	\mathbf{P}\big[
	\dH(\spc(H),\spc(\tilde H)) >
	C\{\|\EE H-\EE\tilde H\| +
	\Delta(H,\tilde H)^{\frac{1}{2}}\sqrt{\log d} \\
	+ (\sigma_*(H)+\sigma_*(\tilde H))x\}\big]
	\le de^{-x^2}
\end{multline*}
for all $x\ge 0$, where $C$ is a universal constant and
$\Delta(H,\tilde H)$ is as in Lemma \ref{lem:gaussperturb}.
\end{prop}

\begin{proof}
Note first that
$$
	\EE[\ntr |z\id - H|^{-2p}]^{\frac{1}{2p}}
	\le
	\EE[\|(z\id - H)^{-1}\|^{2p}]^{\frac{1}{2p}}
	\le
	\EE\|(z\id - H)^{-1}\| +
	C\sqrt{p}\frac{\sigma_*(H)}{(\mathrm{Im}\,z)^2}
$$
for a universal constant $C$ as in the proof of 
Lemma \ref{lem:uppersinglez}. On the other hand,
$$
	\EE[\ntr |z\id - \tilde H|^{-2p}]^{\frac{1}{2p}}
	\ge d^{-\frac{1}{2p}} \EE\|(z\id - \tilde H)^{-1}\|.
$$
Combining these bounds with Lemmas \ref{lem:gconc} and 
\ref{lem:gaussperturb} yields
\begin{multline*}
	\PP\bigg[
	d^{-\frac{1}{2p}}
	\|(z\id - \tilde H)^{-1}\| \ge
	\|(z\id - H)^{-1}\|
	+
	\frac{\|\EE H-\EE\tilde H\|}{(\mathrm{Im}\,z)^2}
	+
	2p
	\,\frac{\Delta(H,\tilde H)}{(\mathrm{Im}\,z)^3}
\\
	+
	C\sqrt{p}\frac{\sigma_*(H)}{(\mathrm{Im}\,z)^2}
	+ 
	\frac{\sigma_*(\tilde H)+\sigma_*(H)}{(\mathrm{Im}\,z)^2}\,x\bigg]
	\le 4e^{-x^2/2}.
\end{multline*}
Choosing $p=\lfloor\log d\rfloor$ and
proceeding as in the proof of Proposition \ref{prop:specunivupper} yields
\begin{multline*}
	\PP\bigg[
	\|(z\id - \tilde H)^{-1}\| \ge
	L\bigg\{
	\|(z\id - H)^{-1}\|
	+
	\frac{\|\EE H-\EE\tilde H\|}{\varepsilon^2}
	+
	\frac{\Delta(H,\tilde H)\log d}{\varepsilon^3}
\\
	+ 
	\frac{\sigma_*(\tilde H)+\sigma_*(H)}{\varepsilon^2}\,x
	\bigg\}\text{ for some }z\in\spc(\tilde H)+i\varepsilon\bigg]
	\le d e^{-x^2}
\end{multline*}
for all $x\ge 0$, where $L$ is a universal constant. The same 
bound holds if we reverse the roles of $H$ and $\tilde H$, and the 
conclusion follows readily from Lemma \ref{lem:rescomparison}.
\end{proof}

We finally formulate a variant of Lemma \ref{lem:gaussperturb} for the
resolvent.

\begin{lem}
\label{lem:gaussperturbres}
Let $H,\tilde H$ be self-adjoint Gaussian random matrices. Then we have
$$
	\big\|\EE[(z\id - H)^{-1}] - \EE[(z\id - \tilde H)^{-1}] \big\|
	\le
	\frac{\|\EE H-\EE\tilde H\|}{(\mathrm{Im}\,z)^2} +
	\frac{\Delta(H,\tilde H)}{(\mathrm{Im}\,z)^3},
$$
where $\Delta(H,\tilde H)$ is as in Lemma \ref{lem:gaussperturb}.
\end{lem}

\begin{proof}
As in the proof of Lemma \ref{lem:gaussperturb}, it suffices to assume
that $\EE H=\EE\tilde H$. Moreover, applying Gaussian interpolation 
as in the proof of Lemma \ref{lem:gaussperturb} yields
\begin{multline*}
	\frac{d}{dt}\EE[\langle v,(z\id -H(t))^{-1}w\rangle]
	=
	\EE[\langle v,(z\id -H(t))^{-1}
	Y(z\id -H(t))^{-1}Y(z\id -H(t))^{-1}
	w\rangle] \\ -
	\EE[\langle v,(z\id -H(t))^{-1}
	\tilde Y(z\id -H(t))^{-1}\tilde Y(z\id -H(t))^{-1}
	w\rangle].
\end{multline*}
Integrating this identity and taking the supremum over $v,w$ with
$\|v\|=\|w\|=1$ readily yields the conclusion, where we use that
$Y,\tilde Y$ are independent of $H(t)$.
\end{proof}

\subsection{The truncation error}
\label{sec:truncerr}

In order to apply the above comparison principles to $G,\tilde G$, we must 
estimate the relevant parameters in this case. 

\begin{lem}
\label{lem:truncerrmn}
$\|\EE G-\EE \tilde G\| \le \sqrt{2}\,\sigma_*(X)$ for
$R\ge \sqrt{2}\,\bar R(X)$.
\end{lem}

\begin{proof}
We first note that as $\EE X=Z_0$ and
by the independence of $Z_1,\ldots,Z_n$
$$
	\EE \tilde G - \EE G =
	\EE \tilde X - \EE X =
	\sum_{i=1}^n \EE[1_{\|Z_i\|\le R} Z_i] =
	\sum_{i=1}^n b_i^{-1} \EE[1_{\max_j\|Z_j\|\le R} Z_i],
$$
where $b_i := \mathbf{P}[\max_{j\ne i}\|Z_j\|\le R]\ge
\mathbf{P}[\max_j\|Z_j\|\le R]$. We therefore obtain 
$$
	\|\EE \tilde G - \EE G\|
	=
	\sup_{\|v\|=\|w\|=1}
	\bigg|\EE\bigg[1_{\max_j\|Z_j\|\le R} 
	\sum_{i=1}^n b_i^{-1}\langle v,Z_iw\rangle
	\bigg]\bigg| \le
	\frac{\sigma_*(X)}{
	\mathbf{P}[\max_j\|Z_j\|\le R]^{\frac{1}{2}}}
$$
by Cauchy-Schwarz. It remains to note that 
$\mathbf{P}[\max_j\|Z_j\|\le R] \ge \frac{1}{2}$
whenever $R\ge \sqrt{2}\,\EE[\max_j\|Z_j\|^2]^{\frac{1}{2}}
=: \sqrt{2}\,\bar R(X)$ by Markov's inequality.
\end{proof}

\begin{lem}
\label{lem:truncerrvar}
$\Delta(G,\tilde G)\le 24\,\bar R(X)\sigma(X)$
for $R\ge \bar R(X)^{\frac{1}{2}}\sigma(X)^{\frac{1}{2}}$.
\end{lem}

\begin{proof}
Suppose first that $M\ge 0$ with $\|M\|\le 1$. We begin by writing
\begin{multline*}
	\EE[(G-\EE G)M(G-\EE G)] -
	\EE[(\tilde G-\EE \tilde G)M(\tilde G-\EE \tilde G)] = \\
	\sum_{i=1}^n
	\big\{
	\EE[1_{\|Z_i\|>R}Z_i M Z_i]
	+ \EE[1_{\|Z_i\|\le R} Z_i]\, M\, \EE[1_{\|Z_i\|\le R} Z_i]\big\}
	\ge 0.
\end{multline*}
Now note that as $\EE[Z_i]=0$, we have $\EE[1_{\|Z_i\|\le R} Z_i]=
-\EE[1_{\|Z_i\|>R} Z_i]$. Moreover, for any self-adjoint random matrix 
$Y$, we have $\EE[Y]\,M\,\EE[Y] \le \EE[Y]^2 \le
\EE[Y^2]$ using $\|M\|\le 1$ and Jensen's inequality. We therefore obtain
$$
	\|\EE[(G-\EE G)M(G-\EE G)] -
	\EE[(\tilde G-\EE \tilde G)M(\tilde G-\EE \tilde G)]\|
	\le 2\!\sup_{\|v\|=1}\sum_{i=1}^n \EE[1_{\|Z_i\|>R} \|Z_iv\|^2].
$$
To proceed, let $Z_i'$ be independent copies of $Z_i$, and note that 
\begin{align*}
	&\sum_{i=1}^n \EE[1_{\|Z_i\|>R} \|Z_iv\|^2] \le
	\sum_{i=1}^n \EE[1_{\max_j\|Z_j\|>R} \|Z_iv\|^2] \\ &\qquad \le
	\sum_{i=1}^n \EE[1_{\max_j\|Z_j\|>R}(\|Z_iv\|^2-\|Z_i'v\|^2)] +
	\PP[{\textstyle\max_j\|Z_j\|>R}]\, \sigma(X)^2
\end{align*}
for $\|v\|=1$.
Moreover, we have
\begin{multline*}
	\sum_{i=1}^n \EE[1_{\max_j\|Z_j\|>R}(\|Z_iv\|^2-\|Z_i'v\|^2)]
	\le
	\EE\bigg[
	\bigg|\sum_{i=1}^n
	(\|Z_iv\|^2-\|Z_i'v\|^2)\bigg|\bigg] =
	\\ 
	\EE\bigg[
	\bigg|\sum_{i=1}^n
	\varepsilon_i(\|Z_iv\|^2-\|Z_i'v\|^2)\bigg|\bigg]
	\le 2\,
	\EE\bigg[
	\bigg|\sum_{i=1}^n
	\varepsilon_i\|Z_iv\|^2\bigg|\bigg]
	\le
	2\,
	\EE\bigg[
        \bigg(\sum_{i=1}^n
        \|Z_iv\|^4\bigg)^{\frac{1}{2}}\bigg],
\end{multline*}
where $\varepsilon_i$ are i.i.d.\ random signs that are independent of the
other variables, we used that the distribution of $(Z_i,Z_i')$ is 
invariant under exchanging $Z_i$ and $Z_i'$ for any $i$, and we applied
Jensen's inequality conditionally on $(Z_i)$ in the last inequality.
Bounding $\|Z_iv\|^4\le(\max_j\|Z_j\|^2) \|Z_iv\|^2$ and applying 
Cauchy-Schwarz yields
$$
	\sum_{i=1}^n \EE[1_{\max_j\|Z_j\|>R}(\|Z_iv\|^2-\|Z_i'v\|^2)] \le
	2\,
	\EE[{\textstyle \max_j \|Z_j\|^2}]^{\frac{1}{2}}
	\sigma(X)
$$
for $\|v\|=1$.
Putting together all the above estimates yields
\begin{multline*}
	\|\EE[(G-\EE G)M(G-\EE G)] -
	\EE[(\tilde G-\EE \tilde G)M(\tilde G-\EE \tilde G)]\|
	\le \\
	2\,\PP[{\textstyle\max_j\|Z_j\|>R}]\, \sigma(X)^2 +
	4\,\EE[{\textstyle \max_j \|Z_j\|^2}]^{\frac{1}{2}}\sigma(X)
	\le
	6\,\EE[{\textstyle \max_j \|Z_j\|^2}]^{\frac{1}{2}}\sigma(X)
\end{multline*}
for $R\ge \EE[{\textstyle \max_j \|Z_j\|^2}]^{\frac{1}{4}}
\sigma(X)^{\frac{1}{2}}$ using Markov's inequality in the last step.

Finally, note that any matrix $M$ with $\|M\|\le 1$ can be written
as $M=\mathrm{Re}\,M+i\,\mathrm{Im}\,M$ with $\|\mathrm{Re}\,M\|=
\frac{1}{2}\|M+M^*\|\le 1$ and $\|\mathrm{Im}\,M\|=
\frac{1}{2}\|M-M^*\|\le 1$. As any self-adjoint matrix
is the difference of its positive and negative parts, we can write
$M=M_1-M_2+iM_3-iM_4$ with $M_i\ge 0$ with $\|M_i\|\le 1$. 
Applying 
the above estimate to each $M_i$ and using the triangle inequality 
concludes the proof.
\end{proof}

Finally, we must bound the matrix parameters of $\tilde X$.

\begin{lem}
\label{lem:truncparm}
We have $R(\tilde X)\le 2R$, 
$\sigma_*(\tilde X)\le\sigma_*(X)$, and
$\sigma(\tilde X)\le\sigma(X)$.
\end{lem}

\begin{proof}
The first inequality follows immediately from
$$
	R(\tilde X) = \Big\|\max_{1\le i\le n} 
	\big\|1_{\|Z_i\|\le R}Z_i - \EE[1_{\|Z_i\|\le R}Z_i]\big\|
	\Big\|_\infty
$$
and the triangle inequality.
Next, note that for any 
(complex) random variable $Y$ and event $A$, we have
$\EE[|1_AY-\EE[1_AY]|^2] \le \EE[1_A|Y|^2] \le \EE[|Y|^2]$.
Thus
\begin{align*}
	\EE[|\langle v,(\tilde X-\EE\tilde X)w\rangle|^2]
	&=
	\sum_{i=1}^n 
	\EE[|1_{\|Z_i\|\le R}\langle v,Z_iw\rangle - \EE[1_{\|Z_i\|\le R} 
	\langle v,Z_iw\rangle]|^2]
	\\ &\le
	\sum_{i=1}^n \EE[|\langle v,Z_i w\rangle|^2]
	=
	\EE[|\langle v,(X-\EE X)w\rangle|^2]
\end{align*}
for any nonrandom vectors $v,w$. The remaining bounds follow, 
respectively, by taking the supremum over $\|v\|=\|w\|=1$, or by summing 
over the coordinate basis $w=e_k$ 
and then taking the supremum over $\|v\|=1$.
\end{proof}

\subsection{Proof of Theorems \ref{thm:specheavy} and
\ref{thm:smunivheavy}}
\label{sec:truncpf}

We now put everything together.

\begin{proof}[Proof of Theorem \ref{thm:specheavy}]
As $X=\tilde X$ on the event $\{\max_i \|Z_i\|\le R\}$, we obtain
\begin{align*}
	&\PP\Big[\dH(\spc(X),\spc(\tilde G)) > C
	\tilde\varepsilon(t),~\max_{1\le i\le n}\|Z_i\|\le R\Big] 
	\\
	&\qquad\qquad\le \PP\big[\dH(\spc(\tilde X),\spc(\tilde G)) > C  
        \tilde\varepsilon(t)\big] \le de^{-t}
\end{align*}
with
$$
	\tilde\varepsilon(t) :=
	\sigma_*(\tilde X)\,t^{\frac{1}{2}} +
	R(\tilde X)^{\frac{1}{3}}\sigma(\tilde X)^{\frac{2}{3}}t^{\frac{2}{3}}
	+R(\tilde X)\,t.
$$
from Theorem \ref{thm:specuniv}. Moreover, we can replace 
$\tilde\varepsilon(t)$ by $\varepsilon_R(t)$ on the left-hand side of the 
above inequality as $\tilde\varepsilon(t)\le\varepsilon_R(t)$ by Lemma 
\ref{lem:truncparm}.

On the other hand, 
Proposition \ref{prop:gausscomp} and
Lemmas \ref{lem:truncerrmn}, \ref{lem:truncerrvar} and 
\ref{lem:truncparm} imply
$$
	\PP\big[\dH(\spc(\tilde G),\spc(G)) >
	C\{ R + \sigma_*(X)\}t^{\frac{1}{2}}\big]
	\le de^{-t}
$$
for all $t\ge 0$ and $R\ge R_0:= \bar 
R(X)^{\frac{1}{2}}\sigma(X)^{\frac{1}{2}}+
\sqrt{2}\,\bar R(X)$.
Here we used that we may assume without loss of 
generality that $t\ge\log d$ in the above estimate (as otherwise
the right-hand side exceeds one and the bound is trivial), and that
$\Delta(G,\tilde G)\le 24 R^2$ by the assumption on $R$ and Lemma 
\ref{lem:truncerrvar}. We may once again replace $\{ R + 
\sigma_*(X)\}t^{\frac{1}{2}}$ by $\varepsilon_R(t)$ on the left-hand side
as $t^{\frac{1}{2}}\lesssim t$ for $t\ge\log d$.

Combining the above bounds, we obtain
\begin{align*}
	&\PP\Big[\dH(\spc(X),\spc(G)) > 2C
	\varepsilon_R(t),~\max_{1\le i\le n}\|Z_i\|\le R\Big]  \\
	&\le \PP\Big[\dH(\spc(X),\spc(\tilde G))+\dH(\spc(\tilde G),\spc(G))
	> 2C
	\varepsilon_R(t),~\max_{1\le i\le n}\|Z_i\|\le R\Big]  \\
	&\le 
	\PP\Big[\dH(\spc(X),\spc(\tilde G))
	> C
	\varepsilon_R(t),~\max_{1\le i\le n}\|Z_i\|\le R\Big] 
	\\ &\qquad\qquad
	+
	\PP\big[\dH(\spc(\tilde G),\spc(G))
	> C
	\varepsilon_R(t)\big] \le 2de^{-t}.
\end{align*}
As in the proof of Theorem \ref{thm:specuniv},
the upper bound can be replaced by $de^{-t}$ if we increase the value of 
the universal constant on the left-hand side.

This concludes the proof of the tail bound. To prove the expectation 
bound, note that choosing $R=R_0t$ in the tail bound yields
$$
	\PP[\dH(\spc(X),\spc(G)) > C\delta(t)]
	\le e^{-t/2}+
	\PP\bigg[\max_{1\le i\le n}\|Z_i\|> R_0 t\bigg] 
$$
for $t\ge 2\log d$, where
$$
	\delta(t):=
	\sigma_*(X)\,t^{\frac{1}{2}} +
        R_0^\frac{1}{3}\sigma(X)^{\frac{2}{3}}t +
        R_0t^2
$$
and we used $de^{-t}\le e^{-t/2}$ for $t\ge 2\log d$.
We now compute
\begin{align*}
	&\EE[\dH(\spc(X),\spc(G))]
	\le
	C\delta(2\log d) +
	\int_{C\delta(2\log 2)}^\infty
	\PP[\dH(\spc(X),\spc(G))>x]\,dx 
\\
	&=
	C\delta(2\log d) +
	C\int_{2\log d}^\infty
	\PP[\dH(\spc(X),\spc(G))>C\delta(t)]\,
	\frac{d\delta(t)}{dt}\,dt 
\\
	&\le
	C\delta(2\log d) +
	C\int_0^\infty
	e^{-t/2} 
	\frac{d\delta(t)}{dt}\,dt 
	+
	C\int_0^\infty
	\PP\bigg[\max_{1\le i\le n}\|Z_i\|> R_0 t\bigg]\,
	\frac{d\delta(t)}{dt}\,dt 
\\
	&\lesssim
	\delta(2\log d) 
	+
	\EE[\delta({\textstyle\max_i\|Z_i\|/R_0})]
	\lesssim \delta(2\log d).\phantom{\int}
\end{align*}
The conclusion follows readily using the assumption $\bar R(X)\, (\log 
d)^3 \lesssim \sigma(X)$.
\end{proof}

\begin{proof}[Proof of Theorem \ref{thm:smunivheavy}]
As $X=\tilde X$ on the event $\{\max_i \|Z_i\|\le R\}$, we have
$$
	\EE[(z\id -X)^{-1}] =
	\EE[(z\id -\tilde X)^{-1}] +
	\EE[\{(z\id -X)^{-1}-
	(z\id -\tilde X)^{-1}\}1_{\max_i\|Z_i\|>R}].
$$
Thus Markov's inequality yields
$$
	\big\|\EE[(z\id -X)^{-1}]-
	\EE[(z\id -\tilde X)^{-1}]\big\|
	\le
	\frac{2\,\mathbf{P}[\max_i\|Z_i\|>R]}{\mathrm{Im}\,z}
	\le
	\frac{2\bar R(X)^2}{R^2\,\mathrm{Im}\,z}.
$$
Applying Theorem \ref{thm:smuniv} and Lemmas
\ref{lem:gaussperturbres}, 
\ref{lem:truncerrmn}, \ref{lem:truncerrvar} and
\ref{lem:truncparm} yields
$$
	\big\|\EE[(z\id -X)^{-1}]-
	\EE[(z\id -G)^{-1}]\big\|
	\lesssim
	\frac{\bar R(X)^2}{R^2\,\mathrm{Im}\,z} 
	+
	\frac{\sigma_*(X)}{(\mathrm{Im}\,z)^2}
	+
	\frac{R^2}{(\mathrm{Im}\,z)^3}
	+
	\frac{R \sigma(X)^2 + R^3\log d}{(\mathrm{Im}\,z)^4}
$$
for $R\ge R_0:= \bar
R(X)^{\frac{1}{2}}\sigma(X)^{\frac{1}{2}}+
\sqrt{2}\,\bar R(X)$.

Now assume first that $\mathrm{Im}\,z\ge R_0$ and 
choose $R=R_0^{\frac{1}{2}}(\mathrm{Im}\,z)^{\frac{1}{2}}$. Then 
we obtain
$$
	\big\|\EE[(z\id -X)^{-1}]-
	\EE[(z\id -G)^{-1}]\big\|
	\lesssim
	\frac{\sigma_*(X)+R_0}{(\mathrm{Im}\,z)^2} 
	+\frac{ R_0^{\frac{3}{2}}\log d}{(\mathrm{Im}\,z)^{\frac{5}{2}}}
	+\frac{R_0^{\frac{1}{2}} \sigma(X)^2}{(\mathrm{Im}\,z)^{\frac{7}{2}}}.
$$
In particular, if $\mathrm{Im}\,z\ge 
R_0^{\frac{1}{5}}\sigma(X)^{\frac{4}{5}}+
R_0(\log d)^{\frac{2}{3}}$, we obtain
$$
	\big\|\EE[(z\id -X)^{-1}]-
	\EE[(z\id -G)^{-1}]\big\|
	\lesssim
	\frac{\sigma_*(X)+
	R_0^{\frac{1}{5}}\sigma(X)^{\frac{4}{5}}
	+R_0 (\log d)^{\frac{2}{3}}}{(\mathrm{Im}\,z)^2}.
$$
On the other hand, for $\mathrm{Im}\,z< 
R_0^{\frac{1}{5}}\sigma(X)^{\frac{4}{5}} +R_0(\log d)^{\frac{2}{3}}$
we can estimate
$$
	\big\|\EE[(z\id -X)^{-1}]-
	\EE[(z\id -G)^{-1}]\big\| \le
	\frac{2}{\mathrm{Im}\,z} \lesssim
	\frac{
	R_0^{\frac{1}{5}}\sigma(X)^{\frac{4}{5}}
	+R_0(\log d)^{\frac{2}{3}}}{(\mathrm{Im}\,z)^2}.
$$
If $\bar R(X) (\log d)^{\frac{5}{3}}\lesssim \sigma(X)$, then
$R_0 (\log d)^{\frac{2}{3}} \lesssim 
R_0^{\frac{1}{5}}\sigma(X)^{\frac{4}{5}} \asymp
\bar R(X)^{\frac{1}{10}}\sigma(X)^{\frac{9}{10}}$, and the first part
of the theorem follows. The second part of the theorem now follows from
\cite[Lemma 5.11]{BBV21} as in the proof of Theorem \ref{thm:smuniv}.
\end{proof}

\section{Applications: Proofs}
\label{sec:applproofs}

\subsection{Random lifts}
\label{sec:pflift}

The aim of this section is to prove Theorem \ref{thm:lift}. We will first 
prove a more general result, and then specialize to the case of lifts.

\subsubsection{Strong convergence}

In this section, we let $\Pi_1,\ldots,\Pi_k$ be i.i.d.\ uniformly 
distributed random $n\times n$ permutation matrices, and we fix 
$A_1,\ldots,A_k\in\mathrm{M}_d(\mathbb{C})$.
We consider the random matrix
$$
	X = \sum_{i=1}^k (A_i\otimes\Pi_i + A_i^*\otimes\Pi_i^*),
$$
and let $X^\perp$ be its restriction to $\mathbb{C}^d\otimes 1^\perp$.
Recall that if $s_1,\ldots,s_{2k}$ is a free semicircular family, then
$c_1,\ldots,c_k$ defined by $c_j = \frac{s_j+is_{k+j}}{\sqrt{2}}$ is 
a free circular family.

\begin{prop}
\label{prop:scperm}
Let $c_1,\ldots,c_k$ be a free circular family, and define
$$
	X_{\mathrm{F}} =
	\sum_{i=1}^k 
	\big\{
	\big(
	(1-\varepsilon^2)^{\frac{1}{2}}A_i+\varepsilon A_i^*
	\big)\otimes c_i +
	\big(
	(1-\varepsilon^2)^{\frac{1}{2}}A_i+\varepsilon A_i^*
	\big)^*\otimes c_i^*
	\big\}
$$
with $\varepsilon = \frac{1}{\sqrt{n-1}(\sqrt{n}+\sqrt{n-2})}$.
Then 
$$
	\mathbf{P}\big[\|X^\perp\| \ge \|X_{\rm F}\| + 
	C\big\{
	v^{\frac{1}{2}}\sigma^{\frac{1}{2}}(\log nd)^{\frac{3}{4}} +
	vt^{\frac{1}{2}} +
	R^{\frac{1}{3}}\sigma^{\frac{2}{3}}t^{\frac{2}{3}} +
	Rt\big\}\big]
	\le 2nd e^{-t}
$$
for all $t\ge 0$, where $C$ is a universal constant and
$$
	\sigma = \bigg\|\sum_{i=1}^k\bigg(A_iA_i^*+A_i^*A_i+
	\frac{A_i^2+A_i^{*2}}{n-1}\bigg)\bigg\|^{\frac{1}{2}},\qquad
	R = 2\max_{1\le i\le k}\|A_i\|,
$$
and
$$
	v = \frac{2}{\sqrt{n-1}} \,
	\bigg\|\Cov\bigg(\sum_{i=1}^k A_ig_i\bigg)\bigg\|^{\frac{1}{2}},
$$
where $g_1,\ldots,g_k$ are i.i.d.\ standard real Gaussians.
\end{prop}

Proposition \ref{prop:scperm} is an immediate consequence of 
Theorem \ref{thm:smconc} once we prove that
$\sigma(X^\perp)=\sigma$, $R(X^\perp)\le R$, $\sigma_*(X^\perp)\le 
v(X^\perp)\le v$, and $\|X^\perp_{\rm free}\|=\|X_{\rm F}\|$. These facts 
will be established in the following lemmas, concluding the proof.

\begin{lem}
\label{lem:stpersig}
We have $\sigma(X^\perp)=\sigma$ and $R(X^\perp)\le R$.
\end{lem}

\begin{proof}
The bound on $R(X^\perp)$ follows immediately from $\|\Pi_i\|=1$. To 
compute $\sigma(X^\perp)$, we note that the restriction of $\Pi_i$ to 
$1^\perp$ a random matrix as in Lemma \ref{lem:groupthy} with $d=n-1$ and 
$s=1$ (as this is an $(n-1)$-dimensional real representation of the 
symmetric group; alternatively, the conclusions of Lemma 
\ref{lem:groupthy} can be verified in this case by a direct computation).
We can therefore compute
$$
	\EE[X^\perp]=0,\qquad
	\EE[(X^\perp)^2] =
	\sum_{i=1}^k\bigg(A_iA_i^*+A_i^*A_i+
        \frac{A_i^2+A_i^{*2}}{n-1}\bigg)\otimes \id,
$$
and the conclusion follows immediately.
\end{proof}

\begin{lem}
We have $\|X_{\rm free}^\perp\|=\|X_{\mathrm{F}}\|$.
\end{lem}

\begin{proof}
Let $\bar c_i := (1-\varepsilon^2)^{\frac{1}{2}}c_i+\varepsilon c_i^*$.
Then we write
$$
	X_{\rm free}^\perp = \sum_{i=1}^k (A_i\otimes 
	\Pi_{i,\mathrm{free}}^\perp + A_i^*\otimes \bar 
	\Pi_{i,\mathrm{free}}^{\perp*}),\qquad
	X_{\rm F} = \sum_{i=1}^k (A_i\otimes \bar c_i + A_i^*\otimes \bar 
	c_i^*).
$$
Now note that
$$
	\tau(\bar c_i)=0,\qquad
	\tau(\bar c_i\bar c_i^*) = 
	\tau(\bar c_i^*\bar c_i) = 1,\qquad
	\tau(\bar c_i^2) = \tau((\bar c_i^*)^2) = \frac{1}{n-1},
$$
while applying Lemma \ref{lem:groupthy} as in the proof of 
Lemma \ref{lem:stpersig} yields
$$
	\mathbf{E}[\Pi_i^\perp] =0,
	\quad
	\mathbf{E}[\Pi_i^\perp\Pi_i^{\perp*}] =
	\mathbf{E}[\Pi_i^{\perp*}\Pi_i^\perp] = \id,\quad
	\mathbf{E}[
	(\Pi_i^\perp)^2] =
	\mathbf{E}[
	(\Pi_i^{\perp*})^2] =
	\frac{1}{n-1}\id.
$$
The conclusion follows as in the proof of \cite[Lemma 7.9]{BBV21}.
\end{proof}

\begin{lem}
We have $v(X^\perp)\le v$.
\end{lem}

\begin{proof}
By Lemma \ref{lem:groupthy}, we have
$$
	\Cov(A_i\otimes\Pi_i^\perp) = \iota(A_i)\iota(A_i)^*\otimes
	\Cov(\Pi_i^\perp) =
	\iota(A_i)\iota(A_i)^*\otimes \frac{1}{n-1}\id,
$$
where $\iota:\mathrm{M}_d(\mathbb{C})\to\mathbb{C}^{d^2}$ maps
a matrix to its vector of entries.
Therefore, as $\Pi_1,\ldots,\Pi_k$ are independent, we have
$$
	\Cov\bigg(\sum_{i=1}^k A_i\otimes\Pi_i^\perp\bigg) =
	\sum_{i=1}^k \Cov(A_i\otimes\Pi_i^\perp) =
	\frac{1}{n-1} \Cov\bigg(\sum_{i=1}^k A_i g_i\bigg)\otimes \id.
$$
The conclusion follows from the triangle inequality
$v(A+B)\le v(A)+v(B)$.
\end{proof}

\subsubsection{Free generators and circular variables}

The operator $X_{\rm F}$ in Proposition \ref{prop:scperm} is defined by a 
free circular family. In the study of random lifts, however, we are 
interested in the analogous operator where the circular variables $c_i$ 
are replaced by the left-regular representation $\lambda(g_i)$ of the free 
generators of $\mathrm{F}_k$. We presently establish a comparison 
principle between these objects.

\begin{prop}
\label{prop:freeclt}
Let $c_1,\ldots,c_k$ be a free circular family and let $g_1,\ldots,g_k$ be 
free generators of $\mathrm{F}_k$. Then for any 
$A_1,\ldots,A_k\in\mathrm{M}_d(\mathbb{C})$, we have
$$
	\bigg\|\sum_{i=1}^k (A_i\otimes c_i+A_i^*\otimes c_i^*)\bigg\|
	\le
	\bigg\|\sum_{i=1}^k (A_i\otimes \lambda(g_i) +
	A_i^*\otimes\lambda(g_i)^*)\bigg\| +
	2\tilde\sigma^{\frac{1}{2}}\tilde R^{\frac{1}{2}},
$$
where $\tilde\sigma^2 = \|\sum_{i=1}^k (A_iA_i^*+A_i^*A_i)\|$
and $\tilde R=\max_{1\le i\le k}\|A_i\|$.
\end{prop}

\begin{proof}
We proceed in several steps.

\medskip
\textbf{Step 1.} We begin by noting that if $c_1,\ldots,c_k$ is a circular 
family, $c_1^*,\ldots,c_k^*$ is also a circular family. We can therefore 
write
$$
	\bigg\|\sum_{i=1}^k (A_i\otimes c_i+A_i^*\otimes c_i^*)\bigg\|
	=
	\bigg\|\sum_{i=1}^k (A_i^*\otimes c_i+A_i\otimes c_i^*)\bigg\|
	=
	\bigg\|\sum_{i=1}^k (\tilde A_i\otimes c_i+\tilde B_i\otimes 
	c_i^*)\bigg\|,
$$
where we defined the self-adjoint matrices
$$
	\tilde A_i = \begin{bmatrix} 0 & A_i \\ A_i^* & 0 \end{bmatrix},\qquad
	\quad
	\tilde B_i = \begin{bmatrix} 0 & A_i^* \\ A_i & 0 \end{bmatrix}.
$$
Similarly, if $g_1,\ldots,g_k$ are free generators of $\mathrm{F}_k$, then 
$g_1^{-1},\ldots,g_k^{-1}$ are as well, and thus the analogous identities 
holds when $c_i$ is replaced by $\lambda(g_i)$.

\medskip
\textbf{Step 2.} For any $A,M\in \mathrm{M}_d(\mathbb{C})_{\rm sa}$ with
$M>0$, define
$$
	\mathrm{R}_A(M) = 
	M^{\frac{1}{2}}\big(
	\big(\id+(M^{-\frac{1}{2}}AM^{-\frac{1}{2}})^2\big)^{\frac{1}{2}}-\id
	\big)M^{\frac{1}{2}}.
$$
Then by \cite[Theorem 1.1 and p.\ 454]{Leh99}, we have
$$
	\bigg\|\sum_{i=1}^k (\tilde A_i\otimes \lambda(g_i)+\tilde 
	B_i\otimes
        \lambda(g_i)^*)\bigg\| =
	\inf_{M>0}\bigg\|
	2M + \sum_{i=1}^k
	(\mathrm{R}_{\tilde A_i}(M) + \mathrm{R}_{\tilde B_i}(M))
	\bigg\|.
$$
On the other hand, as the circular family $c_1,\ldots,c_k$ can be realized 
by setting $c_i = l_i + l_{k+i}^*$ where $l_1,\ldots,l_{2k}$ are canoncial 
creation operators on the free Fock space (cf.\ \cite[\S 1.2]{Leh99}),
we obtain by \cite[Corollary 1.4]{Leh99} 
$$
	\bigg\|\sum_{i=1}^k (\tilde A_i\otimes c_i + 
	\tilde B_i\otimes c_i^*)\bigg\| \le
	\inf_{M>0}\bigg\|
	2M + \sum_{i=1}^k 
	\frac{\tilde A_i M^{-1} \tilde A_i + \tilde B_i M^{-1} \tilde B_i}{2}
	\bigg\|	.
$$

\medskip
\textbf{Step 3.} We claim that for any $\delta>0$, there exists
$M\ge\frac{\delta}{2}\id$ so that
$$
	2M + \sum_{i=1}^k
	(\mathrm{R}_{\tilde A_i}(M) + \mathrm{R}_{\tilde B_i}(M)) =
	\bigg(
	\bigg\|\sum_{i=1}^k (\tilde A_i\otimes \lambda(g_i)+\tilde 
	B_i\otimes
        \lambda(g_i)^*)\bigg\|+\delta\bigg)\id.
$$
Indeed, let $\tilde g_1,\ldots,\tilde g_{2k}$ be the free generators of 
the free product $\mathbb{Z}_2*\cdots*\mathbb{Z}_2$ of $2k$ copies of 
$\mathbb{Z}_2$, and define the operator $\tilde X=\sum_{i=1}^k(\tilde 
A_i\otimes\lambda(\tilde g_i)+\tilde B_i\otimes\lambda(\tilde g_{i+k}))$.
Then \cite[Theorem 1.1, Lemma 2.3 and Proposition 3.1]{Leh99} show that
the above identity is satisfied if we choose the matrix $M$ so that
$(2M)^{-1}=(\mathrm{id}\otimes\tau)\big[\big((\|\tilde 
X\|+\delta)\id-\tilde X\big)^{-1}\big]$.
As $\|\tilde X\|\id-\tilde X\ge 0$, we clearly have
$(2M)^{-1} \le \delta^{-1}\id$, establishing the claim.

\medskip
\textbf{Step 4.} Define the function
$$
	h(x) = \frac{x^2}{2}-((1+x^2)^{\frac{1}{2}}-1)
	= \frac{x^4}{2((1+x^2)^{\frac{1}{2}}+1)^2}.
$$
Then clearly $h(x)\le \frac{1}{8}x^4$. Thus we have
for any self-adjoint $A,M$ with $M>0$
$$
	\frac{AM^{-1}A}{2} \le
	\mathrm{R}_A(M) + \frac{AM^{-1}AM^{-1}AM^{-1}A}{8},
$$
where we used that
$\frac{1}{2}AM^{-1}A-\mathrm{R}_A(M) = 
M^{\frac{1}{2}}h(M^{-\frac{1}{2}}AM^{-\frac{1}{2}})M^{\frac{1}{2}}$.

\medskip
\textbf{Step 5.} We now put everything together. Let $\delta>0$
and choose $M$ as in Step~3. Then we can estimate
\begin{align*}
	&\bigg\|\sum_{i=1}^k (A_i\otimes c_i+A_i^*\otimes c_i^*)\bigg\|
	\le
	\bigg\|
	2M + \sum_{i=1}^k 
	\frac{\tilde A_i M^{-1} \tilde A_i + \tilde B_i M^{-1} \tilde B_i}{2}
	\bigg\|	\\
	&\le
	\bigg\|\sum_{i=1}^k (A_i\otimes \lambda(g_i) +
	A_i^*\otimes\lambda(g_i)^*)\bigg\| +
	\delta \\
	&\qquad\quad+
	\bigg\|
	\sum_{i=1}^k
	\frac{\tilde A_i M^{-1}\tilde A_i M^{-1}\tilde 
	A_iM^{-1}\tilde A_i
	+
	\tilde B_i M^{-1}\tilde B_i M^{-1}\tilde 
	B_iM^{-1}\tilde B_i}{8}
	\bigg\| \\
	&\le
	\bigg\|\sum_{i=1}^k (A_i\otimes \lambda(g_i) +
	A_i^*\otimes\lambda(g_i)^*)\bigg\| +
	\delta +
	\frac{\max_i (\|\tilde A_i\|^2\vee \|\tilde B_i\|^2)}{\delta^3}
	\bigg\|\sum_{i=1}^k (\tilde A_i^2+\tilde B_i^2)\bigg\|,
\end{align*}
where we used
$\langle v,\tilde A_i M^{-1}\tilde A_i M^{-1}\tilde A_iM^{-1}\tilde 
A_iv\rangle \le\|\tilde A_iv\|^2 \|M^{-1}\|^3 \|\tilde A_i\|^2$ 
(and analogously for $\tilde B_i$) and $M\ge\frac{\delta}{2}$
in the last line. As $\max_i (\|\tilde A_i\|^2\vee \|\tilde 
B_i\|^2)=\tilde R^2$ and $\|\sum_i (\tilde A_i^2+\tilde 
B_i^2)\|=\tilde\sigma^2$, the conclusion follows by choosing
$\delta = \tilde\sigma^{\frac{1}{2}}\tilde R^{\frac{1}{2}}$.
\end{proof}

\subsubsection{Random $n$-lifts}

We now specialize the above results to the situation of random $n$-lifts. 
That is, we fix a base graph $H=([d],E_H)$, set $k=|E_H|$, and 
$A_e=e_ie_j^*$ for $e=(i,j)\in E_H$, $i\le j$. Then the random matrix 
$X=X^{(n)}$ is the adjacency matrix of the random $n$-lift of $H$. Let us 
begin by estimating the parameters that appear in Propositions 
\ref{prop:scperm} and \ref{prop:freeclt} in this case.

\begin{lem}
\label{lem:liftparm}
Denote by $\mathrm{D}(H)$ the maximal degree of a vertex of $H$ and by
$\mathrm{M}(H)$ the
maximal multiplicity of an edge of $H$. Then we have
$$
	\tilde\sigma^2 \le 
	\sigma^2 \le 2\,\mathrm{D}(H),\qquad
	v^2 = \frac{4\,\mathrm{M}(H)}{n-1},\qquad
	\tilde R\le R\le 2.
$$
\end{lem}

\begin{proof}
That $\tilde R\le R\le 2$ is immediate. Now let
$A_e=e_ie_j^*$ for $e=(i,j)\in E_H$. Then
$A_eA_e^*+A_e^*A_e=e_ie_i^*+e_je_j^*$ and $A_e^2+A_e^{*2}=
2e_ie_i^*1_{i=j}$. Therefore
$$
	\bigg\|\sum_{e\in E_H} (A_eA_e^*+A_e^*A_e)\bigg\| =
	\mathrm{D}(H),\qquad
	\bigg\|\sum_{e\in E_H} (A_e^2+A_e^{*2})\bigg\| =
	2\,\mathrm{L}(H)\le \mathrm{D}(H),
$$
where $\mathrm{L}(H)$ denotes the maximal number of self-loops attached to
a vertex of $H$. The bounds on $\tilde\sigma,\sigma$ follows.
Finally, note that $\sum_{e\in E_H} A_eg_e$, where $(g_e)_{e\in E_H}$ are 
i.i.d.\ standard real Gaussians, is a matrix with independent entries
such that the variance of its $(i,j)$ entry for $i\le j$ is the number 
of edges in $H$ between vertices $i$ and $j$. The computation of $v$ 
follows immediately.
\end{proof}

Combining Proposition \ref{prop:scperm} and Lemma \ref{lem:liftparm} 
yields an analogue of Theorem \ref{thm:lift}, in which $\varrho(\hat H)$ 
is replaced by $\|X_\mathrm{F}\|$. Moreover, Proposition 
\ref{prop:freeclt} and \eqref{eq:sprcover} readily imply that $\|X_{\rm 
F}\|\le (1+Cn^{-1}+C\mathrm{D}(H)^{-\frac{1}{4}})\varrho(\hat H)$. Thus 
the conclusion of Theorem \ref{thm:lift} holds even when $H$ has 
self-loops, but with an extra $O(n^{-1})$ error term. However, when 
$H$ has no self-loops, the special structure of the coefficients $A_e$ 
enables us to eliminate the $O(n^{-1})$ term, so that the bound can be 
sharp even when $n\not\to\infty$.

\begin{lem}
\label{lem:noloops}
If $H$ has no self-loops, then
$$
	\|X_{\rm F}\| \le
	\bigg(1+\frac{C}{\mathrm{D}(H)^{\frac{1}{4}}}\bigg)
	\varrho(\hat H)
$$
for a universal constant $C$.
\end{lem}

\begin{proof}
Let $\varepsilon$ be as in Proposition \ref{prop:scperm}. As $H$ is 
loopless, all $A_e$ are of the form
$A=e_ie_j^*$ with $i<j$. If we define $A_\varepsilon = 
(1-\varepsilon^2)^{\frac{1}{2}}A+\varepsilon A^*$, then we can compute
$$
	A_\varepsilon MA^*_\varepsilon + A^*_\varepsilon MA_\varepsilon = 
	M_{jj} e_ie_i^* + M_{ii} e_je_j^* +
	\frac{M_{ji} e_ie_j^* + M_{ij} e_je_i^*}{n-1}.
$$
The circular family $(c_e)_{e\in E_H}$ can be realized as
$c_e=l_e+\tilde l_e^*$ where $l_e,\tilde l_e$ are canonical creation 
operators on a free Fock space \cite[\S 1.2]{Leh99},
we have \cite[Theorem 1.3]{Leh99}
$$
	\|X_{\rm F}\| =
	\inf_{M>0} \bigg\|
	M^{-1} + \sum_{e\in E_H} (A_{e\varepsilon}MA_{e\varepsilon}^* +
	A_{e\varepsilon}^*MA_{e\varepsilon})\bigg\|,
$$
and moreover the infimum is attained by an $M$ so that the quantity inside 
the norm on the right-hand side is proportional to the identity. We can 
now reason precisely as in the proof of \cite[Lemma 3.2]{BBV21} that the 
infimum in the above expression can be taken over diagonal matrices only.
Therefore
$$
	\|X_{\rm F}\| =
	\inf_{x\in\mathbb{R}^d:x>0}\max_{i\in [d]}
	\bigg\{ \frac{1}{x_i} + \sum_{j\in [d]:j\sim i} x_j
	\bigg\},
$$
where $j\sim i$ denotes that there is an edge between $i,j$ in $H$.
As the latter expression does not depend on $n$, we can conclude that 
when $H$ has no self-loops, $\|X_F\|$ is unchanged if we set 
$\varepsilon=0$. Then Proposition \ref{prop:freeclt},
\eqref{eq:sprcover}, and Lemma \ref{lem:liftparm} yield
$$
	\|X_{\rm F}\| \le
	\varrho(\hat H) + C\mathrm{D}(H)^{\frac{1}{4}}.
$$
It remains to note that $\varrho(\hat H)\ge \mathrm{D}(H)^{\frac{1}{2}}$
by \eqref{eq:sprcover} and \cite[eq.\ (9.7.2)]{Pis03}.
\end{proof}

We now conclude the proof of Theorem \ref{thm:lift}.

\begin{proof}[Proof of Theorem \ref{thm:lift}]
Applying Proposition \ref{prop:scperm} with $t=(a+2)\log nd$ yields
$$
	\mathbf{P}\bigg[\|X^{(n)\perp}\| \ge
	\bigg(1
	+ C\frac{\mathrm{M}(H)^{\frac{1}{4}}}{n^{\frac{1}{4}}}
	\frac{(\log nd)^{\frac{3}{4}}}{\mathrm{D}(H)^{\frac{1}{4}}}
	+
	C\frac{(\log nd)^{\frac{2}{3}}}{\mathrm{D}(H)^{\frac{1}{6}}}
	+
	C\frac{\log nd}{\mathrm{D}(H)^{\frac{1}{2}}}
	\bigg)
	\varrho(\hat H)
	\bigg]
	\le (nd)^{-a}
$$
using Lemmas \ref{lem:liftparm} and \ref{lem:noloops}, that 
$\varrho(\hat H)\ge \mathrm{D}(H)^{\frac{1}{2}}$ as in the proof of Lemma 
\ref{lem:noloops}, and that $\mathrm{M}(H)\le\mathrm{D}(H)$. Here $C$ is a 
constant that depends on $a$ only.
\end{proof}

\subsection{Smallest singular value}
\label{sec:pfssing}

The aim of this section is to prove Theorem \ref{thm:ssing}. The proof is 
based on the following linearization lemma, which we state in a 
slightly more general form than is needed here as it will be used again in 
section \ref{sec:pfsampcov}.

\begin{lem}
\label{lem:linearize}
Let $Y$ be a $d\times m$ random matrix, and let $B\ge 0$ be a nonrandom 
$m\times m$ positive semidefinite matrix. Define the $(d+2m)\times (d+2m)$ 
random matrix
$$
	\hat Y_\varepsilon :=
	\begin{bmatrix}
	0 & Y^* & (B+4\varepsilon^2\id)^{\frac{1}{2}} \\
	Y & 0 & 0 \\
	(B+4\varepsilon^2\id)^{\frac{1}{2}} & 0 & 0
	\end{bmatrix},
$$
and let $\hat Y_{\varepsilon,\mathrm{free}}$ be its noncommutative 
model. Then
\begin{align*}
	&\spc(\hat Y_\varepsilon)\subseteq \spc(\hat 
	Y_{\varepsilon,\mathrm{free}}) + [-\varepsilon,\varepsilon]
	\qquad\Longrightarrow \\
        &\begin{cases}
        \lambda_{\rm max}(Y^*Y+B+4\varepsilon^2\id)^{\frac{1}{2}} \le
        \lambda_{\rm max}(Y_{\rm free}^*Y_{\rm free}
        +B\otimes\id + 4\varepsilon^2\id)^{\frac{1}{2}}+\varepsilon,\\
        \lambda_{\rm min}(Y^*Y+B+4\varepsilon^2\id)^{\frac{1}{2}} \ge
        \lambda_{\rm min}(Y_{\rm free}^*Y_{\rm free}
        +B\otimes\id + 4\varepsilon^2\id)^{\frac{1}{2}}-\varepsilon
        \end{cases}
\end{align*}
for any $\varepsilon\ge 0$, where
$\lambda_{\rm max}(X):=\sup\spc(X)$ and $\lambda_{\rm min}(X):=\inf\spc(X)$. 
\end{lem}

\begin{proof}
The proof is identical to that of \cite[Lemma 3.13]{BBV21}.
\end{proof}

We can now complete the proof of Theorem \ref{thm:ssing}.

\begin{proof}[Proof of Theorem \ref{thm:ssing}]
We readily compute $\sigma_*(\hat 
Y_\varepsilon)=\sigma_*(Y)$, $\sigma(\hat Y_\varepsilon)=\sigma(Y)$, and
$v(\hat Y_\varepsilon)\le \sqrt{2}\,v(Y)$ by \cite[Lemma 4.10]{BBV21},
while clearly $R(\hat Y_\varepsilon)=R(Y)$ by Remark \ref{rem:nonsa}.
Applying Theorem \ref{thm:smconc} and Lemma \ref{lem:linearize} with $B=0$ 
therefore yields
$$
	\mathbf{P}\big[
        \lambda_{\rm min}(Y^*Y+4\delta(t)^2\id)^{\frac{1}{2}} \le
        \lambda_{\rm min}(Y_{\rm free}^*Y_{\rm free}
        + 4\delta(t)^2\id)^{\frac{1}{2}}-\delta(t)
	\big] 
	\le 6de^{-t},
$$
where 
$$
	\delta(t) = C\big\{v(Y)^{\frac{1}{2}}\sigma(Y)^{\frac{1}{2}}
	(\log d)^{\frac{3}{4}}
	+ \sigma_*(Y)\, t^{\frac{1}{2}} +
        R(Y)^{\frac{1}{3}}\sigma(Y)^{\frac{2}{3}} t^{\frac{2}{3}} +
        R(Y)\,t\big\}
$$
for a universal constant $C$. Using that $\lambda_{\rm min}(Y_{\rm 
free}^*Y_{\rm free}+ 4\delta(t)^2\id)^{\frac{1}{2}} \ge 
\mathrm{s_{min}}(Y_{\rm free})$
and $\lambda_{\rm min}(Y^*Y+4\delta(t)^2\id)^{\frac{1}{2}} \le
\mathrm{s_{min}}(Y)+2\delta(t)$, we obtain
$$
	\mathbf{P}\big[
	\mathrm{s_{min}}(Y) \le
	\mathrm{s_{min}}(Y_{\rm free})
	-3\delta(t)
	\big] 
	\le 6de^{-t}.
$$
It remains to note that we can replace $6d$ by $d$ on the right-hand side 
if we increase the universal constant $C$ (as in the last
step of the proof of Theorem \ref{thm:specuniv}).
\end{proof}

\subsection{Sample covariance matrices}
\label{sec:pfsampcov}

\subsubsection{Proof of Theorem \ref{thm:sguniv}}
\label{sec:pfsguniv}

Gaussian random matrices are unbounded but possess moments of all orders. 
Therefore, a result along the lines of Theorem \ref{thm:sguniv} can be 
proved either using Theorem \ref{thm:specheavy} or using Theorem 
\ref{thm:momentuniv}. These two approaches yield similar conclusions; we 
have chosen the latter approach here as it yields a slightly cleaner bound 
in the present setting. In preparation for the proof, let us estimate the 
relevant matrix parameters of the random matrix $S$ of \eqref{eq:smtx}.

\begin{lem}
\label{lem:pfsgump}
We have
$$
	\sigma(S) =
	\bigg\| \sum_{i=1}^n \big({\tr[\Sigma_i]\Sigma_i +
	\Sigma_i^2}\big)\bigg\|^{\frac{1}{2}},\qquad
	v(S) \le
	\sqrt{2}\,\bigg\|\sum_{i=1}^n \Sigma_i^2\bigg\|^{\frac{1}{2}}.
$$
\end{lem}

\begin{proof}
The identity for $\sigma(S)$ follows readily using
$$
	\EE[(S-\EE S)^2] =
	\sum_{i=1}^n
	\EE[(Y_iY_i^*-\Sigma_i)^2]
	=
	\sum_{i=1}^n
	\big(\EE[Y_iY_i^*\|Y_i\|^2] - \Sigma_i^2\big)
$$
and that
$\EE[Y_iY_i^*\|Y_i\|^2] = \tr[\Sigma_i]\Sigma_i + 2\Sigma_i^2$
by the Wick formula \eqref{eq:wick}.

To bound $v(S)$, we reason analogously. We first note that
$$
	\EE[|{\tr[M(S-\EE S)]}|^2] =
	\sum_{i=1}^n
	\big(\EE[|\langle Y_i,MY_i\rangle|^2]
	- |{\tr[M\Sigma_i]}|^2\big) 
	\le
	2\sum_{i=1}^n
	\tr[M\Sigma_i M^*\Sigma_i],
$$
using $\EE[|\langle Y_i,MY_i\rangle|^2] \le
|{\tr[M\Sigma_i]}|^2 + 2\tr[M\Sigma_i M^*\Sigma_i]$ by \eqref{eq:wick} and 
Cauchy-Schwarz. As
$2\tr[M\Sigma_i M^*\Sigma_i] \le
2\tr[M\Sigma_i^2 M^*]^{\frac{1}{2}}
\tr[\Sigma_i M M^*\Sigma_i]^{\frac{1}{2}} \le
\tr[(M^*M+MM^*)\Sigma_i^2]$ by Cauchy-Schwarz and Young's inequality,
we have
$$
	v(S)^2 = 
	\sup_{\tr |M|^2\le 1}
	\EE[|{\tr[M(S-\EE S)]}|^2] \le
	\sup_{\tr |M|^2\le 1}
	\tr\bigg[(M^*M+MM^*)\sum_{i=1}^n\Sigma_i^2\bigg],
$$
and the conclusion follows readily.
\end{proof}

We must further estimate the parameter $R_q(S)$ in Theorem 
\ref{thm:momentuniv}.

\begin{lem}
\label{lem:pfsgumpr}
For $q\ge 1$, we have
$R_q(S) \lesssim n^{\frac{1}{q}}\max_{i\le n}\{
{\tr \Sigma_i + q\|\Sigma_i\|}\}$.
\end{lem}

\begin{proof}
It follows directly from the definition of $R_q(S)$ that
$$
	R_q(S) 
	\le n^{\frac{1}{q}} \max_{i\le n} 
	\EE[\|Z_i\|^{q}]^{\frac{1}{q}} 
	\le 2 n^{\frac{1}{q}} \max_{i\le n} 
	\EE[\|Y_i\|^{2q}]^{\frac{1}{q}} 
$$
where we used $\|Z_i\| = \|Y_iY_i^*-\EE Y_iY_i^*\| \le \|Y_i\|^2 + 
\EE\|Y_i\|^2$. It remains to note that
$$
	\EE[\|Y_i\|^{2q}]^{\frac{1}{2q}} \le
	\EE\|Y_i\| +
	\EE[(\|Y_i\|-\EE\|Y_i\|)^{2q}]^{\frac{1}{2q}} \lesssim
	(\tr \Sigma_i)^{\frac{1}{2}} +
	\|\Sigma_i\|^{\frac{1}{2}}
	\sqrt{q},
$$
where we used that $\EE\|Y_i\|\le (\tr \Sigma_i)^{\frac{1}{2}}$ by
Cauchy-Schwarz and that $\|Y_i\|$ is $\|\Sigma_i\|$-subgaussian by 
Gaussian concentration \cite[Theorems 5.6 and 2.1]{BLM13}.
\end{proof}

We can now complete the proof of Theorem \ref{thm:sguniv}.

\begin{proof}[Proof of Theorem \ref{thm:sguniv}]
Theorem \ref{thm:momentuniv} and
\cite[Theorem 2.7 and Lemma 2.5]{BBV21} yield
$$
	d^{-\frac{1}{2p}}
	\EE\|S-\EE S\| \le
	\EE[\ntr (S-\EE S)^{2p}]^{\frac{1}{2p}} \le
	2\sigma(S) +
	C v(S)^{\frac{1}{2}}\sigma(S)^{\frac{1}{2}}p^{\frac{3}{4}} +
	C R_{2p}(S) p^2 
$$
for a universal constant $C$. Now let
$p=\lceil\frac{2}{\varepsilon}\log(d+n)\rceil$, so that
$\max(d^{\frac{1}{2p}},n^{\frac{1}{2p}}) \le e^{\frac{\varepsilon}{4}}$.
Moreover,
$C v(S)^{\frac{1}{2}}\sigma(S)^{\frac{1}{2}}p^{\frac{3}{4}} \le
(e^{\frac{\varepsilon}{4}}-1)2\sigma(S)+
\varepsilon^{-1}C^2v(S)p^{\frac{3}{2}}$ 
by Young's inequality and $e^x\ge 1+x$. We therefore obtain for any
$\varepsilon\in (0,1]$
$$
	\EE\|S-\EE S\| \le
	(1+\varepsilon)\,
	2\sigma(S) +
	\frac{K}{\varepsilon^{3}}
	\bigg(v(S) +
	\max_{i\le n} \tr\Sigma_i \bigg)
	\log^3(d+n),
$$
where $K$ is a universal constant and we used Lemma \ref{lem:pfsgumpr} and
$e^{\frac{\varepsilon}{2}}\le 1+\varepsilon$
for $\varepsilon\le 1$. The conclusion follows readily using
Lemma \ref{lem:pfsgump}.
\end{proof}

\subsubsection{A simple lower bound}
\label{sec:pfsimplelower}

The aim of this short section is to show that the leading terms of the 
upper bounds of Theorems \ref{thm:sgbd} and \ref{thm:sguniv} are also 
lower bounds up to a universal constant. These results therefore capture 
the correct ``user-friendly'' quantity in the present setting. In general, 
it is not the case these these terms are optimal to leading order, that is 
up to a factor $1+o(1)$; if such a sharp bound is desired, the proofs of 
Theorems \ref{thm:sgbd} and \ref{thm:sguniv} may be adapted to obtain 
bounds in terms of $\|Y_{\rm free}Y_{\rm free}^*-\EE S\otimes \id\|$ and 
$\|S_{\rm free}-\EE S\otimes \id\|$, respectively.

\begin{lem}
In the setting of section \ref{sec:gaussscov}, we have
$$
	\EE\|S-\EE S\| \gtrsim
	\bigg\|
	\sum_{i=1}^n \tr[\Sigma_i]\Sigma_i
	\bigg\|^{\frac{1}{2}} +
	\max_{i\le n}\tr \Sigma_i.	
$$
\end{lem}

\begin{proof}
We begin by noting that
$$
	\EE\|S-\EE S\| \ge
	\sup_{\|v\|=1} \EE\|(S-\EE S)v\|
	\gtrsim
	\sup_{\|v\|=1} \EE[\|(S-\EE S)v\|^2]^{\frac{1}{2}},
$$
where the last inequality follows by hypercontractivity
\cite[Theorem 3.50]{Jan97} using that $\|(S-\EE S)v\|^2$ is a polynomial 
of degree $4$ of the Gaussian variables $Y_{ij}$. The first part of
the proof of Lemma \ref{lem:pfsgump} therefore yields
$$
	\EE\|S-\EE S\| \gtrsim
	\bigg\|
	\sum_{i=1}^n \big({\tr[\Sigma_i]\Sigma_i+\Sigma_i^2}\big)
	\bigg\|^{\frac{1}{2}}
	\ge
	\bigg\|
	\sum_{i=1}^n \tr[\Sigma_i]\Sigma_i
	\bigg\|^{\frac{1}{2}}
	\ge \max_{i\le n}\|\Sigma_i\|.
$$
On the other hand, we can readily estimate by Jensen's inequality
$$
	\EE\|S-\EE S\|
	\ge \max_{i\le n}\EE\|Y_iY_i^*-\Sigma_i\|
	\ge \max_{i\le n}\tr \Sigma_i - \max_{i\le n}\|\Sigma_i\|,
$$
where we used that $S-\EE S=\sum_{i=1}^n (Y_iY_i^*-\Sigma_i)$ is a sum of 
independent centered random matrices. We can therefore estimate
$$
	\bigg\|
	\sum_{i=1}^n \tr[\Sigma_i]\Sigma_i
	\bigg\|^{\frac{1}{2}} +
	\max_{i\le n}\tr \Sigma_i \le
	\bigg\|
	\sum_{i=1}^n \tr[\Sigma_i]\Sigma_i
	\bigg\|^{\frac{1}{2}} +
	\max_{i\le n}\|\Sigma_i\| +
	\EE\|S-\EE S\|,
$$
and the proof is readily concluded.
\end{proof}

\subsubsection{Proof of Theorem \ref{thm:scovlin}}
\label{sec:pfscovlin}

The proof of Theorem \ref{thm:scovlin} combines our universality 
principles with a linearization argument as in Lemma \ref{lem:linearize}.

\begin{proof}[Proof of Theorem \ref{thm:scovlin}]
Given $B=\|\EE YY^*\|\id - \EE YY^*$, define 
$$
	\breve Y_\varepsilon :=
	\begin{bmatrix}
	0 & Y & (B+4\varepsilon^2\id)^{\frac{1}{2}} \\
	Y^* & 0 & 0 \\
	(B+4\varepsilon^2\id)^{\frac{1}{2}} & 0 & 0
	\end{bmatrix},
$$
and let $\breve H_\varepsilon$ be its Gaussian model.
Then a completely 
analogous argument to the one used in the proof of Lemma 
\ref{lem:linearize} yields
\begin{align*}
	& \dH(\spc(\breve Y_\varepsilon),\spc(\breve H_\varepsilon)) \le
	\varepsilon
	\qquad\Longrightarrow \\
        &\begin{cases}
        \big|\lambda_{\rm max}(YY^*+B+4\varepsilon^2\id)^{\frac{1}{2}} -
        \lambda_{\rm max}(HH^*
        +B + 4\varepsilon^2\id)^{\frac{1}{2}}\big| \le\varepsilon,\\
        \big|\lambda_{\rm min}(YY^*+B+4\varepsilon^2\id)^{\frac{1}{2}} -
        \lambda_{\rm min}(HH^*
        +B + 4\varepsilon^2\id)^{\frac{1}{2}}\big| \le \varepsilon
        \end{cases}
\end{align*}
for any $\varepsilon\ge 0$. Using $|a^{\frac{1}{2}}-b^{\frac{1}{2}}|
(a^{\frac{1}{2}}+b^{\frac{1}{2}}) = |a-b|$ for $a,b\ge 0$, we obtain
\begin{align*}
	& \dH(\spc(\breve Y_\varepsilon),\spc(\breve H_\varepsilon)) \le
	\varepsilon
	\qquad\Longrightarrow \\
        &
        \big| \|YY^*-\EE YY^*\| -
        \|HH^*-\EE HH^*\|\big|
	\le
	\big(\|Y\|+\|H\|+2\|\EE YY^*\|^{\frac{1}{2}}+4\varepsilon\big)
	\varepsilon,
\end{align*}
where we used that $|\|M\|-\|N\|| \le 
|\lambda_{\rm max}(M)-\lambda_{\rm max}(N)|\vee
|\lambda_{\rm min}(M)-\lambda_{\rm min}(N)|$
and $\EE HH^*=\EE YY^*$. Furthermore, we have
$\sigma_*(\breve{Y}_\varepsilon)=\sigma_*(Y)$, 
$\sigma(\breve{Y}_\varepsilon)=\sigma(Y)$, and
$R(\breve{Y}_\varepsilon)=R(Y)$
by \cite[Remark 2.6]{BBV21}.
Thus Theorem \ref{thm:specuniv} yields
\begin{multline*}
	\mathbf{P}\big[
        \big|\|YY^*-\EE YY^*\| -
        \|HH^*-\EE HH^*\|\big|
        >\\
        C\big(\|Y\|+\|H\|+\|\EE YY^*\|^{\frac{1}{2}}+\varepsilon(t)\big)
	\varepsilon(t)
	\big]
	\le (2d+m) e^{-t}
\end{multline*}
for all $t>0$, where $C$ is a universal constant and
$\varepsilon(t)$ is as in Theorem \ref{thm:specuniv}.

To proceed, note that
$$
	\mathbf{P}[\|H\|>\EE\|H\|+ C\varepsilon(t)]\le e^{-t}
$$
by Gaussian concentration as in \cite[Corollary 4.14]{BBV21}, while
$$
	\mathbf{P}[\|Y\|> \EE\|H\| + C\varepsilon(t)]
	\le (d+m)e^{-t}
$$
by Corollary \ref{cor:normuniv} and Remark \ref{rem:nonsa}.
Combining the above bounds yields
$$
	\mathbf{P}\big[
        \big|\|YY^*-\EE YY^*\| -
        \|HH^*-\EE HH^*\|\big|
        >
	C\varepsilon(t)\,\EE\|H\|+C\varepsilon(t)^2
	\big]
	\le C(d+m) e^{-t}
$$
for a universal constant $C$, where we used that 
$$
	\|\EE YY^*\| =
	\|\EE HH^*\| =
	\sup_{\|v\|=1} \EE \|H^*v\|^2 \lesssim
	\sup_{\|v\|=1} (\EE\|H^*v\|)^2 \le
	(\EE\|H\|)^2
$$
by hypercontractivity \cite[Theorem 3.50]{Jan97}.
The conclusion follows by
integrating this tail bound as in the proof of Theorem
\ref{thm:specuniv}.
\end{proof}

\subsubsection{Proof of Theorem \ref{thm:versh}}
\label{sec:pfversh}

Throughout this section we adopt the setting and notation of Theorem 
\ref{thm:versh}. Its proof combines two distinct universality principles. 
Let us begin by applying universality to $Y$.

\begin{prop}
\label{prop:versh1}
We have 
\begin{multline*}
	\EE\|S-\EE S\| \le
	2\|B\|_{\rm HS}\|B\|\sqrt{n} + \|B\|_{\rm HS}^2
	\\ +
	C\big\{\alpha^{\frac{1}{3}}\|B\|^{\frac{1}{3}} 
	(\|B\|\sqrt{n} + \|B\|_{\rm HS})^{\frac{5}{3}}
	+
	\alpha^2\|B\|^2\big\} \log^2(d+n).
\end{multline*}
\end{prop}

\begin{proof}
We will write $Y$ in the form \eqref{eq:model} as
$$
	Y=\sum_{i=1}^N\sum_{j=1}^n Z_{ij},\qquad
	Z_{ij} = A_{ij}\,Be_ie_j^*.
$$
We readily compute
$$
	\EE[YY^*] = n\,BB^*,\qquad
	\EE[Y^*Y] = \|B\|_{\rm HS}^2\id,\qquad
	v(Y) = \|B\|,\qquad
	R(Y) \le \alpha\|B\|
$$
(here we used that $Y$ has independent columns $Y_1,\ldots,Y_n$ with
$\mathrm{Cov}(Y_i)=BB^*$, so that 
$\|\Cov(Y)\|=\max_i\|\Cov(Y_i)\|=\|B\|^2$). 
Applying Theorem \ref{thm:scovlin} yields
$$
	\big|\EE\|S-\EE S\| -
	\EE\|HH^*-\EE HH^*\|\big| \lesssim
	\delta\,(\|B\|\sqrt{n} + \|B\|_{\rm HS}) + \delta^2,
$$
where
$$
	\delta \lesssim 
	\alpha^{\frac{1}{3}}\|B\|^{\frac{1}{3}}
	(\|B\|\sqrt{n} + \|B\|_{\rm HS})^{\frac{2}{3}}
	\log^{\frac{2}{3}}(d+n) +
	\alpha\|B\|\log(d+n)
$$
and we used that $\EE\|H\| \lesssim 
\|B\|\sqrt{n} + \|B\|_{\rm HS}$ (see, e.g., \cite[Lemma 5.4]{vH17}) and 
$\alpha\ge 1$. On the other hand, $\EE\|HH^*-\EE HH^*\|$ can be estimated 
by Theorem \ref{thm:sgbd} with $\Sigma_i = BB^*$. Combining all the above 
bounds yields the conclusion. 
\end{proof}

We now apply universality to $S$. In preparation for the following 
computations, we begin by estimating $\sigma(S)$ and $v(S)$. Recall that
$Y_i$ denotes the $i$th column of $Y$.

\begin{lem}
\label{lem:pfsgvng}
We have
$$
	\sigma(S) \le
	(\|B\|_{\rm HS}\|B\|+
	2\alpha\|B\|^2)\sqrt{n},
	\qquad
	v(S) \le 4\alpha\|B\|^2\sqrt{n}.
$$
\end{lem}

\begin{proof}
By Lemma \ref{lem:shir}, if
$W_1,\ldots,W_m$ are independent centered random variables with unit 
variance and $g_1,\ldots,g_m$ are independent standard Gaussians, then
$$
	\EE[W_iW_jW_kW_l] = \EE[g_ig_jg_kg_l] + 
	(\EE[W_i^4]-3)
	1_{i=j=k=l}.
$$
We will apply this identity in the case that $W_i$ are entries of the 
random matrix $A$. In particular, arguing as in the first part of the 
proof of Lemma \ref{lem:pfsgump}, we obtain
$$
	\EE[(S-\EE S)^2] =
	n \|B\|_{\rm HS}^2 BB^* + n (BB^*)^2
	+ 
	\sum_{i=1}^N
	\sum_{j=1}^n (\EE[A_{ij}^4]-3)
	(B^*B)_{ii} Be_ie_i^*B^*.
$$
As $\EE[A_{ij}^4] \le \alpha^2$, we readily obtain
$$
	\sigma(S)^2 \le
	n \|B\|_{\rm HS}^2 \|B\|^2 + (1+\alpha^2) n \|B\|^4,
$$
and the bound on $\sigma(S)$ follows using $\alpha\ge 1$.

The parameter $v(S)$ can be estimated analogously, but an adequate bound
also follows from a standard concentration argument. Indeed, note that
$$
	v(S)^2 =
	\sup_{\tr |M|^2=1} \mathrm{Var}({\tr MS})
	=
	\sup_{\tr |M|^2=1}\sum_{i=1}^n\mathrm{Var}(\langle Y_i,MY_i\rangle).
$$
By the convex Poincar\'e inequality \cite[Theorem 3.17]{BLM13} and 
Cauchy-Schwarz, we obtain $\mathrm{Var}(\langle Y_i,MY_i\rangle) \le 
16\alpha^2 \|B\|^4$ for $\tr |M|^2=1$, concluding the proof.
\end{proof}

Next, we estimate the parameter $R_q(S)$ in Theorem \ref{thm:momentuniv}.

\begin{lem}
\label{lem:pfscc}
For $q\ge 1$, we have
$R_q(S) \lesssim n^{\frac{1}{q}} 
\{\|B\|_{\rm HS}^2 + \alpha^2 q\|B\|^2\}$.
\end{lem}

\begin{proof}
The convex concentration inequality \cite[Theorems 6.10 and 2.1]{BLM13} 
yields that $\EE[\|Y_i\|^{2q}]^{\frac{1}{2q}}\le \|B\|_{\rm HS} + 
C\alpha\sqrt{q}\|B\|$, where we used $\EE\|Y_i\|\le 
\EE[\|Y_i\|^2]^{\frac{1}{2}} = \|B\|_{\rm HS}$. The conclusion follows
directly as in the proof of Lemma \ref{lem:pfsgumpr}.
\end{proof}

We can now proceed as in the proof of Theorem \ref{thm:sguniv}.

\begin{prop}
\label{prop:versh2}
We have for $\varepsilon\in (0,1]$
$$
	\EE\|S-\EE S\| \le 
	(1+\varepsilon)\,2\|B\|_{\rm HS}\|B\|\sqrt{n} 
	+\frac{C}{\varepsilon^3}
	\big(\|B\|_{\rm HS}^2 + (\alpha\sqrt{n}+\alpha^2) \|B\|^2\big)
	\log^3(d+n).
$$
\end{prop}

\begin{proof}
Apply Lemmas \ref{lem:pfsgvng} and
\ref{lem:pfscc} precisely as in the proof of Theorem \ref{thm:sguniv}.
\end{proof}

We now complete the proof of Theorem \ref{thm:versh}.

\begin{proof}[Proof of Theorem \ref{thm:versh}]
It is convenient to define $\gamma=\frac{\|B\|_{\rm HS}}{\|B\|\sqrt{n}}$ 
and $\delta = \frac{\alpha}{\sqrt{n}}$. In terms of these 
dimensionless parameters, Proposition \ref{prop:versh1} can be expressed 
as
$$
	\frac{\EE\|S-\EE S\|}{n\|B\|^2} \le
	2\gamma + \gamma^2
	+
	C\big\{ 
	\delta^{\frac{1}{3}}
	(1 + \gamma)^{\frac{5}{3}}
	+
	\delta^2\big\} \log^2(d+n)
$$
while Proposition \ref{prop:versh2} yields
$$
	\frac{\EE\|S-\EE S\|}{n\|B\|^2} \le
	2\gamma 
	+C\big\{
	(\gamma^2 + \delta+\delta^2)^{\frac{1}{4}}\gamma^{\frac{3}{4}}
	+ \gamma^2 + \delta+\delta^2\big\}
	\log^3(d+n),
$$
where in the last equation we used
$\inf_{\varepsilon\le 1}\{2\varepsilon\gamma + \frac{K}{\varepsilon^3}\}
\le 3K^{\frac{1}{4}}\gamma^{\frac{3}{4}} + 3K$.

Now note that the assumptions of the theorem imply $\delta\le 1$ and
$\gamma\ge\delta$. Using $\delta^2\le 
\delta^{\frac{1}{3}}(1+\gamma)^{\frac{5}{3}}$,
$\frac{\delta^{\frac{1}{3}}(1 + 
\gamma)^{\frac{5}{3}}}{2\gamma + \gamma^2}\lesssim
\frac{\delta^{\frac{1}{3}}}{\gamma} + 
(\frac{\delta}{\gamma})^{\frac{1}{3}}$ and $\frac{\delta}{\gamma}\le 1$,
we can rearrange and combine the above inequalities to estimate
$$
	\frac{\EE\|S-\EE S\|}{n\|B\|^2} \le
	\bigg(
	1+
	C
	\bigg\{
	\min\bigg(
	\frac{\delta^{\frac{1}{3}}}{\gamma},
	\gamma^{\frac{1}{4}} + \gamma\bigg)
	+
	\frac{\delta^{\frac{1}{4}}}{\gamma^{\frac{1}{4}}}
	\bigg\}
	\log^3(d+n)
	\bigg)
	(2\gamma + \gamma^2).
$$
We conclude with
$\min(\frac{a}{\gamma},\gamma^{\frac{1}{4}}+\gamma) \le
\min(\frac{a}{\gamma},\gamma^{\frac{1}{4}})+
\min(\frac{a}{\gamma},\gamma) \le 
a^{\frac{1}{5}}+a^{\frac{1}{2}}$
and $\delta<1$.
\end{proof}

\begin{rem}[Unbounded entries]
\label{rem:fourmomsc}
The formulation of Theorem \ref{thm:versh} for bounded $A_{ij}$ is not a 
fundamental restriction of our approach: results for 
unbounded entries can be obtained analogously by using our universality 
principles for unbounded random matrices. We have restricted to the 
bounded case largely for simplicity and brevity of exposition. However, in 
order to illustrate some features of the unbounded case, 
let us briefly discuss these here in the context of the slightly simpler 
problem of estimating $\|Y\|$ (as opposed to $\|YY^*-\EE YY^*\|$).

Let $Y=BA$ with $A,B$ as in Theorem \ref{thm:versh}, except that we now 
assume only that $\|A_{ij}\|_s \le \alpha$ for some $4<s<\infty$.
We write $Y$ in the form \eqref{eq:model} as in the proof of Proposition 
\ref{prop:versh1}. Applying Theorem \ref{thm:specheavy} as in the 
proof of Corollary \ref{cor:userfriendly} yields
$$
	\EE\|Y\| \le
	\|\EE YY^*\|^{\frac{1}{2}} +
	\|\EE Y^*Y\|^{\frac{1}{2}} +
	C\big\{
	v(Y)^{\frac{1}{2}}\sigma(Y)^{\frac{1}{2}}
	+
	\bar R(Y)^{\frac{1}{6}}\sigma(Y)^{\frac{5}{6}}\big\} \log(d+n)
$$
provided that $\bar R(Y) \log^3(d+n) \le \sigma(Y)$. All parameters in 
this bound were already computed in the proof of Proposition
\ref{prop:versh1} except $\bar R(Y)$, which we estimate as
\begin{align*}
	\bar R(Y) &=
	\EE\bigg[\max_{i\le N}\max_{j\le n} A_{ij}^2 
	\|Be_i\|^2\bigg]^{\frac{1}{2}} 
	\le
	\EE\bigg[\sum_{i=1}^N\sum_{j=1}^n
	|A_{ij}|^s 
	\|Be_i\|^s\bigg]^{\frac{1}{s}} 
	\\
	&\le
	\alpha n^{\frac{1}{s}}
	\bigg(
	\sum_{i=1}^N
	\|Be_i\|^s\bigg)^{\frac{1}{s}} 
	\le
	\alpha n^{\frac{1}{s}} \|B\|^{1-\frac{2}{s}}
	\|B\|_{\rm HS}^{\frac{2}{s}}.
\end{align*}
Combining the above estimates, we obtain the bound
$$
	\EE\|Y\| \le
	\bigg(
	1+
	\frac{C\log(d+n)}{(n\vee r)^{\frac{1}{4}}} +
	\frac{C\alpha^{\frac{1}{6}}\log(d+n)
	}{(n \vee r)^{\frac{1}{12}-\frac{1}{3s}}}
	\bigg)
	\big(\|B\|\sqrt{n}+\|B\|_{\rm HS}\big)
$$
for $\alpha\log^3(d+n) \le (n \vee r)^{\frac{1}{2}-\frac{2}{s}}$,
where $r=\|B\|_{\rm HS}^2\|B\|^{-2}$ is the effective rank of $B$
and we used that $\|B\|\sqrt{n}+\|B\|_{\rm HS} \ge
(\|B\|\sqrt{n})^{1-t} \|B\|_{\rm HS}^t$ for any $t\in[0,1]$.

This result should be compared to the best previous bound in this 
setting due to Vershynin \cite{Ver11}, which states that $\EE\|Y\| \le 
C(s)\alpha (\|B\|\sqrt{n}+\|B\|_{\rm HS})$ where
$C(s)<\infty$ for $s>4$. In contrast, our bound yields $\EE\|Y\|\le 
(1+o(1))(\|B\|\sqrt{n}+\|B\|_{\rm HS})$ 
as soon as
$n\vee r \gg (\alpha^{\frac{1}{6}}\log(d+n))^{\beta(s)}$, which not only 
yields the best possible constant for a bound of this kind (in view of the 
Bai-Yin law) but also allows $\alpha$ to diverge without affecting the 
estimate to leading order. We emphasize a key feature of both bounds
that was highlighted in \cite{Ver11}: they do not depend on the inner 
dimension $N$, despite that $\bar R(Y)$ is defined as a maximum over
$nN$ random variables.
\end{rem}

\subsection{Strong asymptotic freeness}
\label{sec:pfsaf}

The main aim of this section is to prove Theorem \ref{thm:saf}. In section 
\ref{sec:pfsaflin}, we first develop the special case of bounded random 
matrices by means of a linearization argument. The result is then extended 
to the general setting in section \ref{sec:pfsaftrunc} by employing a 
truncation argument as in section~\ref{sec:trunc}. Finally, Corollary 
\ref{cor:sparsesaf} will be proved in section \ref{sec:pfcorsparse}.

\subsubsection{Linearization}
\label{sec:pfsaflin}

The aim of this section is to prove the following special case of Theorem 
\ref{thm:saf} for bounded random matrices. The general case will be 
deduced from this result in the next section by a truncation argument.

\begin{thm}
\label{thm:bddsaf}
Let $s_1,\ldots,s_m$ be a free semicircular family, and let
$H_1^N,\ldots,H_m^N$ be independent self-adjoint random matrices 
as in Theorem \ref{thm:saf}. Suppose 
$$
	\lim_{N\to\infty}\|\EE[H_k^N]\|=
	\lim_{N\to\infty}\|\EE[(H_k^N)^2]-\id\|= 0
$$
and 
$$
	\lim_{N\to\infty}
	(\log d_N)^{\frac{3}{2}} 
	v(H_k^N) 
	=
	\lim_{N\to\infty} (\log d_N)^{2} R(H_k^N)
	= 0
$$
for every $1\le k\le m$. Then
\begin{align*}
	\lim_{N\to\infty}\ntr p(H_1^N,\ldots,H_m^N) =
	\tau(p(s_1,\ldots,s_m))\quad\mbox{a.s.},\\
	\lim_{N\to\infty}\|p(H_1^N,\ldots,H_m^N)\| =
	\|p(s_1,\ldots,s_m)\|\quad\mbox{a.s.}
\end{align*}
for every noncommutative polynomial $p$.
\end{thm}

The difficulty here is that we are interested in general noncommutative 
polynomials of random matrices, while our universality principles apply 
only to the linear situation of \eqref{eq:model}. To reduce the former to 
the latter, we will use classical linearization arguments that we 
presently recall.

\begin{prop}[Linearization]
\label{prop:masterlin}
Let $H_1^N,\ldots,H_m^N$ be self-adjoint random matrices and
let $s_1,\ldots,s_m$ be a free semicircular family.
\begin{enumerate}[a.]
\itemsep\medskipamount
\item
Suppose that for every $d'\in\mathbb{N}$ and
$A_0,\ldots,A_m\in\M_{d'}(\mathbb{C})_{\rm sa}$
$$
        \spc\big(A_0\otimes\id + \textstyle{\sum_{k=1}^m}A_k\otimes H_k^N
	\big)
        \subseteq
        \spc\big(A_0\otimes\id + \textstyle{\sum_{k=1}^m}A_k\otimes s_k
	\big)
        + [-\varepsilon,\varepsilon]
$$
eventually as $N\to\infty$ a.s.\ for all $\varepsilon>0$. Then
$$
        \limsup_{N\to\infty}\|p(H_1^N,\ldots,H_m^N)\|
        \le \|p(s_1,\ldots,s_m)\|\quad\mbox{a.s.}
$$
for every noncommutative polynomial $p$.
\item
Suppose that for every $d'\in\mathbb{N}$ and
$A_0,\ldots,A_m\in\M_{d'}(\mathbb{C})_{\rm sa}$
$$
	\lim_{N\to\infty}
	\ntr\big[
	\big(A_0\otimes\id + \textstyle{\sum_{k=1}^m}A_k\otimes H_k^N\big)^{2r}
	\big]
	=
	({\ntr}\otimes\tau)\big[\big(A_0\otimes\id + 
	\textstyle{\sum_{k=1}^m}A_k\otimes s_k\big)^{2r}\big]
$$
a.s.\ for all $r\in\mathbb{N}$. Then
$$
	\lim_{N\to\infty} \ntr p(H_1^N,\ldots,H_k^N) =
	\tau(p(s_1,\ldots,s_k))\quad\mbox{a.s.}
$$
for every noncommutative polynomial $p$.
\end{enumerate}
\end{prop}

\begin{proof}
The first part is proved in \cite[Lemma 1 and pp.\ 758–760]{HT05}.
The second part follows directly from the proof of \cite[Lemma 1.1]{DlS10}.
\end{proof}

To apply the linearization argument in the present setting, we use 
the following.

\begin{lem}
\label{lem:pfexploitlin}
Let $H_1^N,\ldots,H_m^N$ be random matrices as in Theorem 
\ref{thm:bddsaf}, and fix $d'\in\mathbb{N}$ and 
$A_0,\ldots,A_m\in\M_{d'}(\mathbb{C})_{\rm sa}$.
Define the random matrix
$$
	\Xi^N = A_0\otimes\id + \sum_{k=1}^m A_k\otimes H_k^N,
$$
and let $\Xi^N_{\rm free}$ be the associated noncommutative model.
Then
$$
	\spc(\Xi^N)\subseteq\spc(\Xi^N_{\rm free})+[-\varepsilon,
		\varepsilon]
$$
eventually as $N\to\infty$ a.s.\ for every $\varepsilon>0$, and
for every $r\in\mathbb{N}$
$$
	\lim_{N\to\infty}
	\big|{\ntr[(\Xi^N)^{2r}]^{\frac{1}{2r}} -
	({\ntr}\otimes\tau)[(\Xi^N_{\rm free})^{2r}]^{\frac{1}{2r}}}
	\big|=0\quad\mbox{a.s.}
$$
\end{lem}

\begin{proof}
It follows as in the proof of \cite[Lemma 7.8]{BBV21} that
$\sigma(\Xi^N)=O(1)$ and $\sigma_*(\Xi^N)\le v(\Xi^N)=o((\log 
d_N)^{-\frac{3}{2}})$.
Moreover, it is clear from the definition of $\Xi^N$ that
$$
	R(\Xi^N) = \max_{k\le m} \|A_k\| R(H_k^N) = o((\log d_N)^{-2}).
$$
Applying Theorem \ref{thm:smconc} with $t=3\log d_N$ yields
$$
	\mathbf{P}\big[\spc(\Xi^N)\subseteq\spc(\Xi^N_{\rm free})
	+ o(1)[-1,1]\big] \ge 1-\frac{2d'}{d_N^2} \ge 1 - \frac{2d'}{N^2}.
$$
The first conclusion follows by the Borel-Cantelli lemma.

On the other hand, Theorem \ref{thm:momentuniv} and
\cite[Theorem 2.7]{BBV21} yield for all $r\in\mathbb{N}$
$$
	\lim_{N\to\infty}
	\big|
	\EE[ \ntr (\Xi^N)^{2r}]^{\frac{1}{2r}} -
	({\ntr}\otimes\tau)[(\Xi^N_{\rm free})^{2r}]^{\frac{1}{2r}}
	\big| = 0.
$$
In particular, $\EE[ \ntr 
(\Xi^N)^{2r}]^{\frac{1}{2r}}=O(1)$ as
$({\ntr}\otimes\tau)[(\Xi^N_{\rm free})^{2r}]^{\frac{1}{2r}} \le
\|\Xi^N_{\rm free}\|\le 2\sigma(\Xi^N)$.
To conclude, we need to show that
$\ntr[(\Xi^N)^{2r}]^{\frac{1}{2r}}$ concentrates around
$\EE[ \ntr (\Xi^N)^{2r}]^{\frac{1}{2r}}$ a.s. To this end, we apply
Lemma \ref{lem:pfsafconc} below with
$t=2\log d_N$ to estimate
$$
	\mathbf{P}\big[
	\big|{\ntr[(\Xi^N)^{2r}]^{\frac{1}{2r}} -
	\EE[ \ntr (\Xi^N)^{2r}]^{\frac{1}{2r}}}\big|
	\ge
	o(1)\big] \le \frac{1}{d_N^2} \le \frac{1}{N^2},
$$
and the conclusion follows by Borel-Cantelli.
\end{proof}

Above we used the following concentration inequality.

\begin{lem}
\label{lem:pfsafconc}
For any random matrix as in \eqref{eq:model}, we have
$$
	\mathbf{P}\big[
	\big|(\ntr X^{2r})^{\frac{1}{2r}} -
	\EE[\ntr X^{2r}]^{\frac{1}{2r}}\big|
	\ge
	C\big(\sigma_*(X)+R(X)^{\frac{1}{2}}\EE[\ntr 
	X^{2r}]^{\frac{1}{4r}}\big)\sqrt{t} +
	CR(X) t
	\big] \le 2e^{-t}
$$
for all $t\ge r$.
\end{lem}

\begin{proof}
We begin by writing
$(\ntr X^{2r})^{\frac{1}{2r}} =
\sup_{f\in\mathcal{F}}\big|\sum_{i=0}^n f(Z_i)\big|$
where
$$
	\mathcal{F}=\{ Z\mapsto x\mathop{\mathrm{Re}}\ntr[MZ] + y 
	\mathop{\mathrm{Im}}\ntr[MZ] :x^2+y^2\le 1, 
	\|M\|_{\frac{2r}{2r-1}}\le 1\}.
$$
Then
$$
	\sup_{f\in\mathcal{F}}\sum_{i=0}^n \mathrm{Var}(f(Z_i)) \le
	\sup_{\|M\|_1\le 1}
	\sum_{i=1}^n \EE |{\ntr MZ_i}|^2 =
	\sigma_*(X)^2
$$
as $\|M\|_1\le \|M\|_{\frac{2r}{2r-1}}$ and as the extreme points of 
$S_1^d$ are rank one matrices, and
$$
	\sup_{f\in\mathcal{F}} \max_{0\le i\le n} 
	\|f(Z_i)-f(Z_i')\|_\infty \le 2R(X)
$$
where $Z_i'$ is an independent copy of $Z_i$.
We now apply\footnote{%
While the statement of \cite[Theorem 3]{Mas00} assumes that 
$\|f\|_\infty\le b$ for all $f\in\mathcal{F}$, only the weaker
assumption $\|f(\xi_i)-f(\xi_i')\|_\infty\le 2b$ is used in the proof.
We also optimized the conclusion over $\varepsilon$.}
\cite[Theorem 3]{Mas00} to estimate
\begin{multline*}
	\mathbf{P}\big[
	\big|(\ntr X^{2r})^{\frac{1}{2r}} -
	\EE[(\ntr X^{2r})^{\frac{1}{2r}}]\big|
	\ge \\
	C\big(\sigma_*(X)+R(X)^{\frac{1}{2}}\EE[(\ntr 
	X^{2r})^{\frac{1}{2r}}]^{\frac{1}{2}}\big)\sqrt{t} +
	CR(X) t
	\big] \le 2e^{-t}
\end{multline*}
for all $t\ge 0$. Consequently
$$
	\big|\EE[\ntr X^{2r}]^{\frac{1}{2r}}-
	\EE[(\ntr X^{2r})^{\frac{1}{2r}}]\big| \lesssim
	\big(\sigma_*(X)+R(X)^{\frac{1}{2}}\EE[(\ntr 
	X^{2r})^{\frac{1}{2r}}]^{\frac{1}{2}}\big)\sqrt{r} +
	R(X) r	
$$
by \cite[Theorem 2.3]{BLM13}, and we conclude by combining the 
above inequalities.
\end{proof}

We can now conclude the proof of Theorem \ref{thm:bddsaf}.

\begin{proof}[Proof of Theorem \ref{thm:bddsaf}]
Let $\Xi^N$ and $\Xi^N_{\rm free}$ be as in Lemma \ref{lem:pfexploitlin}. 
It follows from the proofs of \cite[Lemmas 7.9 and 7.10]{BBV21} that
$$
	\spc(\Xi^N_{\rm free})
	\subseteq\spc(\textstyle{A_0\otimes\id + 
	\sum_{k=1}^m A_k\otimes s_k})+[-\varepsilon,
                \varepsilon]
$$
eventually as $N\to\infty$ for every $\varepsilon>0$, and that
$$
	\lim_{N\to\infty}
	({\ntr}\otimes\tau)[(\Xi^N_{\rm free})^{2r}]
	=
	({\ntr}\otimes\tau)[(\textstyle{A_0\otimes\id +
        \sum_{k=1}^m A_k\otimes s_k})^{2r}]
$$
for all $r\in\mathbb{N}$. Thus Lemma \ref{lem:pfexploitlin} and
Proposition \ref{prop:masterlin} yield
\begin{align*}
	\lim_{N\to\infty}\ntr p(H_1^N,\ldots,H_m^N) =
	\tau(p(s_1,\ldots,s_m))\quad\mbox{a.s.},\\
	\limsup_{N\to\infty}\|p(H_1^N,\ldots,H_m^N)\| \le
	\|p(s_1,\ldots,s_m)\|\quad\mbox{a.s.}
\end{align*}
for every noncommutative polynomial $p$. To conclude, note that
\begin{align*}
	\liminf_{N\to\infty}\|p(H_1^N,\ldots,H_m^N)\|
	&\ge
	\liminf_{N\to\infty} (\ntr 
	|p(H_1^N,\ldots,H_m^N)|^{2r})^{\frac{1}{2r}} \\ &=
	\tau(|p(s_1,\ldots,s_m)|^{2r})^{\frac{1}{2r}}\quad\mbox{a.s.}
\end{align*}
for any $r\in\mathbb{N}$, where we used 
that $|p|^{2r}$ is also a polynomial. As
$$
	\lim_{r\to\infty} \tau(|p(s_1,\ldots,s_m)|^{2r})^{\frac{1}{2r}} =
	\|p(s_1,\ldots,s_m)\|
$$
(here we use that $\tau$ is faithful), the conclusion follows.
\end{proof}

\subsubsection{Proof of Theorem \ref{thm:saf}}
\label{sec:pfsaftrunc}

To prove Theorem \ref{thm:saf} in the general setting, we will combine 
Theorem \ref{thm:bddsaf} with the truncation arguments of section 
\ref{sec:trunc}. Before we proceed to the proof, we state an 
elementary lemma that will be needed below.

\begin{lem}
\label{lem:aszero}
Let $(Y_n)_{n\ge 1}$ be a sequence of real-valued random variables such 
that $|Y_n|\to 0$ a.s.\ as $n\to\infty$. Then there is a nonrandom 
sequence $(a_n)_{n\ge 1}$ with $a_n\to 0$ as $n\to\infty$, such that 
$|Y_n|\le a_n$ eventually as $n\to\infty$ a.s.
\end{lem}

\begin{proof}
Let $Y_n^* := \sup_{m\ge n}|Y_m|$, and let
$n_k := \inf\{n:\mathbf{P}[Y_n^*>2^{-k}] \le 2^{-k}]\}$.
Then clearly $n_k$ is nondecreasing, and
$n_k<\infty$ as we assumed $|Y_n|\to 0$ a.s.
Moreover, we may assume without loss of 
generality that $n_k\to\infty$, as otherwise $Y_n^*=0$ a.s.\ for some $n$ 
and the conclusion is trivial.
We may therefore define 
$(a_n)_{n\ge 1}$ by setting $a_n = 2^{-k}$ for $n_k\le 
n<n_{k+1}$, $k\ge 0$. As by construction
$$
	\mathbf{P}[|Y_n|>a_n\mbox{ for some }n_k\le n<n_{k+1}]
	\le
	\mathbf{P}[Y_{n_k}^*>2^{-k}]\le 2^{-k},
$$
the conclusion follows by the Borel-Cantelli lemma.
\end{proof}

We can now complete the proof of Theorem \ref{thm:saf}.

\begin{proof}[Proof of Theorem \ref{thm:saf}]
We first note that it suffices to prove the a.s.\ version of the theorem,
as the in probability version follows immediately from the a.s.\ version 
using the classical fact that a sequence of random variables converges in 
probability if and only if every subsequence has an a.s.\ convergent 
subsequence.

We therefore assume from now on that the assumptions of the 
theorem hold in the a.s.\ sense.
By Lemma \ref{lem:aszero}, the assumptions imply that there exists a 
nonrandom sequence $(a_N)$ with $a_N\to 0$ as $N\to\infty$ such that
$$
	\max_{1\le k\le m}\max_{1\le i\le M_N} \|Z_{ki}^N\|
	\le (\log d_N)^{-2}a_N
$$
eventually as $N\to\infty$ a.s. Now define the truncated random matrices
$$
	\tilde H_k^N := Z_{k0}^N +
	\sum_{i=1}^{M_N}
	1_{\|Z_{ki}^N\|\le (\log d_N)^{-2}a_N}Z_{ki}^N
$$
as in section \ref{sec:trunc}. 
Then $\tilde H_k^N=H_k^N$ eventually as $N\to\infty$ a.s.\ for all $k$. To 
complete the proof, it therefore suffices to show that 
$\tilde H_1^N,\ldots,\tilde H_m^N$ satisfy the 
assumptions of Theorem \ref{thm:bddsaf}.
To this end, note first that by Lemma \ref{lem:truncparm}
$$
	(\log d_N)^2
	R(\tilde H_k^N) \le
	2a_N \xrightarrow{N\to\infty}0.
$$
Using $\EE[|1_AY-\EE[1_AY]|^2] \le 
\EE[|Y|^2]$ as in the proof of Lemma \ref{lem:truncparm},
we also have
$$
	(\log d_N)^{\frac{3}{2}}v(\tilde H_k^N)
	\le
	(\log d_N)^{\frac{3}{2}}v(H_k^N)
	\xrightarrow{N\to\infty}0.
$$
It remains to estimate $\EE[\tilde H_k^N]$ and
$\EE[(\tilde H_k^N)^2]$.

To bound $\EE[\tilde H_k^N]$, note that it was shown in the proof
of Lemma \ref{lem:truncerrmn} that
$$
	\|\EE[\tilde H_k^N]-\EE[H_k^N]\| \le
	\frac{\sigma_*(H_k^N)}{
	\mathbf{P}[\max_j\|Z_{kj}^N\|\le (\log d_N)^{-2}a_N]^{\frac{1}{2}}}.
$$
But as $\max_j\|Z_{kj}^N\|\le (\log d_N)^{-2}a_N$ eventually a.s., the 
denominator on the right-hand side converges to one and the numerator 
satisfies $\sigma_*(H_k^N)\le v(H_k^N)\to 0$. As by assumption 
$\|\EE[H_k^N]\|\to 0$, we conclude that $\|\EE[\tilde H_k^N]\|\to 0$ as 
well.

To bound $\EE[(\tilde H_k^N)^2]$, note that setting $M=\id$ in the proof 
of Lemma \ref{lem:truncerrvar} yields
\begin{multline*}
	\|\EE[(\tilde H_k^N-\EE \tilde H_k^N)^2] -
	\EE[(H_k^N-\EE H_k^N)^2]\|
	\\ \le
	2\,\PP[{\textstyle\max_j\|Z_{kj}^N\|>(\log d_N)^{-2}a_N}]\, 
	\sigma(H_k^N)^2 +
	4\,\bar R(H_k^N)\sigma(H_k^N).
\end{multline*}
The right-hand side converges to zero as $\max_j\|Z_{kj}^N\|\le (\log 
d_N)^{-2}a_N$ eventually a.s.\ and as $\bar R(H_k^N)\to 0$ and 
$\sigma(H_k^N)=O(1)$ by assumption. As $\|\EE[\tilde H_k^N]\|\to 0$,
$\|\EE[H_k^N]\|\to 0$, and $\|\EE[(H_k^N)^2]-\id\|\to 0$, we conclude
that $\|\EE[(\tilde H_k^N)^2]-\id\|\to 0$.
\end{proof}

\subsubsection{Proof of Corollary \ref{cor:sparsesaf}}
\label{sec:pfcorsparse}

Before we proceed to the proof of Corollary \ref{cor:sparsesaf}, we first 
state another elementary probabilistic lemma.

\begin{lem}
\label{lem:asmax}
Let $(Y_n)_{n\ge 1}$ be a sequence of i.i.d.\ random variables with
$\EE[|Y_n|^p]<\infty$ for some $p>0$. Then
$\lim_{n\to\infty} n^{-\frac{1}{p}}\max_{m\le n}|Y_m| = 0$ a.s.
\end{lem}

\begin{proof}
By the union bound and as $\sum_{k\ge 0} 2^k 1_{2^k\le x} \le 2x$,
we can estimate
$$
	\sum_{k\ge 0}
	\mathbf{P}\bigg[2^{-\frac{k}{p}}\max_{m\le 2^k}|Y_m|
	\ge\varepsilon\bigg]
	\le
	\sum_{k\ge 0}
	2^k\,
	\mathbf{P}\bigg[|Y_1|^p
	\ge 2^{k}\varepsilon^p
	\bigg]\le
	\frac{2\,\EE[|Y_1|^p]}{\varepsilon^p}<\infty
$$
for any $\varepsilon>0$. Thus
$$
	\lim_{k\to\infty}
	\max_{2^{k-1}\le n<2^k}
	n^{-\frac{1}{p}}\max_{m\le n}|Y_m| \le
	2^{\frac{1}{p}}
	\lim_{k\to\infty}
	2^{-\frac{k}{p}}\max_{m\le 2^k}|Y_m|=0\quad\mbox{a.s.}
$$
by the Borel-Cantelli lemma.
\end{proof}

We can now complete the proof of Corollary \ref{cor:sparsesaf}.

\begin{proof}[Proof of Corollary \ref{cor:sparsesaf}]
We prove both parts separately.

\medskip

\textit{Part a.} It suffices to verify that the assumptions of Theorem 
\ref{thm:saf} are satisfied. Let $\mathrm{G}_N=([d_N],E_N)$ be 
$k_N$-regular, and write
$$
	H_k^N = 
	\sum_{i<j:\{i,j\}\in E_N} \frac{\eta_{kij}}{\sqrt{k_N}}
	(e_ie_j^* + e_je_i^*).
$$
Then $\EE[H_k^N]=0$ and $\EE[(H_k^N)^2]=\id$ by construction.
Furthermore, we have
$$
	\lim_{N\to\infty} (\log d_N)^{\frac{3}{2}} v(H_k^N) =
	\sqrt{2}\lim_{N\to\infty} (\log d_N)^{\frac{3}{2}} 
	k_N^{-\frac{1}{2}} = 0
$$
by the assumption of part a. On the other hand, note that
$$
	\EE\bigg[
	\max_{1\le i\le M_N}\|Z_{ki}^N\|^2
	\bigg]
	\le
	\frac{1}{k_N}
	\EE\bigg[
	\max_{i<j:\{i,j\}\in E_N} |\eta_{kij}|^p
	\bigg]^{\frac{2}{p}}
	\le
	\frac{(k_Nd_N)^{\frac{2}{p}}}{k_N}
	\EE[|\eta_{kij}|^p]^{\frac{2}{p}}.
$$
As the assumption of part a.\ implies
$(k_Nd_N)^{\frac{2}{p}} k_N^{-1} \ll (\log d_N)^{-4}$,
we have shown that $(\log d_N)^2\bar R(H_k^N)\to 0$ as $N\to\infty$.
This simultaneously verifies both remaining assumptions of the in 
probability version of Theorem \ref{thm:saf}.

If in addition $E_N$ is increasing, we can use Lemma \ref{lem:asmax} to 
obtain
$$
	\limsup_{N\to\infty}\,
	(\log d_N)^2 \max_{1\le i\le M_N}\|Z_{ki}^N\|
	\lesssim
	\lim_{N\to\infty}
	(k_Nd_N)^{-\frac{1}{p}}
	\max_{i<j:\{i,j\}\in E_N} |\eta_{kij}| =0\quad\mbox{a.s.},
$$
where we used that
$(\log d_N)^2k_N^{-\frac{1}{2}} \lesssim (k_Nd_N)^{-\frac{1}{p}}$
by the assumption of part a. The remaining conclusion of part a.\ then 
follows from the a.s.\ version of Theorem \ref{thm:saf}.

\medskip

\textit{Part b.} Fix any $p>2$. Then we may choose a distribution of the 
entries $\eta_{kij}$ such that
$\EE[\eta_{kij}]=0$,
$\mathrm{Var}(\eta_{kij})=1$, and
$\mathbf{P}[|\eta_{kij}|>x]=(x\log x)^{-p}$ for all 
$x\ge x_0$ (here $x_0>0$ is a sufficiently large constant).
As $\EE[|\eta_{kij}|^p]<\infty$, the assumptions of 
Corollary \ref{cor:sparsesaf} are satisfied. 
Now note that as $\|M\|\ge
\max_{i,j}|M_{ij}|$, we have
$$
	\mathbf{P}\big[\|H_1^N\|>k_N^{-\frac{1}{2}}x
	\big] \ge 
	\mathbf{P}\bigg[\max_{i<j:\{i,j\}\in E_N}
        |(H_1^N)_{ij}|>k_N^{-\frac{1}{2}}x
	\bigg]
	=
	1-\big(1-(x\log x)^{-p}\big)^{\frac{k_Nd_N}{2}}
$$
for $x\ge x_0$. Choosing $x=(k_Nd_N)^{\frac{1}{p}} (\log d_N)^{-1}$
yields $(x\log x)^{-p} \ge (k_Nd_N)^{-1}$ for all sufficiently large $N$, 
where we used $k_N\le d_N$. We have therefore shown that
$$
	\mathbf{P}\big[\|H_1^N\|>k_N^{\frac{1}{p}-\frac{1}{2}}
	d_N^{\frac{1}{p}} (\log d_N)^{-1}
	\big] \ge 
	1-e^{-\frac{1}{2}}
$$
for all large $N$. But the assumption of part b.\ implies that 
$k_N^{\frac{1}{p}-\frac{1}{2}}d_N^{\frac{1}{p}} (\log 
d_N)^{-1}\to\infty$. This stands in contradiction to the conclusion of part 
a., which would imply in particular that $\|H_1^N\|\to \|s_1\|=2$ in 
probability.
\end{proof}

\subsection{Phase transitions in spiked models}
\label{sec:pfspiked}

Here we prove Theorem \ref{thm:spiked}. We first 
prove convergence of the outlier eigenvalues (as in part a.\ of Theorem 
\ref{thm:bbp}), and then consider the eigenvectors (as in part b.\ of 
Theorem \ref{thm:bbp}).

In the following, it will be convenient to introduce the notation
$$
	\mathrm{B}(\theta) := \begin{cases}
	\theta+\frac{1}{\theta} & \mbox{for }\theta>1,\\
	2 & \mbox{for }\theta\le 1,
	\end{cases}
$$
and to write 
$\theta_1\ge\cdots\ge\theta_s>1\ge\theta_{s+1}\ge\cdots\ge\theta_r>0=:\theta_{r+1}$.
The assumptions of Theorem \ref{thm:spiked} will be assumed to 
hold without further comment.

\subsubsection{Outlier eigenvalues}

We begin by a direct application of universality.

\begin{lem}
\label{lem:bbphaus}
$\dH(\spc(A_d+H_d),\spc(A_d+G_d)) \to 0$ as $d\to\infty$ a.s.
\end{lem}

\begin{proof}
We readily compute $\sigma(A_d+H_d)=O(1)$ and
$\sigma_*(A_d+H_d)=O(d^{-\frac{1}{2}})$, while by assumption 
$R(A_d+H_d)=o((\log d)^{-2})$.  Theorem \ref{thm:specuniv}
with $t=3\log d$ yields
$$
	\mathbf{P}[\dH(\spc(A_d+H_d),\spc(A_d+G_d)) >
	o(1)] \le \frac{1}{d^2},
$$
and the conclusion follows by the Borel-Cantelli lemma.
\end{proof}

When combined with Theorem \ref{thm:bbp}, Lemma \ref{lem:bbphaus} suffices 
to detect the presence and locations of any outlier eigenvalues. However, 
Hausdorff convergence is not sufficiently strong to establish convergence 
of individual eigenvalues. In the present setting, this stronger conclusion 
can however be achieved by combining universality with the min-max 
principle. To this end we will use the following lemma.

\begin{lem}
\label{lem:bbpproj}
Fix any orthonormal eigenvectors $v_{d,1},\ldots,v_{d,r}$ of $A_d$ 
with eigenvalues $\theta_1,\ldots,\theta_r$, respectively, and let
$Q_{d,i}$ be the projection onto 
$\{v_{d,1},\ldots,v_{d,i}\}^\perp$.
Then $\lambda_1(Q_{d,i}(A_d+H_d)Q_{d,i}^*)\to \mathrm{B}(\theta_{i+1})$
as $d\to\infty$ a.s.\ for $i=1,\ldots,r$.
\end{lem}

\begin{proof}
The identical argument as in Lemma \ref{lem:bbphaus} yields that
$$
	|\lambda_1(Q_{d,i}(A_d+H_d)Q_{d,i}^*)-\lambda_1(Q_{d,i}(A_d+G_d)Q_{d,i}^*)|
	\xrightarrow{d\to\infty} 0\quad\mbox{a.s.}
$$
But $Q_{d,i}G_d Q_{d,i}^*$ is GOE of dimension
$d-i$ (scaled by a factor $(\frac{d-i}{d})^{\frac{1}{2}}=
1+o(1)$). Thus applying Theorem \ref{thm:bbp}
to $Q_{d,i}(A_d+G_d)Q_{d,i}^*$ yields the conclusion.
\end{proof}

We can now deduce the convergence of the eigenvalues.

\begin{cor}
\label{cor:bbpeig}
$\lambda_i(A_d+H_d)\to \mathrm{B}(\theta_i)$ as $d\to\infty$ a.s.\ for
$1\le i\le r+1$.
\end{cor}

\begin{proof}
Note first that $\lambda_i(A_d+H_d) \le 
\lambda_1(Q_{d,i-1}(A_d+H_d)Q_{d,i-1}^*)$ by the min-max 
principle. Thus $\limsup_{d\to\infty} \lambda_i(A_d+H_d) \le 
\mathrm{B}(\theta_i)$ a.s.\ by Lemma \ref{lem:bbpproj}.

Next, note that the empirical spectral distribution of $H_d$ converges 
a.s.\ to the standard semicircle distribution by Theorem \ref{thm:saf}. 
This implies in particular that $\liminf_{d\to\infty} \lambda_i(A_d+H_d) 
\ge \liminf_{d\to\infty} \lambda_i(H_d)\ge 2$ a.s.\ for all $1\le i\le 
r+1$. Thus we obtain
$\lambda_i(A_d+H_d)\to \mathrm{B}(\theta_i)=2$ as $d\to\infty$ a.s.\ for 
$s+1\le i\le r+1$.

On the other hand, by Lemma \ref{lem:bbphaus} and Theorem 
\ref{thm:bbp}, each $\mathrm{B}(\theta_i)$ with $1\le i\le s$ is
a limit point of the spectrum of $A_d+H_d$ as $d\to\infty$. If 
$\theta_1>\theta_2>\cdots>\theta_s$ are all distinct, this immediately 
implies $\liminf_{d\to\infty} \lambda_i(A_d+H_d) \ge \mathrm{B}(\theta_i)$ 
a.s., concluding the proof. If $\theta_1,\ldots,\theta_s$ are not distinct, 
we can choose $0\le A_d'\le A_d$ of rank $r$ with distinct nonzero 
eigenvalues so that $\lambda_i(A_d')\ge\theta_i-\varepsilon$ for 
all $1\le i\le r$. Then
$$
	\liminf_{d\to\infty}\lambda_i(A_d+H_d) \ge
	\liminf_{d\to\infty}\lambda_i(A_d'+H_d) \ge
	\mathrm{B}(\theta_i-\varepsilon)
	\quad\mbox{a.s.},
$$
and the conclusion follows as $\varepsilon>0$ is arbitrary.
\end{proof}

\subsubsection{Outlier eigenvectors}

A universality statement for eigenvectors of $A_d+H_d$ can be obtained 
using Theorem \ref{thm:smuniv}. This yields the following conclusion.

\begin{lem}
\label{lem:bbpvecuniv}
Let $\varepsilon>0$ be sufficiently small that 
$I_i:=\theta_i+\frac{1}{\theta_i}+[-\varepsilon,\varepsilon]$ are disjoint
for distinct values of $\theta_i$. Then for all $1\le i\le s$ and
nonrandom $v_d\in\mathbb{C}^d$, $\|v_d\|=1$
$$
	|\langle v_d,1_{I_i}(A_d+H_d)v_d\rangle-
	\langle v_d,1_{I_i}(A_d+G_d)v_d\rangle|\xrightarrow{d\to\infty}0
	\quad\mbox{a.s.}
$$
\end{lem}

\begin{proof}
Let $\varphi_i:\mathbb{R}\to[0,1]$ be a smooth function so that
$\varphi_i(x)=0$ for $x\not\in I_i$ and $\varphi_i(x)=1$ for
$x\in \theta_i+\frac{1}{\theta_i}+[-\frac{\varepsilon}{2},\frac{\varepsilon}{2}]$.
Then the second part of Theorem \ref{thm:smuniv} yields
$$
	|\EE[\langle v_d,\varphi_i(A_d+H_d)v_d\rangle] -
	\EE[\langle v_d,\varphi_i(A_d+G_d)v_d\rangle]|\xrightarrow{d\to\infty} 0.
$$
On the other hand, applying Proposition \ref{prop:specconc} with 
$x=2C^{-1}\log d$ yields
$$
	\mathbf{P}\big[\big|\langle v_d,\varphi_i(A_d+H_d)v_d\rangle-
	\EE[\langle v_d,\varphi_i(A_d+H_d)v_d\rangle]\big|
	\ge o(1)
	\big] \le \frac{4}{d^2}, 
$$
and analogously for $A_d+G_d$. Thus the Borel-Cantelli lemma yields
$$
	|\langle v_d,\varphi_i(A_d+H_d)v_d\rangle -
	\langle v_d,\varphi_i(A_d+G_d)v_d\rangle|\xrightarrow{d\to\infty} 0
	\quad\mbox{a.s.}
$$
But
$\varphi_i(A_d+H_d)=1_{I_i}(A_d+H_d)$ and
$\varphi_i(A_d+G_d)=1_{I_i}(A_d+G_d)$ eventually as $d\to\infty$ a.s.\ by 
Corollary \ref{cor:bbpeig} and by
part a.\ of Theorem \ref{thm:bbp}, respectively.
\end{proof}

The desired properties of the eigenvectors now follow directly in the case 
that $\theta_1>\cdots>\theta_s$ are all distinct. The main difficulty is to 
remove the latter requirement by means of a suitable perturbation argument.

\begin{cor}
\label{cor:bbpvec}
For any $1\le i\le s$ and $1\le j\le r$, we have
$$
	\|P_j(A_d)v_i(A_d+H_d)\|^2 \xrightarrow[\text{a.s.}]{d\to\infty}
	\bigg(
	1-\frac{1}{\theta_i^2}\bigg)1_{\theta_j=\theta_i}.
$$
\end{cor}

\begin{proof}
If $\theta_1>\cdots>\theta_s$ are all distinct, then 
$1_{I_i}(A_d+H_d)=v_i(A_d+H_d)v_i(A_d+H_d)^*$ and 
$1_{I_i}(A_d+G_d)=v_i(A_d+G_d)v_i(A_d+G_d)^*$ eventually as $d\to\infty$ 
a.s.\ for $1\le i\le s$ by Corollary \ref{cor:bbpeig} and part a.\ of 
Theorem \ref{thm:bbp}. The conclusion then follows readily by applying 
Lemma \ref{lem:bbpvecuniv} and part b.\ of Theorem \ref{thm:bbp}.

In the general case, let $A_d'=\sum_{m=1}^r \theta_m' \,v_m(A_d)v_m(A_d)^*$ 
be a perturbation of $A_d$ (for a choice of orthonormal eigenvectors 
$v_m(A_d)$) with distinct 
$\theta_1'>\cdots>\theta_r'$ and such that 
$|\theta_m-\theta_m'|\le\varepsilon$ for all $m$.
Let $J_i:=\{k:\theta_k=\theta_i\}$, and let $Q_i(A_d'+H_d)$ be the 
projection on the linear span of $\{v_k(A_d'+H_d):k\in J_i\}$. We claim 
that
$$
	\limsup_{d\to\infty}
	\bigg|
	\|P_j(A_d)\tilde v_{d,i}\|^2 -
	\bigg(
	1-\frac{1}{\theta_i^2}\bigg)1_{\theta_j=\theta_i}
	\bigg|\lesssim\varepsilon\quad\mbox{a.s.}
$$
for any (possibly random) choice of unit vector $\tilde 
v_{d,i}\in\mathop{\mathrm{ran}}(Q_i(A_d'+H_d))$. Indeed, writing
$\tilde v_{d,i}=\sum_{k\in J_i} c_{d,k} v_k(A_d'+H_d)$ with
$\sum_{k\in J_i}c_{d,k}^2=1$, we can compute
\begin{align*}
	\|P_j(A_d)\tilde v_{d,i}\|^2 &=
	\sum_{m\in J_j}
	\sum_{k,l\in J_i}
	c_{d,k} c_{d,l} 
	\langle v_k(A_d'+H_d),v_m(A_d)\rangle
	\langle v_m(A_d),
	v_l(A_d'+H_d)\rangle \\
	&=
	\sum_{k\in J_i}
	c_{d,k}^2 
	\bigg(1-\frac{1}{(\theta_k')^2}\bigg)
	1_{\theta_j=\theta_i}+o(1)\quad\mbox{a.s.}
\end{align*}
as $d\to\infty$,
where we used that 
$|\langle v_m(A_d),v_l(A_d'+H_d)\rangle|^2\to 
(1-\frac{1}{(\theta_l')^2})1_{m=l}$ a.s.\ 
as $A_d'$ has distinct eigenvalues. The claim follows as
$\theta\mapsto 1-\frac{1}{\theta^2}$ is Lipschitz on
$[1,\infty)$.

For any projection matrices $P,Q\in\M_d(\mathbb{C})_{\rm sa}$ and
$c\in\mathbb{R}$, we can write
$$
	\sup_{v\in\mathop{\mathrm{ran}}Q,\|v\|=1}
	\big|\|Pv\|^2-c\,\big| =
	\|Q(P-c\id)Q\|.
$$
Using this identity, the above claim may be rewritten as
$$
	\limsup_{d\to\infty}
	\big\|Q_i(A_d'+H_d)
	\big(P_j(A_d)-(\textstyle{1-\frac{1}{\theta_i^2}})1_{\theta_j=\theta_i}
	\id\big)
	Q_i(A_d'+H_d)\big\|
	\lesssim \varepsilon\quad\mbox{a.s.}
$$
On the other hand, when $\varepsilon$ is sufficiently small, Corollary 
\ref{cor:bbpeig} ensures that the eigenvalues 
$\{\lambda_k(A_d'+H_d):k\in J_i\}$ are separated from the rest of the
spectrum of $A_d'+H_d$ by a positive gap as $d\to\infty$.
A routine 
application of the Davis-Kahan theorem \cite[Theorem VII.3.1 and
Exercise VII.1.11]{Bha97} yields
$$
	\|Q_i(A_d'+H_d)-Q_i(A_d+H_d)\|\lesssim\|A_d'-A_d\|\le\varepsilon
$$
eventually as $d\to\infty$ a.s. As we may choose $\varepsilon>0$ 
arbitrarily small, it follows that
$$
	\big\|Q_i(A_d+H_d)
	\big(P_j(A_d)-(\textstyle{1-\frac{1}{\theta_i^2}})1_{\theta_j=\theta_i}
	\id\big)
	Q_i(A_d+H_d)\big\| \xrightarrow{d\to\infty} 0\quad\mbox{a.s.}
$$
The conclusion follows as 
$v_i(A_d+H_d)\in\mathop{\mathrm{ran}}(Q_i(A_d+H_d))$.
\end{proof}

Combining Corollaries \ref{cor:bbpeig} and \ref{cor:bbpvec} concludes the
proof of Theorem \ref{thm:spiked}.

\SkipTocEntry\subsection*{Acknowledgments}

This work was supported in part by the NSF grants DMS-1856221, 
DMS-2054565, and DMS-2347954. The authors thank Noga Alon, Afonso 
Bandeira, March Boedihardjo, Ioana Dumitriu, Mark Rudelson, Sasha Sodin, 
Joel Tropp, Pierre Youssef, and Yizhe Zhu for helpful discussions on the 
topic of this paper, and the anonymous referees for very helpful comments 
and suggestions.

\bibliographystyle{abbrv}
\bibliography{ref}

\end{document}